\documentclass{memo-l}
\usepackage[french,english]{babel}

\usepackage{amsthm,amsmath,amssymb}

\usepackage{mathtools} 
\usepackage{caption}
\usepackage{textcomp}

\newtheorem{theorem}{Theorem}[chapter]
\newtheorem*{theorem*}{Theorem}
\newtheorem{lemma}[theorem]{Lemma}
\newtheorem{proposition}[theorem]{Proposition}
\newtheorem{corollary}[theorem]{Corollary}

\theoremstyle{definition}
\newtheorem{definition}[theorem]{Definition}
\newtheorem{example}[theorem]{Example}

\theoremstyle{remark}
\newtheorem{remark}[theorem]{Remark}
\newtheorem{problem}{Problem}
\numberwithin{section}{chapter}
\numberwithin{equation}{chapter}

\DeclareMathOperator{\e}{e}
   
\DeclareMathOperator{\LandauO}{\mathcal{O}}            

\DeclareMathOperator{\prob}{\mathbf{Pr}}
\DeclareMathOperator{\conj}{\mathcal{C}}
\DeclareMathOperator{\dist}{\mathrm{dist}}

\newcommand{\tO}{\mathtt 0}    
\newcommand{\tL}{\mathtt 1}    

\newcommand{\N}{\mathbb{N}}
\newcommand{\Z}{\mathbb{Z}}
\newcommand{\R}{\mathbb{R}}
\newcommand{\C}{\mathbb{C}}
\newcommand{\T}{\mathbb{T}}

\newcommand{\sz}{\mathsf z}

\newcommand{\eqdef}{\coloneqq}

\def\bl{\underline{\hphantom{A}}\hspace{1pt}}	

\newcommand{\digit}{\delta}

\newcommand{\ijkl}{i}

\newcommand{\1}{0}
\newcommand{\2}{1}

\newcommand{\3}{4}

\newcommand{\ta}{\mathtt a}           
\newcommand{\tb}{\mathtt b}           
\newcommand{\tc}{\mathtt c}           
\newcommand{\td}{\mathtt d}           

\newcommand{\smallspace}{\hspace{0.5pt}} 

\newcommand{\golden}{\gamma}

\DeclareMathOperator{\ind}{\chi}

\newcommand{\rb}[1]{\left( #1 \right)}

\newcommand{\abs}[1]{\left| #1 \right|}
\newcommand{\norm}[1]{\left\| #1 \right\|}

\newcommand{\bfx}{\mathbf{x}}

\newcommand{\dims}{s}
\newcommand{\gap}{r}

\usepackage{tikz}
\usetikzlibrary{matrix,arrows}
\makeindex

\begin{document}

\frontmatter

\title{Primes as Sums of Fibonacci Numbers}


\author{Michael Drmota}
\address{Institute of Discrete Mathematics and Geometry, TU Wien,\newline Wiedner Hauptstr. 8--10, A-1040 Wien, Austria}
\curraddr{Institute of Discrete Mathematics and Geometry, TU Wien,\newline Wiedner Hauptstr. 8--10, A-1040 Wien, Austria}
\email{michael.drmota@tuwien.ac.at}

\author{Clemens M\"ullner}
\address{Institute of Discrete Mathematics and Geometry, TU Wien,\newline Wiedner Hauptstr. 8--10, A-1040 Wien, Austria}
\curraddr{Institute of Discrete Mathematics and Geometry, TU Wien,\newline Wiedner Hauptstr. 8--10, A-1040 Wien, Austria}
\email{clemens.muellner@tuwien.ac.at}

\author{Lukas Spiegelhofer}
\address{Department Mathematics and Information Technology, Montan\-universit\"at Leoben, Franz-Josef-Strasse 18, 8700 Leoben, Austria}
\curraddr{Department Mathematics and Information Technology, Montanuniversit\"at Leoben, Franz-Josef-Stra\ss e 18, 8700 Leoben, Austria}
\email{lukas.spiegelhofer@unileoben.ac.at}
\thanks{The authors were supported by the FWF (Austrian Science Fund), project F5502-N26, which is a part of the Special Research Program ``Quasi Monte Carlo methods: Theory and Applications'', and by the project ArithRand, which is a joint project between the ANR (Agence Nationale de la Recherche) and the FWF,
grant numbers ANR-20-CE91-0006 and I4945-N.
Moreover, we acknowledge support by the project MuDeRa, 
which is a joint project between the ANR and the FWF, grant numbers ANR-14-CE34-0009 and I-1751.}

\date{}




\dedicatory{Dedicated to Christian Mauduit who passed away too early.}

\begin{abstract}
The purpose of this paper is to discuss the relationship between prime numbers
and sums of Fibonacci numbers.
One of our main results says that for every sufficiently large integer $k$ there exists a prime number that can be represented as the sum of $k$ different and non-consecutive Fibonacci numbers.
This property is closely related to, and based on, a \emph{prime number theorem} for certain morphic sequences.
In our case, these morphic sequences are based on the Zeckendorf expansion of a positive integer $n$ --- we write $n$ as the sum of non-consecutive Fibonacci numbers.
More precisely, we are concerned with the \emph{Zeckendorf sum-of-digits function} $\sz$, which returns the minimal number of Fibonacci numbers needed to write a positive integer as their sum.
The proof of such a prime number theorem for $\sz$, combined with a corresponding \emph{local result}, constitutes the central contribution of this paper, from which the result stated in the beginning follows.

Problems of this type have been discussed intensively in the context of the base-$q$ expansion of integers.
The driving forces of this development were the \emph{Gelfond problems} (1968/1969), more specifically the behavior of the sum-of-digits function in base $q$ along the sequence of primes and along integer-valued polynomials,
and the \emph{Sarnak conjecture}.
Mauduit and Rivat resolved the question on the sum of digits of prime numbers~(2010) and the sum of digits of squares (2009).
Later the second author (2017) proved Sarnak's conjecture for the class of \emph{automatic sequences}, which are based on the $q$-ary expansion of integers, and which generalize the sum-of-digits function in base $q$ considerably.

For the (partial) solution of the Gelfond problems (1967/1968), Mauduit and Rivat have developed a powerful method that 
is based on techniques for ``cutting off digits'',
on sophisticated estimates for Fourier terms, and on estimates for exponential sums.
These techniques --- together with a new decomposition of finite automata --- 
were also the basis for the second author's result on automatic sequences.

In order to obtain corresponding results for Fibonacci numbers,
we have to extend Mauduit and Rivat's method considerably.
In fact, we are departing significantly from this method, proving the statement that $\exp(2\pi i \vartheta\smallspace\sz(n))$ has \emph{level of distribution} $1$.
This latter result forms an essential part of our treatment of the occurring sums of type $\textrm I$ and $\textrm{II}$ and uses \emph{Gowers norms} related to the Zeckendorf sum-of-digits function as a central technical tool.
Gowers norms are a higher order generalization of the above-mentioned Fourier terms, and their appearance in our method is intimately tied to the iterated application of a new generalization of van der Corput's inequality. 
\end{abstract}

\maketitle

\tableofcontents

\chapter*{Preface}

The story begins with the base-$q$ expansion:
we can represent each nonnegative integer $n$ by a sum of powers of $q$ in such a way that each power is taken at most $q-1$ times.
This representation is unique up to the order of the summands.
Written as a linear combination, we have the unique expansion
\[n=\sum_{0\leq j\leq \nu}a_jq^j,\]
where $a_j\in\{0,\ldots,q-1\}$.

The behavior of the base-$q$ expansion under arithmetical operations is not fully understood.
A simple question concerns addition of a constant in the binary case: in which way does the binary expansion of $a+b$ depend on the binary expansions of $a$ and $b$?
The question is easy to formulate, but no closed precise description of the behavior of base-$q$ digits under addition exists~\cite{SW2020}.
It is already a challange to deal with a certain parameter associated to the base-$q$ expansion --- the \emph{sum-of-digits function in base} $q$, in symbols $s_q$.
This function just returns the sum of the base-$q$ digits of its argument; in other words, $s_q(n)$ is the minimal number of powers of $q$ needed to write $n$ as their sum.
The results of this paper are concerned with a parameter of this type --- the \emph{Zeckendorf sum of digits} $\sz(n)$ of $n$, which is the minimal number of Fibonacci numbers needed to represent a given nonnegative integer $n$ as their sum.
Nevertheless, in the proofs we will make use of the full \emph{Zeckendorf expansion} of integers, which carries more information than just the Zeckendorf sum of digits we are interested in.

Extending the above-mentioned question on addition of a constant $d$ by repeatedly adding $d$, we are led to arithmetic progressions.
A.~O.~Gelfond~\cite{Gelfond1967_1968} proved that the sum of digits in base $q$ along $a+d\mathbb N$ is uniformly distributed in certain arithmetic progressions.
\begin{theorem*}[Gelfond]
Suppose that $q,m,b,d,a$ are integers and $q,m,d\geq 2$.
Let $\gcd(m,q-1)=1$.
Then
\begin{equation}\label{eqn_gelfond_AP}
\bigl\lvert\bigl\{1\leq n\leq x:n\equiv a\bmod d,\,s_q(n)\equiv b\bmod m\bigr\}\bigr\rvert=\frac x{dm}+O\bigl(x^\lambda\bigr)
\end{equation}
for some $\lambda<1$ only depending on $q$ and $m$.
\end{theorem*}

The paper~\cite{Gelfond1967_1968} is the source of the so called ``Gelfond problems (1967/68)'',
namely the following three research questions:
\begin{enumerate}
\item \flqq{} Il serait int\'eressant de prouver que \frqq{} --- it would be interesting to prove that ---
for coprime bases $q_1,q_2\geq 2$, and integers $m_1,m_2$ such that $\gcd(m_1,q_1-1)=\gcd(m_2,q_2-1)=1$, the following holds.
There exists some $\delta<1$ such that the number $\psi(x)$ of integers $n\leq x$ satisfying
\[s_{q_1}(n)\equiv \ell_1\bmod m_1\quad\mbox{and}\quad s_{q_2}(n)\equiv \ell_2\bmod m_2\]
is given by
\begin{equation}\label{eqn_gelfond_1}
\psi(x)=\frac{x}{m_1m_2}+O\bigl(x^\delta\bigr).
\end{equation}
\item \flqq{} Il serait aussi int\'eressant \frqq{} --- it would also be interesting --- to find the number of prime numbers $p\leq x$ such that $s_q(p)\equiv \ell \bmod m$.
\item Estimate the number of $n$ such that $s_q(P(n))\equiv\ell\bmod m$, where $P$ is a polynomial taking only nonnegative integer values on $\mathbb N$.
\end{enumerate}

The part that is still open concerns the third problem, where $P$ is a polynomial of degree $\geq 3$.
Although it is known that the Thue--Morse sequence along $n^3$ attains each of its two values infinitely often
\cite{DartygeTenenbaum2006,Moshe2007,Stoll2012}, 
the equidistribution question is completely open. 

Let us describe what is known about these problems.
B\'esineau~\cite{B1972} proved a non-quantitative version of~\eqref{eqn_gelfond_1}, using \emph{pseudorandom properties} of the sum-of-digits function.
D.-H.~Kim~\cite{K1999}, a considerable amount of time later, resolved the precise statement of the first Gelfond problem.

Another ten years later (according to the publication dates), Mauduit and Rivat~\cite{MR2009} published their first major paper on the Gelfond problems. In that paper, the distribution of the sum of digits of $n^2$ in residue classes could be handled, which resolves part of the third Gelfond problem.
Sure enough, their method is sufficient to handle all integer polynomials $P$ of degree two such that $P(\mathbb N)\subseteq \mathbb N$, such as $P(n)=\binom n2$, for example.
This extension has, however, not been treated in the literature so far.
The case of higher degree polynomials has steadily resisted different attempts of proof.
It seems that (in addition to Mauduit and Rivat's work)  new ideas will be needed.

The year after, their second paper on the topic~\cite{MR2010} was published, settling the second Gelfond problem.
That latter paper is the basis for our research presented in the present work.

The first author~\cite[Theorem~4]{D2001} sharpened the first Gelfond problem in that he proved a local result for the joint distribution of sum-of-digits functions in residue classes. 
Combining their efforts, the first author, Mauduit, and Rivat~\cite{DMR2009} could handle a local result on the sum of digits of primes.
For each base $q\geq 2$, there exists an (absolute, effective) constant $k_0$ such that for each $k\geq k_0$ that is 
coprime to $q-1$, there exists a prime number $p$ satisfying 
\begin{equation}\label{eqn_DMR_local}
s_q(p)=k.
\end{equation}
In the present paper, we prove an analogous theorem for the Zeckendorf sum of digits (Theorem~\ref{Th1}), which is our showcase result.

Another result of the present paper is a counterpart to Gelfond's second problem, namely an asymptotic result on
the prime numbers $p\leq x$ such that $z(p)\equiv \ell \bmod m$ (Theorem~\ref{Th4}). Such results are also
called \emph{prime number theorems}.

Besides the Gelfond problems there is a second background problem that has strong links to the present paper, 
namely the {\it Sarnak conjecture} \cite{Sarnak2011}. This conjecture features the \emph{M\"obius Randomness Principle} (MRP), 
which says that any \emph{reasonable} (and bounded) sequence $f$ should satisfy
\[
\sum_{n\leq x}\mu(n)f(n)=o(x).
\]
Sarnak made the informal notion of a \emph{reasonable} sequence precise by stating that
every \emph{deterministic sequence} $f(n)$ should satisfy the MRP. (A sequence is deterministic if it can be written as  $f(n) = F(T^nx_0)$, where $(X,T)$ is a compact, zero topological entropy dynamical system and
$F \in C(X)$.)
This conjecture has received a lot of attention during the last years and could be proved for several instances~\cite{Bourgain2013a, Bourgain2013, BSZ2013, Davenport1937, Drmota2014, DDM2015, DK2015, FKLM2016, Ferenczi2018, GT2012, Green2012,Hanna2017, EKL2016, HKLD2014, ELD2014, HLD2015, Katai2001, Karagulyan2015, Katai1986, K2020,  KL2015, LS2015, MR2010, Mauduit2015,Muellner2017, Peckner2015, SU2015, Veech2016, Wang2017}.
We also note the interesting surveys by Ferenczi, Ku{\l}aga-Przymus and Lema{\'n}czyk \cite{FKL2018} and by Ku{\l}aga-Przymus and Lema{\'n}czyk~\cite{KulagaPrzymus2020}.

The MRP is usually easier to obtain than the corresponding prime number theorem, where $\mu$ has to be replaced by the von Mangoldt function $\Lambda$.
In fact the sum-of-digits case was handled by Dartyge and Tenenbaum~\cite{DT2005}, preceding Mauduit and Rivat's work.

An important class of deterministic sequences 
is given by \emph{automatic sequences}.
They are therefore expected to satisfy the MRP, by Sarnak's conjecture, which was proved in the paper~\cite{Muellner2017} by the second author.
One of the most prominent automatic sequences is the Thue--Morse sequence $t(n) = s_2(n) \bmod 2$. 

Meanwhile, the obvious generalization --- the MRP for \emph{morphic sequences} --- is wide open, and it appears that significant new ideas are needed in order to handle this case.

A very special case was proved by the authors~\cite{DMS2018}.
We could prove the MRP for the Zeckendorf sums of digits function modulo $2$ (which is a morphic sequence), generalizing a method devised by Kropf and Wagner~\cite{KW2016}.
This result is one of the first cases where Sarnak's conjecture was verified for a morphic sequence (apart from automatic sequences, Sturmian words, or substitutions with long repetitions~\cite{Ferenczi2018}).
As is to be expected, this theorem is a lot easier than
the corresponding prime number theorem, which we prove in the present paper.
The method employed in~\cite{DMS2018} is in fact not sufficient for our needs, and we had to take a different path.


\mainmatter
\chapter{Introduction}
The basic object in this paper is the sequence of \emph{Fibonacci numbers}, defined by
\begin{equation*}
F_0=0,\quad F_1=1,\quad\text{and}\quad F_{\ijkl+2}=F_{\ijkl+1}+F_{\ijkl}\quad\text{for }\ijkl\geq 0.
\end{equation*}
By Zeckendorf's theorem~\cite{Z1972}, every nonnegative integer $n$ can be represented uniquely as a sum 
\begin{align}\label{eqn_Zeckendorf_rep}
n=\sum_{\ijkl=2}^L\digit_{\ijkl}(n) F_{\ijkl}
\end{align}
such that $\digit_{\ijkl}(n) \in\{0,1\}$, $\digit_L(n)=1$,
and such that $\digit_{\ijkl+1}(n)=1$ implies $\digit_{\ijkl}(n)=0$ for all $\ijkl\in\{2,\ldots,L-1\}$.

At this point, we note that Lekkerkerker~\cite{Lekkerkerker} published
a proof of Zeckendorf's theorem well before Zeckendorf.
However, Zeckendorf indicated~\cite{Kimberling1998} that he knew the proof as early as 1939.
Even before that, Kempner~\cite{Kempner1936} described the \emph{greedy algorithm} for the closely related $\beta$-numeration systems.
An analogous algorithm (successively subtracting the largest possible Fibonacci number) outputs the unique expansion~\eqref{eqn_Zeckendorf_rep}.

By Binet's formula
\begin{equation*}
F_\ijkl = \frac {\golden^{\ijkl} -(-1/\golden)^{\ijkl}}{\sqrt5},
\end{equation*}
the length of this expansion clearly satisfies
$L  = L(n) = \log n/\log\golden + O(1)$.
Here
\[    
\golden=\frac{\sqrt5+1}2    
\] 
denotes the golden ratio, which is the larger root of the polynomial $x^2-x-1$. In particular, $\golden^2 = 1 + \golden$. 

Due to the uniqueness of the expansion~\eqref{eqn_Zeckendorf_rep}, we may write $\digit_{\ijkl}(n)$ for the
$\ijkl$-th coefficient of $n$ in the Zeckendorf expansion;
occasionally, we omit the argument $n$ if there is no risk of confusion.
Up to a shift of indices, this is a special case of the \emph{Ostrowski expansion} of a nonnegative integer, which is based on the continued fraction expansion of a real number $\alpha$.
The sequence of Fibonacci numbers arises in the continued fraction expansion of $\alpha=\golden$,
being the sequence of denominators of the convergents of $\golden$.
We define the \emph{Zeckendorf sum-of-digits} of $n$, in symbols $\sz(n)$, as the number of nonzero terms in the Zeckendorf expansion of $n$.
This is the minimal number of Fibonacci numbers needed to represent $n$ as their sum.
Note however that there might exist other minimal representations as sums of Fibonacci numbers too, such as $4=1+3=2+2$.
Minimality can be proved using the observation that the Zeckendorf expansion of $n$ is the \emph{lexicographically largest} representation of $n$ as the sum of Fibonacci numbers --- it can be found by the \emph{greedy algorithm}.
The function $\sz$ is uniquely determined by the equation
\[\sz(\digit_2F_2+\digit_3F_3+\cdots)=\digit_2+\digit_3+\cdots\]
for coefficients $\digit_{\ijkl}\in\{0,1\}$, $\ijkl\geq 2$, such that $\digit_{\ijkl+1}=1$ implies $\digit_{\ijkl}=0$.

The Zeckendorf expansion is tied intimately to the distribution of $n\golden\bmod 1$.
In fact, we have
\begin{equation}\label{eqn_central_motivation}
\begin{array}{l@{\hspace{1cm}}l}
\bigl(\digit_2(n),\ldots,\digit_L(n)\bigr)=(\nu_2,\ldots,\nu_L)
&\mbox{ if and only if}\\[2mm]
n\golden\mbox{ lies in a certain interval modulo }1.
\end{array}
\end{equation}
An analogous characterization holds for all Ostrowski expansions (see for example~\cite{RS2011,B2001}). 
Relation~\eqref{eqn_central_motivation} is stated in detail in Lemma~\ref{Lefirstdigits};
this strong connection to irrational rotations on the circle is very helpful for studying the Zeckendorf expansion of integers.

A different point of view is given by \emph{morphic words}~\cite{AS2003},
which are obtained by a fixed point of a general substitution
(over a finite alphabet), followed by a coding.
For example, the sequence $\sz(n)\bmod 2$ (sometimes called Fibonacci--Thue--Morse sequence) is 
given by the following substitution $\sigma$ together with the coding $\pi$ (see~\cite{Bruyere}; we exchanged the roles of $\tb$ and $\td$):
\begin{align}\label{eq_morphic}
\sigma:\left\{\begin{array}{lll}
\ta&\mapsto&\ta\td\\
\tb&\mapsto&\ta\\
\tc&\mapsto& \tc\tb\\
\td&\mapsto&\tc
\end{array}\right\},
\qquad
\pi:\left\{\begin{array}{lll}\ta&\mapsto&\tO\\\tb&\mapsto&\tO\\\tc&\mapsto&\tL\\\td&\mapsto&\tL\end{array}\right\},
\end{align}
and we consider the fixed point starting with $\ta$. Recently Shallit \cite{Shallit2021} 
characterized the subword complexity function of this sequence (proving a conjecture by Dekking). 
Furthermore, M\"obius orthogonality was established by the authors \cite{DMS2018}. 
However, the understanding of properties of general morphic words (such as the behavior along arithmetic subsequences, 
along subsequences of asymptotic density zero, M\"obius orthogonality, 
evaluation along the sequence of prime numbers) is a huge open and important line of research.
One of our central contributions, a \emph{prime number theorem} for the Zeckendorf sum-of-digits function modulo $m$, falls into this field of research and is the first theorem of its kind.

Finally, the sequence of Fibonacci numbers is arguably the simplest nontrivial linear recurrent sequence of degree two.
Not much is known about the relation of prime numbers and values of a linear recurrence to each other, and so our above-mentioned prime number theorem also contributes to this area. Anyway it is worth mentioning that 
Madritsch and Thuswaldner~\cite{MT2019} considered the level of distribution of the sum-of-digits function related
to linear recurrence number systems $G_{k} = a_1 G_{k-1} + \cdots + a_d G_{k-d}$ and showed that the 
level of distribution approaches $1$ if $a_1 \to \infty$. Actually in our present work we show that 
the level of distribution of the Zeckendorf sum-of-digits function equals $1$ (see Chapter~\ref{chap_lod}).

It is a long standing open problem whether there exist infinitely many prime Fibonacci numbers ---
such numbers are called {\it Fibonacci primes}~\cite{Guy2004}.
This question is completely open.
A simple heuristic (similar to Mersenne primes), involving the ideas that (1) a number of size $N$ is prime with probability $1/\log N$ and (2) for $F_{\ijkl}$ to be prime we need $\ijkl=4$ or $\ijkl$ prime, suggests that there are infinitely many of them.
But of course such heuristics have to be examined with great care, and we will not pursue these arguments further.

One of the difficulties with questions of this kind is the following.
It is usually very difficult to find prime numbers in a given {\it sparse subset} of $\mathbb N$. A famous open question concerns prime values of polynomial functions, in particular it is unknown whether there are infinitely many primes of the form $n^2+1$.
Positive results in this direction include work by Fouvry and Iwaniec~\cite{FI1997}, who proved an asymptotic formula for the number of primes of the form $x^2+p^2$, where $p$ is prime;
by Friedlander and Iwaniec~\cite{FI1998}, who proved such a formula for primes of the form $x^2+y^4$; and by Heath-Brown~\cite{H2001}, who could handle $x^3+2y^3$.
A different line of research is represented by the search for prime numbers in \emph{Piatetski-Shapiro sequences} \cite{P1953,Rivat2001,RS2001}:
currently we know that the number of primes $p\leq x$ of the form $p=\lfloor n^c\rfloor$ behaves asymptotically like $x/c\log x$ as long as $1\leq c<2817/2426=1.16117\ldots$.
There are also results on primes with digital restrictions:
Maynard~\cite{Maynard2019} proved that there are infinitely many primes not featuring a certain (arbitrarily chosen) digit in their decimal expansions.
(Note that forbidding a given digit in base $10$ results in a sparse subset of $\mathbb N$.)

Concerning Fibonacci primes $p=F_\ijkl$, currently (2022) the smallest~$36$ of them are known, and the list of their indices $\ijkl$ begins as follows\footnote{\tt http://oeis.org/A001605}\footnote{\tt https://mathworld.wolfram.com/FibonacciPrime.html}:
\[
\ijkl\in \{ 3, 4, 5, 7, 11, 13, 17, 23, 29, 43, 47, 83, 131, 137, 359, 431, 433, 449, 509, 569,\ldots \}.
\]
The present record (as of 2022) is $\ijkl=148091$\footnote{\tt https://primes.utm.edu/top20/page.php?id=39}.
There are $15$ more known Fibonacci probable primes, the largest of which has the index $\ijkl=3340367$\footnote{\tt http://www.primenumbers.net/prptop/searchform.php?form=F(n)}.

Our first result relaxes the requirement a little. We consider a fixed number of Fibonacci numbers.
\begin{theorem}\label{Th1}
Let $k$ be a sufficiently large integer. There exists a prime number $p$ with
\[
\sz(p) = k.
\]
In particular, $p$ can be represented as the sum of $k$ pairwise different and non-consecutive Fibonacci numbers.
\end{theorem}

Note that the phrase ``pairwise different and non-consecutive'' in the corollary is important.
We sketch a proof that for any sufficiently large $k$ there exists a prime that can be written as the sum of $k$ pairwise different Fibonacci numbers. 

Let $k$ be large enough such that there exists a prime number $p$ satisfying
$F_{2k+1} \leq p < F_{2k+2}$ (which is guaranteed by the prime number theorem in short intervals).
Then $p$ can be written via~\eqref{eqn_Zeckendorf_rep} as
\begin{align*}
	p = \sum_{\ijkl = 2}^{2k+1} \digit_{\ijkl}(p) F_{\ijkl},
\end{align*}
where $\digit_{2k+1} = 1$ and at most $k$ of the $2k$ digits are $1$.
If there are less than $k$ digits equal to $1$, then necessarily there exists some $\ijkl \geq 2$ such that $(\digit_{\ijkl+2}, \digit_{\ijkl+1}, \digit_{\ijkl}) = (1,0,0)$. This pattern of digits can now be replaced by $(\digit_{\ijkl+2}, \digit_{\ijkl+1}, \digit_{\ijkl}) = (0,1,1)$, which increases the sum of digits by $1$. This operation can be applied until there are exactly $k$ digits equal to $1$, proving the claim.

In principle our proof methods are effective.
Following the proofs, all of the occurring constants could be made completely explicit --- we do not rely on ineffective arguments introduced by some proofs by contradiction, for example.
Keeping track of the constants would allow us to give an explicit numerical lower bound for $k$ in Theorem~\ref{Th1}.
However, although such an explicit bound would be \emph{nice to have} we quickly realized 
that our calculations would become very messy, and difficult to read.
In order to keep the already long proof clean from numerical values,
we decided, reluctantly, to stick to the ``base version''.
For the moment we have to content ourselves with the
possibility of computing such a bound.
It remains an open, very interesting question to prove that Theorem~\ref{Th1} is true for \emph{all} $k\ge 1$.

A result analogous to Theorem~\ref{Th1}, concerning the base-$q$ expansion instead of the Zeckendorf expansion (writing an integer as a sum, of minimal length, of powers of $q$), is due to the first author, Mauduit, and Rivat~\cite{DMR2009}.
Both this result and our Theorem~\ref{Th1} contribute to the interesting topic represented by the phrase ``mixing of the additive and multiplicative structures of the integers''.

Theorem~\ref{Th1} is actually a direct consequence of a local version of a central limit theorem of the Zeckendorf sum-of-digits function on primes,
which we state now.
We note that the letter $p$, as in the following theorem,
is the notation of choice for a prime number;
its use in many cases entails the condition ``$p$ is prime'',
which would have to be added at appropriate positions if we were to rewrite the paper in a more formal way. 
\begin{theorem}\label{Th2}
For each $\varepsilon>0$, we have
\begin{equation}\label{eqTh1}
  \# \bigl\{ p \le x : \sz(p) = k\bigr\}
  = 
  \frac{\pi(x)}{\sqrt{2\pi \sigma^2 \log_\golden x}}
  \left(
    e^{ - \frac {(k-\mu \log_\golden x)^2}{2\sigma^2 \log_\golden x} }
    +
    O\bigl((\log x)^{-\frac 12+\varepsilon}\bigr)
  \right)
\end{equation}
uniformly for all integers $k\ge 0$, where
\[
\mu = \frac 1{\golden^2 + 1}\quad \mbox{and}\quad \sigma^2 = \frac{\golden^3}{(\golden^2+1)^3},
\]
$\pi(x)$ denotes the number of primes $\le x$,
and $\log_\golden x=\log x/\log\golden$.
\end{theorem}
Clearly, if we specialize to $x = \golden^{k/\mu}$, we get
\[
  \# \bigl\{ p \le \golden^{k/\mu} : \sz(p) = k\bigr\}
  \sim
     \frac{\golden^{k/\mu}}{\sqrt{2\pi (\log \golden)^2 \sigma^2 /\mu^3} \,  k^{3/2} },
\]
which is a quantitative version of Theorem~\ref{Th1} (and shows that there are 
quite a lot of prime numbers $p$ with $\sz(p) = k$ if $k$ is sufficiently large).

As already mentioned, Theorem~\ref{Th2} can be also seen as a local central limit theorem
for $\sz(p)$ when we assume that every prime $p\le x$ is equally likely. 
Actually it is well known that $\sz(n)$, $n\le x$, satisfies a central limit theorem.
This remains true if we restrict ourselves to prime numbers $p\le x$ (see~\cite{DS02}): we have
\begin{equation}\label{eqclt}
\lim_{x\to\infty} \frac 1{\pi(x)} \# \left\{ p\le x : \sz(p) \le \mu \log_\golden x + t \sqrt{ \sigma^2  \log_\golden x } \right\} = \Phi(t)
\end{equation}
for every fixed real $t$, where $\Phi(t)$ denotes the distribution function of the standard normal distribution.

By L\'evy's theorem, a weak limit (like (\ref{eqclt})) is equivalent to a corresponding limiting relation
on the level of Fourier transforms.
This relation can be stated in terms of exponential sums (as usual we use the notation $\e(x) = e^{2\pi ix}$):
\[
\sum_{p\le x} \e\bigl(\vartheta\smallspace\sz(p)\bigr) \sim  \pi(x)\,
 e^{ 2\pi i \vartheta \mu \log_\golden x  - 2\pi^2\vartheta^2 \sigma^2 \log_\golden x}
\]
as $x\to\infty$, where $\vartheta$ is of the form $\vartheta = \beta/\sqrt{\log x}$ for real numbers $\beta$.
Equivalently,
\begin{equation}\label{eqclt2}
\frac 1{\pi(x)} \sum_{p\le x} e\left(\vartheta  \frac{ \sz(p) - \mu \log_\golden x  }{ \sqrt{2\pi \sigma^2 \log_\golden x } }   \right)
\to e^{-\pi \vartheta^2 }
\end{equation}
as $x\to\infty$, for every fixed real $\vartheta$.

The local version of this central limit theorem (Theorem~\ref{Th2}) is a direct consequence
of the following two key properties for the exponential sum $\sum_{p\le x} \e(\vartheta\smallspace\sz(p))$.

\begin{theorem}\label{Th3}
There exists a constant $c>0$ 
such that
\begin{equation}\label{eqTh31}
\sum_{p\le x} \e\bigl(\vartheta\smallspace\sz(p)\bigr) \ll (\log x)^4 x^{1-c \lVert\vartheta\rVert^2},
\end{equation}
uniformly for real $\vartheta$.

Suppose that $0<\nu < \frac 16$ and $0<\eta < \frac \nu 2$.
Then we have
\begin{align}\label{eqTh32}
\hspace{3em}&\hspace{-3em}
\sum_{p\le x} \e\bigl(\vartheta\smallspace\sz(p)\bigr) = \pi(x)\,
\e\bigl(\vartheta \mu \log_\golden x\bigr) \\
& \times \left( e^{- 2\pi^2\vartheta^2 \sigma^2 \log_\golden x}
\bigl( 1 + O\bigl( \vartheta^2 + \lvert\vartheta\rvert^3\log x\bigr)\bigr) +
O\bigl(\lvert\vartheta\rvert \, (\log x)^\nu \bigr)  \right),   \nonumber
\end{align}
uniformly for real $\vartheta$ with $\lvert\vartheta\rvert \le (\log x)^{\eta - \frac 12}$, where $\mu=1/(\golden^2+1)$
and $\sigma^2 = {\golden^3}/{(\golden^2+1)^3}$
\end{theorem}

Theorem~\ref{Th3} implies~\eqref{eqclt2} and, consequently, the central limit relation~\eqref{eqclt}. On the other hand we have by definition 
\[
\# \bigl\{ p \le x : \sz(p) = k\bigr\} = \int_{\lvert \vartheta\rvert = 1} e(-\vartheta k) \sum_{p\le x} e\bigl(\vartheta\smallspace\sz(p)\bigr)\, \mathrm d\vartheta.
\]
Hence, by applying~\eqref{eqTh31} and~\eqref{eqTh32} in order to evaluate this integral asymptotically,
we directly obtain Theorem~\ref{Th2}, and Theorem~\ref{Th1} as a corollary.

It is actually the main goal of this paper to prove Theorem~\ref{Th3}, that is, to establish
the relations~\eqref{eqTh31} and~\eqref{eqTh32}.

We comment first on the second relation~\eqref{eqTh32}, which is a refined version of the
central limit relation~\eqref{eqclt2} (which is in turn equivalent to~\eqref{eqclt}).
The proof of~\eqref{eqTh32} uses a refined version of a moment method of Bassily and K\'atai~\cite{BK1995}.
This method was already used by the first author, Mauduit, and Rivat \cite{DMR2009} 
to establish an analogue of (\ref{eqTh32}) for the $q$-ary sum-of-digits function.
However,  we have to face the additional complication that the digits of the 
Zeckendorf expansion are not asymptotically independent from each other, but can be approximated by a Markov chain.
This is a severe difference to the $q$-ary case, and leads to a much more involved analysis.

Another complication that arises when passing from the $q$-ary case to the Zeckendorf expansion is the 
detection problem for digits. For the detection of Zeckendorf digits with indices in $[a,b)$ we need 
\emph{two-dimensional parallelograms} rather than intervals in $\mathbb R$ as for the base-$q$ expansion. 
This introduces significant technical complications
(see Chapter~\ref{chap:detection}).
In this context, care has to be taken since in the addition of Zeckendorf expansions, carries may propagate ``backwards'' rather than only in direction of more significant digits as in the $q$-ary case.
Consequently, the process of ``cutting away digits'', an essential tool in the works of Mauduit and Rivat~\cite{MR2009,MR2010}, is a much more delicate matter in the Zeckendorf case.
To this end, we introduce a new generalization of van der Corput's inequality.
Van der Corput's inequality is an essential tool used at the base of Mauduit and Rivat's work on squares and primes.
Finding an appropriate replacement suitable for our case therefore proved essential.
A significant deviation from the path mapped out by Mauduit and Rivat was to take the detour via the \emph{level of distribution} of the Zeckendorf sum-of-digits function.
(Note that a partial result in this direction was obtained by Madritsch and Thuswaldner~\cite{MT2019}.)
For this part of our proof, we use in an essential way the method developed by the third author~\cite{S2020} (see Chapter~\ref{chap_lod}), and we establish the fact that the level of distribution of the sequence $\e(\vartheta\smallspace\sz(n))$ equals $1$.
This excursion takes care of the sums of type~\textrm{I} and simultaneously enables us to simplify the occurring sums of type~\textrm{II},
so that they become manageable.
A central property needed in our proof is a \emph{Gowers norm estimate} for the Zeckendorf sum-of-digits function, proved in Chapter~\ref{chap_gowers}.
Gowers norms were introduced by Gowers~\cite{G2001}, who used them to re-prove Szemer\'edi's theorem.
It might be interesting to note that we make use of these norms in both the estimates for sums of type \textrm{I} and \textrm{II}.
Using also the asymptotic independence of $(mp\golden \bmod \Z)_{m\in \N}$ and $(mq \golden \bmod \Z)_{m\in \N}$ when averaging over $p$ and $q$, we get estimates for sums of type \textrm{II} of sufficient quality.
Using Vaughan's identity, we obtain
\begin{equation}\label{eqPro1-1}
\sum_{n\le x} \Lambda(n) \, e\bigl(\vartheta\smallspace \sz(n)\bigr) \ll (\log x)^5 x^{1-c \lVert\vartheta\rVert^2}.
\end{equation}
Here $\Lambda(n)$ denotes the von Mangoldt function, defined by $\Lambda(p^k) = \log p$ for primes $p$ and integers $k\ge 1$, and $\Lambda(n) = 0$ otherwise.
By carrying out a standard summation by parts, we derive~\eqref{eqTh31} from~\eqref{eqPro1-1}, thereby reducing the power of $\log x$ to $(\log x)^4$.


\smallskip

By applying~\eqref{eqTh31} for rational $\vartheta = \ell/m$ and by using discrete Fourier inversion 
we directly obtain the following property that corresponds to Mauduit and Rivat's result \cite{MR2010}
on the $q$-ary sum-of-digits function $s_q(n)$ modulo $m$.
Note that there is no condition such as ${\rm gcd}(q-1,a) = 1$,
since the Zeckendorf sum-of-digits function $\sz(n)$ does not satisfy
a congruence condition like $s_q(n) \equiv n \bmod q-1$.
\begin{theorem}\label{Th4}
Let $m\ge 1$ be an integer.
There exists $c_2>0$ such that, for every integer $a$, the following estimate holds:
\[
\# \bigl\{p \le x : \sz(p) \equiv a \bmod m \bigr\} = \frac{\pi(x)}m + O\bigl(x^{1-c_2}\bigr).
\]
\end{theorem}

A theorem of this kind is sometimes called {\it prime number theorem}. 
If we set 
\[    u_n = \sz(n)\bmod m,    \]
then Theorem~\ref{Th4} can be rewritten to 
\[    \# \bigl\{p \le x : u_p = a \bigr\}
      = \frac{\pi(x)}m + O\left(x^{1-c_2}\right).    \]
The sequence $u=(u_n)_{n\in\N}$ is a {\it morphic sequence} for every fixed integer $m\ge 2$.
(In~\eqref{eq_morphic} this was verified for $m=2$.)


If a morphic sequence can be represented via a substitution $\sigma$ of constant length, it is called \emph{automatic}, see \cite{AS2003}.
The most prominent automatic sequence is the \emph{Thue--Morse sequence} $t=(t_n)_{n\in\N}$ (which we define here on the alphabet $\{1,-1\}$ instead of the customary $\{0,1\}$).
We use the alphabet $\{\ta,\tb\}$, the substitution $\sigma$ given by $\sigma(\ta) = \ta\tb$, $\sigma(\tb) = \tb\ta$ and
the coding $\pi(\ta) = 1$, $\pi(\tb) = -1$.
The resulting sequence is given by
\[
t = (1,-1,-1,1,-1,1,1,-1,\ldots) = \bigl((-1)^{s_2(n)}\bigr)_{n\in\N},
\]
where $s_2$ denotes the binary sum-of-digits function.

In general, automatic or morphic sequences cannot be represented in terms of a simple functional of a numeration system.
Nevertheless, the functions $s_q(n)\bmod m$ and $\sz(n)\bmod m$ can be used as {\it toy examples} in order to get first results in the class of automatic or morphic sequences.

Actually, Theorem~\ref{Th4} is the first prime number theorem for a morphic sequence if we exclude automatic sequences~\cite{Muellner2017} and Sturmian words (as a special case of nilsequences~\cite{Green2012}).
Note that it is a priori not clear that a given morphic sequence (defined by a non-constant length substitution as in~\eqref{eq_morphic}) is in fact non-automatic.
For the case $\sz(n)\bmod 2$ this was proved by the authors \cite{DMS2018}.
The method of proof found there applies for the more general case $\sz(n)\bmod m$ as well.

The relation (\ref{eqPro1-1}) has an important interpretation.
It says that the sequence $u_n = e(\vartheta\smallspace \sz(n))$ is asymptotically orthogonal to the von Mangoldt function (if $\vartheta$ is not an integer). 
Such orthogonality relations play a very prominent role in the context of the already mentioned \emph{Sarnak conjecture}~\cite{Sarnak2011}.
This conjecture states that every \emph{deterministic sequence} $f(n)$ satisfies
\begin{equation}\label{eqop}
\sum_{n\le x} \mu(n)\smallspace f(n) = o(x),
\end{equation}
where $\mu(n)$ is the M\"obius function (defined by $\mu(1) = 1$,
$\mu(p_1p_2\cdots p_\ell) = (-1)^\ell$ for different primes $p_1,\ldots,p_\ell$,
and $\mu(n) = 0$ otherwise), compare with the Preface, too.
In particular, a sequence $f(n)$ that attains only finitely values
is deterministic if the subword complexity is sub-exponential.
More precisely, if $A(L)$ denotes the number of different contiguous subsequences of length $L$ in $f(n)$, then $A(L) = e^{o(L)}$ as $L\to\infty$.

For example, for automatic sequences we have $A(L) = O(L)$, while the subword complexity of morphic sequences satisfies $A(L) = O(L^2)$.
Thus, all automatic and morphic sequences are deterministic and are expected to satisfy the orthogonality property (\ref{eqop}).
For automatic sequences this was verified by the second author~\cite{Muellner2017}, whereas the corresponding problem is wide open for morphic sequences.
The authors~\cite{DMS2018} could prove orthogonality for the (morphic) sequence
$f(n) = (-1)^{\sz(n)}$, but the used proof method is limited.

However, it should be mentioned that the methods developed in the present paper can be also adapted to prove 
\begin{equation}\label{eqPro1-2}
\sum_{n\le x} \mu(n)\smallspace e\bigl(\vartheta\smallspace\sz(n)\bigr) \ll (\log x)^5 x^{1-c \lVert\vartheta\rVert^2}.
\end{equation}
Actually this might be extended to so called \emph{Fibonacci automatic sequences}, see Chapter~\ref{chapter:openproblems}.

As we noted, verifying the Sarnak conjecture for general morphic sequences is an open problem; 
also, so far nothing is known concerning corresponding prime number theorems.
The present paper might be a first step in this direction.

\subsection*{Notation}
In this paper, we will use the following $1$-periodic functions of real numbers.
We write $\e(x)=\exp(2\pi i x)$ for real $x$.
The expression $\lVert x\rVert$ denotes the ``distance of $x$ to the nearest integer'' (although there might be two such integers), $\lVert x\rVert=\min_{n\in\mathbb Z}\lvert x-n\rvert$.
The \emph{fractional part of} $x$ is defined by $\{x\}\eqdef x-\lfloor x\rfloor$.
In some of our estimates, it will be convenient to use the function
\[\log^+(x)\eqdef\left\{\begin{array}{ll}1,&\mbox{ if }x=0;\\\max\bigl(1,\log(x)\bigr),&\mbox{ if }x>0.\end{array}\right.\]
The symbol $\mathbb N$ denotes the set of nonnegative integers.
Throughout this paper, $\golden$ denotes the golden ratio: $\golden=\frac 12(\sqrt{5}+1)$.


\chapter{Plan of the Proofs}

In this chapter we give an overview of the structure of the proofs of our main theorems and the structure of the paper.
In the process, we state two theorems --- Theorem~\ref{thm_lod} and Theorem~\ref{th_sum_2} --- which are concerned with the \emph{level of distribution} of $\e(\vartheta\smallspace\sz(n))$ and a corresponding estimate for sums of \emph{type}~\textrm{II}.

\section{The main result: Theorem~\ref{Th3}}

As outlined in the Introduction,
the key theorem of this paper is Theorem~\ref{Th3}.
We split this theorem into two parts for better readability.

\begin{proposition}\label{Promain1}
There exists a constant $c>0$
such that
\begin{equation}\label{eqPro1}
\sum_{p\le x} \e\bigl(\vartheta\smallspace \sz(p)\bigr)
\ll (\log x)^4 x^{1-c\lVert\vartheta\rVert^2}
\end{equation}
uniformly for real $\vartheta$.
\end{proposition}
\begin{proposition}\label{Promain2}
Suppose that $0<\nu < \frac 16$ and $0<\eta < \frac \nu 2$.
Then we have
\begin{equation}\label{eqPro2}
\begin{aligned}
\hspace{3em}&\hspace{-3em}
\sum_{p\le x} \e\bigl(\vartheta\smallspace\sz(p)\bigr) = \pi(x)\,
e\bigl(\vartheta \mu \log_\golden x\bigr) \\
& \times \left( e^{- 2\pi^2\vartheta^2 \sigma^2 \log_\golden x}
\bigl( 1 + O\bigl( \vartheta^2 + \lvert\vartheta\rvert^3\log x\bigr)\bigr) +
O\bigl(\lvert\vartheta\rvert \, (\log x)^\nu \bigr)  \right)
\end{aligned}
\end{equation}
uniformly for real $\vartheta$ with $\lvert\vartheta\rvert\le(\log x)^{\eta - \frac 12}$,
where $\mu=1/(\golden^2+1)$ and $\sigma^2 = {\golden^3}/{(\golden^2+1)^3}$.
\end{proposition}

Proving Proposition~\ref{Promain1} is the objective of Chapters~\ref{chap_lod} (where the level of distribution of $\e(\vartheta\smallspace\sz(n))$ is considered) and~\ref{ch_type2} (where we estimate a type~\textrm{II}-sum for this function).
These chapters in turn rely on auxiliary results from Chapters~\ref{chap_exponentialsums} (concerning exponential sums and discrepancy),~\ref{chap:detection} (concerned with the detection of Zeckendorf digits by exponential sums), and~\ref{chap_gowers} (the proof of a Gowers uniformity norm estimate).
Proposition~\ref{Promain2} will be proved in Chapter~\ref{chap_local}.
In this latter proof we will be concerned among other things with a quantitative approximation of the Zeckendorf expansion by a Markov process.
Finally, at the end of the paper, in Chapter~\ref{chapter:openproblems},
we will state some possible extensions and open problems that we encountered while working on this paper.

Before discussing the structure of the proofs of Propositions~\ref{Promain1} and \ref{Promain2} in Sections~\ref{sec:proofPromain1} and \ref{sec:proofPromain2} respectively, we show how Theorems~\ref{Th1}, \ref{Th2} and~\ref{Th4}
can be deduced from these two propositions.

\section{Proof of Theorem~\ref{Th4}}
We just assume that Proposition \ref{Promain1} holds. Then we directly get the proposed relation
\begin{align*}
\# \{ p \le x : \sz(p) \equiv a \bmod m \} 
&= \sum_{p\le x} \frac 1m \sum_{j=0}^{m-1} e\left( \frac{j(\sz(p)-a)}m \right) \\
&= \frac 1m \sum_{j=0}^{m-1} e\left( -\frac{ja}m \right)  \sum_{p\le x} e\left( \frac jm \sz(p) \right) \\
&= \frac{\pi(x)} m + O\left( (\log x)^4 x^{1- c_1/m^2} \right) \\
&= \frac{\pi(x)} m + O\left( x^{1- c_2} \right). 
\end{align*}

\section{Proof of Theorem~\ref{Th2}}
We assume now that Propositions~\ref{Promain1} and \ref{Promain2} hold 
(that is, Theorem~\ref{Th3} holds).
Set 
\[
S(\vartheta) = \sum_{p\le x } \e\bigl(\vartheta\smallspace\sz (p)\bigr)
\]
and observe that we have the integral representation
\[
\# \bigl\{ p \le x : \sz (p) = k\bigr\} 
= \int_{-1/2}^{1/2} S(\vartheta) \e(-\vartheta k)\, \mathrm d\vartheta.
\]
We split the integral into two parts:
\[
\int_{-1/2}^{1/2}  = \int_{\lvert\vartheta\rvert \le (\log x)^{\tau-1/2} }
+  \int_{(\log x)^{\tau-1/2}< |\vartheta| \le 1/2 },
\]
where $\tau<1/2$ is chosen later.
The first integral can be easily evaluated with help of Proposition~\ref{Promain2}.
We use the substitution $\vartheta = t/(2\pi \sigma \sqrt{\log_\golden x})$ and obtain
\begin{align*}
\hspace{2em}&\hspace{-2em}
\int_{\lvert\vartheta\rvert \le (\log x)^{\tau-1/2} } S(\vartheta) \e(-\vartheta k) \, \mathrm{d}\vartheta\\
&=  \pi(x)
\int_{\lvert\vartheta\rvert \le (\log x)^{\tau-1/2} }
   \e\bigl(\vartheta (\mu \log_\golden x-k)\bigr)\,
e^{- 2\pi^2\vartheta^2 \sigma^2 \log_\golden x}
\\&\quad\times
\Bigl( 1 + O\bigl( \vartheta^2 +  {\lvert\vartheta\rvert^3}\log x\bigr)\Bigr) \, \mathrm d\vartheta \\
&\quad+ O\left( \pi(x)  \int_{\lvert\vartheta\rvert \le (\log x)^{\tau-1/2} }
\lvert\vartheta\rvert\, (\log x)^\nu\, \mathrm d\vartheta \right) \\
&=  \frac{\pi(x)}{2\pi \sigma \sqrt{\log_\golden x}}
\int_{-\infty}^\infty   e^{it \Delta_k - t^2/2}
\, \mathrm dt + O\left( \pi(x) e^{- 2\pi^2 \sigma^2 (\log x)^{2\tau} } \right) \\
&\quad + O\left(  \frac{\pi(x)}{\log x} \right) +
O\left(  \frac{\pi(x)}{(\log x)^{1-\nu-2\tau}} \right) \\
&= \frac{\pi(x)}{\sqrt{ 2\pi \sigma^2 \log_\golden x}}
\Bigl( e^{-\Delta_k^2/2} + O\bigl( (\log x)^{-\frac 12+\nu+2\tau}\bigr) \Bigr),
\end{align*}
where
\[
\Delta_k = \frac {k - \mu \log_\golden x }{\sqrt{\sigma^2 \log_\golden x} }.
\]
The remaining integral can be directly estimated with Proposition~\ref{Promain1}:
\begin{align*}
\int\limits_{(\log x)^{\tau-1/2}< \lvert\vartheta\rvert \le 1/2} S(\vartheta)\smallspace\e(-\vartheta k) \, \mathrm d\vartheta
&\ll (\log x)^4\, x \, e^{-c_1(\log x)^{2\tau}}\\
& \ll \frac {\pi(x)}{\log x},
\end{align*}
with implied constants that may depend on $\tau$.
Finally, if $\varepsilon$ with $0< \varepsilon < \frac 12$
is given, we can set $\nu = \frac 23 \varepsilon$
and $\tau = \frac 16 \varepsilon$.
Hence $0 < \tau < \frac 13 \nu$ and $\nu + 2 \eta = \varepsilon$.
Thus, Theorem~\ref{Th2} follows immediately:
\begin{equation}\label{eqlocal}
  \# \bigl\{ p \le x : \sz(p) = k\bigr\}
  = 
  \frac{\pi(x)}{\sqrt{2\pi \sigma^2 \log_\golden x}}
  \left(
    e^{ -\frac{(k-\mu \log_\golden x)^2}{2\sigma^2 \log_\golden x} }
    + O\bigl((\log x)^{-\frac 12+\varepsilon}\bigr)
  \right).
\end{equation}

\section{Proof of Theorem~\ref{Th1}}
Next we show that, as we indicated in the Introduction, Theorem~\ref{Th2} implies Theorem~\ref{Th1}.

We specialize~\eqref{eqlocal} to $x = \golden^{k/\mu}$ and obtain
\[
  \# \bigl\{ p \le \golden^{k/\mu} : \sz(p) = k\bigr\}
  = 
     \frac{\golden^{k/\mu}}{\sqrt{2\pi (\log \golden)^2 \sigma^2 /\mu^3} \,  k^{3/2} }
     \Bigl( 1 + O\bigl( k^{-\frac 12 + \varepsilon} \bigr) \Bigr).
\]
In particular, this implies, for $k$ sufficiently large,
\[
 \# \bigl\{ p \le \golden^{k/\mu} : \sz(p) = k\bigr\} > 0.
\]
Of course this implies Theorem~\ref{Th1} --- for sufficiently large $k$ there exists a prime number $p$ with $\sz(p) = k$.

\def\PromainRefa{\ref{Promain1}}
\section{Plan of the Proof of Proposition~\PromainRefa} \label{sec:proofPromain1} 
One of the most classical ways to achieve  estimates for sums over primes is by obtaining good control 
of bilinear sums, usually so called sums of type~\textrm{I} and sums of type~\textrm{II} (sometimes one also uses sums of type~\textrm{III}).
In particular, Vaughan's method can be used to this effect;
the following version can be found for example in~\cite[page 142]{D2000}. 
\begin{lemma}\label{le_vaughan}
	Let $f: \mathbb{N} \to \mathbb{C}$ such that $\lvert f(n)\rvert \le1$ for all $n\geq 1$. For all $N,U,V\ge2$ such that $UV \leq N$ we have
\begin{align}
\hspace{3.5em}&\hspace{-3.5em}
		\sum_{n\leq N} f(n) \Lambda(n) \ll U + (\log N) \sum_{t\leq UV} \max_w \left\lvert\sum_{w\leq r \leq N/t} f(rt)\right\rvert \nonumber\\
			&+\sqrt{N} (\log N)^3
\max_{\substack{U \leq M \leq N/V\\V \leq q \leq N/M}}
\left(\sum_{V< p \leq N/M}\left\lvert\sum_{\substack{M < m \leq 2M\\m \leq \min(N/p,N/q)}} f(mp) \overline{f(mq)}\right\rvert\right)^{1/2},\label{eqn_vaughan}
	\end{align}
with an absolute implied constant.
\end{lemma}
Here the first sum on the right hand side is used as an upper bound for so called sums of type~\textrm{I} and the second sum is used as an upper bound for so called sums of type~\textrm{II}. Therefore, we use the following notation:
\begin{align*}
	S_{\mathrm I}(N, U, V) &\coloneqq (\log N) \sum_{t\leq UV} \max_w \left\lvert\sum_{w\leq r \leq N/t} f(rt)\right\rvert,\\
	S_{\mathrm{II}} (N, U, V) &\coloneqq \sqrt{N} (\log N)^3 \max_{\substack{U \leq M \leq N/V\\V \leq q \leq N/M}}
\left(\sum_{V< p \leq N/M}\left\lvert\sum_{\substack{M < m \leq 2M\\m \leq \min(N/p,N/q)}}\hspace{-1.5em} f(mp) \overline{f(mq)}\right\rvert\right)^{1/2}.
\end{align*}
Of course, we want to use this lemma for the function $f(n) = \e(\vartheta\smallspace\sz(n))$.
It is a priori not clear how to choose the parameters $U$ and $V$.
To this end, we will exploit the fact, clearly visible from this particular version of Vaughan's identity, that better control over one of the two sums allows for more freedom in the treatment of the other.
For our application, we will have very good control over $S_{\mathrm I}$ due to the (optimal) level of distribution $1$; we will choose $U = N^{2/3+\varepsilon}$ and $V = N^{\varepsilon}$.

In order to prove Proposition~\ref{Promain1}, it is therefore sufficient to have good estimates for $S_{\mathrm I}$ and $S_{\mathrm{II}}$.
Chapters~\ref{chap_lod} and~\ref{ch_type2} deal with these sums respectively.

\subsection{Type~\textrm{I} from the level of distribution}

In Chapter~\ref{chap_lod} we will prove the following theorem, stating that
$\e(\vartheta\smallspace\sz(n))$ has \emph{level of distribution} equal to $1$.
Such a result was proved for the classical Thue--Morse sequence by the third author~\cite{S2020}.

\begin{theorem}\label{thm_lod}
Let $\varepsilon>0$.
There exist $c_1=c_1(\varepsilon)>0$ and $C=C(\varepsilon)>0$ depending only on $\varepsilon$ such that for all $\vartheta\in\mathbb R$ and all real $x\geq 1$ we have
\begin{equation}\label{eqn_lod}
\sum_{1\leq d\leq D}
\max_{\substack{y,z\geq 0\\z-y\leq x}}
\max_{0\leq a<d}
\left\lvert
\sum_{\substack{y\leq n<z\\n\equiv a\bmod d}}
\e\bigl(\vartheta\smallspace\sz(n)\bigr)
\right\rvert
\leq C\,(\log^+\!x)^{11/4}\,x^{1-c_1\lVert \vartheta\rVert^2},
\end{equation}
where $D=x^{1-\varepsilon}$.
\end{theorem}
This is a statement on the Zeckendorf sum-of-digits along \emph{very sparse arithmetic subsequences}, having $\asymp N$ elements and common difference $\asymp N^\rho$, where $\rho>0$ is an arbitrarily large exponent.
Note that currently (since Bombieri and Vinogradov) we know that $1/2$ is an admissible level of distribution for the sequence of prime numbers;
meanwhile, the Elliott--Halberstam conjecture~\cite{EH1970} states that~$1$ is admissible.
For more history on the level of distribution, consult the survey paper~\cite{Kontorovich2014} by Kontorovich, the paper~\cite{FouvryMauduit1996} by Fouvry and Mauduit, and Chapter~22 of the book~\cite{FriedlanderIwaniec2010} by Friedlander and Iwaniec.

\begin{remark}
One factor $\log x$ in~\eqref{eqn_lod} comes from extending the summation range (using the classical Lemma~\ref{lem_vinogradov});
another is due to dyadic decomposition of the interval $[1,D]$;
the remaining factor $(\log x)^{3/4}$ is introduced by the divisor function $\tau$, which is used in the part of the proof concerning small $D$.
\end{remark}

In order to reduce type~\textrm{I}-sums to the level of distribution,
almost nothing has to be said.
\begin{corollary}\label{CortypeI}
Suppose that $0 < \varepsilon < \frac 16$. Then we have uniformly for $\vartheta \in \R$ and $N\ge 2$
\begin{align}\label{eqCortypeI}
\begin{aligned}
S_{\mathrm I}\bigl(N, N^{2/3 + \varepsilon}, N^{\varepsilon}\bigr) &= \log(N) \sum_{t\leq N^{2/3 + 2 \varepsilon}} \max_w \left\lvert\sum_{w\leq r \leq N/t} \e\bigl(\vartheta\smallspace\sz(rt)\bigr)\right\rvert\\
&\le C(\log N)^{15/4}N^{1-c_1\lVert \vartheta\rVert^2}
\end{aligned}
\end{align}
for certain constants $c_1 = c_1(\varepsilon) > 0$, $C= C(\varepsilon)>0$.
\end{corollary}

\begin{proof}
Our goal is to estimate $S_{\mathrm I}(N,U,V)$ for $U=N^{2/3+\varepsilon}$ and $V=N^\varepsilon$.
For this purpose we apply Theorem~\ref{thm_lod} for  $D = N^{2/3+ 2\varepsilon}$,
the variable $t$ in the definition of $S_{\mathrm I}$ corresponds to $d$ in~\eqref{eqCortypeI}, the variable $N$ to $x$,
and the sum over 
$\e\bigl(\vartheta\smallspace\sz(rt)\bigr)$ translates to a sum over $\e\bigl(\vartheta\smallspace\sz(n)\bigr)$ 
such that the restrictions $wd\leq n\leq N$ and $n\equiv 0\bmod d$ are satisfied.
Clearly, the factor $\log N$ in the definition of $S_{\mathrm I}$ increases the exponent of the logarithm to $15/4$.
\end{proof}

\subsection{The level of distribution}\label{sec252}
We briefly describe the proof of Theorem~\ref{thm_lod}, which we present in Chapter~\ref{chap_lod}.
There are two main ideas involved: 
(1) truncating the digital expansion of an integer using van der Corput's inequality, and applying a carry propagation lemma (Lemma~\ref{lem_z_carry_lemma});
(2) an estimate for the Gowers norm of $\e(\vartheta\smallspace\sz(n))$ (more precisely on a variant of that, see Theorem~\ref{cor_gowers_estimate}).
Properties of this kind proved essential in the paper~\cite{Mauduit2015} 
by Mauduit and Rivat, which we had to modify suitably in order to fit our needs.

The proof for the level of distribution of $\e(\vartheta\smallspace\sz(n))$ presented in this chapter is based on the recent paper~\cite{S2020} by the third author, which we have to generalize significantly.
The central object are sums of the form
\[\sum_{0\leq n<N}\e\bigl(\vartheta\smallspace\sz(nd+a)\bigr),\]
where $d$ is potentially much larger than $N$.
Applying van der Corput's inequality, we may cut off the most significant Zeckendorf digits, leaving only $\approx \log d$ many digits to be taken into account.
Our goal is to reduce this number further, to the effect that $nd+a$ uniformly runs through all the possible configurations of the remaining digits.
At this important point, we may replace the sum over $nd+a$ by a full sum over $n$.
This successive reduction of digits is carried out by repeated application of van der Corput's inequality (in a suitably generalized form);
this introduces a \emph{Gowers norm} related to the Zeckendorf sum of digits function.
There are several (closely related) Gowers norm notions (see Chapter~\ref{chap_gowers}).
These norms are an essential tool in what is called higher order Fourier analysis~\cite{G2007, T2012}. They were introduced by Gowers~\cite{G2001} in his work on Szemer\'edi's theorem concerning arithmetic progressions in thin subsets of the integers.
More generally, arithmetic progressions in groups can be studied with the help of Gowers norms.
In our context we will use
the Gowers $s$-norm $\norm{f}_{U^s(\mathbb{T})}$ for a bounded, measurable, and $1$-periodic function $f: \R \to \C$, which clearly can be viewed a function from the torus $\mathbb{T}$ to $\C$.
These norms are also called {\it Gowers uniformity norms} (hence the letter $U$).


It is, however, possible to relate a function that depends on the Zeckendorf expansion of $n$ naturally to a function on the torus.
If $x\in [0,1)$ is of the form  $x = \{n \golden\}$ for some non-negative integer $n$, we set for any $\ijkl\geq 2$.
\begin{align*}
	\widetilde{\digit_{\ijkl}}(x) = \digit_{\ijkl}(n).
\end{align*}
Moreover,
\begin{align*}
	\digit_{\ijkl}'(x) = \lim_{z\to x+} \widetilde{\digit_{\ijkl}}(z),
\end{align*}
where the limit is taken from the right side (Lemma~\ref{Lefirstdigits} assures that this limit is well defined).
The function $\digit_{\ijkl}'$ can be extended to a $1$-periodic function and is by definition piecewise constant. We now define the function
\begin{equation}\label{eqgLdef}
g_\lambda(x) = \sum_{\ijkl=2}^\lambda \digit_{\ijkl}'(x),
\end{equation}
which is again $1$-periodic and piecewise constant, and mimics the truncated Zeckendorf sum-of-digits function
\[
\sz_\lambda(n) =  \sum_{\ijkl=2}^\lambda \digit_{\ijkl}(n).
\]
We actually have 
\[
\sz_\lambda(n) = g_\lambda(\gamma n).
\]
Theorem~\ref{cor_gowers_estimate} (which is proved in Section~\ref{sec_gowers_def}) provides a non-trivial estimate for the Gowers norm 
\[
\norm{\e(\vartheta g_\lambda)}_{U^s(\mathbb{T})}.
\]
This Gowers norm estimate is not only important for the level of distribution, but it will be used again in the estimate of sums of type~\textrm{II}.

Complications in this process, compared to the article~\cite{S2020}, arise due to the behavior of Zeckendorf digits, which is very different from the behavior of base-$q$ digits.
For example, it is straightforward to detect base-$q$ digits $a_j(n)$ with indices in an interval, $j\in [A,B)$: we have
\[
\bigl(a_A(n),\ldots,a_{B-1}(n)\bigr)
=(\nu_A,\ldots,\nu_{B-1})
\quad\mbox{if and only if}\quad
\left\{\frac n{q^B} \right\} \in J,\]
where
\[J=\bigl[m/q^{B-A},(m+1)/q^{B-A}\bigr)\quad\mbox{and}\quad m=\sum_{A\leq j<B}\nu_jq^{j-A}.\]
In order to obtain an analogous statement for the Zeckendorf digits, we have to introduce two-dimensional detection parallelograms.
The Zeckendorf digits of $n$ with indices in an interval $[A,B)$ are equal to prescribed values if and only if
\[\left(n/\golden^B\bmod 1,n/\golden^{B+1}\bmod 1\right)\]
is contained in a certain parallelogram modulo $1\times 1$.
This relation is expressed in Corollary~\ref{cor_twodim}.
In order to study the distribution in parallelograms in an ``analytical'' way, we make use of the \emph{isotropic discrepancy}~\eqref{eqn_isotropic_def}, and we adapt the Erd\H{o}s--Tur\'an--Koksma inequality to parallelotopes (Theorem~\ref{thm_ETK_parallelotope}).

The process of cutting away digits with indices in $[A,B)$ is based on this procedure. But we need another important modification, concerning the fundamental inequality of van der Corput.
Mauduit and Rivat~\cite{MR2009} proved a generalization of this inequality in their work on the sum of digits of squares; this variant is not sufficient for our needs, so we had to find an appropriate generalization (Proposition~\ref{prp_vdC_generalized}).

Having found a strong estimate of sums of type~\textrm{I},
we may now approach the treatment of sums of type~\textrm{II} with an optimistic attitude --- by Lemma~\ref{le_vaughan},
we only have to obtain an alleviated type-\textrm{II} estimate, where $U=N^{2/3+\varepsilon}$.

\subsection{Sums of type~\textrm{II}}\label{sec_plan_sum_2}
In Chapter~\ref{ch_type2} we will prove the following Theorem.
\begin{theorem}\label{th_sum_2}
Let $f(n) = \e(\vartheta\smallspace\sz(n))$ and assume that there exists some $c_2>0$ such that
\begin{equation}\label{eqn_type2_condition}
\bigl\lVert\e\bigl(\vartheta g_{\lambda}\bigr)\bigr\rVert_{U^3(\T)}
\ll
F_{\lambda}^{-c_2 \lVert \vartheta\rVert^2}
\end{equation}
holds uniformly for $\vartheta \in \R$, where $F_\lambda$ denotes the $\lambda$th Fibonacci number, 
$g_\lambda(x)$ is defined in (\ref{eqgLdef}) and 
satisfies $g_\lambda(\gamma n) = \sz_\lambda(n)$ ($\sz_\lambda(n) = \sum_{j=2}^\lambda \delta_j(n)$ is
the truncated Zeckendorf sum-of-digits function).

Then for all $N,U,V\geq 2$, such that $UV \leq N$ and for all $\vartheta \in \R$,
\begin{equation}\label{eq_sum_2}
\begin{aligned}
\hspace{1em}&\hspace{-1em}
	S_{\mathrm{II}}(N, U, V)
\\&=\sqrt{N} (\log N)^3 \max_{\substack{U \leq M \leq N/V\\V \leq q \leq N/M}}
\left(\sum_{V< p \leq N/M}\left\lvert\sum_{\substack{M < m \leq 2M\\m \leq \min(N/p,N/q)}} f(mp) \overline{f(mq)}\right\rvert\right)^{1/2}
\\&\ll N \bigl(\log^+ N\bigr)^{5} \left(V^{-c_2 \lVert\vartheta\rVert^2/54} + \frac{N^{1/2 + c_2 \norm{\vartheta}^2/54}}{U^{3/4 + c_2 \lVert\vartheta\rVert^2/54}}\right)^{1/2},
\end{aligned}
\end{equation}
where the implied constant depends at most on $c_2$.
\end{theorem}
\begin{remark}
	For Theorem~\ref{th_sum_2} to give a non-trivial bound, it is sufficient that $V \gg N^{\varepsilon}$ and $U \gg N^{2/3+ \varepsilon}$ for some $\varepsilon>0$.
We also note that the estimate~\eqref{eqn_type2_condition} is provided by Theorem~\ref{cor_gowers_estimate} in Chapter~\ref{chap_gowers}.
Formulating the theorem in this way has the advantage that we can see the dependence of the estimate~\eqref{eq_sum_2} on the quality (given by $c$) of the Gowers norm estimate.
\end{remark}

\begin{corollary}\label{CortypeII}
There exists a constant $c_2>0$ such that for all $0 < \varepsilon < \frac 16$ and $\vartheta \in \R$ we have
\begin{equation}\label{eqCortypeII}
\begin{aligned}
\hspace{2em}&\hspace{-2em}
S_{\mathrm{II}}\bigl(N, N^{2/3+\varepsilon}, N^{\varepsilon}\bigr)\\
&\ll  N \bigl(\log^+N\bigr)^{5} \left(N^{- \varepsilon c_2 \lVert\vartheta\rVert^2/76} + 
N^{- \frac 34 \varepsilon + \left( \frac 13 - \varepsilon \right) c_2 \lVert\vartheta\rVert^2/38  } \right)^{1/2},
\end{aligned}
\end{equation}
uniformly in $\vartheta$.
\end{corollary}

\begin{proof}
First of all we can apply Theorem~\ref{cor_gowers_estimate} 
by specifying $\dims = 3$ and noting that $F_{\lambda} \sim \exp\bigl(\lambda\smallspace \log(\golden) - \log(\sqrt{5})\bigr)$.
This proves (\ref{eqn_type2_condition}) for some $c_2$ such that $0<c_2<1$. Hence, (\ref{eq_sum_2}) holds, where we 
set $U= N^{\frac 23 + \varepsilon}$ and $V = N^\varepsilon$:
\[
\begin{aligned}
\hspace{4em}&\hspace{-4em}
	S_{\mathrm{II}}(N, N^{2/3+\varepsilon}, N^{\varepsilon})
\\&\ll N \bigl(\log^+N\bigr)^{5} \left(N^{- \varepsilon c_2 \lVert\vartheta\rVert^2/76} + 
\frac{N^{1/2 + c_2 \norm{\vartheta}^2/38}}{N^{\left(2/3 + \varepsilon \right)\left(3/4 + c_2 \lVert\vartheta\rVert^2/38 \right) }}\right)^{1/2}
\\ & = N \bigl(\log^+N\bigr)^{5} \left(N^{- \varepsilon c_2 \lVert\vartheta\rVert^2/76} + 
N^{- \frac 34 \varepsilon + \left( \frac 13 - \varepsilon \right) c_2 \lVert\vartheta\rVert^2/38  } \right)^{1/2},
\end{aligned}
\]
as proposed.
\end{proof}

In order to prove Theorem~\ref{th_sum_2}, we use --- as for the sums of type~\textrm{I} --- a carry propagation lemma and a Gowers norm estimate as two of the main ingredients.
Moreover, we will need the asymptotic independence of $(mp\golden\bmod 1)_{m\in \N}$ and $(mq \golden \bmod 1)_{m\in \N}$ when considering averages of $p$ and $q$.

The treatments of our sums of type~\textrm{I} and~\textrm{II} are similar in several aspects; we elaborate here on some details that have been omitted in our description of type~\textrm{I}-sums.

First we use some standard tools (such as Lemma~\ref{lem_vinogradov}) in order to reduce the problem to an estimate for a sum similar to
\begin{align*}
	\sum_{M_1 < p \leq 2 M_1} \left\lvert\sum_{M < m \leq 2M} f(pm) \overline{f(qm)}\right\rvert,
\end{align*}
where $f(n) = \e(\vartheta\smallspace\sz(n))$.
As a first step, we reduce the number of digits that we have to take into account for $\sz(n)$.
This can be done by using first the Cauchy--Schwarz inequality for the sum over $p$ and then van der Corput's inequality for the sum over $m$ (this also allows us to change the order of summation).
This allows us to replace $\e(\vartheta\smallspace\sz(pm)) = f(pm)$ by 
\[\e\bigl(\vartheta (\sz(pm + pr)- \sz(pm))\bigr) = f(pm+pr) \overline{f(pm)}
\eqqcolon \Delta\bigl(\overline{f}; -pr\bigr)(pm)\]
(similarly for $q$), where we take an average over $r \in [-R, R]$.
Considering the effect of adding $pr$ to the Zeckendorf expansion of $pm$, we expect it to change the digits up to position $\log_{\golden}(pr)$ and also have a possible carry (very similar to the base $q$ representation).
However, this carry usually only affects few other digits (see Lemma~\ref{lem_z_carry_lemma}).
Thus, if we take $\lambda$ larger than $\log_{\golden}(M_1 R)$ by a sufficient amount, then any digit at position $\ell \geq \lambda$ should be the same for $pm$ and $pm + pr$ most of the time (similarly for $q$).
Thus, we can replace $\sz$ by $\sz_{\lambda}$, where $\sz_{\lambda}(n) = \sum_{\ell = 2}^{\lambda} \digit_{\ell}(n)$, and also write $f_{\lambda}(n) = \e(\vartheta\smallspace\sz_{\lambda}(n))$.
In total, this means that we are interested in estimating
\begin{align*}
	\sum_{M < m \leq 2 M} \sum_{\abs{r}\leq R} \left\lvert\sum_{M_1 < p \leq 2 M_1} \Delta\bigl(\overline{f_{\lambda}}; -pr\bigr)(pm) \Delta\bigl(f_{\lambda}; -qr\bigr)(qm)\right\rvert.
\end{align*}

Using again the Cauchy--Schwarz inequality and changing the order of summation allows us to take an average over $q$.
Thus, we are interested in
\begin{align*}
	\sum_{\lvert r\rvert \leq R} \sum_{M_1 < p,q \leq 2 M_1} \left\lvert\sum_{M < m \leq 2M} \Delta\bigl(\overline{f_{\lambda}}; -pr\bigr)(pm) \Delta(f_{\lambda}; -qr)(qm)\right\rvert.
\end{align*}

As we noted in the Introduction (see~\eqref{eqn_central_motivation}),
one can detect the Zeckendorf digits of an integer $n$ by considering $n\golden \bmod \Z$
(see Chapter~\ref{chap:detection} for more details on this topic).
Therefore, we can replace $f_{\lambda}(n) = \e(\vartheta\smallspace\sz_\lambda(n))$ by a $1$-periodic function 
$\e(\vartheta g_{\lambda}(x))$ such that  $f_{\lambda}(n) = \e(\vartheta g_{\lambda}(n \golden))$ (compare with (\ref{eqgLdef})).
Thus, the innermost sum depends on $m\golden$ and in particular its multiples $pm \golden$ and $qm \golden$.
Since $m\golden$ is uniformly distributed $\bmod\ \Z$, of excellent quality,
we can replace the sum by an integral via the Koksma--Hlawka inequality
(thus, replacing $m\golden$  by $x$).
This leads to
\begin{align*}
	\sum_{\lvert r\rvert \leq R} \sum_{M_1 < p,q \leq 2 M_1} \left\lvert\int_{0}^{1} \Delta\bigl(\e(-\vartheta g_{\lambda}); -pr\golden\bigr)(px) \Delta\bigl(\e(\vartheta g_{\lambda}); -qr\golden\bigr)(qx)\,\mathrm dx\right\rvert.
\end{align*}

In the next step, we will approximate the function $\Delta(\e(\vartheta g_{\lambda}); -qr\golden)(qx)$ by a trigonometric polynomial of degree $H$, where the size of the coefficients is controlled very well.
This can be done using Vaaler polynomials (see Section~\ref{sec_Vaaler}) and leads us to consider
\begin{align*}
	\sum_{\lvert r\rvert \leq R} \sum_{M_1 < p,q\leq 2M_1} \sum_{\lvert h\rvert \leq H} \left\lvert\int_{0}^{1} \Delta\bigl(\e(-\vartheta g_{\lambda}) ; -pr\golden\bigr)(px) \e(hqx)\, \mathrm dx\right\rvert.
\end{align*}
We recall that $\Delta(\e(-\vartheta g_{\lambda}); -pr\golden)(px)$ is $1$-periodic. Thus, substituting $x$ by $y/p$ results in
\begin{align*}
	\sum_{\lvert r\rvert\leq R} &\sum_{M_1 < p,q\leq 2M_1} \sum_{\lvert h\rvert \leq H} \frac{1}{p} \left\lvert\int_{0}^{p}  \Delta\bigl(\e(-\vartheta g_{\lambda}); -pr\golden\bigr)(y) \e(hqy/p)\,\mathrm dy\right\rvert\\
	& = \sum_{\lvert r\rvert\leq R} \sum_{M_1 < p,q\leq 2M_1} \sum_{\lvert h\rvert \leq H} \left\lvert\frac{1}{p} \sum_{n=0}^{p-1} \e\rb{\frac{hqn}{p}}\right\rvert
\\&\times\left\lvert\int_{0}^{1} \Delta\bigl(\e(-\vartheta g_{\lambda}); -pr\golden\bigr)(y) \e(hqy/p)\,\mathrm dy\right\rvert.
\end{align*}

Next, we apply the inequality of Cauchy--Schwarz on the summation over $p$, which enables us to treat the remaining integral and the sum over $n$ independently.
Note that the sum over $\e(hqn/p)$ originates from the problem of independence of $(mp\golden \bmod 1)_{m\in \N}$ and $(mq \golden \bmod 1)_{m\in \N}$.
We use classical results on linear exponential sums,
where it is essential that we have a sum over both $p$ and $q$.
For the integral, we first note that we can get rid of the term $\e(hqy/p)$ by applying the Cauchy--Schwarz inequality.
Moreover, $rp\golden$ is uniformly distributed modulo $1$,
so that we can replace the sum over $r$ by another integral.\footnote{Here we actually cannot apply the Koksma--Hlawka inequality directly,
but need to be more careful, as the error term would be too large.
However, we still find estimates of sufficient quality, using the ``smoothness'' of our function $g_\lambda$.}
The remaining integral resembles an integral version of a Gowers norm and the remaining expression can be treated by classical tools.
One of the most important ingredients is a good estimate for the Gowers $3$-norm of $\e(\vartheta g_{\lambda}(x))$, which we establish in Chapter~\ref{chap_gowers}.
This finishes the treatment of the sums of type~\textrm{II}.

\subsection{Sums over Primes}
Combining Vaughan's identity (Lemma~\ref{le_vaughan}) and our estimates of sums of types~\textrm{I} and~\textrm{II},
(\ref{eqCortypeI}) and (\ref{eqCortypeII}) (setting $\varepsilon = 1/12$), 
we obtain 
\begin{align*}\label{eqn_Mangoldt_orthogonality}
\sum_{n\leq N}\e\bigl(\vartheta\smallspace\sz(n)\bigr)\Lambda(n)
&\ll N^{\frac 34} + (\log N)^{15/4}N^{1-c_1\lVert \vartheta\rVert^2} \\
&+  N (\log N)^5  
\left(N^{- c_2\lVert\vartheta\rVert^2/912} + N^{-17/304} \right)^{1/2} \\
&\ll \bigl(\log N\bigr)^5 N^{1-c\lVert\vartheta\rVert^2}
\end{align*}
for all $N\geq 2$ and $\vartheta\in\mathbb R$, where $c = \min\{ 17/152,c_1,c_2/1824\}$.

The transition to prime numbers is a standard application of summation by parts, which is formalized in~\cite[Lemme~11]{MR2010}:
we have
\[
\sum_{p\leq N}\e\bigl(\vartheta\smallspace\sz(p)\bigr)\leq \frac 2{\log N}
\max_{t\leq N}\left\lvert\sum_{n\leq t}\e\bigl(\vartheta\smallspace\sz(p)\bigr)\Lambda(n)\right\rvert+O\bigl(\sqrt{N}\bigr),
\]
hence the factor $(\log N)^4$ in Proposition~\ref{Promain1}.
This completes the proof of Proposition~\ref{Promain1} from Theorems~\ref{thm_lod} and~\ref{th_sum_2}.

\def\PromaintwoRef{\ref{Promain2}}
\section{Plan of the Proof of Proposition~\PromaintwoRef}\label{sec:proofPromain2} 
The idea is to approximate the Zeckendorf sum-of-digits
$\sz(p) = \sum_{\ijkl=2}^L \delta_{\ijkl}(p)$ of a \emph{random prime number} $p$
by a sum of random variables $Z_{\ijkl}$ that mimic the random properties of the digits $\delta_{\ijkl}(p)$.
Since $\delta_{\ijkl+1}(n) = 1$ implies $\delta_k(n) = 0$,
it is clear that the random variables $Z_{\ijkl}$ will not be independent.
Actually, if we consider all integers $n\le x$,
it is not difficult to see that the digits $\delta_{\ijkl}(n)$
behave almost like a stationary Markov process (see \cite{DS02} and Section~\ref{sec:8.1}).
It is therefore not unexpected that the digits $\delta_k(p)$ of primes $p$ behave in a similar way.
Consequently the Zeckendorf sum-of-digits function (of primes)
should behave like the sum of a Markov process,
namely like a (properly scaled) Gaussian distribution.
Proposition~\ref{Promain2} is precisely a quantitative version of this heuristic consideration.

Suppose that every $p\le x$ is considered to be equally likely. Then the sum
\[
\frac 1{\pi(x)} \sum_{p\le x} \e \bigl(\vartheta\smallspace\sz (p)\bigr)
\]
is just the characteristic function of this distribution of $\sz (p)$. 
Since we expect that $\sz(p)$ can be approximated by a sum over a stationary Markov process,
the expected value and variance of this distribution should be proportional to the number
$L = \log_\golden x + O(1)$ of digits: $\approx L\mu$ and $\approx L \sigma^2$, respectively,
where $\mu = 1/(\golden^2+1)$ and $\sigma^2 = \golden^3/(\golden^2+1)^3$ (see Section~\ref{sec:8.1}).
Thus it is reasonable to consider the normalized random variable
\[
\frac{\sz(p) - L \mu}{ \sqrt{ L \sigma^2}}.
\]
The corresponding characteristic function is
\[
\phi_1(t) = \frac 1{\pi(x)} \sum_{p\le x} e^{it (\sz(p) - L \mu)/\sqrt{ L \sigma^2} }
\]
and Proposition~\ref{Promain2} just says that
\begin{equation}\label{eqphi1est}
\phi_1(t)
= e^{-t^2/2} \, \left ( 1 +  O \left( \frac {t^2}{{\log x}} \right)+  O \left( \frac {|t|^4}{{(\log x)^{1/2}}} \right) \right)
+ O\left( \frac {|t|}{(\log x)^{\frac 12 - \nu}} \right),
\end{equation}
uniformly for $\lvert  t \rvert\le (\log x)^{\eta}$
(we just have to substitute $\vartheta = t/(2\pi \sigma(\log_\golden x)^{1/2})$).
Note that the asymptotic leading term $e^{-t^2/2}$ is just the characteristic function of the Gaussian distribution.
This is precisely the expected Gaussian behavior.

It turns out that the behavior of the least significant digits as well as the most significant digits is slightly different from the distribution of a \emph{typical digit}.
Therefore one is led to cut off the first and last $L^\nu$ digits,
where $0< \nu < \frac 12$.
More precisely one considers the truncated sum-of-digits function
\[
\sz'(n) = \sum_{L^\nu \le k \le L-L^{\nu}} \delta_k(n)
\]
and the characteristic function of the corresponding normalized distribution
\[
\phi_2(t) = \frac 1{\pi(x)} \sum_{p\le x} e^{it (\sz'(p) - L' \mu)/\sqrt{ L' \sigma^2} },
\]
where $L' =  \# \{ j \in \Z : L^{\nu} \le j \le L- L^{\nu} \}
=  L - 2 L^{\nu} + O(1)$. 

It is easy to show (see Lemma~\ref{Le1}) that $\phi_1(t)$ and $\phi_2(t)$ are very close to each other:
\begin{equation}\label{eqphi12}
\bigl\lvert\phi_1(t) - \phi_2(t)\bigr\rvert = O\left( \frac {\lvert t\rvert}{(\log x)^{\frac 12 - \nu}} \right).
\end{equation}

Thus, it remains to consider $\phi_2(t)$.
As indicated above, the advantage of the use of $\sz'$ is that the digits $\delta_k$ for $L^\nu \le k \le L-L^\nu$ have no side effects in contrast to the first and last digits. 

Let $\overline T_x$ denote the sum 
\[
\overline T_x \coloneqq \sum_{ L^{\nu} \le \ijkl \le L- L^{\nu}} Z_{\ijkl},
\]
where $(Z_{\ijkl})_{\ijkl \ge 0}$ is the stationary Markov process defined by (\ref{eqZ1})--(\ref{eqZ3}).
Then by standard means (see Lemma~\ref{Le2}) it follows that the characteristic function of the normalized random variable $(\overline T_x - L' \mu)/(L' \sigma^2)^{1/2}$ satisfies
\begin{align*}
\phi_3(t) &= \mathbb{E}\, e^{it (\overline T_x - L' \mu)/(L' \sigma^2)^{1/2}}\\
&= e^{-t^2/2} \, \left ( 1 + O \left( \frac {t^2}{\log x} \right) + O \left( \frac {|t|^3}{(\log x)^{1/2}} \right) \right),
\end{align*}
that is, the sums $\overline T_x$ satisfy an asymptotic central limit theorem.

The main step in the proof of Proposition~\ref{Promain2} is to compare $\phi_2(t)$ and $\phi_3(t)$.
This is done in Proposition~\ref{Pro3}:
\begin{equation}\label{eqphicompare}
\bigl\lvert \phi_2(t) - \phi_3(t)\bigr\rvert = O \left( \lvert t\rvert e^{-c_1L^{\kappa}} \right),
\end{equation}
uniformly for real $t$ with $\lvert t\rvert\le L^{\tau}$
(where $\tau$ and $\kappa$ satisfy $0<2\eta < \kappa < \frac 13\nu$
and $c_1$ is a positive constant that depends on $\tau$ and $\kappa$).

Obviously, by putting~\eqref{eqphi12} and~\eqref{eqphicompare} together, this proves~\eqref{eqphi1est} and consequently Proposition~\ref{Promain2}.

\medskip

The proof of~\eqref{eqphicompare} relies on a moment comparison method.
By Taylor's expansion it follows that the difference of two characteristic functions can be compared with (for any integer $D>0$)
\begin{align*}
\mathbb{E} e^{it X} -  \mathbb{E} e^{it Y} &=
\sum_{d< D} \frac {(it)^d}{d!}\left( \mathbb{E}\, X^d - \mathbb{E}\, Y^d \right)\\
&+ O\left( \frac {\lvert t\rvert^D}{D!}\left| \mathbb{E}\, \lvert X\rvert^D - \mathbb{E}\, \lvert Y\rvert^D \right| + 2 \frac {\lvert t\rvert^D}{D!} \mathbb{E}\, \lvert Y\rvert^D \right).
\end{align*}
In particular, we will apply this for $X = (\sz'(p) - L'\mu)/ (L'\sigma^2)^{1/2}$ and
$Y = (\overline T_x - L'\mu)/ (L'\sigma^2)^{1/2}$. 

Lemma~\ref{Le7} states that the corresponding moments of $X$ and $Y$ are actually very close to each other:
\[
\mathbb{E}\, X^d = \mathbb{E}\, Y^d + O \left( e^{-\frac 12 L^\rho}  \right)
\]
uniformly for $1\le d \le L^\kappa$, where $0 < \kappa < \rho < \frac 13 \nu$.
Thus, Lemma~\ref{Le7} (together with a suitable estimate for $\mathbb{E}\, \lvert Y\rvert^D$) proves Proposition~\ref{Pro3}.

\medskip

The proof of Lemma~\ref{Le7} relies on a \emph{Key Lemma}, Lemma~\ref{Le6},
which compares the joint distribution of the Zeckendorf digits of primes with the distribution of the Markov process. It says that 
\begin{align*}
\hspace{4em}&\hspace{-4em}
\frac 1{\pi(x)} \# \bigl\{ p \le x : \delta_{\ijkl_1}(p) = \nu_1,\ldots, \delta_{\ijkl_d}(p) = \nu_d \bigr\}  \\
&= \prob\bigl[Z_{\ijkl_1} = \nu_1, \ldots, Z_{\ijkl_d} = \nu_d\bigr] + O\left( e^{-L^\rho} \right),
\end{align*}
uniformly for $1\le d\le L^{\kappa}$, $L^{\nu} \le \ijkl_1,\ijkl_2,\ldots,\ijkl_d \le L - L^{\nu}$, 
and $\nu_1,\nu_2,\ldots,\nu_d \in \{0,1\}$, where $0< \kappa < \rho < \frac 13 \nu$. 
By expanding the moments $\mathbb{E}\, X^d$ and $\mathbb{E}\, Y^d$ it is easy to see
that Lemma~\ref{Le6} implies Lemma~\ref{Le7} (see the short proof of Lemma~\ref{Le7}).

Thus, it remains to prove the {\it Key Lemma} (Lemma~\ref{Le6}). 

The underlying idea is to use Lemma~\ref{Letiling} to detect a digit.
Let us assume for a moment that we have a precise property of the form
\begin{equation}\label{eqassume}
\delta_\ijkl(n) = 1  \quad\mbox{if and only if}\quad \left( \bigl\{ n \golden^{-\ijkl} \bigr\}, \bigl\{ n \golden^{-\ijkl-1} \bigr\} \right) \in (A_1 \bmod 1),
\end{equation}
where $A_1$ is a certain rectangle whose edges have slopes $\golden^{-1}$ and $-\golden$; see Lemma~\ref{Letiling};
note that the sets $A_0,A_1$ defined there form a Markov partition of the toral automorphism with matrix
\[
\left(  \begin{array}{cc}  1 & 1 \\ 1 & 0 \end{array} \right).
\]
Furthermore let $\psi(x_1,x_2)$ be the function
\[
\psi(x_1,x_2) = \sum_{m_1,m_2\in \mathbb{Z}} \chi_{A_1}(x_1+m_1,x_2+m_2) = \sum_{h_1,h_2\in \mathbb{Z}} c_{h_1,h_2} e(h_1x_1+h_2x_2),
\]
which is the periodic extension of the characteristic function of $A_1$ with Fourier coefficients $c_{h_1,h_2}$. 
Then (assuming that~\eqref{eqassume} holds) 
\begin{align*}
\#\{p\le x : \delta_{\ijkl}(p) = 1\} &= 
\sum_{p\le x} \psi\bigl(p \golden^{-\ijkl}, p \golden^{-\ijkl-1}\bigr) \\
&=   \sum_{h_1,h_2\in \mathbb{Z}} c_{h_1,h_2} \sum_{p\le x} e \left( \bigl(h_1 \golden^{-\ijkl} + h_2 \golden^{-\ijkl-1}\bigr) p \right).
\end{align*}
Thus, (in principle) we have transformed the problem into exponential sums of the form
\begin{equation}\label{eqexpsumalpha}
S= \sum_{p\le x}  \e(\theta p)
\end{equation}
with some (usually) irrational number $\theta$. (Note that $h_1 \golden^{-\ijkl} + h_2 \golden^{-\ijkl-1} = 0$ if and and only if $h_1 = h_2 = 0$.)
It is well known that $S = o(\pi(x))$ for every given irrational $\theta$.
Hence, it is expected that
\[
\#\bigl\{p\le x : \delta_{\ijkl}(p) = 1\bigr\} \sim c_{0,0} \pi(x) = \frac{\pi(x)}{\golden^2+1},
\]
which turns out to be true if $\ijkl$ is not too close to $0$ or to $\log_\golden x$.

In fact our sketch has been a bit imprecise at several places.
First the relation (\ref{eqassume}) is not completely correct as it stands. 
There are only valid implications if we make the set $A_1$ slightly smaller or larger (by an amount of size $O(\gamma^{-\ijkl})$
- see Lemma~\ref{Letiling}). 
Furthermore the Fourier series of $\psi(x_1,x_2)$ is not absolutely convergent
so we cannot directly apply upper bounds for the absolute values of exponential sums.

Both problems can be overcome by {\it smoothing} the characteristic function of $A_1$ so that the Fourier series gets absolutely convergent.
This smoothing gives an error term in the counting problem which can be bounded in the same way as the actual inaccuracy in Lemma~\ref{Letiling}.

Actually the same procedure works if we want to detect several digits 
$\delta_{\ijkl_1}(p),\ldots,$ $\delta_{\ijkl_d}(p)$ at once.
We just have to consider the product of the corresponding (smoothed) characteristic functions.
Fourier analysis therefore leads to exponential sums of type (\ref{eqexpsumalpha}), 
where $\theta$ is of the form
\[
\theta = \sum_{\ell = 1}^d \left( h_{\ell 1} \golden^{-\ijkl_\ell} + h_{\ell 2} \golden^{-\ijkl_\ell-1} \right).
\]
Whereas $h_1 \golden^{-k} + h_2 \golden^{-k-1} = 0$ if and and only if $h_1 = h_2 = 0$, there is no corresponding property if $d> 1$.
Thus, $\theta$ might be zero for several choices of integers $h_{\ell 1},h_{\ell 2}$, $1\le \ell \le d$.
In the proof of Lemma~\ref{Le6} these sets of $2d$-dimensional integer vectors will be denoted by
$\mathcal{M}_0$. 

Summing up, we expect that 
\[
 \# \bigl\{ p \le x : 
\delta_{\ijkl_1}(p) = 1,\ldots, \delta_{\ijkl_d}(p) = 1 \bigr\}  
= \pi(x) \cdot \sum_{(h_{\ell 1},h_{\ell 2})_{1\le \ell \le d} \in \mathcal{M}_0}   
 \prod_{\ell = 1}^d  c_{h_{\ell 1}, h_{\ell 2}} + o(\pi(x))
\]
and that 
\[
\sum_{(h_{\ell 1},h_{\ell 2})_{1\le \ell \le d} \in \mathcal{M}_0}   
 \prod_{\ell = 1}^d  c_{h_{\ell 1}, h_{\ell 2}}  = 
 \prob\bigl[Z_{\ijkl_1} = 1, \ldots, Z_{\ijkl_d} = 1\bigr].
\]

The essential (but quite technical and also lengthy) part of the proof of Lemma~\ref{Le6} is
to make precisely these (and similar) statements rigorous and to quantify the error terms.
As mentioned above one has to smooth out the characteristic functions in order to make the Fourier series absolutely convergent;
the small error in Lemma~\ref{Letiling} has to be taken into account,
and --- most importantly ---
the set $\mathcal{M}_0$ has to be characterized and the exponential sum $S$ has to be bounded.

\section{What is left to prove}
Summarizing, Theorems~\ref{Th1},~\ref{Th2}, and~\ref{Th4} follow from Propositions~\ref{Promain1} and \ref{Promain2}.
Furthermore, Propositions~\ref{Promain1} is a consequence of Theorems~\ref{thm_lod} and~\ref{th_sum_2}.
We will now proceed to the auxiliary chapters (\ref{chap_exponentialsums},~\ref{chap:detection}, and~\ref{chap_gowers}), followed by the proofs of Theorems~\ref{thm_lod} and~\ref{th_sum_2} (Chapters~\ref{chap_lod} and~\ref{ch_type2}).
Finally, in Chapter~\ref{chap_local} we prove Proposition~\ref{Promain2}.


\chapter{Exponential Sums and Uniform Distribution}\label{chap_exponentialsums}

In this chapter we collect useful and mostly well-known results concerning the distribution of points in the unit circle.
\emph{Exponential sums} will feature prominently in these results.
First, in Section~\ref{sec_discrepancy} we will discuss the notion of \emph{discrepancy} and in particular low discrepancy sequences. 
In Section~\ref{sec_Vaaler} we present a useful result on trigonometric approximation by Vaaler,
which we will use to detect points in an interval.
As a consequence, the inequality of Erd\H{o}s--Tur\'an--Koksma can be derived,
which gives an upper bound for the discrepancy of a sequence in terms of exponential sums.
Given a sequence of points $\bfx = (x_1, \ldots, x_N)$ that is uniformly distributed in $[0,1]$,
it is reasonable to expect that one can approximate $\frac{1}{N} \sum_{n=1}^{N} f(x_n)$ by
$\int_{0}^{1} f(x)\,\mathrm dx$.
This can be made precise via the Koksma--Hlawka inequality, which we present in Section~\ref{sec_Koksma}.
In addition to this classical inequality, we present a version that uses some additional smoothness condition for $f$.
In Section~\ref{sec_exp_sum_primes} we will give an upper bound for $\sum_{p \leq x} \e(\vartheta p)$ that works uniformly for $x \geq 2$ and $\vartheta \in \R \setminus \Z$.

\section{Discrepancy}\label{sec_discrepancy}

Let $\mathbf{x} = (\mathbf{x}_j)_{j\in\N}$ be a sequence of points in the $d$-dimensional unit torus $\mathbb{T}^d = \mathbb{R}^d/\mathbb{Z}^d$.
A classical way of measuring the quality of distribution in $\mathbb T^d$ is the discrepancy
\begin{align*}
	D_N(\mathbf{x}) \coloneqq \sup_{\substack{I \subseteq \mathbb{T}^d\\I {\rm an interval} } }
\left\lvert\frac{1}{N} \sum_{n=1}^{N} \chi_I(\mathbf{x}_n) - \lambda_d(I)\right\rvert,
\end{align*}
where $\chi_I$ denotes the characteristic function of $I$ and $I$ is a \emph{$d$-dimensional interval}, that is, $I = [a_1, b_1]\times \cdots \times [a_d, b_d]$.
It is well-known that $\mathbf{x}$ is uniformly distributed in $[0,1]$ if and only if $D_N$ tends to $0$ as $N$ tends to infinity.
Even for $d=1$, there are different notions of discrepancy, but they do not differ by much.
However, when considering $d>1$, there exist other variants of the classical discrepancy which are not as closely related.
One example is the so called \emph{isotropic discrepancy},
\begin{align}\label{eqn_isotropic_def}
	J_N(\mathbf{x}) \coloneqq \sup_{\substack{C\subseteq \mathbb{T}^d\\C {\rm convex} } }
\left\lvert\frac 1N \sum_{n=1}^N \chi_C(\mathbf{x}_n) - \lambda_2(C)\right\rvert.
\end{align}
It is usually difficult to approach the isotropic discrepancy directly.
However, we can use the following inequality (see for example~\cite[Theorem 1.6 (p. 95)]{uniform_distribution}) to relate it to the usual discrepancy (which is much easier to handle):
\begin{align}\label{eq_isotropic}
	D_N(\mathbf{x}) \leq J_N(\mathbf{x}) \leq \bigl(4d\sqrt{d}+1\bigr) D_N(\mathbf{x})^{1/d}.
\end{align}
Although the exponent $1/d$ causes considerable loss,
this estimate is often good enough to get meaningful results.

We will also need the special case where the sequence $\mathbf{x}$ is of the form $x_n = (n\alpha \bmod 1)$.
The discrepancy for such sequences is strongly related to the continued fraction expansion of $\alpha$ (see~\cite{S1984}, for example). 
In particular, we have the following result.
\begin{theorem}[Theorem 3.4 (p. 125) in~\cite{uniform_distribution}]\label{th_bounded_quotients}
	Suppose the irrational $\alpha = [a_0, a_1, \ldots]$ has bounded partial quotients.
Then the discrepancy $D_N(\mathbf{x})$ of $\mathbf{x} = (n\alpha \bmod 1)_{n\in \N}$ satisfies
$ND_N(\mathbf{x}) = O(\log(N))$.
More precisely, if $a_i \leq K$ for $i \geq 1$, we have
	\begin{align*}
		N D_N(\mathbf{x}) &\leq 3 + \left(\frac{1}{\varphi} + \frac{K}{\log(K+1)}\right) \log(N)\\
			&\ll K \log^+(N),
	\end{align*}
	where $\varphi = \log\bigl(\frac{1+\sqrt{5}}{2}\bigr)$.
\end{theorem}

The following well-known lemma gives a correspondence between the continued fraction of $\alpha$ and the quality of approximation by rational numbers.
\begin{lemma}\label{le_bounded_quotients}
	Let $\alpha = [a_0, a_1, \ldots]$ be an irrational number having the property that there exists $\delta > 0$ with
	\begin{align*}
		\left\lvert\alpha - \frac{p}{q}\right\rvert > \frac{\delta}{q^2},
	\end{align*}
	for all rationals $\frac{p}{q}$.
	Then $a_i \leq \frac{1}{\delta}$ for $i \geq 1$.
\end{lemma}
For completeness, we present the short proof of this result.
\begin{proof}
	We denote by $\frac{p_i}{q_i}$ the $i$-th convergent of $\alpha$.
It is well-known that for all $i \geq 0$,
	\begin{align*}
		\frac{1}{q_i(q_i + q_{i+1})} < \left\lvert\alpha - \frac{p_i}{q_i}\right\rvert < \frac{1}{q_i q_{i+1}}.
	\end{align*}
	This implies in particular
	\begin{align*}
		\frac{\delta}{q_i} < \frac{1}{q_{i+1}}.
	\end{align*}
	Moreover, we know that $q_{i+1} = a_{i+1} q_{i} + q_{i-1}$. This gives
	\begin{align*}
		a_{i+1} q_i + q_{i-1} < \frac{1}{\delta} q_i,
	\end{align*}
	and the result follows as $q_i \geq 1$ for all $i \geq 0$ and $q_{-1} = 0$.
\end{proof}

\section{Vaaler polynomials and the Erd\H{o}s--Tur\'an--Koksma inequality}\label{sec_Vaaler}
We start this section by presenting a classical method to detect real numbers in an
interval modulo $1$ by means of exponential sums, due to Vaaler (see~\cite[Theorem 19]{Vaaler1985} and also~\cite[Theorem A.6]{GK1991}).
We give a slightly different formulation of the original result which is better suited for our applications. This version appeared to our knowledge first in~\cite{Mauduit2015}. It was subsequently used in~\cite{DMR2019, Mauduit2018, Muellner2017, Muellner2018} and was also mentioned in~\cite{Hanna2017}. 
Let $I \subset \R$ be an interval and denote by $\chi_{I}$ the
characteristic function of $I$ modulo $1$.

\begin{theorem}\label{th_vaaler}
Let $I \subset \R$ be an interval of length $\ell$. 
Then for every integer $H\geq 1$, 
there exist real-valued trigonometric polynomials $A_{I,H}(x)$ and $B_{I,H}(x)$ 
such that for all $x\in\R$
\begin{equation}\label{eq:vaaler-approximation}
  \left\lvert\chi_I(x) - A_{I,H}(x)\right\rvert
  \leq
  B_{I,H}(x).
\end{equation}
The trigonometric polynomials are defined by
\begin{equation}\label{eq:definition-A-B}
\begin{aligned}
  A_{I,H}(x) &= \sum_{\lvert h\rvert\leq H}  a_h(I,H) \e(h x),\\
  B_{I,H}(x) &= \sum_{\lvert h\rvert\leq H} b_h(I,H) \e(h x),
\end{aligned}
\end{equation}
with coefficients $a_h(I,H)$ and $b_h(I,H)$ satisfying
\begin{equation}\label{eq:vaaler-coef-majoration}
  a_0(I,H) = \ell,\quad
  \left\lvert a_h(I,H)\right\rvert \leq \min\rb{\ell,\tfrac{1}{\pi\abs{h}}},\quad
  \left\lvert b_h(I,H)\right\rvert \leq \tfrac{1}{H+1},
\end{equation}
for all $h$.
\end{theorem}

This approach using exponential sums can be utilized to find an upper bound for the discrepancy of a sequence $\bfx$.
This is particularly useful, as finding the exact value of the discrepancy is usually relatively difficult.
In practice, having an upper bound is often sufficient. 
The following inequality is much older than Theorem~\ref{th_vaaler} and due to Erd\H{o}s, Tur\'an, and Koksma.
\begin{lemma}\label{lem_ETK}
Let $d$ be a positive integer. There exists a constant $C$ such that for all integers $N\geq 1$, all sequences $\mathbf x=(\mathbf{x}_1, \ldots, \mathbf{x}_N)$ in $\mathbb R^d$ and all integers $H\geq 1$ we have
\begin{equation}\label{eqn_ETK}
D_N(\mathbf x)\leq C
\left(\frac 1H+
\sum_{0<\lVert\mathbf h\rVert_\infty<H}
\frac 1{r(\mathbf h)}
\left\lvert\frac 1N\sum_{n=1}^{N}\e(\mathbf h\cdot \mathbf{x}_n)
\right\rvert
\right),
\end{equation}
where $r(\mathbf h)=\prod_{1\leq i\leq s}\max\bigl(1,\lvert h_i\rvert\bigr)$
and here ``\,$\cdot$'' denotes the usual dot product of two vectors in $\mathbb R^d$.
\end{lemma}

\begin{remark}
We will use ``$\cdot$'' both for the scalar product of two vectors and the standard multiplication.
To avoid any possible confusion, we will always write vectors in boldface, such that the meaning of ``$\cdot$'' is clear from the context.
\end{remark}

This inequality has been generalized to measurable sets $\Omega$ in~\cite{gigante2011}. We will use a different notation compared to~\cite{gigante2011} to give a more uniform presentation of the results.
\begin{theorem}[Theorem 2.1 in~\cite{gigante2011}]\label{th_general_erdos}
	Let $\bfx = (\bf{x}_1, \ldots, \bf{x}_N)$ be a sequence of points in the $d$-dimensional torus, and let $\Omega$ be a measurable set with measure $\lambda(\Omega)$, and let $F_H(\mathbf{x}) = 4^{-1} \psi(2H \dist(\mathbf{x},\partial \Omega))$ with $H > 0$ and $\psi(t)$ be a function with fast decay at infinity\footnote{A function $f$ has fast decay at infinity, if for any $\alpha > 0$, there exists $c(\alpha)$ such that $\abs{f(t)} \leq c(\alpha) (1+t)^{-\alpha}$ holds for all $t \geq 0$.}, as in the proof of Corollary 1.2 in \cite{gigante2011}. Then
	\begin{align*}
		&\abs{\lambda(\Omega) - \frac{1}{N} \sum_{n=1}^{N} \chi_{\Omega}(x_n)}\\
			&\qquad \leq \bigl\lvert\hat{F}_H(0)\bigr\rvert + \sum_{0 < \norm{\bf{h}}_{2} < H} \rb{\abs{\hat{\chi}_{\Omega}(h)} + \bigl\lvert\hat{F}_H(h)\bigr\rvert} \abs{\frac{1}{N} \sum_{n=1}^{N} \e(\mathbf{h} \cdot \mathbf{x}_n)}.
	\end{align*}
\end{theorem}

\begin{remark}
	The function $\psi(t)$ in Corollary 1.2 in~\cite{gigante2011} can be made explicit, but it is quite involved.
	Moreover, it is not unique (it depends on the choice of $m$ in the proof of Corollary 1.2 in~\cite{gigante2011}).
	However, the concrete choice of $m$ and, therefore, of $\psi$ does not seem very important for our application as it only changes some constants (see also Remark 1.3 in~\cite{gigante2011} which discusses the optimality of this construction). 
\end{remark}

There is actually a nice analogue of Vaaler polynomials hidden in the proof of Theorem~\ref{th_general_erdos}.

\begin{theorem}\label{th_vaaler_general}
	Let $\Omega$ be a measurable set on the $d$-dimensional torus with measure $\lambda(\Omega)$. Then for any $H \geq 1$ there exist trigonometric polynomials $A_{\Omega, H}(x)$ and $B_{\Omega, H}(x)$ such that for all $x \in \R^d$
	\begin{align*}
		\abs{\chi_{\Omega}(x) - A_{\Omega, H}(x)} \leq B_{\Omega,H}(x).
	\end{align*}
	The trigonometric polynomials are defined by
	\begin{align*}
		A_{\Omega, H}(x) &= \sum_{\norm{h}_{2} \leq H} a_h(\Omega, H) \e(\mathbf{h} \cdot \mathbf{x}),\\
		B_{\Omega, H}(x) &= \sum_{\norm{h}_{2} \leq H} b_h(\Omega, H) \e(\mathbf{h} \cdot \mathbf{x}),
	\end{align*}
	with coefficients $a_h(\Omega, H)$ and $b_h(\Omega, H)$ satisfying
	\begin{align*}
		\abs{a_h(\Omega, H)} \leq \bigl\lvert\hat{\chi}_{\Omega}(h)\bigr\rvert,\
		\abs{b_h(\Omega, H)} \leq \bigl\lvert\hat{F}_{H}(h)\bigr\rvert
	\end{align*}
	and $a_0(\Omega, H) = \lambda(\Omega)$.
\end{theorem}
\begin{proof}
	This follows directly from the proof of Theorem 2.1 in~\cite{gigante2011}.
\end{proof}

\subsection{Polyhedra}
The obvious first step toward using Theorem~\ref{th_general_erdos} is finding good estimates for the appearing Fourier coefficients.
Colzani, Gigante, and Travaglini present such estimates in the case where $\Omega$ is a polyhedron.
\begin{lemma}[Lemma 2.8 in~\cite{gigante2011}]
	If $\Omega$ is a polyhedron in $\R^d$ with diameter $\lambda$, then, 
	\begin{align*}
		\abs{\int_{\Omega} \e(\mathbf{h} \cdot \mathbf{x})\,\mathrm d\mathbf{x}} \leq 2 \sum_{\Omega(d) \supset \ldots \supset \Omega(1)} \prod_{j=1}^{d} \min\rb{\lambda, \bigl(2\pi \bigl\lvert P_{\Omega(j)}(\mathbf{h})\bigr\rvert\bigr)^{-1}}.
	\end{align*}
	The sum is taken over all possible decreasing chains of $j$ dimensional faces $\Omega(j)$ of $\Omega$, and $P_{\Omega(j)}$ is the orthogonal projection on the $j$-dimensional subspace parallel to $\Omega(j)$.
\end{lemma}

\begin{lemma}[Lemma 2.9 in~\cite{gigante2011}]\label{le_H_general}
	Let $\Omega$ be a convex polyhedron in $\R^d$ with diameter $\lambda$. For any $j = 1, 2, \ldots, d-1$, let $\{A(j)\}$ be the collection of all $j$-dimensional subspaces which are intersections of a number of subspaces parallel to the faces of $\Omega$. Finally, let $\psi(t)$ be a function with fast decay at infinity. Then, there exists a positive constant $c$, which depends on $d$ and $\psi(t)$, but not on $\Omega$, such that for every $H > 0$,
	\begin{align*}
		&\abs{\int_{\R^d} \psi\bigl(H \cdot \dist(x, \partial \Omega)\bigr) \e(\mathbf{h} \cdot \mathbf{x})\,\mathrm d\mathbf{x}} \\
			&\qquad \leq c \sum_{j=0}^{d-1} \sum_{A(j) \supset \ldots \supset A(1)} H^{j-d} \prod_{k=1}^{j} \min\rb{\lambda, \rb{2\pi \abs{P_{A(k)}(\mathbf{h})}}}.
	\end{align*}
	When $j = 0$ the inner sum of products is intended to be the number of vertices of the polyhedron, and when $1 \leq j \leq d-1$ the inner sum is taken over all possible decreasing chains of $j$-dimensional subspaces $\{A(j)\}$ and $P_{A(j)}$ is the orthogonal projection on $A(j)$.
\end{lemma}
This shows in particular that $\bigl\lvert \hat{F}_H(0)\bigr\rvert\ll H^{-1}$.

\subsection{Parallelotopes}
For our application, we only consider the case, when $\Omega$ is a parallelotope.
We say a parallelotope $P$ has edges $\mathbf{v}_1, \ldots, \mathbf{v}_d$ if there exists $\mathbf{x}_0$ such that
\begin{align*}
	P \coloneqq \bigl\{\mathbf{x}_0 + t_1 \mathbf{v}_1 + \ldots + t_d \mathbf{v}_d: t_1, \ldots, t_d \in [0,1]\bigr\}.
\end{align*}
In this case, we can give even sharper estimates for the Fourier coefficients.
\begin{lemma}\label{le_fourier_parallel}
	Let $P$ be a $d$ dimensional parallelotope with edges $\mathbf{v}_1, \ldots, \mathbf{v}_d$. Then
	\begin{align*}
		\abs{\hat{\chi}_P(\mathbf{h})} \leq \mu(P) \prod_{i=1}^{d} \frac{1}{\max(1, \pi \cdot \abs{\mathbf{h} \cdot \mathbf{v}_i})}.
	\end{align*}
\end{lemma}
\begin{proof}
	Transforming the integral, we obtain
	\begin{align*}
		\int_P \e(\mathbf{h} \cdot \mathbf{x})\,\mathrm d\mathbf{x} &= \int_{[0,1]^d} \e\bigl(h \cdot (\mathbf{x}_0 + t_1 \cdot \mathbf{v}_1 + \ldots + t_d \cdot \mathbf{v}_d)\bigr) \abs{\det(\mathbf{v}_1, \ldots, \mathbf{v}_d)}\,\mathrm dt_1 \cdots\,\mathrm dt_d\\
			&= \mu(P) \e(\mathbf{h}\cdot \mathbf{x}_0) \prod_{j=1}^{d} \int_{0}^{1} \e(t_j (\mathbf{h} \cdot \mathbf{v}_j))\,\mathrm dt_j.
	\end{align*}

	It is clear that $\abs{\int_0^1 \e(t \cdot r)\,\mathrm dt} \leq 1$. Moreover, if $r \neq 0$,
	\begin{align*}
		\int_0^1 \e(t \cdot r)\,\mathrm dt &= \frac{\e(r) - \e(0)}{2 \pi i r} = \e(r/2) \frac{\sin(\pi r)}{\pi r},
	\end{align*}
	from which the result follows immediately.
\end{proof}

\begin{lemma}\label{le_parallelotope_H}
	Let $P$ be a $d$ dimensional parallelotope with edges $\mathbf{v}_1, \ldots, \mathbf{v}_d$. Then
	\begin{align*}
		\bigl\lvert\hat{F}_H(\mathbf{h})\bigr\rvert \ll_d \prod_{i=1}^{d} \frac{\lVert \mathbf{v}_j\rVert_2}{\max(1, \abs{\mathbf{h} \cdot \mathbf{v}_i})},
	\end{align*}
	holds uniformly for $\norm{\mathbf{h}}_{2} \leq H$, where the implied constant only depends on $d$.
\end{lemma}
\begin{proof}
	Since $\abs{\nabla \dist(\mathbf{x}, \partial \Omega)} = 1$, the coarea formula gives
	\begin{align*}
		\int_{\R^d} &\psi\bigl(H \cdot \dist(\mathbf{x}, \partial \Omega)\bigr) \e(\mathbf{h} \cdot \mathbf{x}) d\mathbf{x}\\
			&= \int_0^{\infty} \rb{\int_{\dist(\mathbf{x},\partial \Omega) = t} \e(\mathbf{h} \cdot \mathbf{x})\,\mathrm d\mathbf{x} } \psi(Ht)\,\mathrm dt.
	\end{align*}
	Next we consider separately the contributions of $\dist(\mathbf{x}, \partial \Omega) = t$ inside and outside of $\Omega$.
	The level sets inside of $\Omega$ are again parallelotopes with edges parallel to $\mathbf{v}_1, \ldots, \mathbf{v}_d$.
	The level sets outside of $\Omega$ are slightly more complicated. They consist of a union of sums of $j$-dimensional faces (parallel to the $j$-dimensional faces of $\Omega$) and portions of $d-j-1$ dimensional spherical surfaces of radius $t$.

	Let us fix one such set $A$, with distance $t$ to $\Omega$. We assume without loss of generality that $A$ is the sum of a face parallel to the face $A(j)$ spanned by $\mathbf{v}_1, \ldots, \mathbf{v}_j$ and a portion of a $d-j-1$-dimensional spherical surface of radius $t$.
	We furthermore use the same estimate as proved in Lemma~\ref{le_fourier_parallel}, yielding
	\begin{align*}
		\abs{\int_{A} \e(\mathbf{h} \cdot \mathbf{x})\,\mathrm d\mathbf{x}} &\leq \abs{\int_{A(j)} \e(\mathbf{h} \cdot \mathbf{x})\,\mathrm d\mathbf{x}} \cdot t^{d-j-1} \frac{2 \pi^{(d-j)/2}}{\Gamma((d-j)/2)}\\
			&\ll t^{d-j-1}\mu\bigl(A(j)\bigr) \prod_{i = 1}^{j} \frac{1}{\max(1, \abs{\mathbf{h} \cdot \mathbf{v}_i})}.
	\end{align*}
	This allows us to show
	\begin{align*}
\hspace{4em}&\hspace{-4em}
		\abs{\int_{0}^{\infty} \rb{\int_{A} \e(\mathbf{h} \cdot \mathbf{x})\,\mathrm d\mathbf{x}} \psi(Ht) dt} \\&\ll \mu\bigl(A(j)\bigr)\prod_{i=1}^{j} \frac{1}{\max(1, \lvert \mathbf{h} \cdot \mathbf{v}_i\rvert)} \int_0^{\infty} \abs{\psi(Ht)} t^{d-j-1}\,\mathrm dt\\
			&=\mu\bigl(A(j)\bigr) \prod_{i=1}^{j} \frac{1}{\max(1, \lvert \mathbf{h} \cdot \mathbf{v}_i\rvert)}  H^{j-d} \int_0^{\infty} \abs{\psi(s)} s^{d-j-1}\,\mathrm ds\\
			&\ll \mu\bigl(A(j)\bigr) \prod_{i=1}^{j} \frac{1}{\max(1, \lvert \mathbf{h} \cdot \mathbf{v}_i\rvert)}  H^{j-d},
	\end{align*}
	where the last inequality is a direct consequence of $\psi$ having fast decay at infinity.
	By applying the Cauchy--Schwarz inequality we see that $$\abs{\mathbf{h} \cdot \mathbf{v}_i} \leq \norm{\mathbf{h}}_2 \norm{\mathbf{v}_i}_2 \leq H \norm{\mathbf{v}_i}_2.$$
	The result follows now immediately, as $\mu(A(j))\leq \norm{\mathbf{v}_1}_2 \cdot \ldots \cdot\norm{\mathbf{v}_j}_2$:
	\begin{align*}
		\mu\bigl(A(j)\bigr) \prod_{i=1}^{j} \frac{1}{\max(1, \lvert \mathbf{h} \cdot \mathbf{v}_i\rvert)}  H^{j-d} &\leq \prod_{i=1}^{j} \norm{\mathbf{v}_i}_2 \cdot \prod_{i=1}^{j} \frac{1}{\max(1, \abs{\mathbf{h} \cdot \mathbf{v}_i})} \cdot \prod_{i=j+1}^{d} \frac{1}{H}\\
		&\leq \prod_{i=1}^{j} \frac{\norm{\mathbf{v}_i}_2}{\max(1,\abs{\mathbf{h} \cdot \mathbf{v}_i})} \cdot \prod_{i=j+1}^{d} \frac{\norm{\mathbf{v}_i}_2}{\max(1,\abs{\mathbf{h} \cdot \mathbf{v}_i})}\\
		&= \prod_{i=1}^{d} \frac{\norm{\mathbf{v}_i}_2}{\max(1,\abs{\mathbf{h} \cdot \mathbf{v}_i})}.
	\end{align*}
\end{proof}

%

\subsection{An Erd\H{o}s--Tur\'an--Koksma inequality for parallelotopes}
We first introduce the discrepancy of a sequence $\mathbf{x} = (\mathbf{x}_j)_{j\in\N}$ of points in the $d$-dimensional unit torus $\mathbb{T}^d = \mathbb{R}^d/\mathbb{Z}^d$ with respect to parallelotopes. Therefore, we define $\mathcal{P}$ as the set of parallelotopes in $\mathbb{T}^d$ with edges parallel to $\mathbf{w}_1, \ldots, \mathbf{w}_d$ and
\begin{align*}
	D_N(\mathbf{x}, \mathcal{P}) \coloneqq \sup_{\substack{P \subseteq \mathbb{T}^d\\P  \in \mathcal{P}}}
\left\lvert\frac{1}{N} \sum_{n=1}^{N} \chi_P(\mathbf{x}_n) - \lambda_d(I)\right\rvert,
\end{align*}

This allows us to prove a version of the Erd\H{o}s--Tur\'an--Koksma inequality for parallelotopes.

\begin{theorem}\label{thm_ETK_parallelotope}
	Let $\mathcal{P}$ be the set of $d$-dimensional parallelotopes with edges parallel to $\mathbf{w}_1, \ldots, \mathbf{w}_d$, where $\mathbf{w}_1, \ldots, \mathbf{w}_d$ are linearly independent unit vectors. Then for any sequence $\mathbf{x} = (\mathbf{x}_1, \ldots, \mathbf{x}_N) \in \R^d$ and $H \in \N$,
	\begin{align*}
		D_N(\mathbf{x}, \mathcal{P}) \ll \frac{1}{H} + \sum_{0 < \norm{\mathbf{h}}_{2} \leq H} \prod_{i=1}^{d} \frac{1}{\max(1, \abs{\mathbf{h} \cdot \mathbf{w}_i})} \frac{1}{N} \abs{\sum_{n=1}^{N} \e(\mathbf{h} \cdot \mathbf{x}_n)},
	\end{align*}
	where the implied constant only depends on $d$.
\end{theorem}
\begin{proof}
	This follows immediately from Theorem~\ref{th_general_erdos}, Lemma~\ref{le_fourier_parallel}, Lemma~\ref{le_parallelotope_H}, and Lemma~\ref{le_H_general} for $h = 0$.
\end{proof}

\begin{remark}
	If we change the range of summation from $0<\norm{\mathbf{h}}_2\leq H$ to $0<\norm{\mathbf{h}} \leq H$, where $\norm{.}$ denotes any norm on $\R^d$, the same statement holds, where at most the implied constant changes (since all norms on $\R^d$ are equivalent).\\
When we consider the case where $\mathbf{w}_i$ is the $i$-th unit vector, i.e. $\mathcal{P}$ denotes the set of intervals, we immediately recover the Erd\H{o}s--Tur\'an inequality.
\end{remark}

\section{The Koksma--Hlawka inequality}\label{sec_Koksma}
One of the main applications of the notion of discrepancy is within numerical integration. 
The discrepancy of the sequence $\bfx = (x_1, \ldots, x_N)$ and the so called \emph{total variation} of a function $f: \R \to \C$ can be combined to yield a sharp bound on the error when approximating $\sum_{n=1}^{N} f(x_n)$ by $\int_{[0,1]}f(x)\,\mathrm dx$.
\begin{definition}
	The \emph{total variation} of a function $f: [a,b] \to \C$ is defined by
	\begin{align*}
		V_a^b(f) = \sup_{P \in \mathcal{P}} \sum_{i=0}^{n_P-1} \abs{f(x_{i+1}) - f(x_i)},
	\end{align*}
	where the supremum runs over the set of all partitions
\[\mathcal{P} \coloneqq \bigl\{P = (x_j)_{0\leq j\leq n_P}: n_P \in \N, a=x_0<x_1<\cdots<x_{n_P}=b\bigr\}.\]
\end{definition}

There is a useful way to describe the total variation of a function if it is continuously differentiable.
\begin{lemma}~\label{le_variation_diff}
	Let $f \in C^1([a,b])$. Then
	\begin{align*}
		V_{a}^{b}(f) = \int_{a}^{b} \abs{f'(x)}\,\mathrm dx.
	\end{align*}
\end{lemma}

Now we state the aforementioned estimate on the error of numerical integration, which is originally due to Koksma.
\begin{theorem}[Theorem 5.1 (p. 143) in~\cite{uniform_distribution}]\label{th_koksma}
	Let $f: [0,1] \to \C$ be a function of bounded variation, and suppose we are given a sequence $\bfx = (x_1, \ldots, x_N)$ in $[0,1]$. Then
	\begin{align*}
		\abs{\frac{1}{N} \sum_{n=1}^{N} f(x_n) - \int_0^1 f(t)\,\mathrm dt} \leq V_0^1(f) D_N(\bfx).
	\end{align*}
\end{theorem}

There is also a well-known $d$-dimensional version of this inequality, which is known as the \emph{Koksma--Hlawka} inequality (see for example~\cite[Theorem 5.5 (p. 151)]{uniform_distribution}).

We will also encounter the case where the bound provided by Theorem~\ref{th_koksma} is not better than the trivial one.
In this case we use some additional smoothness condition satisfied by $f$, which still allows us to improve over the trivial result.

\begin{lemma}\label{lem_1}
	Let $F: \mathbb{R} \to D$ be a Lipschitz continuous function with Lipschitz constant $L>0$, such that $\int_0^1 \abs{F(x)}\,\mathrm dx = \alpha$.
	Then
	\begin{align*}
		\frac{1}{N} \sum_{1\leq n \leq N} \left\lvert F(x_n)\right\rvert \ll
\left(\frac{\alpha^2}{L\,D_N(x_n)}\right)^{1/3} + \bigl(\alpha L\,D_N(x_n)\bigr)^{1/3}.
	\end{align*}
\end{lemma}
\begin{remark}\label{rem_1}
	The Koksma--Hlawka inequality gives under the same conditions
	\begin{align*}
		\frac{1}{N} \sum_{1\leq n \leq N} \left\lvert F(x_n)\right\rvert \leq \alpha + O\bigl( L\,D_N(x_n)\bigr).
	\end{align*}
	This estimate is often stronger, but only useful when $L\,D_N(x_n) \leq 1$, while the result of Lemma~\ref{lem_1} is aimed at the case $L\,D_N(x_n) \geq 1$. 
\end{remark}

%

\begin{proof}
	If $\alpha = 0$, we see that $F(x) = 0$ for all $x$ and the result holds trivially. Thus we assume $\alpha > 0$ from now on.
	We define $M \coloneqq \{y: \lvert F(y)\rvert \geq \alpha_1\}$.
We see directly that necessarily $\lvert M\rvert \leq \alpha/\alpha_1 \coloneqq \alpha_2$.
	Moreover, we define $M' \coloneqq \{y: \lvert F_y\rvert \geq \beta\}$ for some $\beta > \alpha_1$.
	We see that
	\begin{align*}
		\frac{1}{N} \sum_{1\leq n \leq N} \abs{F(x_n)} &= \frac{1}{N} \sum_{\substack{1\leq n \leq N\\ x_n \in M'}} \abs{F(x_n)} + \frac{1}{N} \sum_{\substack{1\leq n \leq N\\ x_n \notin M'}} \abs{F(x_n)}\\
			&\leq \frac{1}{N} \sum_{\substack{1\leq n \leq N\\ x_n \in M'}} \abs{F(x_n)} + \beta.
	\end{align*}
	We aim to show that only few $x_n$ belong to $M'$.
	Therefore, we are interested in the structure of $M'$. Since $F$ is continuous, we know that $M'$ consists of a disjoint union of intervals $M' = I_1 \cup \cdots \cup I_r$, for some $r \in \N$. Each interval $I_i$ is contained in a maximal interval $J_i$ that is a subset of $M$. (It is obviously possible that there exist $i_1 \neq i_2$ such that $J_{i_1} = J_{i_2}$.)

	Since every $J_i$ contains a point $y_i$ such that $\abs{F(y_i)} \geq \beta$, we know that the length of $J_i$ is at least $2\cdot \frac{\beta - \alpha_1}{L}$, as $F$ is Lipschitz continuous with Lipschitz constant $L$.
	We remove duplicates from the list of $J_i$ to obtain disjoint sets $J'_1, \ldots, J'_s$ such that for every $i$ there exists a $j$ such that $I_i \subset J'_j$. Since every $J'_j \subset M$, we find that
	\begin{align*}
		s \cdot 2\frac{\beta - \alpha_1}{L} \leq \sum_{1\leq j \leq s} \abs{J'_j} \leq \abs{M} \leq \alpha_2.
	\end{align*}
	Thus, we have $s \leq \frac{L \alpha_2}{2(\beta- \alpha_1)}$.
	Now we use the discrepancy to obtain
	\begin{align*}
		\frac{1}{N} \bigl\lvert\bigl\{n\leq N : x_n \in M'\bigr\}\bigr\rvert &\leq \sum_{1\leq j \leq s} \frac{1}{N} \bigl\lvert\bigl\{n \leq N: x_n \in J'_j\bigr\}\bigr\rvert\\
			&\leq \sum_{1\leq j \leq s} \abs{J'_j} + s \cdot D_N(x_n)\\
			&\leq \alpha_2 + \frac{L \alpha_2}{2(\beta - \alpha_1)} D_N(x_n).
	\end{align*}
	This gives
	\begin{align*}
		\frac{1}{N} \sum_{1\leq n \leq N} \abs{F(x_n)} \leq \alpha_2 + \frac{L \alpha_2}{2(\beta - \alpha_1)} D_N(x_n) + \beta.
	\end{align*}
	Balancing the second and third term on the right hand side leads to
	\begin{align*}
		\beta = \frac{2\alpha_1 + \sqrt{4 \alpha_1^2 + 8 \alpha_2 L D_N(x_n)}}{4} \leq \alpha_1 + \sqrt{\alpha_2 L D_N(x_n)/2},
	\end{align*}
	where $\beta > \alpha$ since $\alpha_2 > 0, L>0$ and $D_N(x_n)>0$.
	This gives in total
	\begin{align*}
		\frac{1}{N} \abs{F(x_n)} \leq \alpha_2 + 2 \alpha_1 + \sqrt{2 \alpha_2 L D_N(x_n)}.
	\end{align*}
	Balancing again the second and third term leads to
	\begin{align*}
		\alpha_1 = (\alpha L D_N(x_n))^{1/3}>0
	\end{align*}
	and in total
	\begin{align*}
		\frac{1}{N} \sum_{1\leq n \leq N} \abs{F(x_n)} \ll \rb{\frac{\alpha^2}{LD_N(x_n)}}^{1/3} + \rb{\alpha L D_N(x_n)}^{1/3}.
	\end{align*}
\end{proof}

\begin{corollary}\label{cor_1}
	Combining Lemma~\ref{lem_1} and Remark~\ref{rem_1} shows under the same conditions
	\begin{align*}
		\frac{1}{N} \sum_{1\leq n \leq N} \abs{F(x_n)} \ll \alpha + \rb{\alpha L D_N(x_n)}^{1/3}.
	\end{align*}
\end{corollary}
\begin{proof}
	If $\alpha \leq (L D_N(x_n))^2$, then $\frac{\alpha^2}{LD_N(x_n)} \leq \alpha L D_N(x_n)$ proving the result by Lemma~\ref{lem_1}.
	If $\alpha \geq (L D_N(x_n))^2$, then $LD_N(x_n) \leq (\alpha L D_N(x_n))^{1/3}$, proving the result by Remark~\ref{rem_1}.
\end{proof}

\section{Exponential sums over primes}\label{sec_exp_sum_primes}

We will make use of the following bound for exponential sums over primes.

\begin{lemma}\label{Leexpsumprimes}
We have uniformly for $x\ge 2$ and $\theta\in \mathbb{R}\setminus\mathbb{Z}$
\[
\sum_{p\le x} \e(\theta p) \ll (\log x)^3 \left( x \sqrt{ \| \theta\|}  + \sqrt{ \frac x { \| \theta\|}  } + x^{4/5} \right).
\]
\end{lemma}

\begin{proof}
Without loss of generality we assume that $0< \theta \le \frac 12$. Set $Q = 2/\theta$ (which satisfies $Q\ge 4$). 
Then by Dirichlet's approximation theorem there exist integers $a,q$ such that $0< q \le Q$ and $\abs{q\theta - a} < 1/Q$. For the sake of simplicity we assume that $q$ is the smallest integer with this property. 

We first show that $a \ne 0$. Assuming the contrary, we would have $\theta < 1/(qQ) \le 1/Q$ which contradicts our choice  $\theta = 2/Q$.
It is also clear that $a$ has to be positive and since $q$ is chosen minimal it also follows
that $\gcd(a,q) = 1$. We also obtain the bound
\[
\theta \ge \frac aq - \frac 1{qQ} \ge \frac 1q - \frac 1{qQ} \ge \frac 1{2q}.
\]
Since $\theta = 2/Q\le 2/q$ we thus obtain
\begin{equation}\label{eqthetaqrel}
\frac 1{2\theta} \le q \le \frac 2{\theta}.
\end{equation}
Moreover $q \le Q$ also gives $\lvert \theta - a/q\rvert < 1/q^2$.

Finally we apply \cite[Theorem 13.6]{IK04} saying that uniformly for $x\ge 2$, $\theta \in \mathbb{R}$, $\lvert\theta - a/q\rvert \le 1/q^2$, and $\gcd(a,q) = 1$ we have
\[
\sum_{p\le x} \e(\theta p) \ll (\log x)^3 \left( \frac x { \sqrt q} + \sqrt{ xq } + x^{4/5} \right).
\]
Clearly by using (\ref{eqthetaqrel}) this proves the lemma.
\end{proof}

The next lemma will be used in Chapter~\ref{chap_local}.
\begin{lemma}\label{Le4}
Suppose that $0<\Delta< 1/2$ and $A\subseteq \mathbb{R}^2/\mathbb{Z}^2$ a rectangle on the unit torus. We set
\[
U(\Delta) = \bigl\{ \left( x_1 + y_1 - y_2/\golden, x_1 + y_1/\golden + y_2 \right) :
(x_1,x_2)\in \partial A,\, \lvert y_1\rvert\le \Delta/2, \, \lvert y_1\rvert\le \Delta/2 \bigr\},
\]
where $\partial A$ denotes the boundary of $A$.

Then for $L^{\nu}\le  k \le L - L^{\nu}$
and $0<\Delta < 1$ we uniformly have, as $x\to\infty$,
\begin{equation}\label{eqLe42}
\frac 1{\pi(x)}\#\left\{p< x :  \left( \{ p \golden^{-k} \}, \{ p \golden^{-k-1} \} \right) 
 \in U(\Delta) \right\} \ll \Delta + e^{-c_3L^{\nu}},
\end{equation}
where $c_3$ is a certain positive constant.
\end{lemma}

\begin{proof}
Since $U  (\Delta)$ has area $\ll \Delta$ and can be partitioned into $4$ convex sets it follows that
\[
\frac 1{\pi(x)}\#\left\{p< x :  \left( \{ p \golden^{-k} \}, \{ p \golden^{-k-1} \} \right) 
 \in U  (\Delta) \right\} \ll \Delta + \tilde J_{\pi(x)},
\]
where $\tilde J_{\pi(x)}$ refers to the isotropic discrepancy (see Section~\ref{sec_discrepancy}) of the 
$\pi(x)$ points $\left( \{ p \golden^{-k} \}, \{ p \golden^{-k-1} \} \right)$ with primes $p\le x$:
\begin{align*}
	\tilde{J}_{\pi(x)} := \sup_{\substack{C\subseteq \mathbb{T}^2\\C\ \mathrm{convex} } }
\left\lvert\frac 1{\pi(x)} \sum_{\substack{p\leq x\\ p \in \mathbb{P}}} \chi_C\left(\left( \{ p \golden^{-k} \}, \{ p \golden^{-k-1} \} \right)\right) - \lambda_2(C)\right\rvert.
\end{align*}
Analogously, we define the discrepancy of the same points:
\begin{align*}
	\tilde{D}_{\pi(x)} \coloneqq \sup_{\substack{I\subseteq \mathbb{T}^2\\I\ \mathrm{interval} } }
\left\lvert\frac 1{\pi(x)} \sum_{\substack{p\leq x\\ p \in \mathbb{P}}} \chi_I\left(\left( \{ p \golden^{-k} \}, \{ p \golden^{-k-1} \} \right)\right) - \lambda_2(I)\right\rvert.
\end{align*}
Thus, we only have to show that the isotropic discrepancy can be bounded 
by $\tilde J_{\pi(x)} \ll e^{-c_3L^{\nu}}$ for some constant $c_3> 0$.



To do so, we will use Equation~\eqref{eq_isotropic} to relate the isotropic discrepancy to the usual discrepancy and the Erd\H{o}s--Tur\'an--Koksma inequality (Lemma~\ref{lem_ETK}) to find an upper bound for the usual discrepancy.

In our particular case we choose $N = \pi(x)$, the points
$\left( \{ p \golden^{-k} \}, \{ p \golden^{-k-1} \} \right)$ with primes $p\le x$, and 
 $H = \lfloor e^{c L^\nu} \rfloor$, where $c = \frac 12 \log \golden$.
The main issue is to estimate exponential sums of the form
\[
S = \sum_{p \le x} e\left( p\left( \frac{h_1}{\golden^{j}} + \frac{h_2}{\golden^{j+1}} \right) \right).
\]
For convenience we set
\[
\theta = \frac{h_1}{\golden^{j}} + \frac{h_2}{\golden^{j+1}} = \frac{h_1 \golden + h_2}{\golden^{j+1}}.
\]
Clearly, since $\lvert h_1\rvert\le H$, $\lvert h_2\rvert\le H$, and $j \ge L^\nu$ we have
\[
\lvert\theta\rvert \ll \frac{H}{\golden^{L^\nu}} \ll e^{- c L^\nu}.
\]
Furthermore, if $(h_1,h_2)\ne 0$ we have that $h_1 \gamma + h_2$ is a nonzero element of $\mathbb Z[\golden]$, therefore
\begin{align*}
1&\leq\bigl\lvert\mathcal N\left(
h_1\golden+h_2\right)\bigr\rvert=\bigl\lvert h_1\golden+h_2\bigr\rvert
\bigl\lvert h_1\overline{\golden}+h_2\bigr\rvert\\
	& \leq \abs{h_1 \golden + h_2} \rb{\abs{h_1} \abs{\overline{\golden}} + \abs{h_2}} \leq \abs{h_1 \golden + h_2} \rb{\abs{h_1} + \abs{h_2}}
\end{align*}
where $\overline{\golden}=1-\golden$.

It follows directly that
$\lvert h_1 \golden + h_2\rvert \ge 1/(\lvert h_1\rvert+\lvert h_2\rvert)  \gg e^{-c L^\nu}$.
Since $j\le L - L^\nu$ we thus get
\[
\lvert\theta\rvert \gg \frac 1{ e^{cL^\nu} \golden^{L-L^\nu}} \gg \frac{e^{c L^\nu}}x.
\]
Consequently, Lemma~\ref{Leexpsumprimes} gives
\[
S \ll (\log x)^3 x e^{-\frac c2 L^\nu}.
\]
Since 
\[
\sum_{(h_1,h_2)\in \mathbb{Z}^2\setminus \{(0,0)\}, \, \lvert h_1\rvert\le H,\, \lvert h_2\rvert\le H }
\frac 1{(1+\lvert h_1\rvert)(1+\rvert h_2\rvert)}  \ll (\log H)^2 \ll (\log x)^{2\nu},
\]
we get the following upper bound for the usual discrepancy:
\[
\tilde D_{\pi(x)} \ll e^{-cL^\nu} + (\log x)^{4 + 2\nu} e^{-\frac c2 L^\nu} \ll  e^{-\frac c3 L^\nu}.
\]
This also gives $\tilde J_{\pi(x)} \ll e^{-\frac c6 L^\nu}$, which completes the proof of the lemma.
\end{proof}

\section{Geometric series}
As we are dealing routinely with exponential sums, we will need the following results for linear exponential sums, in other words, geometric series.

Therefore, we consider a geometric
series with ratio $\e(\xi), \xi \in \R$ and $L_1,L_2\in\Z$, where $L_1 \leq L_2$:
\begin{align}\label{eq:estimate-geometric-series}
 \begin{split} \left\lvert\sum_{L_1<\ell \leq L_2} \e(\ell \xi)\right\rvert &\leq \min\left(L_2-L_1,\abs{\sin\pi\xi}^{-1}\right)\\
  		&\ll \min\left(L_2-L_1,\norm{\xi}^{-1}\right),
 \end{split}
\end{align}
which is obtained from the formula for finite geometric series.

The following result allows us to find useful estimates for double sums of geometric series, where we additionally take a sum over $n$, where $\xi = \frac{an+b}{m}$.

The following lemma can be found in \cite[Lemma 14]{DMR2019}.
\begin{lemma}\label{le_geometric_series}
  Let  $m\geq 1$ and $A\geq 1$ be integers and $b\in\R$.
  For any real number $U>0$, we have
  \begin{equation}\label{eq:double-sum-min}
    \frac{1}{A} \sum_{1\leq a \leq A}
    \sum_{0\leq n < m}
    \min\rb{
    U, \abs{\sin \rb{\pi \tfrac{ a n + b}{m}}}^{-1}
    }
    \ll
    \tau(m) \ U
    + 
    m\log m.
  \end{equation}
If $\lvert b\rvert\leq 1/2$, we have the sharper bound
  \begin{equation}\label{eq:double-sum-min-sharp}
\begin{aligned}
\hspace{6em}&\hspace{-6em}
    \frac{1}{A} \sum_{1\leq a \leq A}
    \sum_{0\leq n < m}
    \min\rb{
    U, \abs{\sin \rb{\pi \tfrac{ a n + b}{m}}}^{-1}
    }
\\&\ll
    \tau(m) \min\rb{ U, \abs{\sin \rb{\pi \tfrac{b}{m}}}^{-1} }
    + 
    m\log m,
\end{aligned}
  \end{equation}
  where $\tau(m)$ denotes the number of divisors of $m$.
\end{lemma}

\chapter{Detection of Zeckendorf digits}\label{chap:detection}

It is easy to detect base-$q$ (where $q\ge2$ is an integer) digits $\varepsilon_\ijkl(n)$ with indices in an interval $[a,b)$.
Assume that $n=\sum_{\ijkl\geq 0}\varepsilon_\ijkl(n)q^{\ijkl}$ is the base-$q$ expansion of the integer $n\ge0$ and $(\nu_a,\ldots,\nu_{b-1})\in\{0,\ldots,q-1\}^{b-a}$.
Then we have
\[
\bigl(\varepsilon_a(n),\ldots,\varepsilon_{b-1}(n)\bigr)=\bigl(\nu_a,\ldots,\nu_{b-1}\bigr)
\quad\mbox{if and only if}\quad\left\{\frac n{q^b}\right\}\in J,
\]
where
\[J=\left[\frac \omega{q^{b-a}},\frac{\omega+1}{q^{b-a}}\right)\quad\mbox{and}\quad \omega=\sum_{a\leq\ijkl<b}\nu_\ijkl q^{\ijkl-a}.\]
We are interested in related statements on the Zeckendorf expansion.
First of all, the Zeckendorf numeration system is a special case of the \emph{Ostrowski numeration system}~\cite{B2001}, defined on the nonnegative integers (after taking a shift of indices by $1$ into account).
It can be seen as the ``simplest'' Ostrowski expansion, corresponding to the real number $\alpha=\golden-1$ having the continued fraction expansion
\[\bigl[0;\overline1\bigr].\]

The Zeckendorf numeration system exactly describes the irrational rotation $n\mapsto n\golden\bmod 1$, in the sense of~\eqref{eqn_central_motivation}:
the tuple of the lowest $L$ Zeckendorf digits of $n$ equals $\omega$ if and only $n\varphi$ lies in a certain interval $I_\omega$ modulo $1$.
This is the content of Lemma~\ref{Lefirstdigits}.
%

If we want to detect digits with indices in a certain interval $[a,b)$, where not necessarily $a=2$ (the lowest index in the Zeckendorf expansion),
this one-dimensional detection procedure is not good enough for our needs.
For example, detecting the property $\digit_3(n)=\digit_4(n)=0$ requires two intervals $I_{\tO\tO\tO}$ and $I_{\tO\tO\tL}$, and they are separated modulo $1$.
That is,
\[\inf\bigl\{\lVert x-y\rVert: x\in I_{\tO\tO\tO}+\mathbb Z,y\in I_{\tO\tO\tL}+\mathbb Z\bigr\}>0.\]
In general, detecting digits with indices in $[a,b)$ requires $\gg \golden^a$ intervals, which is too large to yield useful estimates.

For this reason, we introduce \emph{two-dimensional detection}, leading to Lemma~\ref{lem_twodim}.
Basically, the interval $[0,1)$ is stretched by a factor $\bigl(F_\lambda^2+F_{\lambda+1}^2\bigr)^{1/2}$ and wrapped around the two-dimensional torus, with slope $-F_{\lambda+1}/F_\lambda$.
This is achieved by considering
\[p(n)=\left(\frac n{\golden^\lambda},\frac n{\golden^{\lambda+1}}\right).\]
In this way, addition of $\golden$ modulo $1$ corresponds to addition of $\golden^{-\lambda}(1,1/\golden)$ modulo $1\times 1$.
This is made precise in the proof of Lemma~\ref{lem_twodim} (see~\eqref{eqn_wrap_identity}).

The procedure of wrapping the unit interval around the torus has the effect that intervals corresponding to lexicographically adjacent digit combinations are placed ``next to each other''.
This very important fact is exploited in Corollary~\ref{cor_twodim}.
Below this corollary, we give a graphical representation of the situation for the case of four significant Zeckendorf digits.

The described one-and two-dimensional detections are used in Chapter~\ref{chap_lod} and~\ref{ch_type2}.
We note that one-and two-dimensional detection can be combined in order to handle the digits of an integer $n$ with indices in $[2,\lambda-1]\cup[a,b-1]$.
The arising exponential sums will contain three parameters, that is, we are dealing wth \emph{three-dimensional} detection.
This combination of Lemma~\ref{Lefirstdigits} and Corollary~\ref{cor_twodim} is carried out in Chapter~\ref{chap_lod}.

It proves convenient to introduce a variant of our two-dimensional detection tailored to the detection of a single digit at index $a$.
In this case, we define rectangles $A_0$ and $A_1$.
These rectangles are independent of $\ijkl$, which makes them easier to work with, but they have the (slight) disadvantage that the property $\digit_\ijkl(n)=0$ is only detected in an asymptotical way.
More precisely, $\digit_\ijkl(n)=0$ for $n<N$ is detected by the interval $A_0$ in all but $\LandauO(N/\golden^\ijkl)$ cases.
We will use this variant in the proof of the second part of Theorem~\ref{Th3} (that is, Proposition~\ref{Promain2}). 

%


\section{One-dimensional detection} 
We wish to detect a block of digits with indices in $[2,\lambda-1]$, where $\lambda\geq 2$ is an integer.
This can be handled by one-dimensional detection,
and is an example of the (well-known) application of the Ostrowski expansion to the study of the sequence $n\alpha$ modulo $1$.
The following function $v(\cdot,\lambda)$ cuts off the Zeckendorf digits with indices $\geq \lambda$:
\begin{equation}\label{eqn_vnL_def}
v(n,\lambda)=\sum_{2\leq \ijkl<\lambda}\digit_{\ijkl}(n)F_{\ijkl}.
\end{equation}
This is the counterpart to the function $n\mapsto n\bmod q^\lambda$
in the case of the $q$-ary representation of integers, and it is not periodic for $\lambda\geq 3$.
In fact, $v(\cdot,3)$ is the \emph{Fibonacci word},
which arises as the Sturmian sequence with slope $\golden-1$.
We define the \emph{truncated Zeckendorf sum-of-digits function},
which only takes into account the digits up to $\lambda$:
\begin{equation}\label{eqn_zL_def}
\sz_\lambda(n)=\sz\bigl(v(n,\lambda)\bigr)=\sum_{2\leq \ijkl<\lambda}\digit_{\ijkl}(n).
\end{equation}
The following statement can be found, for example, in~\cite[Lemma~1]{DMS2018}, 
or~\cite[Proposition~5.7]{Spiegelhofer2014}). 
\begin{lemma}\label{Lefirstdigits}
Assume that $\lambda\geq 2$ is an integer.
We define
  \begin{align*}
	  \widetilde A_\lambda^{(\1)} = (-1)^\lambda \left(-\frac{1}{\golden^{\lambda-1}}, \frac{1}{\golden^\lambda}\right),
    \qquad \widetilde A_\lambda^{(\2)} = (-1)^\lambda \left(-\frac{1}{\golden^{\lambda+1}}, \frac{1}{\golden^\lambda}\right)
	\end{align*}
and
\begin{align*}
		A_\lambda(u) \eqdef u \golden + \left\{\begin{array}{cc} \widetilde A_\lambda^{(\1)}, & 0\leq u < F_{\lambda-1};\\[1mm]
\widetilde A_\lambda^{(\2)}, & F_{\lambda-1} \leq u < F_\lambda. \end{array}\right.
\end{align*}
For all integers $u\in\{0,\ldots,F_\lambda-1\}$ and $n\geq 0$ we have the identity
  \begin{align*}
    v(n,\lambda) = u
  \end{align*}
  if and only if
  \begin{align*}
    n \golden \in A_\lambda(u) + \mathbb Z.
  \end{align*}
\end{lemma}
We note that the sign $(-1)^\lambda$ is at a different position in the paper~\cite{DMS2018},
and we are considering open intervals while~\cite{DMS2018} has half-open intervals.
However, the proof below shows that the endpoints of the intervals are never hit by the sequence $n\golden$, therefore this change is harmless.
Note that $\{n\golden\}$ is dense in $[0,1]$ and our intervals are open sets.
The lemma above therefore shows in particular that the sets $A_\lambda(u) + \Z$,
 where $0\leq u < F_\lambda$, are pairwise disjoint.
Moreover, up to a set of measure zero (in fact a $\golden^{-\lambda+1}$-spaced set of points) they form a partition of $\mathbb R$.
Also, $\bigl\lvert \widetilde A_\lambda^{(\1)}\bigr\rvert = \golden^{-\lambda+2}$ and $\bigl\lvert \widetilde  A_\lambda^{(\2)}\bigr\rvert = \golden^{-\lambda+1}$.
Finally, when we consider the points $a_u=u \golden \bmod 1$ for $0\leq u<F_\lambda$ and arrange them in increasing order $a_{\sigma(0)}<a_{\sigma(1)}<\cdots<a_{\sigma(F_\lambda-1)}$, then the maximum length of the appearing gaps 
\[\bigl\lVert a_{\sigma((n+1)\bmod F_\lambda)}-a_{\sigma(n)}\bigr\rVert\]
(where $0\leq n<F_\lambda$)
is exactly $\golden^{-\lambda+2}$.

This allows us to detect the $\lambda-2$ lowest Zeckendorf digits of $n$ by considering values $n\golden$ in an interval modulo $1$.
Note that for the base case $\lambda=2$, the interval $A_2^{(0)}$ has length $1$,
and indeed there is nothing to detect in this case.

For the convenience of the reader, we give a proof of Lemma~\ref{Lefirstdigits}.
\begin{proof}[Proof of Lemma~\ref{Lefirstdigits}]
The Fibonacci numbers satisfy Binet's formula,
\begin{equation}\label{eqn_binet}
F_{\ijkl}=\frac{\golden^{\ijkl}-(-1/\golden)^{\ijkl}}{\sqrt{5}}
\end{equation}
for $\ijkl\geq 0$,
which we will also use later (see equation~\eqref{eqn_binet_consequence}).
By~\eqref{eqn_binet} we have
\begin{align*}
n\golden&=
v(n,\lambda)\golden+
\sum_{\ijkl\geq \lambda}\digit_{\ijkl}\golden
\frac{\golden^{\ijkl}-(-\golden)^{-\ijkl}}{\sqrt{5}}
\\&=v(n,\lambda)\golden+
\sum_{\ijkl\geq \lambda}\digit_{\ijkl}
\frac{\golden^{\ijkl+1}-(-\golden)^{-(\ijkl+1)}}{\sqrt{5}}+
\sum_{\ijkl\geq \lambda}\digit_{\ijkl}
\frac{-(-\golden)^{-\ijkl}\golden+(-\golden)^{-(\ijkl+1)}}{\sqrt{5}}.
\end{align*}
The second term is a sum of Fibonacci numbers and as such it is an integer.
Moreover, we have $(1+\golden^{-2})/\sqrt{5}=1/\golden$.
Therefore
\begin{equation}\label{eqn_nphi_reduction}
n\golden\equiv v(n,\lambda)\golden
+\sum_{\ijkl\geq \lambda}\digit_{\ijkl}
\frac{(-\golden)^{-\ijkl+1}\left(1+(-\golden)^{-2}\right)}{\sqrt{5}}\bmod 1
\equiv v(n,\lambda)\golden
+s(n,\lambda)\bmod 1,
\end{equation}
where
\begin{equation}\label{eqn_snL_def}
s(n,\lambda)=-\sum_{\ijkl\geq \lambda}\frac{\digit_{\ijkl}}{(-\golden)^{\ijkl}}.
\end{equation}
Clearly, the expression $s(n,\lambda)$ can be written as the difference of nonnegative real numbers as follows:
\begin{equation}\label{eqn_snL_sep}
s(n,\lambda)=\sum_{\substack{\ijkl\geq \lambda\\2\nmid \lambda}}\digit_{\ijkl}\golden^{-\ijkl}
-\sum_{\substack{\ijkl\geq \lambda\\2\mid \lambda}}\digit_{\ijkl}\golden^{-\ijkl}.
\end{equation}
In order to obtain lower and upper bounds for this quantity, we distinguish between the two cases $\digit_{\lambda-1}\in\{0,1\}$.
In the case $\digit_{\lambda-1}=0$, there is no restriction on the digits $\digit_{\ijkl}$ for $\ijkl\geq \lambda$ coming from the lower digits; we easily get
\[-\frac 1{\golden^{\lambda-1}}<s(n,\lambda)<\frac 1{\golden^\lambda}.\]
If $\digit_{\lambda-1}=1$, we necessarily have $\digit_\lambda=0$;
therefore the lower bound increases by $\golden^{-\lambda}$, and we obtain
\[-\frac 1{\golden^{\lambda+1}}<s(n,\lambda)<\frac 1{\golden^\lambda}.\]

Analogously, we handle the case $2\nmid \lambda$:
in this case, the two summands in~\eqref{eqn_snL_sep} switch roles, and we have
\[
\begin{array}{ll}
\displaystyle
-\frac 1{\golden^\lambda}<s(n,\lambda)<\frac 1{\golden^{\lambda-1}}
&\mbox{if }\digit_{\lambda-1}=0,\mbox{ and}\\[4mm]
\displaystyle-\frac 1{\golden^\lambda}<s(n,\lambda)<\frac 1{\golden^{\lambda+1}}
&\mbox{if }\digit_{\lambda-1}=1.
\end{array}
\]
We can summarize the two cases $2\mid \lambda$ and $2\nmid \lambda$ conveniently, by introducing the factor $(-1)^\lambda$ as in the statement of the lemma.
The proof is complete.

\end{proof}

We finish this section with a useful lemma that {\it creates zeros}.

\begin{lemma}\label{le_create_zeros}
  Let $\ell \in \N$ and $n \in \N$.
  Then there exists $0\leq y < F_{\ell}$ such that $v(n+y,\ell) = 0$.
\end{lemma}
\begin{proof}
	We know that $v(n + y,\ell) = 0$ if and only if
	\begin{align*}
		(n + y) \golden \in (-1)^{\ell} \left(-\frac{1}{\golden^{\ell-1}}, \frac{1}{\golden^{\ell}}\right) + \Z,
	\end{align*}
	where the endpoints of the interval are never hit.
	This interval has length $\frac{1}{\golden^{\ell}} + \frac{1}{\golden^{\ell-1}} = \frac{1}{\golden^{\ell-2}}$.
	Let us now consider the points $y\golden \bmod 1$ for $0\leq y < F_{\ell}$. We recall that these points form a sequence of points in $\R \mod \Z$, where any two consecutive points have distance at most $\frac{1}{\golden^{\ell-2}}$. Thus, there exists some $y<F_{\ell}$ such that $(n+y) \golden \in (-1)^{\ell} A_{\ell}^{(\1)}$ and therefore $v(n+y, \ell) = 0$.
\end{proof}

\section{Two-dimensional detection, part one} 
The one-dimensional detection procedure has the drawback that digit combinations
belonging to consecutive integers usually correspond to intervals that are separated.
More precisely, if $\lambda\geq 2$ and $u<F_\lambda-1$, the union of the two intervals $A_\lambda(u)$ and $A_\lambda(u+1)$ is, usually, not connected.
If we are to detect digits with indices in an \emph{interval}
(that is, $\digit_{\ijkl}(n)$, where $a\leq \ijkl <b$),
we obtain a scattered set consisting of $\asymp F_a$ intervals,
which is difficult to handle directly in an analytical way (using the Fourier transform, for example).
For this reason, we introduce a \emph{two-dimensional} detection procedure,
which leads us to \emph{parallelograms} instead of scattered sets.

Assume that $\lambda\geq 2$.
We introduce the function $p:\mathbb N\rightarrow\mathbb R^2$ by
\[p(n,\lambda)=\left(\frac n{\golden^{\lambda}},\frac n{\golden^{\lambda+1}}\right).\]

We define parallelograms $\widetilde B_\lambda^{(\1)}$ and $\widetilde B_\lambda^{(\2)}$ in $\mathbb R^2$ by specifying their defining inequalities:

\begin{equation*}
\widetilde B^{(\1)}_\lambda:
\left\{
\begin{array}{c}
0\leq F_{\lambda+1} x+F_\lambda y<1
\\[1mm]
-\golden\leq -\frac 1\golden x+y< 1
\end{array}
\right\},
\qquad
\widetilde B^{(\2)}_\lambda:
\left\{
\begin{array}{c}
0\leq F_{\lambda+1} x+F_\lambda y<1
\\[1mm]
-\frac 1\golden\leq -\frac 1\golden x+y< 1
\end{array}
\right\}.
\end{equation*}
With their help we define parallelograms $B_\lambda(u)$:

\begin{equation}\label{eqn_def_B}
B_\lambda(u)=
p(u, \lambda)
+\left\{\begin{array}{lr}
\widetilde B_\lambda^{(\1)},& 0\leq u<F_{\lambda-1};\\[1mm]
\widetilde B_\lambda^{(\2)},& F_{\lambda-1}\leq u<F_\lambda.
\end{array}
\right.
\end{equation}
We will see that these sets modulo one are disjoint
and form a partition of the unit square,
which results from the proof of the following lemma.
\begin{lemma}\label{lem_twodim}
For integers $n\geq 0$, $\lambda\geq 2$ and $0\leq u<F_\lambda$ we have
\[v(n,\lambda)=u\qquad\mbox{if and only if}\qquad p(n,\lambda)\in B_\lambda(u)+\mathbb Z^2.\]
\end{lemma}

\begin{proof}
Let $n$ be a nonnegative integer and $\lambda\geq 2$.
Separating the upper from the lower digits we obtain after a short calculation
\begin{align*}
\frac n{\golden^\lambda}&=
\frac 1{\golden^\lambda}
\sum_{2\leq \ijkl<\lambda}\digit_{\ijkl}(n)F_{\ijkl}
+
\frac 1{\golden^\lambda}
\sum_{\ijkl\geq \lambda}\digit_{\ijkl}(n)F_{\ijkl}
\\&=
\frac {v(n,\lambda)}{\golden^\lambda}+
\sum_{\ijkl\geq \lambda}\digit_{\ijkl}(n)F_{\ijkl-\lambda}+
F_\lambda\sum_{\ijkl\geq \lambda}\digit_{\ijkl}(n) (-1)^{\ijkl-\lambda}\golden^{-\ijkl}
\end{align*}
and
\begin{align*}
\frac n{\golden^{\lambda+1}}&=
\frac 1{\golden^{\lambda+1}}
\sum_{2\leq \ijkl<\lambda}\digit_{\ijkl}(n)F_{\ijkl}
+
\frac 1{\golden^{\lambda+1}}
\sum_{\ijkl\geq \lambda}\digit_{\ijkl}(n)F_{\ijkl}
\\&=
\frac {v(n,\lambda)}{\golden^{\lambda+1}}+
\sum_{\ijkl\geq \lambda}\digit_{\ijkl}(n)F_{\ijkl-\lambda-1}
-F_{\lambda+1}\sum_{\ijkl\geq \lambda}\digit_{\ijkl}(n) (-1)^{\ijkl-\lambda}\golden^{-\ijkl}.
\end{align*}
The second term in each of the expressions above is a sum of Fibonacci numbers and therefore an integer.
Writing
\[s'(n,\lambda)=\sum_{\ijkl\geq \lambda}\digit_{\ijkl}(n)(-1)^{\ijkl-\lambda}\golden^{-\ijkl},\]
we see that
\[s'(n,\lambda)=(-1)^{\lambda-1}s(n,\lambda).\]

We get
\begin{equation}\label{eqn_twodim_lines}
\begin{aligned}\frac n{\golden^\lambda}&\equiv \frac {v(n,\lambda)}{\golden^\lambda}+F_\lambda s'(n,\lambda)\mod 1,\\
\frac n{\golden^{\lambda+1}}&\equiv \frac {v(n,\lambda)}{\golden^{\lambda+1}}-F_{\lambda+1} s'(n,\lambda)\mod 1.
\end{aligned}
\end{equation}
We relate this situation to the one-dimensional case.
This lemma is one-dimensional detection in disguise:
using the formula
\begin{equation}\label{eqn_binet_consequence}
F_{c+1}=F_c\golden+\frac{(-1)^c}{\golden^c}
\end{equation}
valid for all integers $c\geq 0$, following from Binet's formula~\eqref{eqn_binet},
we obtain the identities
\begin{equation}\label{eqn_twodim_to_onedim}
\begin{aligned}
-\golden\left(\begin{matrix}F_\lambda\\-F_{\lambda+1}\end{matrix}\right)
  &=\left(\begin{matrix}-F_{\lambda+1}\\F_{\lambda+2}\end{matrix}\right)
+\left(\begin{matrix}1/{\golden^\lambda}\\1/{\golden^{\lambda+1}}\end{matrix}\right),&\mbox{if $2\mid \lambda$};\\
\golden\left(\begin{matrix}F_\lambda\\-F_{\lambda+1}\end{matrix}\right)
  &=\left(\begin{matrix}F_{\lambda+1}\\-F_{\lambda+2}\end{matrix}\right)
+\left(\begin{matrix}1/{\golden^\lambda}\\1/{\golden^{\lambda+1}}\end{matrix}\right),&\mbox{if $2\nmid \lambda$}.
\end{aligned}
\end{equation}
Let us consider the case $2\mid \lambda$, the other one being analogous.
We consider the sets
\[\mathbb R\left(\begin{matrix}F_\lambda\\-F_{\lambda+1}\end{matrix}\right)+[u,u+1)\left(\begin{matrix}1/\golden^\lambda\\1/\golden^{\lambda+1}\end{matrix}\right)+\mathbb Z^2\]
for $0\leq u<F_\lambda$.
Using the identity
\[\frac{F_\lambda}{\golden^{\lambda+1}}+\frac{F_{\lambda+1}}{\golden^\lambda}=1,\]
we see that these sets form a partition of $\mathbb R^2$.
This already explains the first lines of the definitions of $\widetilde B_\lambda^{(i)}$.
Connecting~\eqref{eqn_twodim_lines} and~\eqref{eqn_twodim_to_onedim}, we obtain
\begin{equation}\label{eqn_wrap_identity}
-n\golden\left(\begin{matrix}F_\lambda\\-F_{\lambda+1}\end{matrix}\right)
  \equiv p\bigl(v(n,\lambda)\bigr)
  +s'(n,\lambda)\left(\begin{matrix}F_\lambda\\-F_{\lambda+1}\end{matrix}\right)
  \bmod \left(\begin{matrix}1\\1\end{matrix}\right).
\end{equation}
We see that this identity nicely connects the one-and two-dimensional detection procedures.
One-dimensional detection gives us a partition of the line segment connecting $(0,0)^T$ and $(F_\lambda,-F_{\lambda+1})^T$.
The remainder of the proof is straightforward: reusing the estimates for $s(n,\lambda)$ from the proof of Lemma~\ref{Lefirstdigits}, and treating the second line of~\eqref{eqn_twodim_to_onedim} in an analogous fashion, we can conclude the proof. 
\end{proof}

As we noted before, we wish to detect Zeckendorf digits with indices in an interval $[a,b)$.
For this, we glue together the small parallelograms $B_\lambda(u)$ in order to obtain a larger one.
This possibility is the reason for the introduction of this second type of digit detection.
Let digits $\nu_a,\ldots,\nu_{b-1}$ be given such that no adjacent $\tL$s occur.
We define $M=\sum_{a\leq j<b}\nu_j F_j$.
There are two cases to consider, corresponding to the value $\nu_a\in\{0,1\}$.
If $\nu_a=0$, we set
\[A=\bigcup_{0\leq u<F_a}B_b(M+u).\]
Assume for a moment that we also have $\nu_{b-1}=0$.
Using the identity
\[\frac{F_{b+1}}{\golden^b}+\frac{F_b}{\golden^{b+1}}=1\]
following from~\eqref{eqn_binet},
we see that $\tilde{B}_{\lambda}^{(0)}$ is a parallelogram with corners

\begin{equation}\label{eqn_corners}
\begin{array}{r@{\hspace{1mm}}l@{\hspace{2em}}r@{\hspace{1mm}}l}
\displaystyle
C_1&=\displaystyle\left(  \frac{F_\lambda}{\golden^{\lambda-1}}, -\frac{F_{\lambda+1}}{\golden^{\lambda-1}}\right),&
C_2&=\displaystyle\left(-\frac{F_\lambda}{\golden^\lambda}, \frac{F_{\lambda+1}}{\golden^\lambda}\right),\\[3mm]
C_3&=\displaystyle C_1+\left(\frac1{\golden^\lambda},\frac 1{\golden^{\lambda+1}}\right),&
C_4&=\displaystyle C_2+\left(\frac1{\golden^\lambda},\frac 1{\golden^{\lambda+1}}\right).
\end{array}
\end{equation}
Since $B_b(M+u)$ is obtained from $\tilde{B}_{\lambda}^{(\1)}$ by a shift $\rb{\frac{M+u}{\golden^{\lambda}}, \frac{M+u}{\golden^{\lambda+1}}}$, all of these parallelograms fit together very well.
In particular, its union $A$ is again a parallelogram, defined by the inequalities
\begin{equation*}
\begin{aligned}
M&\leq F_{b+1}x+F_b y<M+F_a,\\[1mm]
-\golden&\leq -\tfrac 1\golden x+y<1.
\end{aligned}
\end{equation*}
The other three cases, that is, $(\nu_a,\nu_{b-1})\in\{(0,1),(1,0),(1,1)\}$, are analogous.
We summarize these considerations in the following, very important, corollary.
\begin{corollary}\label{cor_twodim}
Assume that $2\leq a<b$.
Let $\nu_a,\ldots,\nu_{b-1}\in\{0,1\}$ be given such that the implication $\nu_{\ijkl+1}=1\Rightarrow \nu_{\ijkl}=0$ holds for all $\nu\in\{a,a+1,\ldots,b-2\}$.
Define $M=\sum_{a\leq \ijkl<b}\nu_{\ijkl} F_{\ijkl}$.
\begin{equation*}
\begin{array}{lr@{\hspace{0.5em}}llr@{\hspace{0.2em}}llr@{\hspace{0.2em}}l}
\text{If }&\nu_a&=0, &\text{let }&W&\eqdef F_a,&\quad\text{otherwise}&W&\eqdef F_{a-1};\\[2mm]
\text{if }&\nu_{b-1}&=0,&\text{let }&\alpha&\eqdef-\golden,&\quad\text{otherwise}&\alpha&\eqdef-1/\golden.\end{array}
\end{equation*}
We define the set
\begin{equation}\label{eqn_A_inequalities}
A=\left\{(x,y)\in\mathbb R^2:
\begin{array}{l}
M\leq F_{b+1}x+F_b y<M+W,\\[1mm]
\alpha\leq-\tfrac 1\golden x+y<1.
\end{array}
\right\}.
\end{equation}
For all $n\geq 0$, we have
\[
\digit_j(n)=\nu_j\mbox{ for all }j\in\{a,\ldots,b-1\}
\qquad\mbox{if and only if}\qquad \left(\frac{n}{\golden^b},\frac{n}{\golden^{b+1}}\right)\in A+\mathbb Z^2.\]
\end{corollary}

\newcommand{\shifx}{-0.000125742681113653}
\newcommand{\shify}{0.0655121782080418}
\newcommand{\ang}{-58.3924977537511}
\begin{center}
\begin{tikzpicture}[scale=4.7,rotate=\ang]
        
\foreach \i in {0,...,4} 
{
  \coordinate (C0\i) at (1.37638445577445+\i*\shifx,0.000000000000000+\i*\shify);
  \coordinate (C1\i) at (-0.850652375255634+\i*\shifx,0.000000000000000+\i*\shify);
  \coordinate (C2\i) at (1.37625871309333+\i*\shifx,\shify+\i*\shify);
  \coordinate (C3\i) at (-0.850778117936748+\i*\shifx,\shify+\i*\shify);
  \draw[draw=blue] (C2\i) -- (C0\i) -- (C1\i);
  \draw[draw=blue,dashed] (C1\i) -- (C3\i) -- (C2\i);
}
\foreach \i in {5,...,7} 
{
  \coordinate (C1\i) at (-0.850652375255634+\i*\shifx,0.000000000000000+\i*\shify);
  \coordinate (C3\i) at (-0.850778117936748+\i*\shifx,\shify+\i*\shify);
  \coordinate (C0p\i) at (0.525732080518812+\i*\shifx,0.000000000000000+\i*\shify);
  \coordinate (C2p\i) at (0.525732080518812+\i*\shifx,\shify+\i*\shify);
  \draw[draw=blue] (C2p\i) -- (C0p\i) -- (C1\i);
  \draw[draw=blue,dashed] (C1\i) -- (C3\i) -- (C2p\i);
}

\newcommand{\firstx}{0.524097425664335}
\newcommand{\firsty}{0.851658316704544}
\newcommand{\secondx}{-0.851658316704544}
\newcommand{\secondy}{0.524097425664335}

\draw[>=latex,->] (\firstx*-1.05,\firsty*-1.05) -- (\firstx*1.15,\firsty*1.15);
\draw[>=latex,->] (\secondx*-1.05,\secondy*-1.05) -- (\secondx*1.15,\secondy*1.15);

\draw [dashed] (\firstx*-1.05+\secondx,\firsty*-1.05+\secondy) -- (\firstx*1.05+\secondx,\firsty*1.05+\secondy);
\draw [dashed] (\firstx*-0.05-\secondx,\firsty*-0.05-\secondy) -- (\firstx*1.05-\secondx,\firsty*1.05-\secondy);
\draw [dashed] (-\firstx+\secondx*-0.05,-\firsty+\secondy*-0.05) -- (-\firstx+\secondx*1.05,-\firsty+\secondy*1.05);
\draw [dashed] (\firstx+\secondx*-1.05,\firsty+\secondy*-1.05) -- (\firstx+\secondx*1.05,\firsty+\secondy*1.05);


\node[color=blue,rotate=\ang] (N0001) at (0.45+0*\shifx,0.032+0*\shify) {\small$\mathtt{0000}$};
\node[color=blue,rotate=\ang] (N0010) at (0.45+1*\shifx,0.032+1*\shify) {\small$\mathtt{0001}$};
\node[color=blue,rotate=\ang] (N0100) at (0.45+2*\shifx,0.032+2*\shify) {\small$\mathtt{0010}$};
\node[color=blue,rotate=\ang] (N0101) at (0.45+3*\shifx,0.032+3*\shify) {\small$\mathtt{0100}$};
\node[color=blue,rotate=\ang] (N1000) at (0.45+4*\shifx,0.032+4*\shify) {\small$\mathtt{0101}$};
\node[color=blue,rotate=\ang] (N1001) at (0.45+5*\shifx,0.032+5*\shify) {\small$\mathtt{1000}$};
\node[color=blue,rotate=\ang] (N1010) at (0.45+6*\shifx,0.032+6*\shify) {\small$\mathtt{1001}$};
\node[color=blue,rotate=\ang] (N1010) at (0.45+7*\shifx,0.032+7*\shify) {\small$\mathtt{1010}$};
\node(N10) at (\firstx,\firsty+0.095) {\scriptsize$(1,0)$};
\node(N01) at (\secondx,\secondy+0.095) {\scriptsize$(0,1)$};
\end{tikzpicture}

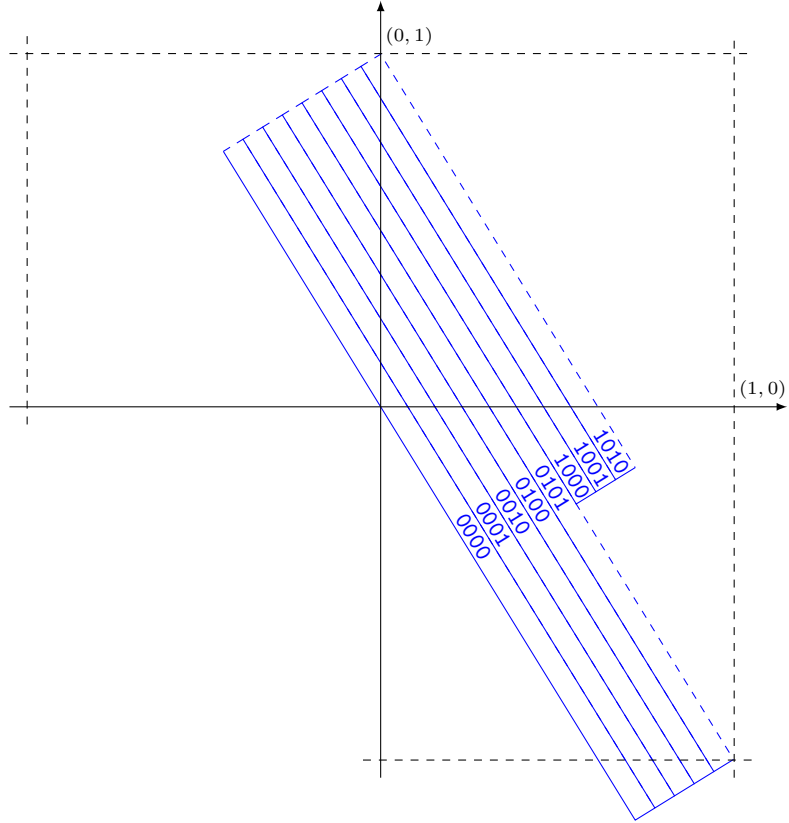
\captionof{figure}{Detecting four Zeckendorf digits}\label{fig_detection}
\end{center}
\begin{example}\label{ex_L_shaped_region}
Figure~\ref{fig_detection} illustrates this result.
In this case, we have $b=6$, and since $F_6=8$, there are eight possibilities for the tuple $(\digit_5,\digit_4,\digit_3,\digit_2)$ formed by the lowest four Zeckendorf digits.
Each of the eight parallelograms in our picture corresponds to one of these digit combinations.
These eight sets $B_6(u)$, modulo $1\times 1$, form a partition of the unit square $[0,1)^2$.
Note that we can clearly see from the definition~\eqref{eqn_def_B}, and the irrationality of $\golden$, that $B_6(u)$ is not a rectangle.

In order to detect integers whose Zeckendorf expansions end with $\mathtt{01\!*\!*}$, for example, we join the sets $B_6(u)$ corresponding to $\tO\tL\tO\tO$ and $\tO\tL\tO\tL$. That is, we form the (disjoint) union $A\coloneqq B_6(3)\cup B_6(4)$.
This is the case $(b,a)=(6,4)$ and $(\nu_5,\nu_4)=(0,1)$ in the above corollary.
As we noted before, the reason for using two-dimensional detection in this paper lies in the observation that such unions are again parallelograms.


\end{example}

\section{Two-dimensional detection, part two} 
The following property is proved in \cite{DS02}, but it could also be derived as a corollary to Lemma~\ref{lem_twodim}.
We can compute (with a small error) the digits $\delta_{\ijkl}(n)$ just by considering
the fractional parts $\{ n \golden^{-\ijkl} \}$ and $\{ n \golden^{-\ijkl-1} \}$.

\begin{lemma}\label{Letiling}
Let $A_0$ and $A_1$ denote the rectangles in the plane $\mathbb{R}^2$ defined as
the convex hulls of the following corners:
\begin{align*}
A_0:&\quad \left(  \frac{-\golden}{\golden^2+1}, \frac{\golden^2}{\golden^2+1} \right),\, (0,1),\,
\left( \frac{\golden^2}{\golden^2+1}, \frac{-1/\golden}{\golden^2+1} \right), \,
\left(  \frac{1}{\golden^2+1}, \frac{-\golden}{\golden^2+1} \right), \\
A_1:&\quad \left(  \frac{1}{\golden^2+1}, \frac{2}{\golden^2+1} \right), \,
\left(  \frac{2}{\golden^2+1}, \frac{\golden^2}{\golden^2+1} \right), \, (1,0),\, 
\left(  \frac{\golden^2}{\golden^2+1}, \frac{-1/\golden}{\golden^2+1} \right).
\end{align*}
Then these two rectangles induce a periodic tiling of the plane with periods $\mathbb{Z}\times \mathbb{Z}$, 
that is, they constitute up to zero measure a partition of the unit square modulo $1$.
Their slopes are $(\golden,1)$ and
$(-1, \golden)$ and their areas are $\golden^2/(\golden^2+1)$ and $1/(\golden^2+1)$, respectively.

If 
\[
\delta_{\ijkl}(n) = \omega  \qquad (\omega\in \{0,1\}) 
\]
then
\begin{equation}\label{eqZdigcalc}
\left( \{ n \golden^{-\ijkl} \}, \{ n \golden^{-\ijkl-1} \} \right) \in (A_\omega \bmod 1) + O\bigl(\golden^{-\ijkl}\bigr).
\end{equation}
\end{lemma}

\begin{remark}
	The main advantage of this method is that $A_0$ and $A_1$ do not depend on $\ijkl$ --- which comes at the cost of the error term $\LandauO\rb{\golden^{-\ijkl}}$.
\end{remark}

\chapter{Gowers Norms}\label{chap_gowers}

In his work on Szemer\'{e}di's theorem~\cite{G2001}, 
Gowers introduced a new family of norms, which are nowadays known as \emph{Gowers norms} or \emph{Gowers uniformity norms}.
These norms are a fundamental object in what is now known as higher order Fourier analysis (see for example 
\cite{G2007} 
or \cite{T2012} 
for more background on Gowers norms). 
In this chapter we will show that the Gowers norms of the Zeckendorf sum-of-digits function $\sz$ are very small, i.e. exponentially decreasing to zero.
It turns out that it is beneficial to not use $\sz_{\lambda}$ directly, but work with the function $g_{\lambda}$ introduced in~\eqref{eqgLdef} (see also Lemma~\ref{Lefirstdigits}, which motivated the definition of $g_{\lambda}$).
Thus, we will use a variant of Gowers norms using \emph{integrals} instead of sums.
Both kinds of norms are introduced in Section~\ref{sec_gowers_def} along with their basic properties.

Then, we show in Section~\ref{sec_gowers_proof} that the function $\e(\vartheta\smallspace g_\lambda)$ is Gowers uniform of any order for $\lambda \to \infty$. This proof relies on a recursion formula~\eqref{eq_recursion} and a single cancellation in the appearing sum on the right hand side.

\section{Definition and properties of Gowers Norms}\label{sec_gowers_def}

There are different definitions of Gowers norms, used depending on the context.
We present here a few of them.
Let $G$ be a finite abelian group and $f: G \to \C$, then the Gowers uniformity $s$-norm is defined via
\begin{align*}
	\norm{f}^{2^s}_{U^s(G)} = \mathbb{E}_{x,h_1, \ldots, h_s \in G} \prod_{\varepsilon_1, \ldots, \varepsilon_s \in \{0,1\}} \conj^{\varepsilon_1 + \cdots + \varepsilon_s} f(x + h_1 \varepsilon_1 + \cdots + h_s \varepsilon_s),
\end{align*}
where $\conj$ denotes complex conjugation and $\mathbb{E}_{x\in G} h(x)$ denotes the average
\[
\mathbb{E}_{x \in G} h(x) = \frac 1{|G|} \sum_{x\in G} h(x).
\]

We can also define the Gowers $s$-norm for a complex valued function $f$ on a segment $[N] \coloneqq \{0,\ldots, N\}$ via
\begin{align*}
	\norm{f}_{U^s[N]} = \bigl\lVert\tilde{f}\bigr\rVert_{U^s(\Z / \tilde{N} \Z)} / \norm{\mathbf{1}_{[N]}}_{U^s(\Z / \tilde{N} \Z)},
\end{align*}
where $\tilde{N}$ is an arbitrary integer larger than $2^s N$\footnote{It is often useful to choose $\tilde{N}$ to be a prime number.}, $\tilde{f}(x)$ is equal to $f(x)$ for $x \in \{0,\ldots, N\}$ and $0$ otherwise and $\mathbf{1}_{[N]}$ is the indicator function of $[N]$.

There are also extensions to compact groups $G$,
where the expected value is defined as integration with respect to the Haar measure of the group
(see for example~\cite{HK2012} 
and~\cite{HK2005}).              
Thus, we can also define a Gowers $s$-norm for an integrable and $1$-periodic function $f: \R \to \C$ (we can view $f$ as a function from the torus $\mathbb{T} = \R/\Z$ to $\C$) via
\begin{align*}
	\norm{f}^{2^s}_{U^s(\mathbb{T})} \coloneqq \int_{[0,1]} \int_{[0,1]^{\dims}} \prod_{\varepsilon_1, \ldots, \varepsilon_{\dims} \in \{0,1\}}\conj^{\varepsilon_1 + \cdots + \varepsilon_{\dims}} f \rb{x + \sum_{1\leq i \leq {\dims}} y_i \varepsilon_i}\,\mathrm d(y_1,\ldots,y_s)\,\mathrm dx.
\end{align*}

In the context of Gowers norms, it is often useful to define a difference operator. Therefore, let $f: \R \to \C$ and $y \in \R$. Then
\begin{align}
	\Delta(f; y)(x) \eqdef f(x) \overline{f(x+y)}.
\end{align}

We will also use the iterated difference function, which is inductively defined for $y_1, \ldots, y_{\dims} \in \R$ as follows:
\begin{align*}
	\Delta\bigl(f; y_1, \ldots, y_{\dims}\bigr)(x) \eqdef \Delta(\Delta(f; y_1, \ldots, y_{\dims-1}); y_{\dims})(x).
\end{align*}
As we are only dealing with complex-valued functions, the appearing terms commute, which gives the following form
\begin{align}
	\Delta(f;y_1, \ldots, y_{\dims})(x) = \prod_{\varepsilon_1, \ldots, \varepsilon_{\dims} \in \{0,1\}}\rb{\conj^{\varepsilon_1 + \cdots + \varepsilon_{\dims}} f} \rb{x + \sum_{1\leq i \leq {\dims}} y_i \varepsilon_i}.
\end{align}
Thus, we can write
\begin{align}\label{eq_gowers_int_delta}
	\norm{f}^{2^{\dims}}_{U^{\dims}(\mathbb{T})} = \int_{[0,1]} \int_{[0,1]^{\dims}} \Delta(f; x_1, \ldots, x_{\dims})(x)\,\mathrm d(x_1, \ldots, x_{\dims})\,\mathrm dx.
\end{align}

Many of the properties of the classical Gowers norms carry over to this setting without any major difficulties. We only need very basic results and include the proofs for the convenience of the reader.

First we present an equivalent formulation for the Gowers norm of a function $f$.
\begin{lemma}\label{lem_uniform}
	Let $f$ be an integrable function from $\mathbb{T} = \R/\Z$ to $\mathbb{C}$. Then, 
	\begin{align*}
		\norm{f}_{U^{\dims +1}(\T)}^{2^{s+1}} = \int_{[0,1]^{\dims}} \left\lvert\int_{[0,1]} \Delta(f; x_1, \ldots, x_{\dims})(x)\,\mathrm dx\right\rvert^2\,\mathrm d(x_1, \ldots, x_{\dims}).
	\end{align*}
\end{lemma}
\begin{proof}
	This follows directly from the fact that for any integrable and $1$-periodic function $g$,
	\begin{align*}
		\abs{\int_{[0,1]}g(x)\,\mathrm dx}^2 &= \int_{[0,1]} g(x)\,\mathrm dx \int_{[0,1]} \overline{g(y)}\,\mathrm dy\\
			&= \int_{[0,1]^2} g(x) \overline{g(y)}\,\mathrm d(x,y)\\
			&= \int_{[0,1]^2} g(x) \overline{g(x+z)}\,\mathrm d(x,z).
	\end{align*}
\end{proof}

We also need the fact that if $f$ has small Gowers norm of order $\dims+1$, then $\Delta(f; z)$ has also small Gowers norm of order $\dims$ for most $z \in \T$.
\begin{lemma}\label{le_gowers_1}
	Let $f$ be an integrable function from $\mathbb{T}$ to $\C$, where $\norm{f}_{\infty} \leq 1$. Then there exists a set $\mathcal{M}$ of measure $\lambda(\mathcal{M}) \leq \norm{f}_{U^{\dims+1}(\T)}^{2^{\dims}}$ such that for any $z \notin \mathcal{M}$,
	\begin{align*}
		\norm{\Delta(f;z)}_{U^{\dims}(\T)} \leq \norm{f}_{U^{\dims+1}(\T)}.
	\end{align*}
\end{lemma}
\begin{proof}
	A simple reordering of integrals shows
	\begin{align*}
		\norm{f}_{U^{\dims +1}(\T)}^{2^{\dims+1}} = \int_{[0,1]}\norm{\Delta(f;z)}_{U^{\dims}(\T)}^{2^{\dims}}\,\mathrm dz.
	\end{align*}
	Let us now consider
	\begin{align*}
		\mathcal{M} \eqdef \bigl\{z : \norm{\Delta(f;z)}_{U^{\dims}(\T)} > \norm{f}_{U^{\dims+1}(\T)}\bigr\},
	\end{align*}
	and assume that $\lambda(\mathcal{M}) > \norm{f}_{U^{\dims+1}(\T)}^{2^{\dims}}$.
	This implies immediately 
	\begin{align*}
		\int_{[0,1]}\norm{\Delta(f;z)}_{U^{\dims}(\T)}^{2^{\dims}}\,\mathrm dz > \norm{f}_{U^{\dims+1}(\T)}^{2^{\dims}} \cdot \norm{f}_{U^{\dims+1}(\T)}^{2^{\dims}},
	\end{align*}
	which gives a contradiction.
\end{proof}

\section{Gowers Norms for the Zeckendorf sum of digits}\label{sec_gowers_proof}

The Gowers norm of automatic sequences has already been studied by Konieczny for the Thue--Morse sequence and the Rudin--Shapiro sequence~\cite{Konieczny2019} and for general automatic sequences by Byszewski, Konieczny and the second author~\cite{Byszewski2020}.
The result presented in this section is the first estimate of Gowers norms of a morphic sequence.
The strategy used in this section is similar to the one used for the mentioned results for automatic sequences: First we find a recursion for the Gowers norm which relies on the structure of morphic sequences. Then we find some cancellation in this recursion, which is already sufficient to obtain the result.

Throughout this section we fix some positive integer $\dims$ and denote by $\gap$ the smallest integer such that $\dims+1 < F_{\gap} < \golden^{\gap-1}$.
We are interested in estimating the Gower's norm of order $s$ of the aforementioned function $\e(\vartheta\smallspace g_\lambda)$ (see~\eqref{eqgLdef}).
We start by defining the Zeckendorf digits for a real number $x \in [0,1)$. 
Therefore, we recall that for $x = \{n \golden\}$ and any $\ijkl\geq 2$, we defined
\begin{align*}
	\widetilde{\digit_{\ijkl}}(x) = \digit_{\ijkl}(n).
\end{align*}
Moreover,
\begin{align*}
	\digit_{\ijkl}'(x) = \lim_{z\to x+} \widetilde{\digit_{\ijkl}}(z),
\end{align*}
where the limit is taken from the right.
Up to a finite set, the function $\digit_{\ijkl}'$ is defined by the closure of the graph of $\widetilde{\digit_{\ijkl}}$; since $\digit_{\ijkl}$ has values in $\mathbb N$, the function $\digit_{\ijkl}'$ is piecewise constant.

Thus, we can also talk about the Zeckendorf expansion of a real number $x \in [0,1)$ and also define
\begin{align*}
	v'(x, \lambda) = \sum_{\ijkl = 2}^{\lambda} \digit_{\ijkl}'(x) F_{\ijkl}.
\end{align*}
The Zeckendorf expansion of a real number shares the properties of the Zeckendorf expansion of integers as $v(n,\lambda) = v'(n \golden, \lambda)$.

This of course relates back to our definition of $g_{\lambda}$
\begin{align}\label{eq_def_g_L}
	g_\lambda(x) = \sum_{\ijkl = 2}^{\lambda} \digit_{\ijkl}'(x) = \sz(v'(x,\lambda)),
\end{align}
which is piecewise constant and the main focus of this section.

First, we give a result that allows to decompose the contribution of high and low digits in this new setting.
\begin{lemma}\label{le_separate_s}
	Let $x_0, \ldots, x_{\dims} \in [0,1]$ be such that there exists some $k \in \N$ and integers $m_{\ijkl}< F_{k}$ with $v(x_{\ijkl}, k+2\gap) = m_{\ijkl}$.
	Then we have for all $\ell \in \N$,
	\begin{align*}
		\digit_{\ell}'(x_0 + \cdots + x_{\dims}) = \digit_{\ell}'(x_0 + \cdots + x_{\dims} - m_0 \golden - \cdots - m_{\dims} \golden) + \digit_{\ell}(m_0 + \cdots + m_{\dims}).
	\end{align*}
\end{lemma}
\begin{remark}
	This lemma tells us how to separate the contribution of ``low digits'' ($m_i$) if they are separated from the remaining ``high digits'' by at least $2r$ zeros.
	This is due to the fact, that we have at most $r$ carries to the left or to the right, when adding $s+1$ numbers (we recall that $r$ is defined to satisfy $\dims+1 < F_r < \golden^{r-1}$).
\end{remark}
\begin{proof}
	We start by stating the following trivial equation,
	\begin{align*}
		\golden \cdot (x_0 + \cdots + x_{\dims}) = \golden \cdot (x_0 + \cdots + x_{\dims} - m_0 - \cdots - m_{\dims}) + \golden \cdot (m_0 + \cdots + m_{\dims}).
	\end{align*}
	Therefore, the main point of the proof is to show that the sum of the Zeckendorf expansions (for real numbers) of the two terms on the right-hand side give again a valid Zeckendorf expansion.
	We see that $v'(x_i, k+2\gap) = m_i < F_{k}$ is equivalent to $x_i \in \golden m_i + A_{k+2\gap}^{(\1)} + \Z$.
	Thus, we have
	\begin{align*}
		\sum_{i=0}^{m} x_i &\in \sum_{i=0}^{\dims}\rb{\golden m_i + A_{k+2\gap}^{(\1)}} + \Z\\
			&\subset \golden \sum_{i=0}^{\dims} m_i + (\dims+1) A_{k+2\gap}^{(\1)} + \Z\\
			&\subset \golden \sum_{i=0}^{\dims} m_i + A_{k+\gap}^{(\1)} + \Z.
	\end{align*}
	In other words, $v'(x_0 + \cdots + x_{\dims}, k+\gap) = m_0 + \cdots + m_{\dims}$. Moreover, we see that $m_0 + \cdots + m_{\dims} < (\dims+1) F_{k} < F_{k+\gap-1}$. The last inequality follows directly from $\dims+1 < F_{\gap}$ and the well-known identity
	\begin{align*}
		F_m \cdot F_n = F_{m+n-1} - F_{m-1} \cdot F_{n-1}.
	\end{align*}
	This shows that the non-zero digits of $x_0 + \cdots + x_{\dims} - (m_0 + \cdots + m_{\dims})$ and $m_0 + \cdots + m_{\dims}$ are separated by at least one zero and the result follows immediately.
\end{proof}

\begin{corollary}\label{cor_separate_s}
	Let $n_0, \ldots, n_{\dims} \in \N$ be such that there exists some $k\in \N$ and $m_0, \ldots, m_{\dims} < F_{k}$ and $t_0, \ldots, t_{\dims}$ with $v(t_i, k+2\gap) = 0$ and $n_i = m_i + t_i$. Then we have for all $\ell \in \N$,
	\begin{align*}
		\digit_{\ell}(n_0 + \cdots + n_{\dims}) = \digit_{\ell}(m_0 + \cdots + m_{\dims}) + \digit_{\ell}(t_0 + \cdots + t_{\dims}).
	\end{align*}
\end{corollary}

Now we come back to our Gowers norm~\eqref{eq_gowers_int_delta} and decompose the interval $[0,1]$ into disjoint intervals $R_{\mu}(n) = n \golden + A_{\mu}^{(\delta_{\mu-1}(n))}$ for $0 \leq n < F_{\mu}$. The interval $R_{\mu}(n)$ corresponds exactly to the real numbers $x$ which have the same digits as $n$ up to position $\mu-1$ (see Lemma~\ref{Lefirstdigits}).
This gives
\begin{align}\label{eq_gowers_recursion}
	&\norm{\e(\vartheta g_{\lambda})}^{2^{\dims}}_{U^{\dims}(\T)} 
		= \sum_{n_0, \ldots, n_{\dims} < F_{\mu}} S_{\mu}(n_0, \ldots, n_{\dims}),
\end{align}
where
\begin{align*}
	S_{\mu}(n_0, \ldots, n_{\dims}) \eqdef \int_{R_{\mu}(n_0) \times \cdots \times R_{\mu}(n_{\dims})} \Delta(\e(\vartheta g_{\lambda}(.)); x_1, \ldots, x_{\dims})(x_0)\,\mathrm d(x_0, \ldots x_{\dims}).
\end{align*}

By the definition of $R_{\mu}(n_i) = n_i \golden + A_{\mu}^{(\delta_{\mu-1}(n_i))}$, we can write
\begin{align*}
\hspace{0em}&\hspace{-0em}
	S_{\mu}(n_0, \ldots, n_{\dims}) =\int_B
\Delta\bigl(\e(\vartheta g_{\lambda}); x_1 + \golden n_1, \ldots, x_{\dims} + \golden n_{\dims}\bigr)(x_0 + \golden n_0)\,\mathrm d(x_0, \ldots, x_{\dims}),
\end{align*}
where we set
\[B=A_{\mu}^{(\delta_{\mu-1}(n_0))} \times \cdots \times A_{\mu}^{(\delta_{\mu-1}(n_{\dims}))}.\]
If the non-zero digits of $n_0,\ldots,n_{\dims}$ separate into two blocks, then we can apply Lemma~\ref{le_separate_s} to separate the contribution of these blocks.

Furthermore, we can also decompose each individual summand analogously, which gives for $\mu' > \mu$
\begin{align}\label{eq_recursion}
	S_{\mu}(n_0, \ldots, n_{\dims}) = \sum_{\substack{\mathbf{n}' < F_{\mu'}\\ v(\mathbf{n}', \mu) = \mathbf{n} } }  S_{\mu'}(n_0', \ldots, n_{\dims}'),
\end{align}
where $\mathbf{n'} = (n_0', \ldots, n_{\dims}')$, $\mathbf{n}' < F_{\mu'}$ means $n_i' < F_{\mu'}$ for $i = 0, \ldots, \dims$ and $v(\mathbf{n}', \mu) = \mathbf{n}$ means $v(n_i', \mu) = n_i$ for $i = 0, \ldots, \dims$.
Therefore, we are interested in the number of summands in equation~\eqref{eq_recursion}. This number depends only on the $(\dims+1)$-tuple $\delta_{\mu-1}(\mathbf{n}) \eqdef (\delta_{\mu-1}(n_0), \ldots, \delta_{\mu-1}(n_{\dims}))$. As we aim to iteratively apply~\eqref{eq_recursion}, we also need to keep track of $\delta_{\mu'-1}(\mathbf{n'})$. This motivates the following definitions
\begin{align*}
	M(\mu, \mu', \mathbf{n}, \underline{\delta}') &\eqdef \{\mathbf{n}' \in [0,F_{\mu'})^{\dims+1}: \delta_{\mu'-1}(\mathbf{n'}) = \underline{\delta}', v(\mathbf{n'}, \mu) = \mathbf{n}\}\\
	N(\mu, \mu', \mathbf{n}, \underline{\delta}') &\eqdef \abs{M(\mu, \mu', \mathbf{n}, \underline{\delta}')} = \prod_{i=0}^{\dims} F_{\mu'-\mu-3-\delta_{\mu-1}(n_i) - \delta'_i}.
\end{align*}
We see in particular, that $N(\mu, \mu', \mathbf{n}, \underline{\delta}')$ does only depend on $\delta_{\mu-1}(\mathbf{n})$. Thus, we will also denote it by $N(\mu, \mu', \delta_{\mu-1}(\mathbf{n}), \underline{\delta}')$ instead.

These definitions allow us to find the following recursion. 
\begin{proposition}\label{pr_recursion}
	We have for $\mu \leq \lambda - 5\gap-\3$,
	\begin{align*}
		\abs{S_{\mu}(\mathbf{n})} &\leq \rb{1-\frac{4\norm{\vartheta}^2}{\golden^{5\gap(\dims+1)} } } \sum_{\underline{\delta}' \in \{1,2\}^{\dims+1}} N(\mu, \mu+5\gap+\3, \mathbf{n}, \underline{\delta}')
\\&\times\max_{\mathbf{n}' \in M(\mu, \mu+5\gap+\3, \mathbf{n}, \underline{\delta}')} \abs{S_{\mu+5\gap+\3}(\mathbf{n'})}.
	\end{align*}
\end{proposition}
The same recursion without the factor $\rb{1-\frac{4\norm{\vartheta}^2}{\golden^{5\gap(\dims+1)} } }$ can be directly obtained from~\eqref{eq_recursion}, but this is not sufficient to prove Theorem~\ref{cor_gowers_estimate}.
\begin{proof}
	We recall that by \eqref{eq_recursion} and the triangle inequality,
	\begin{align*}
		\abs{S_{\mu}(\mathbf{n})} \leq \sum_{\underline{\delta}' \in \{\1,\2\}^{m+1}}  \abs{\sum_{\mathbf{n}' \in M(\mu, \mu+5\gap+\3, \mathbf{n}, \underline{\delta}')} S_{\mu+5\gap+\3}(\mathbf{n'})}.
	\end{align*}
	We will treat each choice of $\underline{\delta}'$ individually (but uniformly).
	Therefore, we fix $\underline{\delta}'$ and aim to find two choices of $n_i^{(1)}, n_i^{(2)} \in M(\mu, \mu+5\gap+\3, \mathbf{n}, \underline{\delta}')$ which cancel at least partially.
	We note that $(\dims-1) F_{\mu + 2\gap} < F_{\mu + 3\gap}$ and denote by $y$ the least integer such that $v(y, \mu + 2\gap) = 0$ and $y + (\dims-1) F_{\mu + 2\gap} \geq F_{\mu + 3\gap}$. By Lemma~\ref{le_create_zeros} we know that $y + (\dims-2) F_{\mu + 2\gap} < F_{\mu + 3\gap}$. This guarantees also $y + (\dims-1) F_{\mu + 2\gap} < F_{\mu + 3\gap+1}$.

We define,
	\begin{align*}
		n_i^{(1)}= n_i^{(2)} &= \delta_i \cdot F_{\mu+5\gap+3} + F_{\mu+2\gap} + n_i \qquad \mbox{for } 2\leq i \leq \dims,\\
		n_1^{(1)} = n_1^{(2)} &=\delta_1 \cdot F_{\mu+5\gap+3} + y + n_1,\\
		n_0^{(1)} &= \delta_0 \cdot F_{\mu+5\gap+3} + F_{\mu + 3\gap+2} + n_0,\\
		n_0^{(2)} &= \delta_0 \cdot F_{\mu+5\gap+3} + F_{\mu + 3\gap +1} + n_0.
	\end{align*}
	
	We see that all of these integers decompose into three summands having possible non-zero digits in the Zeckendorf representation at positions $[2,\mu-1], [\mu + 2\gap, \mu + 3\gap+2], [\mu + 5\gap + 3, \mu + 5\gap + 3]$.
	Therefore, we decompose $n_i^{(j)} = t_i^{(j)} + u_i^{(j)} + v_i^{(j)}$, where $t_i^{(j)}, u_i^{(j)}, v_i^{(j)}$ correspond to the high, middle and low digits respectively.
	The most significant digit is chosen to ensure $\delta_{\mu + 5\gap+3}(\mathbf{n}^{(j)}) = \underline{\delta}'$ and the least significant digits to ensure $v(\mathbf{n}^{(j)}, \mu) = \mathbf{n}$. Both of them are independent of $j$. The digits in the middle are chosen to guarantee some cancellation.
	Lemma~\ref{le_separate_s} and Corollary~\ref{cor_separate_s} imply that we can treat the sum of digits of the three parts independently:
	\begin{align*}
		S_{\mu}(\mathbf{n}^{(j)}) &= \int_B \Delta(\e(\vartheta g_\lambda); x_1 + \golden n_1^{(j)}, \ldots, x_{\dims} + \golden n_{\dims}^{(j)})(x_0 + \golden n_0^{(j)})\,\mathrm d(x_0, \ldots, x_{\dims})\\
		&= \Delta(\e(\vartheta g_\lambda); \golden(u_1^{(j)} + v_1^{(j)}), \ldots, \golden(u_{\dims}^{(j)} + v_{\dims}^{(j)}))(\golden(u_0^{(j)} + v_0^{(j)}))\\
		&\quad\times\int_B \Delta(\e(\vartheta g_\lambda); x_1 + \golden t_1^{(j)}, \ldots, x_{\dims} + \golden t_{\dims}^{(j)})(x_0 + \golden n_0^{(j)})\,\mathrm d(x_0, \ldots, x_{\dims})\\
		&= \Delta(\e(\vartheta \sz_\lambda); v_1^{(j)}, \ldots, v_{\dims}^{(j)})(v_0^{(j)}) \cdot \Delta(\e(\vartheta \sz_\lambda); u_1^{(j)}, \ldots, u_{\dims}^{(j)})(u_0^{(j)})\\
		&\quad\times\int_B \Delta(\e(\vartheta g_\lambda); x_1 + \golden t_1^{(j)}, \ldots, x_{\dims} + \golden t_{\dims}^{(j)})(x_0 + \golden n_0^{(j)})\,\mathrm d(x_0, \ldots, x_{\dims}),
	\end{align*}
where we use the abbreviation
\[B=A_{\mu}^{(\delta_0')} \times \cdots \times A_{\mu}^{(\delta_{\dims}')}.\]

	As $t_i^{(1)} = t_i^{(2)}$ and $v_i^{(1)} = v_i^{(2)}$ for $i = 0, \ldots, \dims$, it remains to consider the contribution of $u_i^{(j)}$:
	\begin{align*}
		\Delta(\e(\vartheta &\sz_\lambda); -u_1^{(j)}, \ldots, -u_{\dims}^{(j)})(u_0^{(j)})\\
			 &= \prod_{\varepsilon_1, \ldots, \varepsilon_{\dims} \in \{0,1\}}\conj^{\varepsilon_1 + \cdots + \varepsilon_{\dims}} \e(\vartheta \sz_\lambda) \rb{u_0^{(j)} + \sum_{1\leq i \leq {\dims}} u_i^{(j)} \varepsilon_i}\\
			&= \prod_{\varepsilon_1, \ldots, \varepsilon_{\dims} \in \{0,1\}} \e\rb{(-1)^{\varepsilon_1 + \cdots + \varepsilon_{\dims}}\vartheta \cdot \sz_{\lambda}\rb{u_0^{(j)} + y \varepsilon_1 + F_{\mu + 3\gap + 2} (\varepsilon_2 + \cdots + \varepsilon_{\dims})}}.
	\end{align*}
	
	We recall that $y + (\dims - 2) F_{\mu + 2\gap} < F_{\mu + 3\gap}, (\dims -1) F_{\mu + 2\gap} < F_{\mu + 3\gap}$.
	As $u_0^{(1)} = F_{\mu + 3\gap + 2}, u_0^{(2)} = F_{\mu + 3\gap +1}$, we see that the digits of $u_0^{(j)}$ and $y \varepsilon_1 + F_{\mu + 3\gap + 2} (\varepsilon_2 + \cdots + \varepsilon_{\dims})$ are separated by a $0$ at position $\mu+3\gap$ as long as $\varepsilon_i = 0$ for at least one $i \in \{1, \ldots, \dims\}$. This implies for this case
	\begin{align*}
\hspace{4em}&\hspace{-4em}
		\sz_{\lambda}\rb{u_0^{(1)} + y \varepsilon_1 + F_{\mu + 3\gap + 2} (\varepsilon_2 + \cdots + \varepsilon_{\dims})} 
\\&= \sz_{\lambda}\rb{u_0^{(2)} + y \varepsilon_1 + F_{\mu + 3\gap + 2} (\varepsilon_2 + \cdots + \varepsilon_{\dims})}.
	\end{align*}
	For the remaining case $\varepsilon_i = 1$ for all $i = 1, \ldots, \dims$, we have
	\begin{align*}
		\sz_{\lambda}\rb{u_0^{(1)} + y \varepsilon_1 + F_{\mu + 3\gap + 2} (\varepsilon_2 + \cdots + \varepsilon_{\dims})} &= \sz_{\lambda}(F_{\mu + 3\gap + 2} + y + (\dims-1) F_{\mu + 2\gap})\\
		\sz_{\lambda}\rb{u_0^{(2)} + y \varepsilon_1 + F_{\mu + 3\gap + 2} (\varepsilon_2 + \cdots + \varepsilon_{\dims})} &= \sz_{\lambda}(F_{\mu + 3\gap + 1} + y + (\dims-1) F_{\mu + 2\gap}).
	\end{align*}

	As $F_{\mu + 3\gap} \leq y + (\dims-1) F_{\mu + 2\gap} < F_{\mu + 3\gap +1}$, we see that the digits of $y  (\dims-1) F_{\mu + 2\gap}$ and $F_{\mu + 3\gap+2}$ are again separated by a $0$ at position $\mu + 3\gap+1$, but the digits of $y  (\dims-1) F_{\mu + 2\gap}$, $F_{\mu + 3\gap+2}$ and $F_{\mu + 3\gap + 1}$ have non-zero digits at position $\mu + 3\gap$ and $\mu + 3\gap +1$. This creates one new non-zero digit at position $\mu + 3\gap +2$ instead of the two ones. This gives in total
	\begin{align*}
		\sz_{\lambda}(F_{\mu + 3\gap + 2} + y + (\dims-1) F_{\mu + 2\gap}) = \sz_{\lambda}(F_{\mu + 3\gap + 1} + y + (\dims-1) F_{\mu + 2\gap}) + 1.
	\end{align*}
	Combining all these calculations gives
	\begin{align*}
		S_{\mu + 5\gap+\3}(\mathbf{n}^{(1)}) = S_{\mu + 5\gap + \3}(\mathbf{n}^{(2)}) \cdot \e((-1)^{\dims} \vartheta).
	\end{align*}
	This shows together with the triangle inequality and the estimate $\abs{1 + \e((-1)^m \vartheta)} = 2\cos(\pi \norm{\vartheta}) = 2(1-2\sin^2(\norm{\vartheta} \pi/2)) \leq 2 - 4\norm{\vartheta}^2$,
	\begin{align*}
		 &\abs{\sum_{\mathbf{n}' \in M(\mu, \mu+5\gap+\3, \mathbf{n}, \underline{\delta}')} S_{\mu+5\gap+\3}(\mathbf{n'})}\\
			&\qquad  \leq \rb{N(\mu, \mu+5\gap+\3, \mathbf{n}, \underline{\delta}') - 4\norm{\vartheta}}^2 \max_{\mathbf{n}' \in M(\mu, \mu+5\gap+\3, \mathbf{n}, \underline{\delta}')} \abs{S_{\mu+5\gap+\3}(\mathbf{n'})}.
	\end{align*}
	Thus, the statement follows from the fact that
\[N(\mu, \mu+5\gap+\3, \mathbf{n}, \underline{\delta}') \leq (F_{5\gap+1})^{\dims+1} \leq \golden^{5\gap(\dims+1)}.\]
\end{proof}

Applying Proposition~\ref{pr_recursion} iteratively gives the following theorem.

\begin{theorem}\label{cor_gowers_estimate}
	For any $\dims\in \N$, there exists some $c = c(\dims)>0$ such that
	\begin{align*}
		\bigl\lVert\e\bigl(\vartheta g_{\lambda}\bigr)\bigr\rVert^{2^{\dims}}_{U^{\dims}(\mathbb{T})}
    \ll_s \exp\left(- c \lambda\lVert\vartheta\rVert^2\right),
	\end{align*}
	for $\lambda \to \infty$.
\end{theorem}
\begin{proof}
	We write $\lambda = (5\gap + \3) \lambda' + k$, where $k \in \{0,1,\ldots, 5\gap + 3\}$ and $\lambda'$ is an integer. Then we apply~\eqref{eq_gowers_recursion} for $\mu = k$ to find
	\begin{align*}
		\norm{\e\rb{\vartheta g_\lambda}}^{2^{\dims}}_{U^{\dims}(\T)} = \sum_{n_0, \ldots, n_{\dims} < F_k} S_k (n_0, \ldots, n_{\dims}).
	\end{align*}
	Our goal is to prove by induction on $\ell$ that
	\begin{align}\label{eq_induction_gowers}
	\begin{split}
		\norm{\e\rb{\vartheta g_\lambda}}^{2^{\dims}}_{U^{\dims}(\T)} &\leq \rb{1-\frac{4\norm{\vartheta}^2}{\golden^{5\gap(\dims+1)} } }^{\ell} \sum_{n_0, \ldots, n_{\dims} < F_k} 
\sum_{\underline{\delta}' \in \{\1,\2\}^{\dims+1}} N(k, k+\ell(5\gap+\3), \mathbf{n}, \underline{\delta}') 
\\&\times \max_{\mathbf{n}' \in M(k, k+\ell(5\gap+\3), \mathbf{n}, \underline{\delta}')} \abs{S_{k+\ell(5\gap+\3)}(\mathbf{n'})}
	\end{split}
	\end{align}
	holds for every $0\leq \ell \leq \lambda'$.
	The case $\ell = 0$ follows directly, as
	\begin{align*}
		N(k,k, \mathbf{n}, \underline{\delta}') = \left\{
		\begin{array}{cl} 1, & \text{if }\delta_k(\mathbf{n}) = \underline{\delta}';\\ 0, & \text{otherwise.} \end{array}\right.
	\end{align*}
	Now suppose that~\eqref{eq_induction_gowers} holds for $\ell<\lambda'$. Then, by applying Proposition~\ref{pr_recursion}, we find
	\begin{align*}
		&\abs{S_{k + \ell(5\gap + \3)}(\mathbf{n}')} \\
		&\leq \rb{1-\frac{4\norm{\vartheta}^2}{\golden^{5\gap(\dims+1)} } } \sum_{\underline{\delta}'' \in \{\1,\2\}^{\dims+1}} N(k + \ell(5\gap + \3), k+(\ell+1)(5\gap+\3), \mathbf{n}', \underline{\delta}'')
\\&\qquad \times\max_{\mathbf{n}'' \in M(k + \ell(5 \gap + \3), k+ (\ell+1)(5\gap+\3), \mathbf{n}', \underline{\delta}'')} \abs{S_{k+(\ell+1)(5\gap+\3)}(\mathbf{n}'')}\\
		&\leq \rb{1-\frac{4\norm{\vartheta}^2}{\golden^{5\gap(\dims+1)} } } \sum_{\underline{\delta}'' \in \{\1,\2\}^{\dims+1}} N(k + \ell(5\gap + \3), k+(\ell+1)(5\gap+\3), \mathbf{n}', \underline{\delta}'')
\\&\qquad \times\max_{\mathbf{n}'' \in M(k, k+ (\ell+1)(5\gap+\3), \mathbf{n}, \underline{\delta}'')} \abs{S_{k+(\ell+1)(5\gap+\3)}(\mathbf{n}'')}.
	\end{align*}
	
Finally, noting that
\begin{align*}
	\sum_{\underline{\delta}' \in \{\1,\2\}^{\dims+1}} &N\bigl(k, k+\ell(5 \gap + \3), \mathbf{n}, \underline{\delta}'\bigr) N\bigl(k+\ell(5 \gap + \3), k+(\ell+1)(5 \gap + \3), \underline{\delta}', \underline{\delta}''\bigl) \\
		&= N(k, k+(\ell+1)(5\gap + \3), \mathbf{n}, \underline{\delta}''),
\end{align*}
	finishes the induction step.
	
	Moreover, we will use the trivial estimate
	\begin{align*}
		\abs{S_\lambda(\mathbf{n}')} &\leq \abs{A_\lambda^{(\delta_{\lambda-1}(n_0'))} \times \cdots \times A_\lambda^{(\delta_{\lambda-1}(n_{\dims}'))}}\\
				& = \frac{1}{\golden^{\lambda-2+\delta_{\lambda-1}(n_0')}} \times \cdots \times \frac{1}{\golden^{\lambda-2+\delta_{\lambda-1}(n_{\dims}')}}
	\end{align*}
	This shows in total that
	\begin{align*}
		\abs{S_\lambda(\mathbf{n}')} \cdot N(k, \lambda, \mathbf{n}, \underline{\delta}'') &\leq \prod_{i=0}^{s} \frac{F_{\lambda-k-3-\delta_{k}(n_i') - \delta_i''}}{\golden^{\lambda-2+\delta_{\lambda-1}(n_i')}},
	\end{align*}
	which is bounded by an absolute constant (depending on $\dims$).
	Thus, equation~\eqref{eq_induction_gowers} for $\ell = \lambda'$ shows, that (since $k$ is bounded by $5\gap + 3$ which only depends on $\dims$)
	\begin{align*}
		\norm{\e\rb{\vartheta g_\lambda}}^{2^{\dims}}_{U^{\dims}(\T)} &\leq \rb{1-\frac{4\norm{\vartheta}^2}{\golden^{5\gap(\dims+1)} } }^{\lambda'} \sum_{n_0, \ldots, n_{\dims} < F_k} 
\sum_{\underline{\delta}' \in \{\1,\2\}^{\dims+1}} N(k, \lambda, \mathbf{n}, \underline{\delta}') 
\\&\times \max_{\mathbf{n}' \in M(k, \lambda, \mathbf{n}, \underline{\delta}')} \abs{S_{\lambda}(\mathbf{n'})}\\
			&\ll_s \rb{1-\frac{4\norm{\vartheta}^2}{\golden^{5\gap(\dims+1)} } }^{\lambda'}
			\leq \rb{1-\frac{4\norm{\vartheta}^2}{\golden^{5\gap(\dims+1)} } }^{\lambda/(5\gap + \3) - 1}\\
			&\ll \exp\left(-\frac{4\lambda\lVert\vartheta\rVert^2}{(5\gap + \3) \golden^{5\gap(\dims+1)} } \right),
	\end{align*}
	which finishes the proof for $c(\dims) = \frac{4}{(5\gap + \3) \golden^{5\gap(\dims+1)}}$ which only depends on $\dims$ as $\gap$ depends only on $\dims$ (we recall that $\gap$ was defined as the smallest integer such that $\dims+1 < F_{\gap} < \golden^{\gap-1}$).
\end{proof}
\begin{remark}
	We note that $c(\dims)$ decreases very fast, i.e. $c(\dims) \approx \frac{1}{\log(\dims) \dims^{5 \dims}}$.
\end{remark}


\chapter{The level of distribution of $\sz$}\label{chap_lod}
In this chapter, we prove that the Zeckendorf sum-of-digits function has level of distribution $1$, that is, Theorem~\ref{thm_lod}.

The very rough idea of proof is to follow the ideas devised in the recent paper~\cite{S2020} by the third author.
In that paper, the problem of determining the level of distribution of the Thue--Morse sequence was reduced to a Gowers norm estimate.
The estimate for the Zeckendorf--Gowers norm found in Chapter~\ref{chap_gowers} will therefore be essential for the proof.
For our purposes, it will turn out useful to deal with the integral variant of this notion.
The reason for the usefulness of this variation (replacing sums by integrals) lies in the fact that Zeckendorf digits are not as well-behaved as $q$-ary digits with respect to addition.
Due to \emph{inverse carry propagation}, knowing the initial $L$ Zeckendorf digits of the summands $a$ and $b$ is not enough to determine the first $L$ digits of $a+b$.
Since Gowers norms are (among other things) concerned with repeated addition, this feature of the Zeckendorf expansion comes into play and has to be dealt with.
For this reason, we switch to the ``continuous'' version,
using the one-dimensional detection procedure via $n\golden$ and discrepancy estimates instead of working with digits directly.
On this level, addition does not cause any problem --- a uniformly distributed sequence, rotated by a constant, is still uniformly distributed --- and we avoid technical complications.

Plan of the chapter.
Section~\ref{sec_average} is concerned with detection of Zeckendorf digits on arithmetic progressions $nd+a$, where an average over $d$ is introduced.
This section comprises three important propositions that we will use at the core of our argument leading to the main result.
In Section~\ref{sec_proofs} we prove our main theorem, using the Gowers norm estimate from Theorem~\ref{cor_gowers_estimate}.
\section{Zeckendorf digits along arithmetic progressions}\label{sec_average}
\subsection{Introducing average discrepancy}
Detection of the lowest $L-2$ Zeckendorf digits can be performed by applying Lemma~\ref{Lefirstdigits}.
In fact, it will turn out to be convenient to detect \emph{general} wrapped intervals $\subseteq[0,1)$ (as defined in~\eqref{eqn_wrapped_intervals} below), and not just those corresponding to initial Zeckendorf digits.
This corresponds to the \emph{discrepancy} of $n\alpha$-sequences, in fact we will use the following rotation-invariant variant.
For a subset $A\subseteq \mathbb R$, let $\ind_A$ be the indicator function of $A$ defined by \[\ind_A(x)=\begin{cases}1,&\text{ if }x\in A;\\0,&\text{ otherwise.}\end{cases}\]

For a real number $\alpha$ we denote the discrepancy $D_N(n\alpha)$ of the sequence $(n\alpha)_{n=0}^{N-1}$ (see Section~\ref{sec_discrepancy}) by 
$\tilde D_N(\alpha)$.

In the level of distribution-statement we will encounter a sum over the common difference $d$ of the considered arithmetic progressions.
This will lead to the averages of $\tilde D_T(d\golden)$, as $d$ runs.
A rough estimate of this sum can be obtained as follows.
\begin{lemma}\label{lem_average_discrepancy}
Assume that $D$ and $T$ are positive integers.
Then
\begin{equation}\label{eqn_average_discrepancy}
\sum_{D\leq d<2D}\tilde D_T(d\golden)\ll \frac{D\log^+(DT)}{\sqrt{T}}
\end{equation}
holds with an absolute implied constant.
\end{lemma}
\begin{proof}
Let $H\geq 1$ be an integer. We apply the Erd\H{o}s--Tur\'an inequality (Lemma~\ref{lem_ETK}) in the one-dimensional case 
and obtain uniformly for all integers $D,H,T\ge1$
\begin{align*}
\sum_{D\leq d<2D}D_T(d\golden)
&\ll \frac DH+\frac 1T\sum_{1\leq h<H}\frac 1h\sum_{D\leq d<2D}\left\lvert\sum_{0\leq t<T}\e(htd\golden)\right\rvert.
\end{align*}
Applying~\eqref{eq:estimate-geometric-series}, and extending the range of summation over $d$ we obtain
\begin{align*}
\sum_{1\leq h<H}\frac 1h\sum_{D\leq d<2D}\left\lvert\sum_{0\leq t<T}\e(htd\golden)\right\rvert
&\leq \sum_{1\leq h<H}\frac 1h\sum_{D\leq d<2D}\min\bigl(T,\lVert hd\golden\rVert^{-1}\bigr)\\
&\leq \sum_{1\leq h<H}\frac 1h\sum_{hD\leq d<2hD}\min\bigl(T,\lVert d\golden\rVert^{-1}\bigr).
\end{align*}

At the cost of a longer summation over $d$, we have won a full sequence $d\mapsto d\golden$ instead of a subsequence thereof.
The values $d\golden$ are distributed in a very uniform manner, and the discrepancy of this sequence is only logarithmic
(see Theorem~\ref{th_bounded_quotients}).
Rotation by $\xi$ does not change the logarithmic behavior of the discrepancy of $d\golden+\xi\bmod 1$.
In particular, we can consider any interval for $d$.
For brevity, we write $I=[hD,2hD)$, which contains $B=hD$ integers.
Note that $0\not\in I$; we obtain $\lVert d\golden\rVert\gg (hD)^{-1}$ for all $d\in I$ since $\golden$ is badly approximable. The implied constant is absolute.
Thus we can replace $\min\rb{T, \norm{d\golden}^{-1}}$ by $\min\rb{hD, \norm{d\golden}^{-1}}$.
Therefore, by Theorem~\ref{th_koksma}, Lemma~\ref{le_variation_diff}, and Theorem~\ref{th_bounded_quotients},
\begin{align*}
\hspace{4em}&\hspace{-4em}
	\frac{1}{hD} \sum_{d\in I} \min\rb{T, \norm{d\golden}^{-1}} \ll \frac{1}{hD} \sum_{d\in I} \min\rb{hD, \norm{d\golden}^{-1}}\\
		& = \int_{[0,1]} \min\rb{hD, \norm{x}^{-1}}\mathrm dx + O\rb{V_0^1\Bigl(\min\Bigl(hD, \norm{\cdot}^{-1}\Bigr)\Bigr) D_{hD}(\golden)}\\
		&\ll \log^+(hD) + hD \frac{\log^+(hD)}{hD}
		\ll \log^+(hD).
\end{align*}
This gives in total
\begin{align*}
	\sum_{D\leq d < 2D} D_T(d\golden) &\ll \frac{D}{H} + \frac{1}{T} \sum_{1\leq h < H} D \log^+(hD)\\
		&\ll D \log^+(HD)\rb{\frac{1}{H} + \frac{H}{T}}.
\end{align*}
Choosing $H = \bigl\lfloor\sqrt{T}\bigr\rfloor$ gives the desired result.
\end{proof}
We define the set $\mathcal I$ of \emph{wrapped intervals} in $[0,1]$,
\begin{equation}\label{eqn_wrapped_intervals}
\mathcal I=\{I\subseteq [0,1]:\mbox{there exists an interval $J$ in $\mathbb R$ such that } I+\mathbb Z=J+\mathbb Z\}.
\end{equation}
The Lebesgue measure $\lambda(I)$ of a wrapped interval $i\in\mathcal I$ is simply the sum of the lengths of the connected components of $I$, of which there are at most two.
From the above discrepancy estimate we obtain the following  proposition.
\begin{proposition}\label{prp_onedim}
Let $D$ and $T$ be positive integers. Then
\begin{equation}\label{eqn_average_lowest_digits}
\frac 1D
\sum_{D\leq d<2D}
\sup_{\beta\in\mathbb N}
\sup_{I\in\mathcal I}
\left\lvert\frac
1T\#\bigl\{0\leq t<T:\{(td+\beta)\golden\}\in I\bigr\}-\lambda(I)\right\rvert
\ll
\frac{\log^+(DT)}{\sqrt{T}}
\end{equation}
holds with an absolute implied constant.
\end{proposition}
By Proposition~\ref{Lefirstdigits}, this discrepancy estimate contains in particular an estimate for the distribution of the lowest Zeckendorf digits along arithmetic progressions.

\subsection{Applying two-dimensional detection}
We wish to detect digits of $td+\beta$ with indices in $[a,b)$. More precisely, our goal is to find an estimate of the number
\[\#\bigl\{0\leq t<T:\digit_{\ijkl}(td+\beta)=\nu_{\ijkl}\mbox{ for }a\leq \ijkl<b\bigr\}\]
for given digits $\nu_{\ijkl}\in\{0,1\}$ with no adjacent $1$'s.
Again, we will make substantial use of the sum over $d$, and we prove only an averaged estimate of the above expression.

An essential tool will be again the {Erd\H{o}s--Tur\'an--Koksma inequality}.
While this inequality is originally formulated for the usual discrepancy (involving \emph{axis-parallel rectangles}, see Lemma~\ref{lem_ETK}), 
we are interested in \emph{parallelograms} coming from two-dimensional detection.
In Theorem~\ref{thm_ETK_parallelotope} the inequality was adapted to this situation.

Let the setup be as in Corollary~\ref{cor_twodim}.
In order to count the number of $t\in[0,T)$ such that $td+\beta$ has a given digit combination between $a$ and $b$, we detect whether
\begin{equation*}
p(td+\beta,b)=\left(\frac{td+\beta}{\golden^b},\frac{td+\beta}{\golden^{b+1}}\right)
\end{equation*}
lies in a certain parallelogram $A+\mathbb Z^2$.
We apply Theorem~\ref{thm_ETK_parallelotope} in order to estimate the number $G$ of integers $t\in[0,T)$ such that $p(td+\beta,b)\in A+\mathbb Z$,
\[G=\#\bigl\{t\in[0,T):p(td+\beta,b)\in A+\mathbb Z^2\bigr\}.\]
The parallelogram $A$ is spanned by the linearly independent vectors $\widetilde w^{(1)}=(-F_b,F_{b+1})$ and $\widetilde w^{(2)}=(\golden,1)$.
Let $w^{(1)}$ and $w^{(2)}$ be the corresponding unit vectors.
We also have to choose the parameter $H\geq 1$ later on.
This yields
\begin{equation}\label{eqn_G1_ETK}
G=T\lambda(A)
+\LandauO\left(
\frac TH
+\sum_{\substack{-H<h_1,h_2<H\\(h_1,h_2)\neq (0,0)}}
\frac1{r(h_1,h_2)}
\min\left(T,\left \lVert \frac{h_1d}{\golden^b}+\frac{h_2d}{\golden^{b+1}}\right\rVert^{-1}\right)
\right),
\end{equation}
where
\[r(h_1,h_2)=
\max\bigl(1,\bigl\lvert h_1w^{(1)}_1+h_2w^{(1)}_2\bigr\rvert\bigr)
\max\bigl(1,\bigl\lvert h_1w^{(2)}_1+h_2w^{(2)}_2\bigr\rvert\bigr).\]

This estimate is uniform in $\beta$, and the implied constant is absolute.
In order to bound the error term in~\eqref{eqn_G1_ETK},
we first investigate the quantity 
\begin{equation}\label{eqn_def_x}
\theta=\frac{h_1}{\golden^b}+\frac{h_2}{\golden^{b+1}}.
\end{equation}

First we handle the case $\theta=0$.
This is equivalent to $h_1\golden+h_2=0$, which can only happen if $(h_1,h_2)=(0,0)$, and this case is excluded. Therefore
\[\theta\neq 0.\]

We need to ensure that $\theta$ is not too small. For this, we use basic algebraic number theory.
We have $5\equiv 1\bmod 4$, therefore the ring of integers in $\mathbb Q(\sqrt{5})$ is given by $\mathbb Z[\golden]$;
the norm of $k+\ell\golden$ is therefore a nonzero integer for $(k,\ell)\neq (0,0)$.
We know that $\golden^{b+1}\theta$ is a nonzero element of $\mathbb Z[\golden]$, therefore
\begin{align*}
1&\leq\bigl\lvert\mathcal N\left(
h_1\golden+h_2\right)\bigr\rvert=\bigl\lvert h_1\golden+h_2\bigr\rvert
\bigl\lvert h_1\overline{\golden}+h_2\bigr\rvert,
\end{align*}
where $\overline{\golden}=1-\golden$.
The second factor is strictly bounded by $\lvert H(1-\golden)\lvert+H=H\golden$, therefore
\[\lvert \theta \rvert> \frac 1{\golden^{b+2}H}.\]
Also, we have $\lvert \theta \rvert<H\golden^{-b+1}$, summarizing
\begin{equation}\label{eqn_x_bounds}
\frac 1{\golden^{b+2}H}
<
\lvert \theta \rvert
<\frac H{\golden^{b-1}}.
\end{equation}

Using this nesting, we are now going to estimate the sum
\begin{equation}\label{eqn_2dim_abssum}
S=\sum_{D\leq d<2D}\min\left(T,\frac1{\lVert d\theta\rVert}\right),
\end{equation}
where $\theta$ is defined by~\eqref{eqn_def_x}.
We introduce a parameter $P\geq 1$ that we will choose in a moment.
We split the set of $d\in[D,2D)$ into two parts, corresponding to whether $\lVert d\theta\rVert\leq P/T$.
By~\eqref{eqn_x_bounds}, the set
\[J=\bigl\{d\in [D,2D):\lVert d\theta\rVert\leq P/T\bigr\}\]
consists of at most
$DH\golden^{-b+1}+1$ intervals of length bounded by $2\golden^{b+2} PH/T+1$.
The number $\#J$ of exceptional integers $d$ is therefore bounded by
\[
\golden^2\left(D+\frac{\golden^b}H\right)\left(\frac{2PH^2}T+\frac H{\golden^b}\right).
\]
On these exceptional integers, we estimate $\min\bigl(T,\lVert d\theta\rVert^{-1}\bigr)\leq T$,
while on the remaining (at most $D$) integers we have a contribution
$\min\bigl(T,\lVert d\theta\rVert^{-1}\bigr)\leq T/P.$
That is, we have
\begin{equation}\label{eqn_S_twodim_estimate1}
S\leq
\golden^2\left(D+\frac{\golden^b}H\right)\left(2PH^2+\frac {TH}{\golden^b}\right)+\frac{DT}{P}.
\end{equation}
If $T\le H^2$, this estimate
clearly holds: trivially $S\leq DT$, and the summand $\gamma^2 D\cdot 2PH^2$ on the right hand side contributes at least $DT$.
If $T>H^2$, we choose $P=\lfloor\sqrt{T}/H\rfloor\ge1$,
and obtain
\begin{equation}\label{eqn_S_twodim_estimate2}
S\ll 
DT\left(\frac{H}{T^{1/2}}+\frac{H}{\golden^b}+\frac{\golden^b}{DT^{1/2}}+\frac 1D\right).
\end{equation}
The right hand side is bounded below by $DT$ as soon as $T\leq H^2$,
therefore the estimate is trivially satisfied in this case.
Consequently,
equation~\eqref{eqn_S_twodim_estimate2} holds for all integers $D,T,H\geq 1$ and $b\geq 0$.
This estimate does not contain the integers $h_1$ and $h_2$, we may therefore write
\begin{equation}\label{eqn_twodim_free_H}
\begin{aligned}
\hspace{1em}&\hspace{-1em}
\frac 1D
\sum_{D\leq d<2D}
\sup_{\beta\in\mathbb N}
\left\lvert\frac 1T\#\bigl\{0\leq t<T:p(td+\beta,b)\in A+\mathbb Z\bigr\}-\lambda(A)\right\rvert
\\&\ll
\frac 1H+
\left(\frac{H}{T^{1/2}}+\frac{H}{\golden^b}+\frac{\golden^b}{DT^{1/2}}+\frac 1D\right)
\sum_{\substack{-H<h_1,h_2<H\\(h_1,h_2)\neq (0,0)}}
\frac1{r(h_1,h_2)}
.
\end{aligned}
\end{equation}
This estimate is valid uniformly for all parallelograms $A\bmod 1\times1$ spanned by $w^{(1)}$ and $w^{(2)}$ (these vectors occur in our definition of the function $r$, before~\eqref{eqn_G1_ETK}).
The summation only introduces a term of size $(\log^+\!H)^2$;
we sketch the proof of this statement.
Let us decompose $\mathbb R^2$ into lozenges with sides
parallel to $\mathbf w^{(1),\mathsf L}$ and $\mathbf w^{(2),\mathsf L}$,
where the unit vector $\mathbf w^{(j),\mathsf L}$ results from $\mathbf w^{(j)}$ by a rotation by $\pi/2$ to the left.
These lozenges are shifts of the set $F=\{\alpha \mathbf w^{(1),\mathsf L}+\beta \mathbf w^{(2),\mathsf L}:\alpha,\beta\in [0,1]\}$.
In each set $F_{m,n}=F+m\mathbf w^{(1),\mathsf L}+n\mathbf w^{(2),\mathsf L}$, for $m,n\in\mathbb Z$, we can find at most two lattice points from $\mathbb Z^2$ by an elementary argument.
Also, for a lattice point
$\mathbf z=(m+\alpha)\mathbf w^{(1),\mathsf L}+(n+\beta)\mathbf w^{(2),\mathsf L}\in F_{m,n}$ we have
$p_1=(n+\alpha)\bigl(\mathbf w^{(2),\mathsf L}\cdot \mathbf w^{(1)}\bigr)$ and
$p_2=(m+\beta)\bigl(\mathbf w^{(1),\mathsf L}\cdot \mathbf w^{(2)}\bigr)$,
where
$p_j(\mathbf z)=\mathbf z\cdot \mathbf w^{(j)}$, for $j\in\{1,2\}$,
occurs in Theorem~\ref{thm_ETK_parallelotope}.
Since the vectors $\mathbf w^{(1)}$ and $\mathbf w^{(2)}$ (see~\eqref{eqn_corners}) form an angle lying in the interval $[\pi/2-\varepsilon,\pi/2+\varepsilon]$ for some $\varepsilon<\pi/4$, we have $\lvert p_1\rvert\asymp n+\alpha$ and $\lvert p_2\rvert\asymp m+\beta$, where the implied constants are absolute.
If we consider all $(h_1,h_2)\in\{-H+1,\ldots,H-1\}^2$, this set of points is contained in the set of lattice points in a lozenge
\[\bigcup_{-H'<m,n<H'}F_{m,n};\]
by the same argument on the angle between $\mathbf w^{(1)}$ and $\mathbf w^{(2)}$, 
we have $H'\ll H$ with an absolute constant.
Combining these ideas, the $\log^2$-estimate follows, with an absolute implied constant.

We choose
\[H=\min\left(\bigl\lfloor T^{1/4}\bigr\rfloor,
\bigl\lfloor \golden^{b/2}\bigr\rfloor\right)
,\]
and obtain for all integers $D,T\geq 1$ and $b\geq 2$
\begin{equation}
\begin{aligned} 
\hspace{1em}&\hspace{-1em}
\frac 1D
\sum_{D\leq d<2D}
\sup_{\beta\in\mathbb N}
\left\lvert\frac 1T\#\bigl\{0\leq t<T:p(td+\beta,b)\in A+\mathbb Z\bigr\}-\lambda(A)\right\rvert
\\&\ll
\left(\frac1{T^{1/4}}+\frac1{\gamma^{b/2}}+\frac{\golden^b}{DT^{1/2}}+\frac 1D\right)\bigl(\log^+T\bigr)^2,
\end{aligned}
\end{equation}
where the implied constant is absolute.
Note that this estimate is valid for parallelograms $A$ as in Corollary~\ref{cor_twodim}.
By this corollary, and using the observation that we may shift our parallelogram modulo $1\times 1$ without changing the error terms, we immediately obtain the following statement.
\begin{proposition}\label{prp_twodim_estimate}
Let $a,b,T,D$ be positive integers such that $2\leq a<b$.
Assume that $\nu_j\in\{0,1\}$ for $a\leq j<b$ and $\nu_{j+1}=1\Rightarrow\nu_j=0$ for all $j\in\{a,\ldots,b-2\}$.
Let $A$ be the detection parallelogram defined in Corollary~\ref{cor_twodim}.
Then
\begin{equation}\label{eqn_prp_twodim_estimate}
\begin{aligned}
\hspace{4em}&\hspace{-4em}
\frac 1D
\sum_{D\leq d<2D}
\sup_{\beta\in\mathbb N}
\left\lvert\frac 1T\#\bigl\{0\leq t<T:
\digit_j(td+\beta)=\nu_j\mbox{ for }a\leq j<b\bigr\}-\lambda(A)\right\rvert
\\&\ll
  \left(\frac1{T^{1/4}}+\frac1{\golden^{b/2}}+\frac{\golden^b}{DT^{1/2}}+\frac 1D\right)\bigl(\log^+T\bigr)^2
\end{aligned}
\end{equation}
with an absolute implied constant.
\end{proposition}
\begin{remark}
	With some more work, one can replace the upper bound by
	\begin{align*}
		\rb{\frac{1}{T^{1/2}} + \frac{1}{\gamma^{b/2}} + \rb{\frac{\gamma^b}{DT}}^{1/3} + \frac{1}{D}} (\log^+ T)^2.
	\end{align*}
	As this sharper upper bound does only improve some constants, we only present this shorter proof.
\end{remark}

\subsection{Three-dimensional detection}
We are interested in the joint distribution of digits of $td+\beta$ in $[2,L)\cup[a,b)$.
More generally, we study the distribution of $\{(td+\beta)\golden\}$ in wrapped intervals $I\in\mathcal I$, where the digits of $td+\beta$ between $a$ and $b$ are fixed.
Let us define the quantity
\[\mathfrak p(n)=\left(\frac n{\golden^b},\frac{n}{\golden^{b+1}},n\golden\right).
\]
We ask for the number of $t\in[0,T)$ such that
\begin{equation}\label{eqn_3dim_parallelogram}
\mathfrak p(td+\beta)\in B+\mathbb Z^3,
\end{equation}
where $B=A\times I$ (up to coordinate projections),
$A$ is the parallelogram from Corollary~\ref{cor_twodim} corresponding to a given digit combination with indices in $\{a,a+1,\ldots,b-1\}$ and $I$ is any wrapped interval.
That is, we define
\[G=\#\bigl\{t\in [0,T):\mathfrak p(td+\beta)\in B+\mathbb Z^3\bigr\}.\]

This quantity can be estimated by Theorem~\ref{thm_ETK_parallelotope} again.
For this, we define the unit vectors $w^{(j)}$, for $1\leq j\leq 3$, corresponding to the edges of the parallelepiped $B$: we have $w^{(1)}=c_1\bigl(-F_b,F_{b+1},0\bigr)$, $w^{(2)}=c_2\bigl(\golden,1,0\bigr)$, and $w^{(3)}=(0,0,1)$, where the factors $c_1$ and $c_2$ are normalization factors.
Let $p_j$ be the $j$-th projection corresponding to the vectors $w^{(1)}$, $w^{(2)}$, and $w^{(3)}$, and
\[r(h_1,h_2,h_3)=\prod_{j=1}^3 \min\bigl(1,\bigl\lvert p_j(h_1,h_2,h_3)\bigr\rvert\bigr).\]
The quantity $G$ can be estimated by
\begin{multline}\label{eqn_G_threedim}
G=T\lambda(B)
+\LandauO\left(
\frac TH
+\sum_{\substack{\mathbf h\in\mathbb Z^3\setminus\{\mathbf 0\}\\\lVert \mathbf h\rVert<H}}
\frac1{r(\mathbf h)}
\min\left(T,\frac1{\left \lVert \frac{h_1d}{\golden^b}+\frac{h_2d}{\golden^{b+1}}+h_3d\golden\right\rVert}\right)
\right).
\end{multline}

The contribution of the cases where $h_3=0$ can be estimated using~\eqref{eqn_S_twodim_estimate2}, yielding

\begin{equation}\label{eqn_h3_zero_contrib}
\begin{aligned}
\hspace{8em}&\hspace{-8em}
\sum_{D\leq d<2D}
\sum_{\substack{-H<h_1,h_2<H\\(h_1,h_2)\neq (0,0)}}
\frac1{r(h_1,h_2,0)}
\min\left(T,\left \lVert \frac{h_1d}{\golden^b}+\frac{h_2d}{\golden^{b+1}}\right\rVert^{-1}\right)
\\&\ll
DT\left(\frac H{T^{1/2}}+\frac H{\golden^b}+\frac{\golden^b}{DT^{1/2}}+\frac 1D\right)\bigl(\log^+H\bigr)^2
\end{aligned}
\end{equation}

In order to handle the $\lVert\cdot\rVert$-part for $h_3\neq 0$
we prohibit certain $(h_1,h_2,h_3)$ with the property that
$h_1\golden^{-b}+h_2\golden^{-b-1}+h_3\golden$ lies close to a rational number with denominator bounded by a parameter $Q$ to be chosen later.
Let
\[\theta=\frac{h_1}{\golden^b}+\frac{h_2}{\golden^{b+1}}+h_3\golden.\]
The idea is the following: the term $h_3\golden$ avoids rational numbers with denominators $\leq Q$ due to the bad approximability of $\golden$, and the other summands are small since $\golden^b$ will be much larger than $H$.
It follows that $\theta$ still avoids rational numbers with denominators $\leq Q$.
We work out the details.

We consider the norm in $\mathbb Q(\sqrt{5})$ again.
Assume that $1\leq \lvert h_3\rvert<H$ and $1\leq q\leq Q$,
and choose the integer $p$ such that $\lvert h_3\golden-p/q\rvert$ is minimal,
that is,
\[p=\left\lfloor h_3q\golden+\frac 12\right\rfloor.\]
Clearly, $\lvert p-h_3q\golden\rvert<1/2$.
Then
\[1\leq\bigl\lvert\mathcal N\left(
qh_3\golden-p\right)\bigr\rvert=
\bigl\lvert h_3q\golden-p\bigr\rvert
\bigl\lvert h_3q\overline{\golden}-p\bigr\rvert.
\]
Since
$\bigl\lvert h_3q\overline{\golden}-p\bigr\rvert
\leq 
q\lvert h_3\rvert \bigl\lvert \overline{\golden}-\golden\bigr\rvert+1/2
<\bigl(2\golden-1\bigr)HQ$, we obtain
\[\left\lVert h_3\golden-\frac pq\right\rVert>\frac 1{(2\golden-1)HQ^2}\]
for $1\leq \lvert h_3\rvert<H$.
The summands $h_1\golden^{-b}$ and $h_2\golden^{-b-1}$ introduce a small perturbation $<H\golden^{-b+1}$,
and we get
\begin{equation}\label{eqn_rat_lower_bound}
\left\lVert \theta-\frac pq\right\rVert>\frac 1{(2\golden-1)HQ^2}-\frac{H}{\golden^{b-1}}
\end{equation}
for all integers $h_1,h_2,h_3$ such that $0\leq \lvert h_1\rvert,\lvert h_2\rvert,\lvert h_3\rvert<H$, and for all integers $p$ and $q$ such that $1\leq q\leq Q$.
At the end of the proof, we will choose the parameters in such a way that
\[\golden^{b-1}\geq \golden^2H^2Q^2,\]
so that the right hand side will be bounded below by $1/(16HQ^2)$.
For now, we have to keep this requirement in mind.

Now $d$ runs. Again, we are interested in the sum
\begin{equation}\label{eqn_average_Norm}
\sum_{D\leq d<2D}\min\left(T,\lVert d\theta\rVert^{-1}\right),
\end{equation}
where
\[\theta=\frac{h_1}{\golden^b}+\frac{h_2}{\golden^{b+1}}+h_3\golden.\]
As in the two-dimensional case, we assume that $P\geq 1$ is an integer that we chose later; we require
\begin{equation}\label{eqn_T_assumption}
T\geq 32PHQ^2.
\end{equation}
We estimate the number of elements of
\[J=\bigl\{d\in [D,2D):\lVert d\theta\rVert\leq P/T\bigr\}.\]
We claim that
\[\#J<D/Q+1.\]
In order to obtain a contradiction, let us assume that $\#J\geq D/Q+1$.
If all distances between consecutive elements of $J$ were bounded below by $Q+1$, 
the cardinality of $J$ would be bounded above by $D/(Q+1)+1$, which is not possible by our assumption; therefore there exist integers $t\in[1,Q]$ and $d\in[D,2D-t)$ such that
\[\bigl \lVert d\theta\bigr \rVert\leq\frac PT\quad\mbox{and}\quad\bigl \lVert(d+t)\theta\bigr \rVert\leq \frac PT.\]
By the triangle inequality, we get $\lVert t\theta\rVert\leq 2P/T$, which implies the existence of integers $q\in[1,Q]$ (namely $q=t$, for example) and $p$ such that
\[\left\lVert \theta-\frac pq\right\rVert\leq \frac{2P}{qT}\leq 2\frac PT.\]
By the assumption~\eqref{eqn_T_assumption},
we get a contradiction to~\eqref{eqn_rat_lower_bound}.

Estimating the summands $\min(T,\lVert d\theta\rVert^{-1})$ by $T$ if $d\in J$, and by $T/P$ otherwise, it follows that the sum~\eqref{eqn_average_Norm}
is bounded by
\[DT\left(\frac 1D+\frac 1Q+ \frac 1P\right).\]

We collect the contributions and the corresponding requirements on the variables --- note that we need the estimate~\eqref{eqn_S_twodim_estimate2} from the two-dimensional part for the case $h_3=0$ ---
and obtain the following statement: 
if $32PHQ^2\leq T$
and $H^2Q^2\leq \golden^{b-3}$, then we have
\begin{equation}\label{eqn_threedim_free_HPQ}
\begin{aligned}
\hspace{5em}&\hspace{-5em}
\frac 1D\sum_{D\leq d<2D}
\sup_{\beta\in\mathbb N} 
\sup_{I\in \mathcal I}
\left\lvert\frac 1T\#\bigl\{t\in[0,T):\mathfrak p(td+\beta)\in A\times I\bigr\}-\lambda(A)\lambda(I)\right\rvert
\\&\ll
\frac 1H+
\left(
\frac 1D+\frac 1P+\frac 1Q\right)\bigl(\log^+H\bigr)^3
\\&+\left(\frac H{T^{1/2}}+\frac H{\golden^b}+\frac{\golden^b}{DT^{1/2}}\right)\bigl(\log^+H\bigr)^2
\end{aligned}
\end{equation}
with some absolute implied constant.
Note that the triple sum over $(h_1,h_2,h_3)$ causes the factor $(\log^+H)^3$ in the same way that a double sum generated the factor $(\log^+H)^2$ in the two-dimensional detection case.

Later we will face the problem that $\golden^b$ will not necessarily be larger than $T$; the tempting choice $H=P=Q\asymp T^{1/4}$ is therefore, due to the requirement $H^2Q^2\leq\golden^{b-3}$, too restrictive.
As a remedy, we simply put
\begin{equation}\label{eqn_HPQ_choice}
H=P=Q=\left\lfloor\min\left(\frac T{32},\golden^{b-3}\right)^{1/4}\right\rfloor,
\end{equation}
and obtain the following statement.
\begin{proposition}\label{prp_threedim_estimate}
Assume that $a,b,T,D$ are positive integers such that $2\leq a<b$.
Assume that $\nu_j\in\{0,1\}$ for $a\leq j<b$ and $\nu_{j+1}=1\Rightarrow\nu_j=0$ for $a\leq j<b-1$.
Let $A$ be the detection parallelogram defined in Corollary~\ref{cor_twodim}. Then
\begin{equation}\label{eqn_prp_threedim_estimate}
\begin{aligned}
\hspace{4em}&\hspace{-4em}
\frac 1D
\sum_{D\leq d<2D}
\sup_{\beta\in\mathbb N}\,
\sup_{I\in\mathcal I}\,
\biggl\lvert\frac 1T\#\bigl\{0\leq t<T:
\digit_j(td+\beta)=\nu_j\mbox{ for }a\leq j<b
\bigr.\biggr.\\&\biggl.\bigl.
\qquad\qquad\mbox{ and }(td+\beta)\golden\in I+\mathbb Z\bigr\}-\lambda(A)\lambda(I)\biggr\rvert
\\&\ll
\left(\frac 1D+\frac 1{T^{1/4}}+\frac 1{\golden^{b/4}}\right)\bigl(\log^+T\bigr)^3
+\frac {\golden^b}{DT^{1/2}}\bigl(\log^+T\bigr)^2
\end{aligned}
\end{equation}
with an absolute implied constant.
\end{proposition}
Note that the cases $T<32$ have to be treated separately in order to obtain this statement, since we need $H,P,Q\geq 1$ in the argument above; but this amounts only to (possibly) increasing the implied constant in the proposition. 

With Propositions~\ref{prp_onedim},~\ref{prp_twodim_estimate}, and~\ref{prp_threedim_estimate} behind us, we may now attack Theorem~\ref{thm_lod}.
\section{Proofs}\label{sec_proofs}
\subsection{Lemmas}
As in the papers~\cite{MR2010,MR2009} by Mauduit and Rivat, essential ingredients in our proof are \emph{van der Corput's inequality} and a \emph{carry propagation lemma}.
In this proof, we will also meet a \emph{Gowers uniformity norm}, which appears quite spontaneously by iterated application of van der Corput's inequality --- this link has been exploited in the recent paper~\cite{S2020} 
by the third author.
We will apply the inequality of van der Corput in order to cut off digits from above, which we will make precise in a moment.
\begin{lemma}\label{lem_vdc}
Let $I$ be a finite interval in $\mathbb Z$ containing $N$ integers and
let $z_n$ be a complex number for $n\in I$.
For all integers $R\geq 1$ we have
\begin{equation}\label{eqn_vdc}
\begin{aligned}
  \Biggl \lvert\sum_{n\in I}z_n\Biggr \rvert^2
&\leq
  \frac{N+R-1}R
  \sum_{0\leq\lvert r\rvert<R}\left(1-\frac {\lvert r\rvert}R\right)
  \sum_{\substack{n\in I\\n+r\in I}}z_{n+r}\overline{z_n}\\
  &=  \frac{N+R-1}R
  \sum_{0\leq\lvert r\rvert<R}\left(1-\frac {\lvert r\rvert}R\right)
  \sum_{\substack{n\in I\\n+r\in I}}\Delta(z; r)(n).
\end{aligned}
\end{equation}
\end{lemma}
We will need a carry propagation lemma for the Zeckendorf sum-of-digits function. The following result follows from~\cite[Lemma~2.6]{S2018}. 
\begin{lemma}\label{lem_z_carry_lemma}
Let $\lambda\geq 2$ and $N,r\geq 0$ be integers.
Then
\begin{equation}\label{eqn_carry}
\bigl \lvert
\bigl\{
0\leq n<N:\sz(n+r)-\sz(n)\neq \sz_\lambda(n+r)-\sz_\lambda(n)
\bigr\}
\bigr \rvert\leq
N\frac r{F_{\lambda-1}}
.
\end{equation}

\end{lemma}
Note that in that paper we work with the Ostrowski expansion of an integer;
in order to obtain Lemma~\ref{lem_z_carry_lemma} from~\cite[Lemma~2.6]{S2018}, we have to take care of shifts by one.
With the notations $q_i$ and $\psi_\lambda$ from~\cite{S2018}, we have $q_i=F_{i+1}$ for $i\geq 0$, and $\psi_\lambda(n)=\sz_{\lambda+1}(n)$ for $\lambda\geq 1$.
The following standard result (see, for example,~\cite[Lemma~5.2.3]{Huxley1996},~\cite[Lemme~2]{MR2010}) allows us to extend the range of a summation in exchange for a controllable factor.
\begin{lemma}\label{lem_vinogradov}
Let $x\leq y\leq z$ be real numbers and $a_n\in\mathbb C$ for $x\leq n<z$.
Then
\begin{equation*}
    \left\lvert\sum_{x\leq n<y}a_n\right\rvert
  \leq
    \int_0^1
      \min\left\{y-x+1,\left\lVert\xi\right\rVert^{-1}\right\}
      \left\lvert\sum_{x\leq n<z}a_n\e(n\xi)\right\rvert
    \,\mathrm d \xi.
\end{equation*}
\end{lemma}
\begin{proof}
We use the Kronecker Delta $\delta_{i,j}$.
Since
$\int_0^1\e(k\xi)\,\mathrm d \xi=\delta_{k,0}$ for $k\in\mathbb Z$
we have
\begin{equation*}
    \sum_{x\leq n<y}a_n
  =
    \sum_{x\leq n<z}a_n\sum_{x\leq m<y}\delta_{n-m,0}
  =
    \int_0^1
    \sum_{x\leq m<y}\e(-m\xi)
    \sum_{x\leq n<z}a_n\e(n\xi)
    \,\mathrm d \xi,
\end{equation*}
from which the statement follows.
\end{proof}

\subsection{Method of proof}
Let us give a short overview of the method of proof of Theorem~\ref{thm_lod}.
It will become clear in a moment that it is sufficient to find an upper bound for certain exponential sums of the form
\begin{equation}\label{eqn_sum_form}
\sum_{0\leq n<N}\e\bigl(\vartheta\smallspace\sz(nd+a)\bigr)\e(n\xi).
\end{equation}
We will distinguish between two different situations, the cases ``\emph{long arithmetic progressions}'' and ``\emph{short arithmetic progressions}'' respectively.
In the first case, the common difference $d$ will be small compared to the length of summation $N$ (say, $d\leq N^\varepsilon$ for some $\varepsilon>0$).
In this situation, we strongly build on the paper~\cite{S2018}, dealing with a certain pseudorandom property of Ostrowski sum-of-digits functions, of which $\sz$ is one example.
This pseudorandom property is related to \emph{correlations} --- such as $\e(\vartheta\smallspace\sz(n+t)-\vartheta\smallspace\sz(n))$. We will see that these correlations appear when the Cauchy--Schwarz inequality is applied to the sum~\eqref{eqn_sum_form}.

The second case can be regarded as the centerpiece of our method, and the method relies heavily on the ideas introduced in~\cite{S2020}.
We are now in the situation of short arithmetic progressions, where $N$ may be very small compared to $d$, that is, an arbitrarily small power of $d$.
Note, however, that the full statement on the level of distribution is not needed for our results on prime numbers --- any level strictly above $2/3$ suffices.
Meanwhile, we have no doubt that Theorem~\ref{thm_lod} in its generality (the level of distribution equals $1$) is of strong independent interest.

The problem that arises in this situation may be described informally as follows.
We consider the Zeckendorf expansion along the (short) arithmetic progression $n\mapsto nd+a$.
The Zeckendorf expansions of $nd+a$ and $(n+1)d+a$ usually differ at the lowest $\asymp\log d$ positions.
Applying a ``carry lemma'' such as introduced in the work of Mauduit and Rivat~\cite{MR2010,MR2009} we may discard the digits above $A=C\log d$, for some $C>0$.
Since the sum over $n$ is short, the remaining Zeckendorf digits $\digit_{A-1}(nd+a),\digit_{A-2}(nd+a),\ldots,\digit_2(nd+a)$ can only attain few of the admissible tuples $(\omega_{A-1},\omega_{A-2},\ldots,\omega_2)$, as $n$ runs.
We were not able to describe the structure of the set of appearing tuples as $n$ runs through a short interval $[0,N)$, even less to prove the theorem in this manner.

In order to overcome this difficulty, the proof will proceed by ``cutting off'' intervals of digits of length $\ll \log N$ repeatedly.
This uses a variant of van der Corput's inequality (Proposition~\ref{prp_vdC_generalized}), iteratively:
each application of this proposition enables us to discard one interval of length $\ll\log N$.
To this end, we use \emph{detection of Zeckendorf digits} as introduced in Chapter~\ref{chap:detection}.
In the process, \emph{higher order correlations} are introduced, which, quite inevitably, lead us to \emph{Gowers norms}.

Now let us begin the proof of Theorem~\ref{thm_lod}.
\subsection{Proof of Theorem~\ref{thm_lod}}
For real numbers $D,N\geq 1$ and $\xi$ we define
\begin{equation}\label{eqn_S0_def}
S_0=S_0(D,N,\vartheta,\xi)=
\frac 1D\sum_{D\leq d<2D}
\sup_{a\in\mathbb N}
\left\lvert
\frac 1N\sum_{0\leq n<N}
\e\bigl(\vartheta\smallspace\sz(nd+a)\bigr)\e(n\xi)\right\rvert.
\end{equation}

By Lemma~\ref{lem_vinogradov} we extend the summations over $n$ occurring in~\eqref{eqn_lod}, so that they are of equal length $N\asymp x/D$ for $d\in[D,2D)$.
This introduces a factor $\log x$.
Dyadic decomposition of $[1,D]$ explains another factor $\log x$.
We therefore see (and it will be made more precise later) that it is sufficient to prove the following statement.
\begin{center}
\begin{minipage}{0.8\textwidth}
For each real number $\rho>0$, there exist constants $c>0$ and $C$ such that for all positive integers $N$ and $D$ satisfying $1\leq D\leq N^\rho$, and all reals $\vartheta$ and $\xi$,
\end{minipage}
\end{center}
\begin{equation}\label{eqn_sufficient}
S_0\bigl(D,N,\vartheta,\xi\bigr)\leq CN^{-c\lVert\vartheta\rVert^2}\bigl(\log^+\!N)^{3/4}.
\end{equation}
The major part of the proof will therefore be concerned with an estimate for $S_0$.

By Cauchy--Schwarz and Lemma~\ref{lem_vdc} we obtain for all positive integers $R$
\begin{multline*}
\bigl\lvert S_0(D,N,\vartheta,\xi)\bigr\rvert^2
\leq
\frac{N+R-1}{RDN}
\sum_{D\leq d<2D}
\sup_{a\in\mathbb N}
\sum_{0\leq \lvert r\rvert<R}
\biggl(1-\frac{\lvert r\rvert}{R}\biggr)
\e\bigl(r\xi\bigr)
\\
\times
\frac 1N
\sum_{\substack{0\leq n<N\\0\leq n+r<N}}
\e\bigl(\vartheta\smallspace\sz\bigl((n+r)d+a\bigr)
-\vartheta\smallspace\sz(nd+a)
\bigr).
\end{multline*}

We apply the carry propagation lemma (Lemma~\ref{lem_z_carry_lemma}), thereby introducing the positive parameter $\lambda$.
Treating the summand $r=0$ separately,
omitting the condition $0\leq n+r<N$,
and considering $r$ and $-r$ simultaneously,
we obtain
\begin{align}\label{eqn_S0_squared}
\left\lvert S_0(D,N,\vartheta,\xi)\right\rvert^2
\ll E_0
+\frac 1{RD}
\sum_{1\leq r<R}
\sum_{D\leq d<2D}
\sup_{a\in\mathbb N}\,
\lvert S_1 \rvert,
\end{align}
where
\begin{equation*}
S_1=\frac 1N\sum_{0\leq n<N}
\e\bigl(\vartheta\smallspace\sz_\lambda(nd+a+rd)-\vartheta\smallspace\sz_\lambda(nd+a)\bigr)
\end{equation*}
and
\[ E_0=\frac 1{R}+\frac{RD}{F_\lambda}+\frac RN.  \]

Equation~\eqref{eqn_S0_squared} is the point of departure for two cases concerning \emph{small} $D$ and \emph{large} $D$ respectively.
These cases will be treated quite differently.
The case of small $D$ can be handled using the \emph{pseudorandomness} of the Zeckendorf sum-of-digits function (see~\cite{S2018}).
The harder case concerning large $D$ consists in the reduction of the problem to a \emph{Gowers norm} related to the function $\sz$.

\bigskip
\noindent
\subsubsection{Long arithmetic progressions.}
First we treat the case $1\leq D\leq N^{1/3}$.
(The numerical value $1/3$ has no significance as any value in $(0,1/2)$ will do, but it is more convenient to fix a value.)
In this case, the summation over $n$ will clearly be of length at least $D^3$ --- the arithmetic progression $(nd+a)_{0\leq n<N}$ is \emph{long}.
On average, taken over $d\in[D,2D)$, the sequence $(nd+a)\golden$ will have small discrepancy $\bmod\, 1$; this is an application of Proposition~\ref{prp_onedim}.
Using the discrepancy estimate from this proposition, we are now going to transform the sum over $n$, introducing the function $g_\lambda$ defined in Section~\ref{sec252}.
Let us define the $1$-periodic function
\[
G_d(x)=\e\bigl(\vartheta g_\lambda(x+rd\golden)-\vartheta g_\lambda(x)\bigr).
\]
This function is piecewise constant, featuring $\leq 2F_\lambda$ wrapped intervals $I$ on which it is constant (the factor $2$ coming from the fact that $G_d$ is a product of two piecewise constant functions on $F_\lambda$ wrapped intervals).

Let $\mathcal J_d$ be the decomposition into these wrapped intervals.
For $I\in\mathcal J_d$, let $A_{d,I}$ be the value that $G_d$ takes on the interval $I$.
For most $d$, the sequence $\{(nd+a)\golden\}$ distributes nicely to these intervals, and we can get rid of the arithmetic progression.
This lends itself to an application of the inequality in Theorem~\ref{th_koksma} due to Koksma.
Using also Proposition~\ref{prp_onedim}, we obtain
\begin{align}\label{eqn_S1_KH_inequality}
\frac 1{D}
\sum_{D\leq d<2D}\sup_{a\in\mathbb N}
\,\bigl \lvert
S_1\bigr\rvert
&=
\frac 1D
\sum_{D\leq d<2D}
\left\lvert
\sum_{I\in \mathcal J_d}
A_{d,I}\lambda(I)
\right\rvert
+\LandauO\left(E^{(1)}_1\right),
\end{align}
where
\[E^{(1)}_1=\frac{F_\lambda\log^+(DN)}{\sqrt{N}},\]
since the total variation of $G_d$ is bounded by $2F_\lambda$.
Using the low discrepancy of $\{n\golden\}$, we transform this back to a sum over $n$, using Theorem~\ref{th_koksma} again, and obtain
\begin{align}\label{eqn_S1_KH_inequality_2}
\frac 1{D}
\sum_{D\leq d<2D}
\left\lvert
\sum_{I\in \mathcal J_d}
A_{d,I}\lambda(I)
\right\rvert
&=
\frac 1{D}
\sum_{D\leq d<2D}
\bigl\lvert \omega_{rd}(N)\bigr\rvert+\LandauO\left(E^{(2)}_1\right),
\end{align}
where
\[
\omega_t(N)=
\omega_t(\vartheta,N)=
\frac 1N
\sum_{0\leq n<N}
\e\bigl(\vartheta\smallspace\sz_\lambda(n+t)-\vartheta\smallspace\sz_\lambda(n)\bigr)
\]
and
\[E^{(2)}_1=\frac{F_\lambda\log^+\!N}N.\]

From~\eqref{eqn_S0_squared},~\eqref{eqn_S1_KH_inequality} and~\eqref{eqn_S1_KH_inequality_2} we obtain
\begin{equation}\label{eqn_S0_S2}
\begin{aligned}
\bigl\lvert
S_0(D,N,\vartheta,\xi)
\bigr\rvert^2
&\ll
S_2(R,D,N,\vartheta)
+E_0+E^{(1)}_1+E^{(2)}_1,
\end{aligned}
\end{equation}
where
\begin{equation*}
\begin{aligned}
S_2(R,D,N,\vartheta)
&=
\frac 1{RD}\sum_{D\leq d<2D}\sum_{1\leq r<R}\bigl\lvert \omega_{rd}(\vartheta,N)\bigr\rvert
\\&=
\frac 1{RD}\sum_{D\leq d<2D}\sum_{1\leq r<R}\left\lvert
\frac 1N
\sum_{0\leq n<N}
\e\bigl(\vartheta\smallspace\sz_\lambda(n+rd)-\vartheta\smallspace\sz_\lambda(n)\bigr)
\right\rvert.
\end{aligned}
\end{equation*}

We will now derive a nontrivial estimate for $S_0(D,N,\vartheta,\xi)$ for small $D$, followed by a shorter sketch of a proof of the same estimate, using Gowers norms. The first proof is longer, but it examines the situation from a different viewpoint.
The reader might find the elementary arguments contained therein helpful.
In particular, the auxiliary result given in Proposition~\ref{prp_zeckendorf_fourier} could also be proven with the help of Corollary~\ref{cor_gowers_estimate}.

The correlation $\omega_t(\vartheta,N)$ can be estimated using the method from~\cite{S2018}. 
For convenience, we reproduce the essential parts from that paper needed to prove such an estimate.
Note that we need to take care of a shift of indices by $1$ caused by the slightly differing definitions of the Ostrowski and the Zeckendorf numeration systems.
\begin{lemma}\label{lem_S2017_lem1}
Assume that $\lambda\geq 3$.
Let $(w^{(\lambda)}_j)_{j\geq 0}$ be the increasing enumeration of the nonnegative integers $n$ such that $\digit_2(n)=\cdots=\digit_{\lambda-1}(n)=0$.
The intervals $\bigl[w^{(\lambda)}_j,w^{(\lambda)}_{j+1}\bigr)$ constitute a partition of the set $\mathbb N$ into intervals of the two possible lengths $F_\lambda$ and $F_{\lambda-1}$.
We have $w^{(\lambda)}_{j+1}-w^{(\lambda)}_j=F_{\lambda-1}$ if and only if $\digit_2(j)=1$.
\end{lemma}
For example, the sequence of integers whose Zeckendorf expansions end with $\tO\tO$ starts with
$0,3,5,8,$ $11,13,16,18,21,\ldots$, having gaps
$3,2,3,3,2,3,2,3,\ldots$.
This is just the sequence $\digit_2$ with renamed values.

Note that the last condition in the Lemma originally reads $\digit_\lambda\bigl(w^{(\lambda)}_j\bigr)=1$.
However, in our special case all partial quotients are equal to $1$, which induces a certain shift-invariance of our numeration system.
The integers having zeros below $\lambda$ in the Zeckendorf expansion are therefore indexed by a \emph{generalized Beatty sequence}: treating the trivial case $\lambda=2$ separately, we obtain
\[
w^{(\lambda)}_j=F_\lambda j-F_{\lambda-2}\bigl\lfloor j(\golden-1)\bigr\rfloor
\]
for $j\geq 0$ and $\lambda\geq 2$
(compare the comment after~\eqref{eqn_vnL_def}).

We will also use Fourier coefficients related to the Zeckendorf numeration system: set
\begin{equation}\label{eqn_Z_fourier}
G_\lambda(h)=G_\lambda(\vartheta,h)=\frac 1{F_\lambda}\sum_{0\leq u<F_\lambda}\e\left(\vartheta\smallspace\sz(u)-huF_\lambda^{-1}\right).
\end{equation}

The following lemma is
a slight extension of~\cite[Lemma~2.7]{S2018}.
\begin{lemma}\label{lem_Z_fourier_correlation}
Let $\lambda\ge2$ and $t\ge0$ be integers.
If  $i$ is such that $w^{(\lambda)}_{i+1}-w^{(\lambda)}_i=F_\lambda$, we have
\begin{equation}\label{eqn_Z_fourier_correlation_1}
\sum_{h=0}^{F_\lambda-1}\bigl \lvert G_\lambda(\vartheta,h)\bigr \rvert^2\e\bigl(htF_\lambda^{-1}\bigr)
=\frac1{F_\lambda}\sum_{v=w^{(\lambda)}_i}^{w^{(\lambda)}_{i+1}-1}\e\bigl(\vartheta\smallspace\sz_\lambda(v+t)-\vartheta\smallspace\sz_\lambda(v)\bigr)+\LandauO\left(\frac t{F_\lambda}\right)
\end{equation}
for all $i\geq 0$, with an absolute implied constant.
If  $w^{(\lambda)}_{i+1}-w^{(\lambda)}_i=F_{\lambda-1}$, we have
\begin{equation*}
\sum_{h=0}^{F_{\lambda-1}-1}\bigl \lvert G_{\lambda-1}(\vartheta,h)\bigr \rvert^2\e\bigl(htF_{\lambda-1}^{-1}\bigr)
=\frac1{F_{\lambda-1}}\sum_{v=w^{(\lambda)}_i}^{w^{(\lambda)}_{i+1}-1}\e\bigl(\vartheta\smallspace\sz_\lambda(v+t)-\vartheta\smallspace\sz_\lambda(v)\bigr)+\LandauO\left(\frac t{F_{\lambda-1}}\right).
\end{equation*}
\end{lemma}
The first part follows directly from~\cite[Lemma~2.7]{S2018}.
The second part follows from the first, treating the integers $v$ such that $v+t\geq w^{(\lambda)}_{i+1}$ separately and replacing $\sz_\lambda$ by $\sz_{\lambda-1}$.
The second part of this lemma is used in~\cite{S2018}, without writing it down explicitly; we added it here for clarity of exposition.

We set $T=2RD$ and write $\tau$ for the divisor function, which counts the number of positive divisors of a positive integer $n$.
Wilson~\cite{Wilson1923} proved, using Perron's formula, that
\[
\frac 1T\sum_{t\leq T}\tau(t)^2\sim Tp(\log T)+\LandauO\bigl(T^{1/2+\varepsilon}\bigr)
\]
for some cubic polynomial $p$ with leading coefficient $1/\pi^2$, thus verifying and strenghening an earlier claim by Ramanujan~\cite{Ramanujan1916}.
In particular,
\begin{equation}\label{eqn_tau_squared}
\frac 1T\sum_{t\leq T}\tau(t)^2\sim \frac 1{\pi^2}\bigl(\log T\bigr)^3.
\end{equation}

The estimate~\eqref{eqn_S0_S2} implies
\begin{equation}
\begin{aligned}
\left\lvert S_0(D,N,\vartheta,\xi)\right\rvert^4
\ll
\left\lvert
\frac 1{RD}\sum_{1\leq r<R}
\sum_{D\leq d<2D}
\bigl \lvert\omega_{rd}(\vartheta,N)\bigr \rvert
\right\rvert^2
+E_0+E^{(1)}_1+E^{(2)}_1.
\end{aligned}
\end{equation}
Note that we do not need to square the error term nor consider the mixed terms; this is the case since the left hand side as well as $S_2(R,D,N,\vartheta)$ are bounded by $1$, and all constituents are nonnegative numbers.
We will use similar considerations again a couple of times.

By Cauchy--Schwarz,
\begin{equation}\begin{aligned}
\hspace{2em}&\hspace{-2em}
\left\lvert
\frac 1{RD}\sum_{1\leq r<R}
\sum_{D\leq d<2D}
\bigl \lvert\omega_{rd}(\vartheta,N)\bigr \rvert
\right\rvert^2
\leq
\left\lvert
\frac 2T\sum_{0\leq t<T}
\tau(t)\cdot
\bigl \lvert \omega_t(\vartheta,N)\bigr \rvert
\right\rvert^2
\\&\ll
\bigl(\log^+\!RD\bigr)^3
\frac 1T\sum_{0\leq t<T}
\bigl \lvert \omega_t(\vartheta,N)\bigr \rvert^2
\leq
\bigl(\log^+\!RD\bigr)^3
\frac 1T\sum_{0\leq t<T}
\bigl \lvert \omega_t(\vartheta,N)\bigr \rvert,
\end{aligned}\end{equation}
which implies
\begin{equation}\label{eqn_S0_eight_power}
\bigl \lvert S_0(D,N,\vartheta,\xi)\bigr \rvert^8
\ll
\bigl(\log^+\!RD\bigr)^6
\left(\frac 1T\sum_{0\leq t<T}
\bigl \lvert \omega_t(\vartheta,N)\bigr \rvert\right)^2
+E_0+E^{(1)}_1+E^{(2)}_1.
\end{equation}

By applying Lemma~\ref{lem_Z_fourier_correlation} we decompose the summation over $N$ into pieces.
Set $k=\max\bigl\{j:w^{(\lambda)}_j\leq N\bigr\}$. Assume that $a$ is the number of indices $0\leq j<k$ such that $w^{(\lambda)}_{j+1}-w^{(\lambda)}_j=F_\lambda$ and 
$b$ is the number of indices $0\leq j<k$ such that $w^{(\lambda)}_{j+1}-w^{(\lambda)}_j=F_{\lambda-1}$.

Choose the complex numbers $\varepsilon_t$ and $\varepsilon'_t$ such that $\lvert \varepsilon_t\rvert=1$ and $\lvert \varepsilon'_t\rvert=1$, and in such a way that
\[\varepsilon_t\sum_{0\leq h<F_\lambda}
\bigl \lvert G_\lambda(h)\bigr \rvert^2
\e\left(htF_\lambda^{-1}\right)
\quad\mbox{and}\quad
\varepsilon'_t\sum_{0\leq h<F_{\lambda-1}}
\bigl \lvert G_{\lambda-1}(h)\bigr \rvert^2
\e\left(htF_{\lambda-1}^{-1}\right)
\]
are nonnegative real numbers.
Then
\begin{align*}
\hspace{4em}&\hspace{-4em}
\frac 1T\sum_{0\leq t<T}
\left\lvert
\frac 1{w_k}
\sum_{0\leq n<w_k}
\e\bigl(
\vartheta\smallspace\sz_\lambda(n+t)-\vartheta\smallspace\sz_\lambda(n))
\bigr)
\right\rvert
\\&=
\frac 1T\Biggl \lvert \frac{aF_\lambda}{w_k}
\sum_{0\leq t<T}\varepsilon_t\sum_{0\leq h<F_\lambda}
\bigl \lvert G_\lambda(h)\bigr \rvert^2
\e\left(\frac{ht}{F_\lambda}\right)
\\&+
\frac{bF_{\lambda-1}}{w_k}
\sum_{0\leq t<T}\varepsilon'_t\sum_{0\leq h<F_{\lambda-1}}
\bigl \lvert G_{\lambda-1}(h)\bigr \rvert^2
\e\left(\frac{ht}{F_{\lambda-1}}\right)
\Biggr \rvert
+\LandauO\left(\frac{aT}{w_k}+\frac{bT}{w_k}\right)
\\&\leq
\frac 1T\Biggl \lvert
\sum_{0\leq h<F_\lambda}
\bigl \lvert G_\lambda(h)\bigr \rvert^2
\sum_{0\leq t<T}\varepsilon_t
\e\left(\frac{ht}{F_\lambda}\right)
\\&+
\sum_{0\leq h<F_{\lambda-1}}
\bigl \lvert G_{\lambda-1}(h)\bigr \rvert^2
\sum_{0\leq t<T}\varepsilon'_t
\e\left(\frac{ht}{F_{\lambda-1}}\right)
\Biggr \rvert
+\LandauO\left(\frac{T}{F_\lambda}\right).
\end{align*}
Assuming that $T\leq F_\lambda$, Cauchy--Schwarz implies
\begin{align*}
\hspace{2em}&\hspace{-2em}
\frac1{T^2}
\left\lvert\sum_{0\leq h<F_\lambda}
\bigr \rvert G_\lambda(h)\bigr \rvert^2
\sum_{0\leq t<T}\varepsilon_t
\e\left(\frac{ht}{F_\lambda}\right)
\right \rvert^2
\\&\leq
\frac 1{T^2}
\sum_{0\leq h<F_\lambda}
\bigl \lvert G_\lambda(h) \bigr \rvert^4
\sum_{0\leq h<F_\lambda}
\left \lvert \sum_{0\leq t<T}
\varepsilon_t\e\left(\frac{ht}{F_\lambda}\right)
\right \rvert^2
\\&=
\frac 1{T^2}\sum_{0\leq h<F_\lambda}
\bigl \lvert G_\lambda(h)\bigr \rvert^4
\sum_{0\leq h<F_\lambda}
\sum_{0\leq t_1,t_2<T}
\varepsilon_{t_1}\overline{\varepsilon_{t_2}}
\e\left(h\frac{t_1-t_2}{F_\lambda}\right)
\\&=\frac{F_\lambda}{T^2}
\sum_{0\leq h<F_\lambda}
\bigl \lvert G_\lambda(h)\bigr \rvert^4\sum_{0\leq t_1,t_2<T}
\varepsilon_{t_1}\overline{\varepsilon_{t_2}}\delta_{t_1,t_2}
=
\frac{F_\lambda}{T}
\sum_{0\leq h<F_\lambda}
\bigl \lvert G_\lambda(h)\bigr \rvert^4,
\end{align*}
and analogously,
\begin{equation}
\begin{aligned}
\hspace{2em}&\hspace{-2em}
\frac1{T^2}
\left\lvert\sum_{0\leq h<F_{\lambda-1}}
\bigr \rvert G_{\lambda-1}(h)\bigr \rvert^2
\sum_{0\leq t<T}\varepsilon'_t
\e\left(\frac{hr}{F_{\lambda-1}}\right)
\right \rvert^2
&\leq
\frac{F_{\lambda-1}}{T}
\sum_{0\leq h<F_{\lambda-1}}
\bigl \lvert G_{\lambda-1}(h)\bigr \rvert^4,
\end{aligned}
\end{equation}
under the condition that $T\leq F_{\lambda-1}$ (which implies $t_1\equiv t_2\bmod F_{\lambda-1}\Rightarrow t_1=t_2$).

We obtain
\begin{equation}\label{eqn_corr_mean_estimate}
\begin{aligned}
\hspace{1em}&\hspace{-1em}
\left(
\frac 1T
\sum_{t=0}^{T-1}
\bigl \lvert
\omega_t(\vartheta,N)
\bigr \rvert
\right)^2
=
\left(
\frac 1T
\sum_{t=0}^{T-1}
\left \lvert
\frac 1N
\sum_{n=0}^{N-1}
\e\bigl(\vartheta\smallspace\sz_\lambda(n+t)-\vartheta\smallspace\sz_\lambda(n)\bigr)
\right \rvert
\right)^2
\\&\ll
\left(
\frac 1T
\sum_{0\leq t<T}
\left \lvert
\frac 1{w_k}
\sum_{0\leq n<w_k}
\e\bigl(\vartheta\smallspace\sz_\lambda(n+t)-\vartheta\smallspace\sz_\lambda(n)\bigr)
\right \rvert
\right)^2
+\LandauO\left(\frac{F_\lambda}N\right)
\\&\ll
\frac{F_\lambda}T
\left(
\sum_{0\leq h<F_\lambda}
\bigl \lvert G_\lambda(h)\bigr \rvert^4
+
\sum_{0\leq h<F_{\lambda-1}}
\bigl \lvert G_{\lambda-1}(h)\bigr \rvert^4
\right)
+\LandauO\left(\frac{T}{F_\lambda}+\frac{F_\lambda}N\right).
\end{aligned}
\end{equation}
In order to estimate the terms $G_\lambda(h)$, uniformly in $h$, we could use Gowers norms (cf. Chapter~\ref{chap_gowers}). 
However, for this elementary case it is instructive to present an independent proof.
The following proposition is basically contained in the third author's thesis~\cite{Spiegelhofer2014}. 

\begin{proposition}\label{prp_zeckendorf_fourier}
For $\lambda\geq 0$ let $\widetilde G_\lambda(\vartheta,\beta)$ be defined as the following variant of~\eqref{eqn_Z_fourier}:
\[\widetilde G_\lambda(\vartheta,\beta)=\frac 1{\golden^\lambda}\sum_{0\leq u<F_\lambda}\e\left(\vartheta\smallspace\sz(u)+\beta u\right).\]
There exist constants $c>0$ and $C$ such that
for all $\lambda\geq 2$, $\vartheta\in\mathbb R$, and $\beta\in\mathbb R$,
\begin{equation}\label{eqn_zeckendorf_fourier_estimate}
\bigl \lvert \widetilde G_\lambda(\vartheta,\beta)\bigr \rvert
\leq C e^{-c\lambda\lVert \vartheta\rVert^2}.
\end{equation}
\end{proposition}
\begin{proof}
Clearly, by periodicity, we only need to consider $\vartheta,\beta\in[0,1]$.
In this proof, we omit the arguments $\vartheta$ and $\beta$ of the function $\widetilde G_\lambda$; they do not change in the course of the proof.
By the relation $\sz(u+F_\lambda)=1+\sz(u)$ that holds for $\lambda\geq 2$
and $0\leq u<F_\lambda-1$ we see that for all $\lambda\geq 2$,
\begin{multline*}
\widetilde G_{\lambda+1}=\frac 1{\golden^{\lambda+1}}
\sum_{0\leq u<F_\lambda}
\e\bigl(\vartheta\smallspace\sz(u)+\beta u\bigr)
+
\frac 1{\golden^{\lambda+1}}
\sum_{0\leq u<F_{\lambda-1}}
\e\bigl(\vartheta\smallspace \sz(u+F_\lambda)+\beta u+\beta F_\lambda\bigr)
\\=
\frac 1{\golden}\widetilde G_\lambda+
\frac 1{\golden^2}
\e\bigl(\vartheta+\beta F_\lambda\bigr)
\widetilde G_{\lambda-1}.
\end{multline*}

We write $\alpha_\lambda=\alpha_\lambda(\vartheta,\beta)=\e\bigl(\vartheta+\beta F_\lambda\bigr)$
and $A_\lambda=A_\lambda(\vartheta,\beta)=
\left(\begin{smallmatrix}
\golden^{-1}&\golden^{-2}\alpha_\lambda
\\1&0
\end{smallmatrix}\right)
$.
We obtain
\[\left(\begin{matrix}\widetilde G_{\lambda+1}
\\\widetilde G_\lambda\end{matrix}\right)
=A_\lambda\left(\begin{matrix}\widetilde G_\lambda
\\\widetilde G_{\lambda-1}\end{matrix}\right).
\]
A short calculation reveals that
\[\left(\begin{matrix}\widetilde G_{\lambda+5}
\\\widetilde G_{\lambda+4}\end{matrix}\right)
=A_{\lambda+4}A_{\lambda+3}A_{\lambda+2}A_{\lambda+1}A_\lambda
\left(\begin{matrix}\widetilde G_\lambda
\\\widetilde G_{\lambda-1}\end{matrix}\right)
=\left(\begin{matrix}a_\lambda&b_\lambda\\c_\lambda&d_\lambda\end{matrix}\right)
\left(\begin{matrix}
\widetilde G_\lambda\\\widetilde G_{\lambda-1}
\end{matrix}\right)
\]
for $\lambda\geq 2$,
where
\[
\begin{aligned}
a_\lambda &= \golden^{-5}\bigl(1+\alpha_{\lambda+1}+\alpha_{\lambda+2}+\alpha_{\lambda+3}(1+\alpha_{\lambda+1})+\alpha_{\lambda+4}(1+\alpha_{\lambda+1}+\alpha_{\lambda+2})\bigr),\\
b_\lambda &= \golden^{-6}\alpha_\lambda\bigl(1+\alpha_{\lambda+2}+\alpha_{\lambda+3}+\alpha_{\lambda+4}(1+\alpha_{\lambda+2})\bigr),\\
c_\lambda &= \golden^{-4}\bigl(1+\alpha_{\lambda+1}+\alpha_{\lambda+2}+\alpha_{\lambda+3}(1+\alpha_{\lambda+1})\bigr),\\
d_\lambda &= \golden^{-5}\alpha_\lambda\bigl(1+\alpha_{\lambda+2}+\alpha_{\lambda+3}\bigr).
\end{aligned}
\]
To obtain the result,
we use the row-sum norm $\lVert\cdot\rVert_\infty$ for matrices,
which is derived from the maximum norm for vectors
and which is sub-multiplicative.
Since $\lVert A_\lambda\rVert_\infty\leq 1$,
it suffices to prove that
\begin{equation}\label{eqn_matrixproduct_sup_bound}
\sup_{\substack{\beta\in\mathbb R\\\lambda\geq 2}}
\,\bigl \lVert A_{\lambda+4}A_{\lambda+3}A_{\lambda+2}A_{\lambda+1}A_\lambda\bigr \rVert_\infty
\leq \exp\bigl(-c'\lVert\vartheta\rVert^2\bigr)
\end{equation}
for some positive absolute constant $c'$. Note that $\beta$ and $\vartheta$ occur in the definition of $A_\lambda$, since $\alpha_\lambda$ depends on these quantities.
We apply the following lemma, appearing for example in Delange~\cite{D1972}. 

\begin{lemma}\label{lem_delange}
Let $z_1,\ldots,z_{q-1}$ be complex numbers such that $|z_j|\leq 1$
for $1\leq j<q$. Then
\[\left\lvert\frac 1q(1+z_1+\cdots+z_{q-1})\right\rvert
\leq
1-\frac 1{2q}\max_{1\leq j<q}\left(1-\mathrm{Re}z_j\right)
.
\]
\end{lemma}

By this lemma and the elementary estimate
$\cos(2\pi x)\leq 1-2\lVert x\rVert^2$ we get
\begin{equation}\label{eqn_angle_estimate}
  k-\left\lvert 1+\e(x_1)+\cdots+\e(x_{k-1})\right\rvert
\geq
  \frac 12\max_{1\leq i<k}(1-\mathrm{Re}\e(x_i))
\geq
  \max_{1\leq i<k}\lVert x_i\rVert^2
\end{equation}
for any family $(x_1,\ldots,x_{k-1})$ of real numbers.

We have
\begin{align*}
\lvert a_\lambda\rvert\leq 8\golden^{-5},&\qquad\lvert b_\lambda\rvert\leq 5\golden^{-6},\\
\lvert c_\lambda\rvert\leq 5\golden^{-4},&\qquad\lvert d_\lambda\rvert\leq 3\golden^{-5}.
\end{align*}
We assume that $\lvert a_\lambda\rvert+\lvert b_\lambda\rvert\geq 1-\varepsilon$.
Then $8\golden^{-5}-\lvert a_\lambda\rvert\leq \varepsilon$ and
$5\golden^{-6}-\lvert b_\lambda\rvert\leq \varepsilon$, since in the other case we get a contradiction by the identity
$8\golden^{-5}+5\golden^{-6}=1$.
We apply~(\ref{eqn_angle_estimate}), inserting five dummy terms, and obtain therefore
\begin{equation*}
  \max\bigl\{\lVert\vartheta+\beta F_{\lambda+1}\rVert^2,\lVert\vartheta+\beta F_{\lambda+2}\rVert^2,\lVert\vartheta+\beta F_{\lambda+3}\rVert^2 \bigr\}
\leq 8-\golden^5\lvert a_\lambda\rvert\leq \golden^{5}\varepsilon
.
\end{equation*}
We apply an analogous argument for the quantity $c_\lambda$.
The assumption
$\lvert c_\lambda\rvert+\lvert d_\lambda\rvert \geq 1-\varepsilon$
leads to the same estimate (it yields an upper bound $\golden^4\varepsilon$ for the maximum, which is bounded by $\golden^5\varepsilon$).
Now if
$\lVert A_{\lambda+4}A_{\lambda+3}A_{\lambda+2}A_{\lambda+1}A_\lambda\rVert_{\infty}\geq 1-\varepsilon$,
we have $\lvert a_\lambda\rvert+\lvert b_\lambda\rvert\geq 1-\varepsilon$ or $\lvert c_\lambda\rvert+\lvert d_\lambda\rvert\geq 1-\varepsilon$ and therefore
\begin{multline*}
\lVert\vartheta\rVert
=\lVert\vartheta+\beta F_{\lambda+1}+\vartheta
+\beta F_{\lambda+2}
-(\vartheta+\beta F_{\lambda+3})\rVert
\\\leq\lVert\vartheta+\beta F_{\lambda+1}\rVert
+\lVert\vartheta+\beta F_{\lambda+2}\rVert
+\lVert\vartheta+\beta F_{\lambda+3}\rVert
\leq 3\sqrt{\golden^5\varepsilon}.
\end{multline*}

By contraposition and continuity, it follows that
\[\lVert A_{\lambda+4}A_{\lambda+3}A_{\lambda+2}A_{\lambda+1}A_\lambda\rVert_{\infty}\leq 1-c'\lVert\vartheta\rVert^2
\leq \exp\left(-c'\lVert \vartheta\rVert^2\right),
\]
where $c'=1/(9\golden^5)$.
This estimate is independent of $\beta$ and $\lambda$.
By decomposition of the matrix product
$A_{\lambda-1}\cdots A_1$ into blocks of length five, and the fact that $\lVert A_k\rVert_\infty\leq 1$, we obtain the statement of the proposition.
The values of $c,C$ could be made explicit without any problem, but this is not necessary for our main theorem.
\end{proof}

The argument finishing the case ``long arithmetic progressions'' starts with~\eqref{eqn_corr_mean_estimate}.
We are going to choose $F_\lambda$ slightly larger than $T$ in order to obtain a nontrivial error term; we see that we need a nontrivial estimate for the fourth power of Fourier terms.
Using Proposition~\ref{prp_zeckendorf_fourier} and Parseval's identity, we obtain
\[
\sum_{0\leq h<F_\lambda}
\bigl \lvert G_\lambda(h)\bigr \rvert^4
\leq 
\sup_{h\in\mathbb Z}\,\bigl \lvert G_\lambda(h)\bigr \rvert^2
\sum_{0\leq h<F_\lambda}
\bigl \lvert G_\lambda(h)\bigr \rvert^2
\ll
\exp\bigl(-2c\lambda\lVert\vartheta\rVert^2\bigr).
\]
We choose
\[R\asymp N^{1/12}.\]
Since $D^3\ll N$, there is enough room to choose $F_\lambda$ between $T=2RD$ and $\sqrt{N}$ --- note the presence of the error terms $T/F_\lambda$ and $E^{(1)}_1=F_\lambda\bigl(\log DN\bigr)^3/\sqrt{N}$.

For any $\lambda$ satisfying $2RD\leq F_\lambda\leq N$, we have
\[\bigl \lvert G_\lambda(h)\bigr \rvert^2
\ll
\exp\bigl(-c\mu\lVert\vartheta\rVert^2\bigr)\]
with the constant $c$ from Proposition~\ref{prp_zeckendorf_fourier},
where $N\asymp 2^\mu$.
Choose
\[F_\lambda\asymp T\exp\left(\frac{c\mu}2\lVert\vartheta\rVert^2\right).\]
With these choices of $R$ and $F_\lambda$, we obtain from~\eqref{eqn_corr_mean_estimate}
\begin{equation}\label{eqn_corr_mean_estimate_2}
\begin{aligned}
\left(
\frac 1T
\sum_{t=0}^{T-1}
\bigl \lvert
\omega_t(\vartheta,N)
\bigr \rvert
\right)^2
&\ll
\exp\left(-\frac{c\mu}2\lVert\vartheta\rVert^2\right)
+
\exp\left(\frac{c\mu}2\lVert\vartheta\rVert^2\right)
N^{-1/6}
\\&\ll N^{-c'\lVert\vartheta\rVert^2}
\end{aligned}
\end{equation}
for some positive absolute constant $c'$, and an absolute implied constant.
We can finally handle~\eqref{eqn_S0_eight_power}:
the error terms can be bounded by similar arguments as the error terms in~\eqref{eqn_corr_mean_estimate}, where the error term $E^{(1)}_1$ is accounted for by the choices of $R$ and $F_\lambda$.
We obtain
\[\bigl \lvert S_0(D,N,\vartheta,\xi)\bigr \rvert
\leq C
\bigl(\log^+\!N\bigr)^{3/4}
N^{-c''\lVert\vartheta\rVert^2}
\]
for some absolute constants $c''>0$ and $C$.
This estimate is valid uniformly in the variables $D,N,\vartheta,$ and $\xi$, where $D^3\leq N$.
This finishes the case $1\leq D\leq N^{1/3}$.

\bigskip
\noindent
\subsubsection{Short arithmetic progressions.}\label{sec_short}
This case is harder, and uses the generalization~\eqref{eqn_vdc_generalized} of van der Corput's inequality,
two-and three-dimensional detection, and Gowers norms.
We assume that $N^{1/3}\leq D\leq N^{\rho_2}$ for some real number $\rho_2\geq 1/3$.
\begin{remark}
[Various remarks]
Note that we start from~\eqref{eqn_S0_squared},
but it is not necessary to keep the same choice of $R$ that we had for the case ``long arithmetic progressions''. We will choose $R$ at the end.
In contrast to the first case, the summation over $n$ cannot be guaranteed to be longer than $\golden^\lambda$ --- note that $\golden^\lambda\gg D$ in order to obtain a nontrivial error term $E_0$.
Therefore the arithmetic progression $nd+a$ cannot yet be dispensed with.
As a remedy, we cut away digits repeatedly, using van der Corput's inequality and the two-and three-dimensional detection procedures.
An analogous method was applied successfully for the case of the \emph{Thue--Morse sequence}~\cite{S2020}.
In fact, we proceed slightly different to that paper, cutting away digits beginning \emph{from the left} (that is, starting at the more significant digits) instead of \emph{from the right}.
In order to do so, we employ the variant~\eqref{eqn_vdc_generalized} of van der Corput's inequality.
This statement has practically the same proof as the usual inequality of van der Corput.
However, this minor variation has a huge impact, as the problem simplifies considerably; in fact we could not tackle the problem without using this tool.
\end{remark}

In the following, constants implied by $\LandauO$-estimates may depend on the variable $m$. This variable is used to denote the number of times that we apply van der Corput's inequality.

Cauchy--Schwarz applied to~\eqref{eqn_S0_squared} implies
\begin{align*}
\left\lvert\frac{S_0(D,N,\vartheta,\xi)}{ND}\right\rvert^4
\ll E_0
+\frac 1{RD}
\sum_{1\leq r<R}
\sum_{D\leq d<2D}
\sup_{a\in\mathbb N}
\,\lvert S_1\rvert^2.
\end{align*}
Again, we do not need to square the error term.

The modified version of van der Corput's inequality that we will present in Proposition~\ref{prp_vdC_generalized} below will be an essential tool in our proof.
Lemme~17 of~\cite{MR2009} is not sufficient for our needs --- this is due to the non-periodicity of Zeckendorf digits, and we need to admit more general sets of shifts than just the set $\{kr:\lvert r\rvert<R\}$ appearing in~\cite{MR2009}.
Another generalization of van der Corput exists in the literature:
Lemme~5 in the paper~\cite{RS2001} by Rivat and Sargos on the Piatetski-Shapiro prime number theorem.
This generalization is similar in spirit to Proposition~\ref{prp_vdC_generalized} below.
It admits an arbitrary sequence $(x_1,\ldots,x_M)$ of reals as parameters, but it seems to go in a slightly different direction.
In particular, we suspect that Proposition~\ref{prp_vdC_generalized} does not simply follow by choosing real numbers $x_1,\ldots,x_M$ in that result.

Although the (short) proof of the following proposition is practically the same as for~\cite[Lemme~17]{MR2009}, the statement appears to be new.
\begin{proposition}[Generalized van der Corput inequality]\label{prp_vdC_generalized}
Let $I$ be a finite interval in $\mathbb Z$ containing $M$ integers and $x_m\in\mathbb C$ for $m\in I$.
Assume that $K\subset \mathbb N$ is a finite nonempty set.
Then
\begin{equation}\label{eqn_vdc_generalized}
\left\lvert \sum_{m\in I}x_m\right\rvert^2
\leq \frac{M+\max K-\min K}{\lvert K\rvert^2}
\sum_{k_1,k_2\in K}
\sum_{\substack{m\in\mathbb Z\\m,m+k_1-k_2\in I}}
x_m\overline {x_{m+k_1-k_2}}.
\end{equation}
\end{proposition}
\begin{proof}
For convenience, we set $x_m=0$ for $m\not\in I$.
Moreover, let $A=\min I-\max K$ and $B=\max I-\min K$.
Then by Cauchy--Schwarz
\begin{align*}
\left\lvert
\lvert K\rvert
\sum_{m\in I}x_m
\right\rvert^2
&=
\left\lvert
\sum_{k\in K}
\sum_{m\in\mathbb Z}
x_{m+k}
\right\rvert^2
=
\left\lvert
\sum_{m\in\mathbb Z}
\sum_{k\in K}
x_{m+k}
\right\rvert^2
\\&\leq
\bigl(M+\max K-\min K\bigr)
\sum_{A\leq m\leq B}
\sum_{k_1,k_2\in K}
x_{m+k_2}\overline{x_{m+k_1}}
\end{align*}
and
\[
\sum_{A\leq m\leq B}
\sum_{k_1,k_2\in K}
x_{m+k_2}\overline{x_{m+k_1}}
=
\sum_{k_1,k_2\in K}
\sum_{\substack{m\in\mathbb Z\\m+k_1,m+k_2\in I}}
x_{m+k_2}\overline{x_{m+k_1}}.
\]
A change of variables $m\mapsto m'-k_2$ yields the claim.
\end{proof}

We obtain, introducing the finite nonempty set $K_1=K_1(d)\subseteq\mathbb N$ to be defined later,
\begin{align*}
\bigl\lvert S_1\bigr\rvert^2
\ll
\frac{1}{\bigl\lvert K_1\bigr\rvert^2}
\sum_{k_1,k'_1\in K_1}
S'_2 + \LandauO(E_1),
\end{align*}
where
\[  S'_2=\frac 1N\sum_{\substack{0\leq n<N\\0\leq n+k_1-k'_1<N}}
    \prod_{\varepsilon_0,\varepsilon_1\in\{0,1\}}
    \e\left(
    (-1)^{\varepsilon_0+\varepsilon_1}
    \vartheta\smallspace
    \sz_\lambda\bigl(nd+a+\varepsilon_0 rd+\varepsilon_1 (k_1-k'_1)d\bigr)
    \right)  \]

and
\[  E_1=\frac{\max K_1-\min K_1}{N}.  \]

Since we are working with the truncated Zeckendorf sum-of-digits function $\sz_\lambda$, we easily see that $a\mapsto S'_2(a)$ takes each of its values infinitely often:
given $a$, we only have to add a sufficiently large Fibonacci number $F_{\ell(a)}$ to $a$, and obtain $S'_2(a)=S'_2\bigl(a+F_{\ell(a)}\bigr)$.
It follows that
\[\sup_{a\ge0}
\bigl\lvert S'_2\bigr\rvert
=\sup_{a\ge a_0}
\bigl\lvert S'_2\bigr\rvert
\]
for all $a_0$, and we set
\[a_0\coloneqq 2D\sum_{1\leq j\leq m}\bigl(\max K_j-\min K_j\bigr),\]
where $m$ and the sets $K_1,\ldots,K_j$ are chosen later.
With this choice, and given that $a\ge a_0$, the expression
$nd+a+\varepsilon_0 rd+\varepsilon_1 (k_1-k'_1)$
is a nonnegative integer, and we may set
\[  S_2\coloneqq\frac 1N\sum_{0\leq n<N}
    \prod_{\varepsilon_0,\varepsilon_1\in\{0,1\}}
    \e\left(
    (-1)^{\varepsilon_0+\varepsilon_1}
    \vartheta
    \sz_\lambda\bigl(nd+a+\varepsilon_0 rd+\varepsilon_1 (k_1-k'_1)d\bigr)
    \right).   \]
It follows that
\[\sup_{a\ge0}
\bigl\lvert S'_2\bigr\rvert
=\sup_{a\ge a_0}
\bigl\lvert S_2\bigr\rvert
+\LandauO(E_1).
\]
Therefore
\begin{equation}\label{eqn_S0_square_0}
\begin{aligned}
\hspace{8em}&\hspace{-8em}
\bigl\lvert S_0(D,N,\vartheta,\xi)\bigr\rvert^4
\ll\frac 1D
\sum_{D\leq d<2D}
\frac1R
\sum_{1\leq r<R}
\\&
\frac1{\bigl\lvert K_1(d)\bigr\rvert^2}
\sum_{k_1,k'_1\in K_1(d)}
\sup_{a\geq a_0}
\bigl\lvert S_2\bigr\rvert
+E_0+E_1.
\end{aligned}
\end{equation}
We apply Cauchy--Schwarz and the generalized van der Corput inequality alternatingly. Using the definition of $a_0$ above, we obtain in analogy to~\cite[equation (5.1)]{S2020}:
\begin{equation}\label{eqn_S0_estimate}\begin{aligned}
  \bigl\lvert S_0(D,N,\vartheta,\xi)\bigr\rvert^{2^{m+1}}
  &\ll\frac 1D\sum_{D\leq d<2D}\frac 1R\sum_{1\leq r<R}
  \frac 1{\,\bigl\lvert K_1(d)\bigr\rvert^2\cdots
  \bigl\lvert K_m(d)\bigr\rvert^2}
  \\&\times\sum_{\substack{k_j,k'_j\in K_j(d)\\\text{for }1\leq j\leq m}}
  \sup_{a\geq a_0}\,\bigl\lvert S_3\bigr\rvert
  +E_0+\cdots+E_m,
\end{aligned}\end{equation}
where
\begin{multline*}
  S_3=\sum_{0\leq n<N}
  \prod_{\varepsilon_0,\ldots,\varepsilon_m\in\{0,1\}}
  \e\bigl(
  (-1)^{\varepsilon_0+\cdots+\varepsilon_m}
  \vartheta\smallspace
  \sz_\lambda\bigl(nd+a+\varepsilon_0 rd
\\
  +\varepsilon_1 (k_1-k'_1)d+\cdots
  +\varepsilon_m(k_m-k'_m)d\bigr)
  \bigr),
\end{multline*}
\begin{equation}\label{eqn_Ej_choice}
\begin{aligned}
E_0&=\frac 1{R}+\frac{RD}{F_\lambda}+\frac RN,
&&\mbox{and}\\
E_j(d)&=
\frac{\max K_j(d)-\min K_j(d)}N&&\mbox{for }1\leq j\leq m.
\end{aligned}
\end{equation}
(Note that we could, alternatively, use the difference operator $\Delta$ employed in Chapter~\ref{ch_type2} 
in order to express the sum $S_3$, and the sums following later on.
However, we think that the notation used in the present chapter makes the influence of the $k_j, k_j'$ are more apparent.)

Let us choose the sets $K_1(d),\ldots,K_m(d)$ in such a way that the number of digits to be taken into account is reduced successively.

After the first application of van der Corput's inequality, only the interval $[2,\lambda)$ of digits remained (see~\eqref{eqn_S0_squared}), and we replaced $\sz$ by the truncated version $\sz_\lambda$.
The set $K_1(d)$ will be responsible for removing the digits with indices in $\bigl[\lambda-\mu,\lambda\bigr)$, where $\mu$ is chosen later.
In general, for $1\leq j<m$, the set
$K_j(d)$ will exclude the digits with indices in $\bigl[\lambda-j\mu,\lambda-(j-1)\mu\bigr)$.
The last step is slightly different in that we remove more digits:
$K_m(d)$ will take care of excluding the digits with indices in $\bigl[\lambda-(m+3)\mu,\lambda-(m-1)\mu\bigr)$.
In order to do this, we will define the sets $K_j(d)$ suitably in~\eqref{eqn_Ki_def}, and apply Proposition~\ref{prp_vdC_generalized} for each $1\leq j\leq m$.

Due to carry propagation (to the left \emph{and} to the right) we will need margins of a certain width $\sigma\geq 5$ to be defined later.
On the margins, we prohibit certain digit combinations.
We will see that, for all $k_j$ and $k'_j$, these forbidden digit combinations are avoided \emph{for most} $d$ and $n$.
For $d$ and $n$ in the remaining ``good'' set it will be the case that for all $j_0\in [1,m]$ and for all combinations
\[(\varepsilon_0,\ldots,\varepsilon_{j_0-1},\varepsilon_{j_0+1},\ldots,\varepsilon_m)\in \{0,1\}^m,\]
the integers
\begin{align*}
\kappa=nd+a+\varepsilon_0rd+\sum_{\substack{j=1\\j\neq j_0}}^m \varepsilon_j(k_j-k'_j)d
\quad\mbox{and}\quad
\kappa'=\kappa+(k_{j_0}-k'_{j_0})d
\end{align*}
have the same Zeckendorf digits with indices in $\bigl[\lambda-j_0\mu,\lambda-(j_0-1)\mu\bigr)$. (See Lemma~\ref{lem_same_digits} below.)
Since the Zeckendorf digit sums of these two integers appear as a difference --- \emph{thanks to van der Corput's inequality} --- this allows us to discard the digits in $\bigl[\lambda-j_0\mu,\lambda-(j_0-1)\mu\bigr)$.

We fix some notation.
For $1\leq j<m$, let
\begin{equation}\label{eqn_interval_def_regular}
\begin{array}{ll}
a_j=\lambda-j\mu-\sigma,&a_j'=\lambda-j\mu,\\
b_j'=\lambda-(j-1)\mu,&b_j=\lambda-(j-1)\mu+\sigma,
\end{array}
\end{equation} 
and for the case $j=m$ set
\begin{equation}\label{eqn_interval_def_last}
\begin{array}{ll}
a_m=\lambda-(m+3)\mu-\sigma,&
a_m'=\lambda-(m+3)\mu,\\
b_m'=\lambda-(m-1)\mu,&
b_m=\lambda-(m-1)\mu+\sigma.
\end{array}
\end{equation} 

These integers define intervals
$\bigl[a'_j,b'_j\bigr)\subset \bigl[a_j,b_j\bigr)$. Based on these intervals,
we will choose $K_j(d)$ in such a way that the
summand $k_jd$, for all $k_j\in K_j(d)$, has
no nonzero digit in the larger interval $\bigl[a_j,b_j\bigr)$.
Therefore it will not change the digits of $\kappa$ in the smaller interval $\bigl[a'_j,b'_j\bigr)$ when added to an integer $\kappa$ from the ``good'' set.
In total, the ranges of the digits cover the interval $[\lambda-(m+3)\mu-\sigma,\lambda)$,
and only digits below
\begin{equation}\label{eqn_nu_def}
\nu\eqdef a_m'=\lambda-(m+3)\mu
\end{equation}
will remain.

We introduce another parameter $B$ to be chosen later.
This parameter will be a bound on the diameter $\max K_j(d)-\min K_j(d)$ appearing in Proposition~\ref{prp_vdC_generalized}.
For $1\leq j\leq m$, we set
\begin{equation}\label{eqn_Ki_def}
K_j(d)\coloneqq\bigl\{0\leq k< B:\digit_i(kd)=0\mbox{ for }a_j\leq i<b_j\bigr\}.\end{equation}

Next we will be concerned with the exceptional cases where carry- or borrow propagation from outside the interval $[a_j,b_j)$ into the smaller interval $[a'_j,b'_j)$ occurs upon adding $(k_j-k'_j)d$.
We exclude certain digit combinations on the margins: we define
\begin{align*}
A_j&=\left\{t\in\mathbb N:
\left\{\begin{array}{ll}
\digit_i(t)=1 \mbox{ for }a_j\leq i<a_j'\mbox{ and }2\mid i &\mbox{or}\\
\digit_i(t)=1 \mbox{ for }a_j\leq i<a_j'\mbox{ and }2\nmid i &\mbox{or}\\
\digit_i(t)=0 \mbox{ for }a_j\leq i<a_j'
\end{array}\right\}
\right\},\\
B_j&=\left\{t\in\mathbb N:
\left\{\begin{array}{ll}
\digit_i(t)=1 \mbox{ for }b_j'
\leq i<b_j-3\mbox{ and }2\mid i &\mbox{or}\\
\digit_i(t)=1 \mbox{ for }b_j'
\leq i<b_j-3\mbox{ and }2\nmid i
\end{array}\right\}
\right\}.
\end{align*}
The set $A_{j_0}$ guarantees that adding $(k_{j_0}-k'_{j_0})d$ will not cause a carry- or borrow propagation \emph{from the right} into the interval $[a'_{j_0},b'_{j_0})$ (that is, coming from the less significant digits).
The set $B_{j_0}$ handles carry- and borrow propagation \emph{from the left} into this interval.
Note that the digits with indices in $[b_j-3,b_j)$ 
may be arbitrary.
In this way, the exceptional set $B_j$ is enlarged.
The use of this will become obvious in the proof of Lemma~\ref{lem_same_digits} below.

For each choice of $d\in[D,2D)$, $r\in[1,R]$, and $(k_j,k'_j)\in K_j(d)^2$ we have to exclude those $n$ having the following property.

\begin{equation}\label{eqn_exclude}
\begin{aligned}
&\mbox{For some $1\leq j_0\leq m$ and $(\varepsilon_0,\ldots,\varepsilon_{j_0-1},\varepsilon_{j_0+1},\ldots,\varepsilon_m)\in \{0,1\}^m$,}\\
&\mbox{we have }nd+a+\varepsilon_0rd+\sum_{\substack{1\leq j\leq m\\j\neq j_0}} \varepsilon_j(k_j-k'_j)d\in A_{j_0}\cup B_{j_0}.
\end{aligned}
\end{equation}

We are interested in the number of these exceptional integers $n\in[0,N)$.
For each $j_0$ (there are $m$ of them) and for $\LandauO\bigl(2^m\bigr)$ choices of $(\varepsilon_0,\ldots,\varepsilon_m)$,
we have to exclude $\LandauO(1)$ digit combinations.
By our convention that implied constants may depend on the variable $m$, the total number of digit combinations to be excluded is $\LandauO(1)$.
In order to determine the number of exceptions to~\eqref{eqn_exclude}, we want to apply Proposition~\ref{prp_twodim_estimate}, but we have to keep in mind that the sets $K_j(d)$ depend on $d$.
Using discrepancy as in~\eqref{eqn_prp_twodim_estimate}, we have an upper bound for this number;
an important point is to note that this upper bound is \emph{independent of the shift} caused by the integers $a,r,k_j$, and $k'_j$.
For our purpose of estimating the number of exceptions, we can therefore dispense with the average over $k_j,k'_j$ and apply Proposition~\ref{prp_twodim_estimate}.
This proposition introduces an error term
\begin{equation}\label{eqn_END_bound}
F=\golden^{-\delta}+\left(\frac1{N^{1/4}}+\frac1{\golden^{\nu/2}}+\frac{\golden^{\lambda+\sigma}}{DN^{1/2}}+\frac 1D\right)\bigl(\log^+N\bigr)^2,
\end{equation}
where the first summand accounts for the expected number of exceptions.
Note that the values $\nu\coloneqq a'_m=\lambda-(m+3)\mu$
and $\lambda+\sigma=b_1$ are (up to $\LandauO(1)$) the upper endpoints of the outermost detection intervals in question, and for the error $F$ we take the larger contribution of each summand in~\eqref{eqn_prp_twodim_estimate}.

We will see later that our choices of $a_j',b_j$ (depending on $\sigma\geq 5$, $\lambda$, and $\mu$) imply a total error $F\ll\golden^{-\sigma}$.

For the remaining good indices $n$ we have the following important fact.
\begin{lemma}\label{lem_same_digits}
Assume that the margin $\sigma$ is at least $5$.
If $n$ does not satisfy~\eqref{eqn_exclude}, then for all $1\leq j_0\leq m$, and for all
\[(\varepsilon_0,\ldots,\varepsilon_{j_0-1},\varepsilon_{j_0+1},\ldots,\varepsilon_m)\in \{0,1\}^m,\]
the integers
\begin{align*}
\kappa=nd+a+\varepsilon_0rd+\sum_{\substack{1\leq j\leq m\\j\neq j_0}} \varepsilon_j(k_j-k'_j)d\quad\mbox{and}\quad
\kappa'=\kappa+(k_{j_0}-k'_{j_0})d
\end{align*}
have the same Zeckendorf digits with indices in $\bigl[a'_{j_0},b'_{j_0}\bigr)$
.\end{lemma}
\begin{proof}
We begin with the summand $x=k_{j_0}d$.
We split this number into two parts (note that $x$ does not have nonzero digits between $a_{j_0}$ and $b_{j_0}$ by definition of $K_{j_0}$):
\[x_1=\sum_{2\leq i<a_{j_0}}\digit_i(x)F_i\quad\mbox{and}\quad x_2=\sum_{i\geq b_{j_0}}\digit_i(x)F_i.\]
We begin with the contribution $x_1$ of the lower digits.
For the current case concerning $k_{j_0}$, we need lines $1$ and $2$ of the definition of $A_{j_0}$.
Since $\kappa$ has a block $\tO\tO$ in the margin $[a_{j_0},a_{j_0}')$,
addition of $x$ does not cause a carry propagation from the right beyond $a'_{j_0}$.
Analogously, subtraction of $k'_{j_0}d$ can be split into two parts,
and for the lower part we need the third line of the definition of $A_{j_0}$.
Excluding the block consisting only of zeros avoids borrows to occur, which would propagate into the interval $[a_{j_0}',b_{j_0}')$.

We proceed to the upper digits. This case is more involved and needs the one-dimensional detection of digits via intervals.
We are interested in the digits of $\kappa$ below $b'_{j_0}$ and want to show that they do not change when adding $x_2$.
Define \[   u=v(\kappa,b_{j_0}').   \]
The corresponding (wrapped) detection interval $I$ is obtained from the interval
\[    (-1)^{b_{j_0}'}\left(-\golden^{-b_{j_0}'\pm 1}, \golden^{-b_{j_0}'}\right)    \]
via the rotation $x\mapsto x+u\golden\bmod 1$.
Since $\kappa$ is not an element of $B_{j_0}$ and $\sigma\geq 5$, we have a block $\tO\tO$ in the interval $[b_{j_0}',b_{j_0}-3]$.
Here the summand $-3$ in the definition of $B_j$ comes into play.
Inspecting Equation~\eqref{eqn_snL_sep} in the proof of Proposition~\ref{Lefirstdigits},
we see that the point $v(\kappa,b_{j_0})\golden$ has a distance of at least $\golden^{-b_{j_0}+3}$ from the endpoints of the interval $I$.
By similar reasoning, the remaining digits $\geq b_{j_0}$ contribute less than
\[
\sum_{\substack{i\geq b_{j_0}\\2\mid i-b_{j_0} } }\lVert F_i\golden\rVert \leq 
\sum_{\substack{i\geq b_{j_0}\\2\mid i-b_{j_0} } }\golden^{-i}
=\golden^{-b_{j_0}}\frac 1{1-\golden^{-2}} 
=\golden^{-b_{j_0}+1}.
\]
Adding $x_2$ to $\kappa$ corresponds to adding $x_2\golden\bmod 1$ to $v(\kappa,b_{j_0})\golden$.
Since $2\gamma^{-b_{j_0}+1}<\gamma^{-b_{j_0}+3}$, the procedure can be repeated for $k'_{j_0}$.
We see that adding $\bigl(k_{j_0}-k'_{j_0}\bigr)d$ does not cause a carry propagation from the left into $\bigl[2,b'_{j_0}\bigr)$.
That is, the digits below $b_{j_0}'$ are not changed.
This completes the proof of Lemma~\ref{lem_same_digits}.
\end{proof}
Note that we used the requirement $\kappa'=\kappa+\bigl(k_{j_0}-k'_{j_0}\bigr)d\geq 0$ implicitly in this proof, but this is guaranteed since $a$ is large.

For the ``good indices'' $n$, we may therefore replace $\sz_\lambda$ by $\sz_\nu$, where
\[\nu\eqdef a'_m=\lambda-(m+3)\mu.\]
Note that the set of these good indices $n$ depends on $d$, $a$, $r$, $k_j$, and $k'_j$.
After this replacement, we extend the sum over $n$ to the full range $0\leq n<N$ again, introducing the error $F$ caused by Proposition~\ref{prp_twodim_estimate} (see~\eqref{eqn_END_bound} above).

\smallskip
We therefore obtain
\[  S_3=S_4+\LandauO(F),  \]
where
\begin{equation}\label{eqn_S4_expression}\begin{aligned}
\hspace{4cm}&\hspace{-4cm}
S_4=
\sum_{0\leq n<N}
\prod_{\varepsilon_0,\ldots,\varepsilon_m\in\{0,1\}}
\e\bigl(
(-1)^{\varepsilon_0+\cdots+\varepsilon_m}
\vartheta\smallspace\sz_\nu\bigl(nd+a+\varepsilon_0 rd
\\&
+ \varepsilon_1 (k_1-k'_1)d+\cdots
+ \varepsilon_m(k_m-k'_m)d\bigr)
\bigr)
\end{aligned}\end{equation}
and $\nu=\lambda-(m+3)\mu$.
We have therefore cleared the first hurdle: sufficiently many digits have been cut off so that we are in a situation similar to the case ``long arithmetic progressions''.
The length of the sums over $0\leq n<N$, $1\leq r<R$, and $k_j,k'_j\in K_j(d)$ (which are still present in~\eqref{eqn_S0_estimate}) will be large compared to the Fibonacci number $F_{\nu}$, so that we can hope for equidistribution.
Consequently, we will be able to replace the summands in the exponential in~\eqref{eqn_S4_expression} by full sums; the path towards a Gowers-type norm is now clearly visible.

In what follows, we will need one-and three-dimensional detection only.
Two-dimensional detection was only used for avoiding certain digit combinations on the margins of width $\sigma$.

\noindent
\textbf{Reduction to a Gowers norm.}
We start from $S_4$.
Using one-dimensional detection (Proposition~\ref{prp_onedim}), we first get rid of the arithmetic progression $nd+a$.

We use the function $g_\nu$ defined in~\eqref{eq_def_g_L}.
The $m+1$-fold product in $S_4$ amounts to considering a function built from $g_\nu$ that is constant on $2^{m+1}F_\nu$ wrapped intervals.
Using Theorem~\ref{th_koksma}, and Proposition~\ref{prp_onedim} this introduces an error
\[F^{(1)}=\frac{F_\nu}{\sqrt{N}}\log^+(DN),\]
 and we obtain
\begin{align*}
\bigl\lvert S_0(D,N,\vartheta,\xi)\bigr\rvert^{2^{m+1}}
\ll
S_5+
E_0+\cdots+E_m+F+F^{(1)},
\end{align*}
where
\begin{align*}
S_5&=
\frac1{DR
\bigl\lvert K_1(d)\bigr\rvert^2
\cdots
\bigl\lvert K_m(d)\bigr\rvert^2
}
\sum_{D\leq d<2D}
\sum_{1\leq r\leq R}
\sum_{\substack{k_j,k'_j\in K_j(d)\\\text{for }1\leq j\leq m}}
\\
&
\biggl\lvert
\int_0^1
\prod_{\varepsilon_0,\ldots,\varepsilon_m\in\{0,1\}}
\e\bigl(
(-1)^{\varepsilon_0+\cdots+\varepsilon_m}
  \vartheta
  g_\nu\bigl(x+\varepsilon_0 rd\golden
+ \varepsilon_1 (k_1-k'_1)d\golden
\\
&
+\cdots
+ \varepsilon_m(k_m-k'_m)d\golden\bigr)
\bigr)
\,\mathrm dx
\biggr\rvert.
\end{align*}
Similarly as in the argument introducing the error~\eqref{eqn_END_bound}, we don't lose the summation over $d$ in the process of applying Proposition~\ref{prp_onedim}.
This is because we use this proposition to separate the main term (an average in $d$ over integrals in $x$) from the error term $F^{(1)}$.
A similar argument also applies to the replacement of $r$ and of $k_j$ by integrals, which we will perform in a moment.

Note that in the above argument leading to $S_5$, the variable $a$ vanishes.
This yields the important uniformity property in the level of distribution-statement~\eqref{eqn_lod}.
The replacement of $(nd+a)\golden$ by $x$ is the centerpiece of the proof of Theorem~\ref{thm_lod}.
It is made possible by the \emph{repeated truncation process}, which was introduced in the paper~\cite{S2020} by the third author for the
case of the Thue--Morse sequence.

By Cauchy--Schwarz, we get rid of the absolute value in the above expression.
This introduces another integration variable, and we obtain
\begin{align*}
\bigl\lvert S_5\bigr\rvert^2&\leq
\frac1{DR
\bigl\lvert K_1(d)\bigr\rvert^2
\cdots
\bigl\lvert K_m(d)\bigr\rvert^2
}
\sum_{D\leq d<2D}
\sum_{1\leq r\leq R}
\sum_{\substack{k_j,k'_j\in K_j(d)\\\text{for }1\leq j\leq m}}
\\&
\int_0^1
\int_0^1
\prod_{\varepsilon,\varepsilon_0,\ldots,\varepsilon_m\in\{0,1\}}
\e\bigl(
(-1)^{\varepsilon+\varepsilon_0+\cdots+\varepsilon_m}
  \vartheta
  g_\nu\bigl(x+\varepsilon y+\varepsilon_0 rd\golden
\\
&
+ \varepsilon_1 (k_1-k'_1)d\golden+\cdots
+ \varepsilon_m(k_m-k'_m)d\golden\bigr)
\bigr)
\,\mathrm dx
\,\mathrm dy.
\end{align*}

Next we treat the summand $\varepsilon_0rd\golden$.
In contrast to the case ``long arithmetic progressions'', the sum over $r$ will be long compared to $F_\nu$.
On average (in $d$), we can therefore expect reasonably low discrepancy of $rd\golden\bmod 1$.
We move the summation over $r$ inside the integral; for given $x,y,d,k_j$, and $k'_j$ the product gives a function $G$ in the continuous variable $z$ (replacing $rd\golden$) that is constant on $\ll F_\nu$ wrapped intervals $I\in\mathcal J$ forming a partition of $[0,1)$.
Using Koksma (Theorem~\ref{th_koksma}) and the average discrepancy estimate from Proposition~\ref{prp_onedim} again, this introduces an error
\[F^{(2)}=\frac{F_\nu}{\sqrt{R}}\log^+(DR),\]
and we obtain
\begin{align*}
\bigl\lvert S_0(D,N,\vartheta,\xi)\bigr\rvert^{2^{m+2}}
\ll
S_6+
E_0+\cdots+E_m+F+F^{(1)}+F^{(2)},
\end{align*}
where
\begin{align*}
S_6&=
\frac1{D
\bigl\lvert K_1(d)\bigr\rvert^2
\cdots
\bigl\lvert K_m(d)\bigr\rvert^2
}
\sum_{D\leq d<2D}
\sum_{\substack{k_j,k'_j\in K_j(d)\\\text{for }1\leq j\leq m}}
\\&
\int_{[0,1]^3}
\prod_{\varepsilon,\varepsilon_0,\ldots,\varepsilon_m\in\{0,1\}}
\e\bigl(
(-1)^{\varepsilon+\varepsilon_0+\cdots+\varepsilon_m}
  \vartheta
  g_\nu\bigl(x+\varepsilon y+\varepsilon_0 z
\\&
+ \varepsilon_1 (k_1-k'_1)\,d\golden+\cdots
+ \varepsilon_m(k_m-k'_m)\,d\golden\bigr)
\bigr)
\,\mathrm d(x,y,z).
\end{align*}

We proceed to the sets $K_j(d)$.
Now we finally need the three-dimensional detection procedure.
The case $j=m$ will be handled separately, we therefore assume first that $1\leq j<m$.
The digits of $kd$ below $\nu$ should be uniformly distributed as $k$ runs through $K_j(d)$.
More generally, we will show that the sequence $(kd\golden)_{k\in K_j(d)}$ has small discrepancy modulo $1$.
This generalization is used in order to pass to the differences $k_j-k'_j$, and to replace the sums over $k_j$ and $k'_j$ by integrals;
the use of digits would complicate these tasks unnecessarily.
We are exactly in the situation where Proposition~\ref{prp_threedim_estimate} comes into play:
the set $K_j(d)$ was defined as the set of integers $k$ in the interval $[0,B)$ such the digits of $kd$ with indices in $[a_j,b_j)$ are zero.
For each $d$, the number of $k_j\in K_j(d)$ such that $k_jd\golden\in I+\mathbb Z$ appears as the cardinality of the set in~\eqref{eqn_prp_threedim_estimate} (for $B=T$). 
The shifts $k'_j$ are irrelevant, since the $\sup$ in Proposition~\ref{prp_threedim_estimate} is over all intervals $I\in\mathcal J$, and we can rotate by any value, for example by $-k'_jd\golden$.
That is, we have the case $\nu_a=\cdots=\nu_{b-1}=0$ in Proposition~\ref{prp_threedim_estimate}, where $T=B$ and $k_j$ serves as $t$.
For each $d$, we replace the sum over $k_j$ by an integral (of a step function,  constant on $\LandauO(F_\nu)$ intervals whose union is $[0,1]$),
which induces an error
\[F^{(3)}=
\golden^{\mu+2\sigma}F_\nu
\left(
\left(\frac 1D+\frac 1{B^{1/4}}+\frac1{\golden^{b_j/4}}\right)\bigl(\log^+B\bigr)^3
+\frac{\golden^{b_j}}{DB^{1/2}}
\bigl(\log^+B\bigr)^2
\right).
\]
Note that we only divide by $T\alpha_j$ instead of $T$ as in Proposition~\ref{prp_threedim_estimate}.
This explains the first factor since
\[  \alpha_j\asymp\golden^{b_j-a_j}=\golden^{\mu+2\sigma}\asymp \lambda(A')^{-1},  \]
where $A'$ is the two-dimensional detection parallelogram corresponding to the digits $\digit_{a_j}=\cdots=\digit_{b_j-1}=0$.
We also note that the third summand in this error term is the reason why we chose the $m$th interval to be larger than the previous ones.
Our choice of variables will make sure that $b_j\geq 5\mu$ for $1\leq j<m$, and $\nu$ and $\sigma$ are small enough so that $\gamma^{\mu+2\sigma+\nu-b_j/4}$ gives a nontrivial gain.

Carrying this out for all $1\leq j<m$ --- recall that $m=\LandauO(1)$ by our convention --- we obtain
\begin{align*}
\bigl\lvert S_0(D,N,\vartheta,\xi)\bigr\rvert^{2^{m+2}}
\ll
S_7+
E_0+\cdots+E_m+F+F^{(1)}+F^{(2)}+F^{(3)},
\end{align*}
where
\begin{align*}
S_7&=
\frac1{D
\bigl\lvert K_m(d)\bigr\rvert^2
}
\sum_{D\leq d<2D}
\sum_{k_m,k'_m\in K_m(d)}
\\&
\int_{[0,1]^{m+2}}
\prod_{\varepsilon,\varepsilon_0,\ldots,\varepsilon_m\in\{0,1\}}
\e\bigl(
(-1)^{\varepsilon+\varepsilon_0+\cdots+\varepsilon_m}
  \vartheta
  g_\nu\bigl(x+\varepsilon y+\varepsilon_0 z
\\&
+ \varepsilon_1 x_1+\cdots+\varepsilon_{m-1} x_{m-1}
+ \varepsilon_m(k_m-k'_m)\,d\golden\bigr)
\bigr)
\,\mathrm d(x,y,z,x_1,\ldots,x_{m-1}).
\end{align*}
It remains to handle the $m$th interval.
The replacement of the sum over $x_m$ by an integral requires a bit more work.
The last window that we have cut out is $[a'_m,b'_m)$, where $a'_m=\lambda-(m+3)\mu=\nu$ and $b'_m=\lambda-(m-1)\mu$.
It is four times the size of the intervals $[a'_j,b'_j)$ for $1\leq j<m$.
In order to cut out this interval, we have set the digits of $kd$ to zero in the larger interval $[a_m,b_m)$;
this interval overlaps with $[2,\nu)$ by the margin $[a_m,a_m+\sigma)$.
It follows that we do not have uniform distribution of the digits of $kd$ below $\nu$ under the constraint $k\in K_m(d)$ --- the digits of $kd$ on the margin $[a_m,a_m+\sigma)$ are zero.
Again, we work with the one-dimensional characterization of Zeckendorf digits instead of considering digits directly.
That is, we study the discrepancy of $k_md\golden$, on average in $d$, as $k_m$ runs through $[0,B)$.

Setting the digits of $kd$ in $[a_m,b_m)$ zero, for this is how $K_m(d)$ was defined,
corresponds to choosing $F_{a_m}$ wrapped intervals $I\in\mathcal J$ of length $\asymp\golden^{-b_m}$.
Let $J_1$ be the union of these wrapped intervals.
With the notation from Proposition~\ref{Lefirstdigits} we have the disjoint union
\begin{equation}\label{eqn_J1_identity}
\begin{aligned}
J_1+\mathbb Z
&=
\bigcup_{0\leq u<F_{a_m}}
\bigl(A_{b_m}(u)+\mathbb Z\bigr)
\\&=\bigcup_{0\leq u<F_{a_m}}
\left(u\golden+(-1)^{b_m}\left(-\frac{1}{\golden^{b_m-1}}, \frac{1}{\golden^{b_m}}\right)+\mathbb Z\right).
\end{aligned}
\end{equation}
On wrapped intervals $I\subseteq J_1$ we expect uniform distribution of $(\{k_md\golden\})_{k_m\in K_m(d)}$, while the remaining part $[0,1)\setminus J_1$ of the interval $[0,1)$ is not hit by $\{k_md\golden\}$.
More precisely, setting $\alpha_m=F_{a_m}\golden^{-b_m+2}=\lambda(J_1)$ (which is also the volume of the two-dimensional detection parallelogram corresponding to $\delta_{a_m}=\cdots=\delta_{b_m-1}=0$), Proposition~\ref{prp_onedim} yields the crucial estimate
\begin{equation}\label{eqn_crucial}
\begin{aligned}
\hspace{2em}&\hspace{-2em}\frac 1D\sum_{D\leq d<2D}
\sup_{\substack{I\in \mathcal I\\I\subseteq J_1}}
\biggl\lvert
\frac 1{B\alpha_m}
\#\bigl\{
k_m\in K_m(d):\{k_md\golden\}\in I\bigr\}-\frac 1{\alpha_m}\lambda(I)
\biggr\rvert
\\&=
\frac 1{\alpha_m}
\frac 1D\sum_{D\leq d<2D}
\sup_{\substack{I\in \mathcal I\\I\subseteq J_1}}
\biggl\lvert
\frac 1B
\#\bigl\{
0\leq k<B:\{kd\golden\}\in I\bigr\}-\lambda(I)
\biggr\rvert
\\&\ll
\frac 1{\alpha_m}\frac{\log(DB)}{\sqrt{B}}.
\end{aligned}
\end{equation}
We see that this gives the expected number of elements $(kd\golden)_{k\in K_m(d)}$ falling into intervals $I\subseteq J_1$, 
once the parameters have been chosen appropriately --- the points $\{k_md\golden\}$ are concentrated to $J_1$ and therefore an interval $I\subseteq J_1$ is hit with a frequency scaled up by a factor $\alpha_m^{-1}$.
We replace the sum over $k_m$ by an integral over $J_1$, using~\eqref{eqn_crucial}.
For all $d\in[D,2D)$, $x,y,z,x_1,\ldots,x_{m-1}\in[0,1)$ and $k'_m\in K_m(d)$ we see that the product in the definition of $S_7$ gives rise to a function in a new continuous variable $x_m$ (replacing $k_md\golden$). This function is a step function with $\ll F_\nu$ points of discontinuity in $[0,1)$, and we are interested in the restriction to $J_1$.
In each maximal wrapped interval $J\subseteq J_1$ (corresponding to digits $\delta_2,\ldots,\delta_{a_m-1}$ and $\delta_{a_m}=\cdots=\delta_{b_m-1}=0$), we have at most $\LandauO(1)$ many points of discontinuity of this piecewise constant function.
Applying~\eqref{eqn_crucial} yields $\alpha_m^{-1}$ times the integral over $J_1$; the error term has to be multiplied by $F_{a_m}$ since $J_1$ is a union of that many wrapped intervals.
We obtain
\begin{align*}
\bigl\lvert S_0(D,N,\vartheta,\xi)\bigr\rvert^{2^{m+2}}
\ll
S_8+
E_0+\cdots+E_m+F+F^{(1)}+F^{(2)}+F^{(3)}+F^{(4)},
\end{align*}
where
\begin{align*}
S_8&=
\frac1{D
\bigl\lvert K_m(d)\bigr\rvert
}
\sum_{D\leq d<2D}
\sum_{k'_m\in K_m(d)}
\\&
\int_{[0,1]^{m+2}}
\frac1{\alpha_m}
\int_{J_1-k'_md\golden}
\prod_{\varepsilon,\varepsilon_0,\ldots,\varepsilon_m\in\{0,1\}}
\e\bigl(
(-1)^{\varepsilon+\varepsilon_0+\cdots+\varepsilon_m}
  \vartheta
  g_\nu\bigl(x+\varepsilon y+\varepsilon_0 z
\\&
+ \varepsilon_1 x_1+\cdots+\varepsilon_{m-1} x_{m-1}
+ \varepsilon_mx_m\bigr)
\bigr)
\,\mathrm dx_m
\,\mathrm d(x,y,z,x_1,\ldots,x_{m-1})
\end{align*}
and
\[
F^{(4)}
=\frac{F_{a_m}}{\alpha_m}
\frac{\log(DB)}{\sqrt{B}}.
\]
We define
\[G(u)=
\prod_{\varepsilon,\varepsilon_0,\ldots,\varepsilon_{m-1}\in\{0,1\}}
\e\bigl(
(-1)^{\varepsilon+\varepsilon_0+\cdots+\varepsilon_{m-1}}
  \vartheta
  g_\nu\bigl(u+\varepsilon y+\varepsilon_0 z
+ \varepsilon_1 x_1+\cdots+\varepsilon_{m-1} x_{m-1}
\bigr)\bigr),
\]
omitting some arguments of $G$ for clarity, and
\[H(v)=
\left\lvert
\int_{[0,1]}
G(u)\overline{G(u+v)}
\,\mathrm du
\right\rvert.
\]
The function $H$ is continuous, bounded by $1$ in absolute value, and $1$-periodic in $\mathbb R$.
The function $G$ is a step function and $1$-periodic.
It has at most $2^{m+1}F_\nu$ points of discontinuity in $[0,1]$; at these points, the ``height'' $\lvert \lim_{x\nearrow z} G(x)-\lim_{x\searrow z}G(x)\rvert$ of the jumps is bounded by $2$.
Consequently, for any real numbers $v$ and $v'$, the functions $u\mapsto G(u+v)$ and $u\mapsto G(u+v')$ on $[0,1)$ differ only on a set $I\subseteq [0,1)$ of measure $\ll F_\nu\lvert v-v'\rvert$ (for each position $x$ where a jump occurs we cut out the interval $[x-\max(v,v'),x-\min(v,v')]$ of length $\lvert v-v'\rvert$).
 This implies
\[\bigl\lvert H(v)-H(v')\bigr\rvert \leq L \bigl\lvert v-v'\bigr\rvert\]
with a Lipschitz constant $L\ll F_\nu$.

The next step consists in an application of Fubini's theorem.
Note that, by~\eqref{eqn_J1_identity}, the set $J_1$ is a union of $F_{a_m}$ copies of the interval
\[
I=(-1)^{b_m}\left(-\frac{1}{\golden^{b_m-1}}, \frac{1}{\golden^{b_m}}\right),
\]
rotated by $\ell\golden$.
Therefore, we find for any $x\in \R$,

\begin{align*}
\frac1{\alpha_m}
\int_{J_1-k'_md\golden}
H(x_m)
\,\mathrm dx_m
&\leq
\golden^{b_m-2}
\int_{I-k'_md\golden}
\left\lvert
\frac 1{F_{a_m}}
\sum_{0\leq \ell<F_{a_m}}
H(x+u+\ell\golden)
\right\rvert
\,\mathrm du.
\end{align*}
By the bad approximability of $\golden$, in symbols, $\lVert n\golden\rVert\geq C/n$ for some absolute constant $C>0$,
the set $V=\bigl\{\{x+u+\ell\golden\}:0\leq \ell<F_{a_m}\bigr\}$
is $CF_{a_m}^{-1}/2$-spaced --- that is,
$[v-\rho,v+\rho]\cap V\subseteq \{v\}$ for all $v\in\mathbb R$,
where $\rho=CF_{a_m}^{-1}/2$.
For $v\in V$, we obtain from the Lipschitz condition the bound
\[\int_{[v-\rho,v+\rho)}H(v')\,\mathrm dv'
\gg \frac{H(v)^2}{F_\nu}
\]
with some absolute implied constant.
This implies
\begin{equation}\label{eqn_spaced_sum_integral}
\frac 1{F_{a_m}}
\sum_{0\leq \ell<F_{a_m}}
H(x+u+\ell\golden)^2
\ll
\frac{F_{\nu}}{F_{a_m}}
\int_{[0,1]}H(v)\,\mathrm dv%
\end{equation}
and by Cauchy--Schwarz,
\begin{align*}
\left\lvert
\frac 1{\alpha_m}
\int_{J_1-k'_md\golden}
H(x_m)\,\mathrm dx_m
\right\rvert^2
&\leq
\golden^{b_m-2}
\int_{I-k'_md\golden}
\frac 1{F_{a_m}}
\sum_{0\leq \ell<F_{a_m}}
H(x+u+\ell\golden)^2
\,\mathrm du
\\&\ll
\frac{F_\nu}{F_{a_m}}
\int_{[0,1]}H(v)\,\mathrm dv.
\end{align*}
It follows that
\begin{align*}
\bigl\lvert S_8\bigr\rvert^2
&\ll
\frac1D\frac{F_\nu}{F_{a_m}}
\sum_{D\leq d<2D}
\int_{[0,1]^{m+2}}
H(v)
\,\mathrm d(v,y,z,x_1,\ldots,x_{m-1}).
\end{align*}
Since
\[\bigl\lvert H(v)\bigr\rvert^2
=
\int_{[0,1]^2}
G(u)\overline{G(u+v)}
\overline{G(u+x)}
G(u+x+v)
\,\mathrm d(x,u),
\]
we only have to add two more variables to our Gowers norm.
Noting also that $F_\nu/F_{a_m}\ll \golden^\sigma$ and using Cauchy--Schwarz,
we obtain
\begin{equation}\label{eqn_S0_S8}
\begin{aligned}
\bigl\lvert S_0(D,N,\vartheta,\xi)\bigr\rvert^{2^{m+4}}
&\ll
\golden^{2\sigma}
S_9+
E_0+\cdots+E_m
\\&+F+F^{(1)}+F^{(2)}+F^{(3)}+F^{(4)},
\end{aligned}
\end{equation}
where
\begin{multline*}
S_9=
\frac1D
\sum_{D\leq d<2D}
\int_{[0,1]^{m+4}}
\\
\prod_{\varepsilon,\varepsilon_1,\ldots,\varepsilon_{m+3}\in\{0,1\}}
\e\bigl(
(-1)^{\varepsilon+\varepsilon_1+\cdots+\varepsilon_{m+3}}
  \vartheta
  g_\nu\bigl(x+\varepsilon_1 x_1+\cdots+\varepsilon_{m+3} x_{m+3}
\bigr)\bigr)\\
\,\mathrm d(x,x_1,\ldots,x_{m+3}).
\end{multline*}
We see that the margin $\sigma$ appears in the additional factor $\golden^{2\sigma}$.
This will not be a problem as this margin may be chosen very small, depending on the quality of the Gowers norm.
More precisely, the quantity $\sigma$ appears in two error terms:
Proposition~\ref{prp_twodim_estimate} yields among others an error $\golden^{-\sigma}$; the above computation combined with the Gowers norm estimate will give us $\golden^{2\sigma}N^{-c\lVert\vartheta\rVert^2}$.
Balancing these two terms leads to the choice $\golden^\sigma\asymp N^{c\lVert\vartheta\rVert^2/3}$, and the error $N^{-c\lVert\vartheta\rVert^2/3}$ remains.

Equation~\eqref{eqn_S0_S8}
holds for all integers $D,N\geq 1$ and all real numbers $\vartheta,\xi$, and all
$B,R,m,\lambda,\mu,\sigma$ satisfying the mild constraints
\begin{equation}\label{eqn_mild}
\begin{aligned}
&B,R,m,\lambda,\mu\geq 1,\\
&\sigma\geq 5,\\
&\lambda-(m+3)\mu-\sigma\geq 2.
\end{aligned}
\end{equation}
In order to obtain a significant estimate, we use the hypothesis
\[N^{1/3}\leq D\leq N^{\rho_2}\]
where $\rho_2\geq 1/3$ is a real number
(stated at the very beginning of Section~\ref{sec_short}).
So far, we have not used this hypothesis.
We proceed similarly to~\cite[page~2581]{S2020}.
First, we determine the number $m$ of times we have to apply van der Corput's inequality. This is related to the ratio $\log D/\log N$, and the quantity of interest is bounded by $C\rho_2$ for an absolute constant $C$.
More precisely, choose the integers $L,k\geq 1$ in such a way that
\[2^{L-1}\leq D<2^L\quad\mbox{and}\quad N^{k-1}\leq 2^L<N^k\]
and set
\[    m=10k.    \]
The intervals we cut out are of length $\mu$ for  $1\leq j<m$, of length $4\mu$ for $j=m$, and the remaining interval $[2,\nu)$ should have length $\asymp \mu/16$. We will also see in a moment that we have to introduce an additional window of size $\asymp \mu/4$ in order to account for the variable $R$ (which was introduced by the first application of van der Corput's inequality, before~\eqref{eqn_S0_squared}).
According to this, we choose
\[
\mu=\left\lfloor 
\frac{L}{m+3+5/16}
\right\rfloor,
\]
and we set
\[R=\lfloor \gamma^{\mu/4}\rfloor,\quad
B=\lfloor2^{9\mu}\rfloor.
\]
Choose $\sigma$ in such a way that
\[
2^\sigma
\asymp
32\min\left(\left\lfloor 2^{\mu/32},N^{c\lVert\vartheta\rVert^2/3}\right\rfloor\right).
\]
(The factor $32$ takes care of the condition $\sigma\geq 5$.)
Furthermore, choose $\lambda$ in such a way that
\[   RD/F_\lambda\asymp\golden^{-\sigma}.   \]

For all $N$ larger than some constant depending only on $\rho_2$ we see that the constraints~\eqref{eqn_mild} are satisfied.
Carefully inspecting all of the error terms appearing in~\eqref{eqn_S0_S8}, and inserting the Gowers norm estimate (Theorem~\ref{cor_gowers_estimate}), we obtain
\[
\bigl\lvert S_0(N,D,\vartheta,\xi)\bigr\rvert^{2^{m+4}}\leq C\left(\golden^{-\sigma}+\golden^{2\sigma}N^{-c\lVert \vartheta\rVert^2}\right)\]
for some $c,C$ depending only on $\rho_2$.
Since $\mu\asymp \log N$ with an implied constant depending only on $\rho_2$, we obtain, after taking $2^{m+4}$th roots,
\[
S_0(N,D,\vartheta,\xi)\ll N^{-c'\lVert \vartheta\rVert^2}\]
for some positive $c'$ depending only on $\rho_2$.

Dyadic composition in $D$ and extending the summation over $n$ by means of Lemma~\ref{lem_vinogradov} contributes a factor $(\log^+N)^2$.
Note that for the application of Lemma~\ref{lem_vinogradov} we introduced the variable $\xi$, which is used only at this point.
The proof of Theorem~\ref{thm_lod} is complete.







\chapter{Type II Sums}\label{ch_type2}

\section{Statement of the result}
This chapter is devoted to proving the estimate for sums of type~\textrm{II} given in Theorem~\ref{th_sum_2}.
We note that the estimate~\eqref{eqn_type2_condition} required in the theorem was considered in Theorem~\ref{cor_gowers_estimate}.

In~\cite{DMS2018} we proved the following statement, with $\vartheta=1/2$:
for any $0 < p <q$, we have
\begin{align}\label{eqn_DMS2018}
	\lim_{M \to \infty} \frac{1}{M} \sum_{m\leq M} \e\bigl(\vartheta \sz(mp)\bigr) \e\bigl(-\vartheta \sz(mq)\bigr) = 0.
\end{align}
It appears reasonable that the method of proof used in that paper can be adapted to prove~\eqref{eqn_DMS2018} for all $\vartheta\in\R\setminus\Z$.
This seems to indicate that we already have the tools necessary for dealing with our sums of type~\textrm{II}.
However, the big obstacle is to have a good quantitative control for such sums.
In particular, in~\cite{DMS2018} we were only able to show cancellation in~\eqref{eqn_DMS2018} for $M \gg \golden^{\max(p,q)}$.

This means that we need another approach for proving the prime number theorem for $\sz$.
In particular, one can achieve much better estimates, when we consider ``typical'' $p$ and $q$, for which we expect a cancellation already for $M = O(\max(p,q))$.
We note that questions of this kind (for the base-$b$ expansion, $b\geq 2$ an integer) were considered in the paper~\cite{DT2005} by Dartyge and Tenenbaum.
Thus, taking an average over $p$ is a crucial aspect for our result and we need a new approach.

\def\thsumtwo{\ref{th_sum_2}}
\section{Proof of Theorem~\thsumtwo}

This section is devoted to the proof of Theorem~\ref{th_sum_2}. 
The main ingredients are a order-$3$ Gowers norm estimate for $f(n) = \e(\vartheta \sz(n))$ and a carry propagation lemma.
Moreover, we will use the asymptotic independence of $(mp\golden \bmod \Z)_{m\in \N}$ and $(mq \golden \bmod \Z)_{m\in \N}$ when we take an average over $p$ and $q$.

We start by considering the first line of equation~\eqref{eq_sum_2} and split the summation over $p$ into at most $\log(N/M)$ dyadic intervals and apply Lemma~\ref{lem_vinogradov} to extend the range of summation over $m$ at the cost of a factor $\log(M)$. Thus,
\begin{align}\label{eq_SII_to_S0}
	S_{II}(N, U, V) \leq \sqrt{N} (\log N)^3 \max_{\substack{U\leq M \leq N/V\\ V \leq q \leq N/M\\ V < M_1 \leq N/M}} \bigl(\log(N) \log(M) S_0(M_1, M, q, \xi)\bigr)^{1/2},
\end{align}
where
\begin{align*}
	S_0 \eqdef \sum_{M_1 < p\leq 2M_1} \left\lvert\sum_{M<m\leq 2M} f(pm) \overline{f(qm)} \e(\xi m)\right\rvert.
\end{align*}
As a first goal, we want to reduce the number of digits that contribute in the sum above.
Therefore, we apply the inequalities of Cauchy--Schwarz and van der Corput (Lemma~\ref{lem_vdc}); for any $1\leq R \leq M$, we obtain
\begin{align*}
	S_0^2 &\leq M_1 \frac{M+R-1}{R} \sum_{M_1 < p \leq 2 M_1} \sum_{0 \leq\lvert r\rvert< R} \left(1-\frac{\lvert r\rvert}{R}\right)\\
		&\qquad \times\sum_{M < m, m+r \leq 2M} \Delta(f;pr)(pm) \Delta(\overline{f}; qr)(qm)\\
		&\ll \frac{M_1M}{R} \sum_{M_1 < p \leq 2 M_1} \sum_{0 < \lvert r\rvert < R} \rb{1-\frac{\lvert r\rvert}{R}}\sum_{M < m \leq 2M} \Delta(f;pr)(pm) \Delta\bigl(\overline{f}; qr\bigr)(qm)\\
			&\qquad +O\bigl(M_1^2 M (M + R^2)/R\bigr),
\end{align*}
where we excluded the contribution of $r = 0$ and extended the summation over $m$ in the last step.
Now we use the carry propagation lemma (Lemma~\ref{lem_z_carry_lemma}) in order to replace $f$ by $f_{\lambda}$,
where $f_{\lambda}(n) = \e(\vartheta \sz_{\lambda}(n))$. 
This gives
\begin{align*}
\hspace{2em}&\hspace{-2em}
	\frac{S_0^2}{(M M_1)^2} \ll \frac{1}{M_1} \sum_{M_1 < p \leq 2 M_1} \frac{1}{R}\sum_{0< \lvert r\rvert < R} \left(1-\frac{\lvert r\rvert}{R}\right) \frac{1}{M}
\\&\times \sum_{M < m \leq 2M} \Delta (f_{\lambda}; pr)(pm) \Delta\bigl(\overline{f_{\lambda}}, qr\bigr)(qm)+ O\left(\frac{M_1 R}{F_{\lambda}}\right) + O\left(\frac{1}{R} + \frac{R}{M}\right)\\
		&\ll \frac{1}{R} \sum_{0 < \lvert r\rvert < R} \frac{1}{M}\sum_{M < m \leq 2M} \left\lvert\frac{1}{M_1} \sum_{M_1 < p \leq 2 M_1} \Delta (f_{\lambda}; pr)(pm)\right\rvert\\
		&\qquad + O\left(\frac{M_1 R}{F_{\lambda}} + \frac{1}{R} + \frac{R}{M}\right).
\end{align*}
Now we want to reorganize the sums, in order to exploit the fact that the sum over $m$ is very long.
We use the Cauchy--Schwarz inequality once again, obtaining
\begin{align*}
\hspace{2em}&\hspace{-2em}
	\frac{S_0^4}{(M M_1)^4} \ll \frac{1}{R} \sum_{0 < \lvert r\rvert < R} \frac{1}{M}\sum_{M < m \leq 2M} \left\lvert\frac{1}{M_1}\sum_{M_1< p \leq 2 M_1} \Delta (f_{\lambda}; -pr)(pm)\right\rvert^2\\
		&+ O\left(\frac{(M_1R)^2}{F_{\lambda}^2} + \frac{1}{R^2} + \frac{R^2}{M^2}\right)\\
	&= \frac{1}{R} \sum_{0 < \lvert r\rvert < R} \frac{1}{M} \sum_{M < m \leq 2M} \frac{1}{M_1^2} \sum_{M_1< p,q\leq M_1} \Delta (f_{\lambda}; -pr)(pm) \Delta\bigl(\overline{f_{\lambda}}; -qr\bigr)(qm)\\
		&+ O\left(\frac{(M_1R)^2}{F_{\lambda}^2} + \frac{1}{R^2} + \frac{R^2}{M^2}\right ).
\end{align*}

Thus, we are interested in estimating
\begin{align*}
  S_1 \eqdef \frac{1}{R} \sum_{0 < \lvert r\rvert < R} \frac{1}{M_1^2} \sum_{M_1 < p,q \leq 2M_1} \frac{1}{M}\left\lvert\sum_{M<m\leq 2M} \Delta\bigl(f_{\lambda}; -rp\bigr) (mp)\Delta \bigl(\overline{f_{\lambda}}, -rq\bigr)(mq)\right\rvert.
\end{align*}

We consider the innermost sum and replace once more $f_{\lambda}(n)$ by $\e(\vartheta g_{\lambda}(n\golden))$ 
(see~\eqref{eq_def_g_L}).
This gives
\begin{align*}
	S_2 &\eqdef \frac{1}{M}\sum_{M < m \leq 2M} \Delta \bigl(f_{\lambda}; rp\bigr)(mp) \Delta \bigl(\overline{f_{\lambda}}; rq\bigr)(mq)\\
		 &= \frac{1}{M} \sum_{M < m \leq 2M} \Delta \bigl(\e(\vartheta g_{\lambda}) ;-rp \golden\bigr) (p m \golden) \Delta \bigl(\e(-\vartheta g_{\lambda}); rq\golden\bigr)(q m \golden).
\end{align*}
Now we aim to replace the sum over $m$ by an integral.
Therefore, we recall that $g_{\lambda}$ is a $1$-periodic step function with at most $F_{\lambda}$ points of discontinuity in any interval of length $1$.
Moreover, the product of two step functions $f_1, f_2$  is again a step function for which the number of steps is bounded by the sum of the number of steps of $f_1$ and $f_2$.
Thus, $\Delta(\e(\vartheta g_{\lambda}) ; rp\golden)$ is again a step function with at most $2F_{\lambda}$ steps.
Also, if $f(x)$ is a step function with $k$ steps, then $f(l x)$ is a step-function with at most $l k$ steps.
This shows in total that
\begin{align*}
	V_0^1\bigl(\Delta \bigl(\e(\vartheta g_{\lambda}) ; rp \golden\bigr) (p\cdot\bl) \Delta \bigl(\e(-\vartheta g_{\lambda}); rq\golden\bigr)(q\cdot\bl)\bigr) \leq 2(p+q) F_{\lambda}.
\end{align*}

We have the simple identity $\{nx\} = \{n \{x\}\}$, which allows us to use the Koksma--Hlawka inequality (Theorem~\ref{th_koksma}),
\begin{align*}
	S_2 &= \int_{[0,1]} \Delta \bigl(\e(\vartheta g_{\lambda}) ; rp \golden\bigr) (p x) \Delta\bigl(\e(-\vartheta g_{\lambda}); rq\golden\bigr)(q x)\,\mathrm dx + O\bigl(2(p+q)F_{\lambda}  D_M(\golden)\bigr)\\
		&= \underbrace{\int_{[0,1]} \Delta \bigl(\e(\vartheta g_{\lambda}) ; rp \golden\bigr) (p x) \Delta \bigl(\e(-\vartheta g_{\lambda}); rq\golden\bigr)(q x)\,\mathrm dx}_{\eqqcolon S_2'} + O\bigl(M_1 F_{\lambda} D_M(\golden)\bigr).
\end{align*}
The error term is clearly negligible if $M$ is large enough compared to $M_1$.

Next we use the fact that $\Delta(\e(-\vartheta g_{\lambda}); rq\golden)(x)$ is a step function having $O(F_{\lambda})$ jumps in $[0,1)$.
We denote the set of intervals on which it is constant by $\mathcal{I}$ and the value of $\Delta(\e(-\vartheta g_{\lambda}); rq\golden)$ on an interval $I \in \mathcal{I}$ by $v_I$, while we denote its length by $\ell_I$ (where $\ell_I \ll F_{\lambda}^{-1}$).
We obtain
\begin{align*}
	S_2' &= \sum_{I \in \mathcal{I}} v_I \int_{[0,1]} \ind_{[qx \in I]} \Delta \bigl(\e(\vartheta g_{\lambda}) ; rp\golden\bigr)(px)\,\mathrm dx.
\end{align*}
We use Theorem~\ref{th_vaaler} (by Vaaler) in order to approximate $\ind_{[px \in I]}$ using exponential sums.
This gives for any $H \geq 1$,
\begin{align*}
	\lvert S_2'\rvert &\leq \left\lvert\sum_{I \in \mathcal{I}} v_I \int_{[0,1]} A_{I, H} (qx) \Delta \bigl(\e(\vartheta g_{\lambda}) ; rp\golden\bigr)(px)\,\mathrm dx\right\rvert\\
			&\qquad  + \left\lvert\sum_{I \in \mathcal{I}} v_I \int_{[0,1]} \left(\ind_{[qx \in I]} - A_{I, H}(qx)\right) \Delta \bigl(\e(\vartheta g_{\lambda}) ; rp\golden\bigr)(px)\,\mathrm dx\right\rvert \\
		&\leq \sum_{I \in \mathcal{I}} \left\lvert\int_{[0,1]} \sum_{\lvert h\rvert \leq H} a_{h}(I, H) \e(hqx) \Delta \bigl(\e(\vartheta g_{\lambda}) ; rp\golden\bigr)(px) \,\mathrm dx\right\rvert\\
			&\qquad + \sum_{I \in \mathcal{I}} \int_{[0,1]} \left\lvert\ind_{[qx \in I]} - A_{I, H}(qx)\right\rvert\,\mathrm dx\\
		&\leq \sum_{I \in \mathcal{I}} \sum_{\lvert h\rvert \leq H} \left\lvert a_h(I, H) \int_{[0,1]} \e(hqx) \Delta \bigl(\e(\vartheta g_{\lambda}) ; rp\golden\bigr)(px)\,\mathrm dx\right\rvert + \sum_{I \in \mathcal{I}} \int_{[0,1]} B_{I, H} (qx)\,\mathrm dx\\
		&\ll \sum_{I \in \mathcal{I}} \sum_{\lvert h\rvert \leq H} \min\left(1/F_{\lambda}, 1/\lvert h\rvert\right)\left\lvert\int_{[0,1]} \e(hqx) \Delta (\e(\vartheta g_{\lambda}) ; rp\golden)(px) \,\mathrm dx\right\rvert\\
			&\qquad + \sum_{I \in \mathcal{I}} \sum_{\lvert h\rvert \leq H} \int_{[0,1]} b_{h}(I, H) \e(hqx)\,\mathrm dx.
\end{align*}
The last term is $0$ unless $h = 0$.
In this case we have $O(F_{\lambda})$ contributions of size $\frac{1}{H}$ (since $\lvert b_h(\alpha, H)\rvert \leq \frac{1}{H+1}$). Thus, we need $F_{\lambda} = o(H)$ in order to obtain a non-trivial estimate.

The remaining integral can be rewritten as
\begin{align*}
\hspace{4em}&\hspace{-4em}
	\int_{[0,1]} \e(hqx) \Delta \bigl(\e(\vartheta g_{\lambda}) ; rp\golden\bigr)(px)\,\mathrm dx = \frac{1}{p}\int_{[0,p]} \e\rb{\frac{hqy}{p}} \Delta \bigl(\e(\vartheta g_{\lambda}) ; rp\golden\bigr)(y)\,\mathrm dy\\
		&= \frac{1}{p} \sum_{n=0}^{p-1} \int_{[n, n+1]}\e\rb{\frac{hqy}{p}} \Delta \bigl(\e(\vartheta g_{\lambda}) ; rp\golden\bigr)(y)\,\mathrm dy\\
		&= \frac{1}{p} \sum_{n=0}^{p-1} \int_{[0,1]} \e\rb{\frac{hq(y+n)}{p}} \Delta \bigl(\e(\vartheta g_{\lambda}) ; rp\golden\bigr)(y)\,\mathrm dy\\
		&= \frac{1}{p} \sum_{n=0}^{p-1} \e\rb{\frac{hqn}{p}} \int_{[0,1]} \e\rb{\frac{hqy}{p}} \Delta \bigl(\e(\vartheta g_{\lambda}) ; rp\golden\bigr)(y)\,\mathrm dy .
\end{align*}
This shows that
\begin{align*}
\hspace{4em}&\hspace{-4em}
	\abs{\int_{[0,1]} \e(hqx) \Delta \bigl(\e(\vartheta g_{\lambda}) ; rp\golden\bigr)(px)\,\mathrm dx}
\\&\ll \frac{1}{p} \min\rb{p, \abs{\sin\left(\pi\frac{hq}{p}\right)}^{-1}} \abs{ \int_{[0,1]} \e\rb{\frac{hqy}{p}} \Delta \bigl(\e(\vartheta g_{\lambda}) ; rp\golden\bigr)(y)\,\mathrm dy},
\end{align*}
and in total
\begin{align*}
	\abs{S_2} &\ll F_{\lambda} \sum_{\abs{h} \leq H} \min\rb{\frac{1}{F_{\lambda}}, \frac{1}{\abs{h}}} \frac{1}{p} \min\rb{p, \abs{\sin\left(\pi\frac{hq}{p}\right)}^{-1}}\\&\times\max_{\alpha \in \mathbb{R}} \abs{\int_0^1 \e\rb{\alpha y} \Delta \bigl(\e(\vartheta g_{\lambda}) ; rp\golden\bigr)(y)\,\mathrm dy}
+ O\rb{\frac{F_{\lambda}}{H} + M_1 F_{\lambda} D_M(\golden)}.
\end{align*}
The remaining integral amounts to a statement on Gowers uniformity;
the term $\min\bigl(p, \lvert\sin(\pi hq/p)\rvert^{-1}\bigr)$
concerns the  independence of $mp\golden$ and $mq\golden$.
At this point, we have been successful in separating these two factors.
This implies
\begin{align*}
	S_1 &\ll  \frac{1}{R} \sum_{0<\abs{r}<R} \frac{1}{M_1} \sum_{M_1 < p \leq 2 M_1} \max_{\alpha \in \mathbb{R}} \abs{\int_0^1 \e(\alpha y)  \Delta \bigl(\e(\vartheta g_{\lambda}) ; -rp\golden\bigr)(y)\,\mathrm dy}\\
			&\times \rb{\frac{1}{M_1} \sum_{M_1 < q \leq 2 M_1} \sum_{\abs{h} \leq H} \min\rb{1, \frac{F_{\lambda}}{\abs{h}}} \frac{1}{p} \min\rb{p, \abs{\sin\left(\pi\frac{hq}{p}\right)}^{-1}}}\\
		&\qquad + O\rb{\frac{F_{\lambda}}{H} + M_1 F_{\lambda} D_M(\golden)}.
\end{align*}
Next we use the Cauchy--Schwarz inequality for the summation over $p$ to use independent estimates for the integral and the remaining factor.
This gives us the estimate
\begin{align*}
\hspace{1em}&\hspace{-1em}
	S_1 \ll  \frac{1}{R} \sum_{0<\abs{r}<R} \frac{1}{M_1} \left(\sum_{M_1 < p \leq 2 M_1} \max_{\alpha \in \mathbb{R}} \abs{\int_0^1 \e(\alpha y)  \Delta \bigl(\e(\vartheta g_{\lambda}) ; -rp\golden\bigr)(y)\,\mathrm dy}^2\right)^{1/2}\\
			&\times\left(\sum_{M_1 < p \leq 2 M_1} \left(\frac{1}{M_1} \sum_{M_1 < q \leq 2 M_1} \sum_{\abs{h} \leq H} \min\rb{1, \frac{F_{\lambda}}{\abs{h}}}
\right.\right.\\&\left.\left.\hspace{6em}\times
 \frac{1}{p} \min\rb{p, \abs{\sin\left(\pi\frac{hq}{p}\right)}^{-1}}\right)^{2}\right)^{1/2}+ O\rb{\frac{F_{\lambda}}{H} + M_1 F_{\lambda} D_M(\golden)}.
\end{align*}

We first aim to estimate the second factor.
Therefore, we split the summation over $h$ into intervals of length $p$, which yields
\begin{align*}
	S_3 &\eqdef \frac{1}{M_1} \sum_{M_1 < q \leq 2 M_1} \sum_{\abs{h} \leq H} \min\rb{1, \frac{F_{\lambda}}{\abs{h}}} \frac{1}{p} \min\rb{p, \abs{\sin\left(\pi\frac{hq}{p}\right)}^{-1}}  \\
		&\ll \sum_{0 \leq j\leq H/p} \min\rb{1, \frac{F_{\lambda}}{pj}} \frac{1}{p} \sum_{0 \leq h' < p} \frac{1}{M_1} \sum_{M_1 < q \leq 2 M_1} \min\rb{p, \abs{\sin\left(\pi\frac{h'q}{p}\right)}^{-1}}.
\end{align*}

We use Lemma~\ref{le_geometric_series} to estimate the summation over $h'$ and $q$, which gives
\begin{align*}
	S_3 &\ll \sum_{0 \leq j \leq H/p} \min\rb{1, \frac{F_{\lambda}}{pj}} \frac{1}{p} \bigl(\tau(p) p + p \log p\bigr)\\
		&\ll \frac{F_{\lambda}}{M_1} \rb{1 + \log\rb{\frac{H}{F_{\lambda}}}} \bigl(\tau(p) + \log(p)\bigr).
\end{align*}
By using again~\eqref{eqn_tau_squared} --- the estimate
$\sum_{n\leq N} \tau(n)^2 \sim N/\pi^2 \log(N)^3$ --- we obtain
\begin{align*}
	\rb{\sum_{M_1 < p \leq 2M_1} S_3^{2}}^{1/2} &\ll \frac{F_{\lambda}}{M_1} \rb{1 + \log\rb{\frac{H}{F_{\lambda}}}} \bigl(M_1 \bigl(\log(M_1)^{3} + \log(M_1)^{2}\bigr)\bigr)^{1/2}\\
			&\ll \frac{F_{\lambda}}{M_1^{1/2}} \rb{1 + \log\rb{\frac{H}{F_{\lambda}}}} \log(M_1)^{3/2}.
\end{align*}

It only remains to estimate
\begin{align*}
	S_4 := \frac{1}{R} \sum_{0<\abs{r}<R} \frac{1}{M_1} \rb{\sum_{M_1 < p \leq 2 M_1} \max_{\alpha \in \mathbb{R}} \abs{\int_0^1 \e(\alpha y)  \Delta (\e(\vartheta g_{\lambda}) ; rp\golden)(y) dy}^2}^{1/2}.
\end{align*}
We use the Cauchy--Schwarz inequality for the summation over $r$ to find
\begin{align*}
	S_4^2 &\ll \frac{1}{R} \sum_{0<\abs{r} < R} \frac{1}{M_1^2} \sum_{M_1 < p \leq 2 M_1} \max_{\alpha \in \mathbb{R}} \int_0^1\int_0^1  \Delta (\e(\vartheta g_{\lambda}) ; rp\golden)(x) \overline{ \Delta (\e(\vartheta g_{\lambda}) ; rp\golden)(y)} 
\\&\hspace{20em}\times\e\bigl(\alpha(x-y)\bigr)\,\mathrm dy\,\mathrm dx\\
		&= \frac{1}{R} \sum_{0<\abs{r} < R} \frac{1}{M_1^2} \sum_{M_1 < p \leq 2 M_1} \max_{\alpha \in \mathbb{R}} \int_0^1 \int_0^1  \Delta (\e(\vartheta g_{\lambda}) ; rp\golden)(x) \overline{ \Delta (\e(\vartheta g_{\lambda}) ; rp\golden)(x+z)}
\\&\hspace{20em}\times\e(-\alpha z)\,\mathrm dz\,\mathrm dx\\
		&\leq \frac{1}{R} \sum_{0<\abs{r} < R} \frac{1}{M_1^2} \sum_{M_1 < p \leq 2 M_1} \max_{\alpha \in \mathbb{R}} \int_0^1 \abs{\e(\alpha z)} \abs{\int_0^1 \Delta (\e(\vartheta g_{\lambda}) ; rp\golden, z)(x)\,\mathrm dx}\mathrm dz\\
		&= \frac{1}{R} \sum_{0<\abs{r} < R} \frac{1}{M_1^2} \sum_{M_1 < p \leq 2 M_1}\int_0^1 \abs{\int_0^1  \Delta (\e(\vartheta g_{\lambda}) ; rp\golden, z)(x)\,\mathrm dx}\mathrm dz\\
		&= \frac{1}{M_1} \int_0^1 \frac{1}{R} \sum_{0<\abs{r} < R} \frac{1}{M_1} \sum_{M_1 < p \leq 2 M_1}\abs{\int_0^1 \Delta\bigl(\Delta(\e(\vartheta g_{\lambda}) ;z); rp\golden\bigr)(x)\,\mathrm dx}\mathrm dz.
\end{align*}
Thus, we were able to remove the term $\e(\alpha y)$ at the cost of introducing one more difference operator.
 By Lemma~\ref{le_gowers_1} we see that $\norm{\Delta(\e(\vartheta g_{\lambda}) ; z)}_{U^2(\T)}^{2^2} \leq \norm{\e(\vartheta g_{\lambda}) }_{U^3(\T)}^{2^2}$ unless $z$ is in a set of measure $\norm{\e(\vartheta g_{\lambda}) }_{U^3(\T)}^{2^2}$, which we call $\mathcal{M}$. This shows
\begin{align*}
	M_1 \cdot S_4^2 &\leq \norm{\e(\vartheta g_{\lambda}) }_{U^3(\T)}^{2^2}
\\&+ \int_{[0,1] \setminus \mathcal{M}} \frac{1}{R} \sum_{0<\abs{r} < R} \frac{1}{M_1} \sum_{M_1 < p \leq 2 M_1}\abs{\int_0^1 \Delta\bigl(\Delta(\e(\vartheta g_{\lambda}) ;z); rp\golden\bigr)(x)\,\mathrm dx}\mathrm dz.
\end{align*}

Now we want to apply Corollary~\ref{cor_1}, where $F(y) = \int_0^1 \Delta(\Delta(\e(\vartheta g_{\lambda}) ;z); y)(x) dx$ and $x_n = rn \golden$.
We recall that $\Delta(\e(\vartheta g_{\lambda}) ;z)$ is a $1$-bounded step-function having
at most $O(F_{\lambda})$ jumps in $[0,1)$.
Thus, $F$ is continuous and piecewise linear with gradient at most $O(F_{\lambda})$. 
In particular, $F$ is Lipschitz continuous with Lipschitz-constant $O(F_{\lambda})$.
We finally see by the Cauchy--Schwarz inequality and Lemma~\ref{lem_uniform},
\begin{align*}
	\abs{\int_0^1 F(y) dy}^2 &= \abs{\int_0^1 \int_0^1 \Delta\bigl(\Delta(\e(\vartheta g_{\lambda}) ;z);y\bigr)(x)\,\mathrm dx\,\mathrm dy}^2\\
		&\leq \int_0^1 \abs{\int_0^1 \Delta\bigl(\Delta(\e(\vartheta g_{\lambda}) ;z);y\bigr)(x)\,\mathrm dx}^2\!\!\mathrm dy \\
		&= \norm{\Delta(\e(\vartheta g_{\lambda});z)}_{U^{2}(\T)}^{2^2}.
\end{align*}
This means that we can apply Corollary~\ref{cor_1} with
\[L = O(F_{\lambda}),\quad \alpha = \norm{\Delta(\e(\vartheta g_{\lambda}) ;z)}_{U^{2}(\T)}^{2} \leq \norm{\e(\vartheta g_{\lambda}) }_{U^{3}(\T)}^{2}\]
 and it only remains to estimate the discrepancy of the sequence $(rn \golden)_{M_1<n\leq 2 M_1}$.

We know by Lemma~\ref{le_bounded_quotients} that the partial quotients of $r\golden$ are bounded by $O(r)$. This shows together with Theorem~\ref{th_bounded_quotients} 
\begin{align*}
	D_{M_1}(r\golden) &\ll \frac{\abs{r} \log^+(M_1)}{M_1}.
\end{align*}

In total, we obtain
\begin{align*}
	M_1 \cdot S_4^2 &\ll \norm{\e(\vartheta g_{\lambda}) }_{U^3(\T)}^{2^2} + \frac{1}{R} \sum_{0<\abs{r}<R} \norm{\e(\vartheta g_{\lambda}) }_{U^3(\T)}^{2}\\
		&\qquad + \frac{1}{R} \sum_{0<\abs{r}<R} \rb{\norm{\e(\vartheta g_{\lambda}) }_{U^3(\T)}^{2} F_{\lambda} \frac{\abs{r} \log^+(M_1)}{M_1}}^{1/3}\\
		&\ll \rb{\norm{\e(\vartheta g_{\lambda}) }_{U^3(\T)}^{2} F_{\lambda} \frac{R \log^+(M_1)}{M_1}}^{1/3},
\end{align*}
where we assumed that $RF_{\lambda}/M_1 \geq 1$ for the last inequality.

This in turn shows
\begin{align*}
	S_1 &\ll \frac{F_{\lambda}}{M_1^{1/2}} \log^{+}\rb{\frac{H}{F_{\lambda}}} \log(M_1)^{3/2} M_1^{-1/2} \rb{\norm{\e(\vartheta g_{\lambda}) }_{U^3(\T)}^{2} F_{\lambda} \frac{R \log^+(M_1)}{M_1}}^{1/6}\\
		&\qquad + \frac{F_{\lambda}}{H} + M_1 F_{\lambda} \frac{\log^+(M)}{M}.
\end{align*}
Choosing $H = M/M_1 + O(1)$ gives
\begin{align*}
	\frac{S_0^4}{(M M_1)^4} &\ll S_1 + \frac{(M_1 R)^2}{F_{\lambda}^2} + \frac{1}{R^2} + \frac{R^2}{M^2}\\
		&\ll \frac{F_{\lambda}}{M_1} \log^+(M) \log(M_1)^{10/6} \rb{\norm{\e(\vartheta g_{\lambda}) }_{U^3(\T)}^{2} F_{\lambda} \frac{R}{M_1}}^{1/6}\\
			&\qquad + M_1 F_{\lambda} \frac{\log^+(M)}{M} + \frac{(M_1 R)^2}{F_{\lambda}^2} + \frac{1}{R^2} + \frac{R^2}{M^2}.
\end{align*}

%
We recall that there exists $c>0$ such that $\norm{\e(\vartheta g_{\lambda}) }_{U^3(\T)} \ll F_{\lambda}^{-c \norm{\vartheta}^2}$.
We choose $R = M_1^{\rho} + O(1)$ and $\lambda = \log_{\golden}(M_1) (1+ 2\rho) + O(1)$ which gives $F_{\lambda} \asymp M_1^{1+ 2\rho}$. Thus, we have
\begin{align*}
	\frac{S_0^4}{(M M_1)^4} &\ll M_1^{2 \rho} \log^+(M) \bigl(\log^{+}M_1\bigr)^{10/6} \rb{M_1^{-c\norm{\vartheta}^2 - 2c \norm{\vartheta}^2\rho+ 3\rho}}^{1/6}\\
		&\qquad + \frac{M_1^{2+ 2\rho}}{M} \log^+(M) + M_1^{-2\rho} + \frac{M_1^{2\rho}}{M^2}
.
\end{align*}
Next we choose $\rho = c\norm{\vartheta}^2/27$ to find (we neglect the term $-2c\norm{\vartheta}^2\rho$)
\begin{align*}
	\frac{S_0^4}{(M M_1)^4} &\ll M_1^{-2 c \norm{\vartheta}^2 /27} \log^+(M) \bigl(\log^{+}M_1\bigr)^{10/6}\\
		&\qquad + \frac{M_1^{2+ 2c\norm{\vartheta}^2/27}}{M} \log^+(M).
\end{align*}
Finally, this gives
\begin{align*}
	S_0 \ll M M_1 \log^+(M)\log^+(M_1) \rb{M_1^{-c \norm{\vartheta}^2/54} + \frac{M_1^{1/2 + c \norm{\vartheta}^2/54}}{M^{1/4}}}.
\end{align*}

This upper bound is monotonously increasing in $M_1$ and we recall that $V \leq M_1 \leq N/M$.
Thus, we find by~\eqref{eq_SII_to_S0}
\begin{align*}
	S_{II} (M) 
		&\ll N \bigl(\log^+M\bigr) \bigl(\log^+N\bigr)^4 \rb{(N/M)^{-c\norm{\vartheta}^2/54} + \frac{(N/M)^{1/2 + c \norm{\vartheta}^2/54}}{M^{1/4}}}^{1/2}.
\end{align*}

As $U \leq M < N/V$, this finishes the proof of Theorem~\ref{th_sum_2}.


\chapter{Local limit theorem}\label{chap_local}

The goal of this chapter is to prove the asymptotic relation~\eqref{eqPro2} for
\[
\sum_{p\le x} \e\bigl(\vartheta \sz(p)\bigr),
\]
which is stated in Proposition~\ref{Promain2}.

The proof is in principle close to the proof given in~\cite{DMR2009}.
However, instead of approximating the sum-of-digits function by a sum of independent random variables,
we approximate it by a sum over a suitable Markov chain.

\section{Approximation of $\sz(p)$ by a sum over a Markov chain}\label{sec:8.1}

We start by arguing why the digits of the Zeckendorf expansion behave
as a Markov process (compare also with \cite{DS02}, where more generally digital expansions related
to sequences $G_k$ that satisfy the recurrence $G_k = a G_{k-1} + G_{k-2}$ for some integer $a\ge 1$ 
are discussed). 

Consider the set $A_k$ of nonnegative integers $n$ for which the Zeckendorf expansions of $n$
has length $L = k$, that is, $\delta_k(n) = 1$ and $\delta_{\ijkl}(n) = 0$ for all $\ijkl > k$.
It is easy to see that $\lvert A_k\rvert = F_{k-1}$.
For $2\le \ijkl \le k$ let $A_{k,\ijkl}(b)$, where $b\in \{0,1\}$, denote the set of $n\in A_k$
satisfying $\delta_{\ijkl}(n) = b$.
If $\delta_{\ijkl}(n) = 0$, then the two blocks of digits $\delta_2(n), \ldots, \delta_{\ijkl-1}(n)$
and $\delta_{\ijkl+1}(n), \ldots, \delta_{k-2}(n)$
are only restricted by the condition that no consecutive $0$s appear.
Hence, $\bigl\lvert A_{k,\ijkl}(0)\bigr\rvert = F_{\ijkl} F_{k-\ijkl}$.
Similarly, if $\delta_{\ijkl}(n) = 1$ then $\delta_{\ijkl-1}(n) = \delta_{\ijkl+1}(n) = 0$ and, thus,
the two smaller blocks of digits $\delta_2(n), \ldots, \delta_{\ijkl-2}(n)$ and
$\delta_{\ijkl+2}(n), \ldots, \delta_{k-2}(n)$ are only restricted by the same condition.
Consequently $\lvert A_{k,\ijkl}(1)\rvert = F_{\ijkl-1} F_{k-\ijkl-1}$.
Thus, the \emph{probabilities} that the $\ijkl$-th digit equals $0$ and $1$ respectively are given by
\[
\frac{\lvert A_{k,\ijkl}(0)\rvert}{\lvert A_k\rvert} = \frac{F_\ijkl F_{k-\ijkl}}{F_{k-1}}\quad\mbox{and}\quad 
\frac{\lvert A_{k,\ijkl}(1)\rvert}{\lvert A_k\rvert} = \frac{F_{\ijkl-1} F_{k-\ijkl-1}}{F_{k-1}}
\]
respectively. If $\ijkl\to\infty$ and $k-\ijkl\to\infty$ these probabilities converge to 
\[
\frac{F_\ijkl F_{k-\ijkl}}{F_{k-1}} \to \frac{\golden^2}{\golden^2+1} \quad\mbox{and}\quad 
\frac{F_{\ijkl-1} F_{k-\ijkl-1}}{F_{k-1}} \to \frac{1}{\golden^2+1}
\]
respectively. 

Next we consider the subsets $A_{k,\ijkl}(b,c)$ of $A_{k,\ijkl}(b)$, where $\delta_{\ijkl+1}(n) = c$.
By the same kind of arguments as above we have
\begin{alignat*}{3}
\bigl\lvert A_{k,\ijkl}(0,0)\bigr\rvert &= F_{\ijkl} F_{k-\ijkl-1}, &\qquad \bigl\lvert A_{k,\ijkl}(0,1)\bigr\rvert &= F_{\ijkl} F_{k-\ijkl-2}, \\
\bigl\lvert A_{k,\ijkl}(1,0)\bigr\rvert &= F_{\ijkl-1} F_{k-\ijkl-1}, &\qquad \bigl\lvert A_{k,\ijkl}(1,1)\bigr\rvert&= 0.
\end{alignat*}
Thus, the {\it conditional probabilities} that $\delta_{\ijkl+1}(n) = c$ given that $\delta_{\ijkl}(n) = b$
are given by
\begin{alignat*}{3}
\frac{\lvert A_{k,\ijkl}(0,0)\rvert}{\lvert A_{k,\ijkl}(0)\rvert} &= \frac {F_{k-\ijkl-1}}{F_{k-\ijkl}}, &\qquad
\frac{\lvert A_{k,\ijkl}(0,1)\rvert}{\lvert A_{k,\ijkl}(0)\rvert} &= \frac {F_{k-\ijkl-2}}{F_{k-\ijkl}}, \\
\frac{\lvert A_{k,\ijkl}(1,0)\rvert}{\lvert A_{k,\ijkl}(1)\rvert} &=  1, &\qquad
\frac{\lvert A_{k,\ijkl}(1,1)\rvert}{\lvert A_{k,\ijkl}(1)\rvert} &= 0
\end{alignat*}
respectively. If $k-\ijkl\to \infty$ (even if $\ijkl$ is fixed), these probabilities converge to the entries
of the matrix
\begin{equation}\label{eqPdef}
{\bf P} = \left( \begin{matrix} \frac 1{\golden} & \frac 1{\golden^2} \\ 1 & 0 \end{matrix} \right).
\end{equation}

Summing up, by considering just numbers in $A_k$ the Zeckendorf digits behave, as $k\to\infty$, approximately
as a discrete Markov process $(Z_{\ijkl})_{\ijkl \ge 0}$, where the intial distribution is given
\begin{equation}\label{eqZ1}
\prob [Z_0 = 0] = \frac {\golden^2}{\golden^2+1}, \quad \prob [Z_0 = 1] = \frac 1{\golden^2+1}
\end{equation}
and the transition probabilities are defined by
\begin{alignat}{3}
\prob \bigl[Z_{\ijkl+1} = 0\,\mid\, Z_{\ijkl} = 0\bigr] &= \frac 1{\golden}, &\qquad \prob \bigl[Z_{\ijkl+1} = 1\,\mid\, Z_{\ijkl} = 0\bigr] &= \frac 1{\golden^2}, \label{eqZ2}\\
\prob \bigl[Z_{\ijkl+1} = 0\,\mid\, Z_{\ijkl} = 1\bigr] &= 1, &\qquad \prob \bigl[Z_{\ijkl+1} = 1\,\mid\, Z_{\ijkl} = 1\bigr] &= 0.  \label{eqZ3}
\end{alignat}
This means that the transition matrix of this Markov process is given by (\ref{eqPdef})
and that the stationary distribution (as well as the inital distribution) is given by
\[
p_0 = \frac {\golden^2}{\golden^2+1}, \quad p_1 = \frac 1{\golden^2+1}.
\]

As mentioned above these approximation properties were already discussed in~\cite{DS02}.
Actually it is possilbe to consider the joint distribution of several digits.
More precisely we state one property from \cite{DS02} (Lemma~4) that we will also use later in the proof 
of Lemma~\ref{Le6}. Note that the most and least significant digits are not included (a more detailed analysis
would show that they behave differently).

\begin{lemma}\label{LeDS02}
Let $0 < \eta < \frac 12$, and $d \ge 1$ be a fixed integer, moreover $(Z_{\ijkl})_{\ijkl}$ the stationary Markov 
process from above.
Then we have uniformly for all integers $\ijkl_1,\ldots,\ijkl_d$ satisfying
\[
(\log x)^\eta  \le \ijkl_1 < \ijkl_2 < \ldots < \ijkl_d < \log_{\golden}x  - (\log x)^\eta,
\]
and for all $\nu_1,\ldots, \nu_d \in \{0,1\}$,
\begin{align*}
\hspace{4em}&\hspace{-4em}
\frac 1{x} \# \bigl\{ n < x : \delta_{\ijkl_1}(n) = \nu_1, \ldots, \delta_{\ijkl_d}(n) = \nu_d \bigr\}
\\&= \prob\bigl[Z_{\ijkl_1} = \nu_1, \ldots, Z_{\ijkl_d} = \nu_d\bigr]
+ O\left( \frac 1{(\log x)^\lambda} \right)
\end{align*}
for every fixed $\lambda > 0$.
\end{lemma}

Since we are interested in the distribution of the Zeckendorf sum-of-digits function $\sz(n)$ it is, thus, natural 
to compare it with the distribution of $S_L = Z_0 + \cdots + Z_{L-1}$. 
Since $(Z_{\ijkl})_{\ijkl}$ is a stationary Markov process it is well known that $S_L$ satisfies
a central limit theorem. Actually we need a refined form with uniform estimates for the moments (Lemma~\ref{LeCLT}).

In order to study sums $S_L = Z_0 + \cdots + Z_{L-1}$ we note that the probability generating function
of $S_L$ is given by
\begin{equation}\label{eqEvSn}
\begin{aligned}
\mathbb{E}\, v^{S_L} &= \left( \begin{matrix} \frac {\golden^2}{\golden^2+1}  & \frac v{\golden^2+1}  \end{matrix} \right)
\left( \begin{matrix} \frac 1{\golden} & \frac v{\golden^2} \\ 1 & 0 \end{matrix} \right)^{L-1}
\left( \begin{matrix} 1 \\ 1  \end{matrix} \right)\\
&= a(v) \lambda_1(v)^L + b(v) \lambda_2(v)^L,
\end{aligned}
\end{equation}
where
\[
\lambda_1(v) = \frac{1 + \sqrt{1+4v}}{2\golden}, \quad \lambda_2(v) = \frac{1 - \sqrt{1+4v}}{2\golden}, 
\]
and
\[
a(v) = \frac{ \frac{\golden^2 + v}{\golden^2+1} - \lambda_2(v) }{ \lambda_1(v) - \lambda_2(v)}, \quad
b(v) = \frac{ \frac{\golden^2 + v}{\golden^2+1} - \lambda_1(v) }{ \lambda_2(v) - \lambda_1(v)}.
\]
In particular this gives 
\[
\mathbb{E}\, S_L =  \frac{L}{\golden^2 + 1} \quad\mbox{and}\quad
\mathbb{V}{\rm ar}\, S_L =  L \frac{\golden^3}{(\golden^2+1)^3} + \frac 2{25} - \frac 2{25} \left( - \golden^{-2} \right)^L.
\]
For convenience we set
\[
\mu = \frac 1{\golden^2 + 1}\quad \mbox{and}\quad \sigma^2 = \frac{\golden^3}{(\golden^2+1)^3}.
\]

\begin{lemma}\label{LeCLT}
Let $(Z_{\ijkl})_{\ijkl}$ denote the stationary Markov process defined above and $S_L = Z_0+ \cdots + Z_{L-1}$.
Then $S_L$ satisfies a central limit theorem of the form
\[
\frac{S_L - \mu L}{\sqrt{\sigma^2 L}} \to N(0,1).
\]
In particular we have for the characteristic function
\begin{equation}\label{eqcharfuncSn}
\mathbb{E}\, e^{it (S_L - L \mu)/(L \sigma^2)^{1/2}}
= e^{-t^2/2} \, \left ( 1 + O \left( \frac {t^2}{L} \right) + O \left( \frac {\lvert t\rvert^3}{L^{1/2}} \right) \right),
\end{equation}
which is uniform for $\lvert t\rvert\le L^{\frac 16}$.
Furthermore, the centralized moments satisfy
\[
\mathbb{E}\left( \frac{S_L - L \mu}{(L \sigma^2)^{1/2}}\right)^d = 
\left\{ \begin{array}{cl} 
\frac{d!}{2^{d/2}(d/2)!} \left(1 + O(d^2/L^{1/2}) \right) & \mbox{if $d$ is even;} \\[2mm]
O\left( d^{d/2}e^{-d/2} d^{3/2}/L^{1/2}  \right) & \mbox{otherwise,} 
\end{array} \right.
\]
uniformly for $d \le L^{1/4}$.
\end{lemma}
 
\begin{proof}
By using~\eqref{eqEvSn}, setting
\[
v = e^{it/(L\sigma^2)^{1/2}},
\]
and the Taylor expansions $a(e^{is}) = 1 + O(s^2)$ and 
\[
\log\left( \lambda_1(e^{is})  \right)
= i\mu s - \frac 12 \sigma^2 s^2 + O(s^3),
\]
as well as the (trivial) upper bound $\lvert b(e^{is}) \lambda_2(e^{is})^{n}\rvert \ll e^{-c n}$ 
(for some $c> 0$) we immediately obtain~\eqref{eqcharfuncSn}.

Clearly,\eqref{eqcharfuncSn} implies that $S_L$ satisfies the proposed central limit theorem.

In order to handle the centralized moments of $S_L$ we note that the characteristic function
$\phi_Y(t) = \mathbb{E}\, e^{it Y}$ of a random variable $Y$
is closely related to the moment generating function:
\begin{align*}
\sum_{d\ge 0} \mathbb{E}\, Y^d \frac{w^d}{d!} &=  \mathbb{E}\,e^{wY} \\
&= \phi_Y(-i w).
\end{align*}
By applying this for $Y = (S_L-\mu L)/(\sigma^2 L)^{1/2}$ we, thus, obtain
\[
\sum_{d\ge 0} \mathbb{E}\, Y^d \frac{w^d}{d!}  
= e^{w^2/2} \left ( 1 + O \left( \frac {w^2}{L} \right) + O \left( \frac {\lvert w\rvert^3}{L^{1/2}} \right) \right)
\]
that holds uniformly for $\lvert w\rvert\le L^{\frac 16}$.
The moments can be then computed with the help of Cauchy's formula
\[
\mathbb{E}\, Y^d = \frac{d!}{2\pi i} \int_{\lvert w\rvert = w_0}
e^{w^2/2} \left ( 1 + O \left( \frac {w^2}{L} \right) + O \left( \frac {\lvert w\rvert^3}{L^{1/2}} \right) \right)  \frac {\mathrm dw}{w^{d+1}}.
\]
Asymptotically, integrals of this type can be evaluated with help of a saddle point method,
where the saddle point $w_0$ (of the dominating part of the integrand $e^{w^2/2 - d\log w}$)
is given by $w_0 = \sqrt d$. Of course this only works if $d= o\bigl(L^{\frac 13} \bigr)$.
We note that 
\[
\frac{d!}{2\pi i} \int_{\lvert w\rvert = w_0}
e^{w^2/2} \frac {\mathrm dw}{w^{d+1}} 
= \begin{cases} 
\frac{d!}{2^{d/2}(d/2)!}, & \mbox{if $d$ is even;} \\
0, & \mbox{otherwise,} 
\end{cases}
\]
are just the moments of the standard normal distribution.
Observe that for even $d$ we have $\mathbb{E}\, Y^d \sim \sqrt 2 d^{d/2} e^{-d/2}$. 
Furthermore we can estimate the remaining integrals by
\begin{eqnarray*}
\frac{d!}{2\pi i} \int_{\lvert w\rvert = w_0}
e^{w^2/2} \frac {\mathrm dw}{w^{d+1}} O \left( \frac {w^2}{L} + \frac {\lvert w\rvert^3}{L^{1/2}} \right) \frac {\mathrm dw}{w^{d+1}} 
&\ll &  d! e^{d/2} d^{-d/2} \frac{d^{3/2}}{L^{1/2}} \\
& \ll & \frac{d!}{2^{d/2}(d/2)!} \frac {d^2}{L^{1/2}} 
\end{eqnarray*}
which completes the proof of the lemma. Note that this estimate is only significant it
$d \le L^{1/4}$.
\end{proof}

With the help of Lemma~\ref{LeDS02} it is possible to compare the centralized moments of 
the Zeckendorf sum-of-digits function $\sz(n)$, $n\le x$, with the centralized moments of $S_L$,
where $L = \log_\gamma x + O(1)$ denotes the maximal length of the Zeckendorf expansion of $n\le x$. 
This immediately proves a central limit theorem for the Zeckendorf sum-of-digits function $\sz(n)$.

The main idea behind the proof of Proposition~\ref{Promain2} is that the {\it overall distribution}
of the Zeckendorf sum-of-digits function does not change drastically if we just consider
primes $p\le x$ instead of all natural numbers $n\le x$.


More precisely, for (sufficiently large) $x\ge 2$ we consider the set of primes
\[
\{ p \in \mathbb{P} : p \le x \}
\]
and assume that every prime in this set is equally likely.
Then the sum-of-digits function $\sz (p)$ can be interpreted as a random variable
\[
S_x = S_x(p) = \sz (p) = \sum_{\ijkl \le L} \delta_{\ijkl}(p).
\]
Of course, $D_{\ijkl} = D_{\ijkl,x} = \delta_{\ijkl}$, the $\ijkl$-digit, is also considered as a random variable.

We can now reformulate Proposition~\ref{Promain2}.
Let $L = \log_\golden x + O(1)$ denote the maximal length of a Zeckendorf expansion of an integer $\le x$. 
Then the asymptotic formula (\ref{eqPro2}) is equivalent to the relation
\begin{equation}\label{eqPro2refomulated}
\begin{aligned}
\phi_1(t) &\coloneqq \mathbb{E}\, e^{it (S_x - L \mu)/(L \sigma^2)^{1/2}}
\\&= e^{-t^2/2} \, \left ( 1 +  O \left( \frac {t^2}{{\log x}} \right)+  O \left( \frac {\lvert t\rvert^3}{{(\log x)^{1/2}}} \right) \right)
+ O\left( \frac {\lvert t\rvert}{(\log x)^{\frac 12 - \nu}} \right),
\end{aligned}
\end{equation}
which is uniform for $\lvert t\rvert\le (\log x)^{\eta}$.
We just have to set $\vartheta = t/(2\pi \sigma(\log_\golden x)^{1/2})$.

Clearly, $\phi_1(t)$ is just the characteristic function of the normalized 
Zeckendorf sum-of-digits-function of primes $p\le x$. And since $\phi_1(t) \to e^{-t^2/2}$ (for every real $t$) is equivalent
to a central limit theorem, Proposition~\ref{Promain2} can be seen as a strengthened form 
of a central limit theorem.

For proof technical reasons we have to truncate the Zeckendorf sum-of-digits appropriately.
Set  
\[
L' = \# \bigl\{ \ijkl \in \mathbb{Z} : L^{\nu} \le \ijkl \le L- L^{\nu} \bigr\}
=  L - 2 L^{\nu} + O(1), 
\]
where $0<\nu< \frac 12$ is fixed, and
\[
T_x = T_x (p) =  \sum_{ L^{\nu} \le \ijkl \le L- L^{\nu}} \delta_{\ijkl}(p) =
\sum_{ L^{\nu} \le \ijkl \le L- L^{\nu}} D_{\ijkl}.
\]
First we observe that $\phi_1(t)$ and
\begin{equation}\label{eqdefphi2}
\phi_2(t) \coloneqq \mathbb{E}\, e^{it (T_x - L' \mu)/(L' \sigma^2)^{1/2}}
\end{equation}
do not differ significantly.

\begin{lemma}\label{Le1}
We have, uniformly for all real $t$,
\[
\bigl\lvert \phi_1(t) - \phi_2(t)\bigr\rvert = O\left( \frac {\lvert t\rvert}{(\log x)^{\frac 12 - \nu}} \right).
\]
\end{lemma}

\begin{proof}
We only have to observe that $\lvert L-L'\rvert\ll L^{\nu}$, $\lVert S_x-T_x\rVert_\infty \ll L^{\nu}$,
$\lVert S_x\rVert_\infty \ll L$ and
that $\lvert e^{it}- e^{is}\rvert \le \lvert t-s\rvert$. Consequently
\begin{align*}
\bigl\lvert \phi_1(t) - \phi_2(t)\bigr\rvert & \le \lvert t\rvert\, \mathbb{E}\ \left\lvert \frac{S_x - L \mu}{(L \sigma^2)^{1/2}} -
\frac{T_x - L' \mu}{(L' \sigma^2)^{1/2}} \right\rvert
\\&\ll \lvert t\rvert\left( \frac {\lVert S_x-T_x\rVert_\infty }{L^{1/2}} + \frac {\lvert L - L'\rvert}{L^{1/2}} +
\lVert S_x\rVert_\infty \left( \frac 1{L'^{1/2}} - \frac 1{L^{1/2}} \right)\right) \\
&\ll \frac {\lvert t\rvert}{(\log x)^{\frac 12 - \nu}}.
\end{align*}
This proves the lemma.
\end{proof}


In a final step we approximate $T_x$ by a sum $\overline T_x$ of random variables from the
above defined Markov chain. More precisely we set 
\[
\overline T_x \coloneqq \sum_{ L^{\nu} \le \ijkl \le L- L^{\nu}} Z_{\ijkl},
\]
that is, $\overline T_x$ has the same distribution as $S_{L'}$.
In particular we have the following property that follows immediately from Lemma~\ref{LeCLT}.
\begin{lemma}\label{Le2}
The characteristic function of the normalized random variable $\overline T_x$ is
asymptotically given by
\begin{equation}\label{eqLe2}
\begin{aligned}
\phi_3(t) &\coloneqq \mathbb{E}\, e^{it (\overline T_x - L' \mu)/(L' \sigma^2)^{1/2}}\\
&= e^{-t^2/2} \, \left ( 1 + O \left( \frac {t^2}{\log x} \right) + O \left( \frac {\lvert t\rvert^3}{(\log x)^{1/2}} \right) \right),
\end{aligned}
\end{equation}
which is also uniform for $\lvert t\rvert\le (\log x)^{\frac 16}$.
\end{lemma}

It remains to compare $\phi_2(t)$ and $\phi_3(t)$, which is actually the main 
difficulty in the proof of Proposition~\ref{Promain2}.
\begin{proposition}\label{Pro3}
Suppose that $0< \nu < \frac 12$ and that $\tau$ and $\kappa$ satisfy $0<2\tau < \kappa < \frac 13\nu$.
Furtermore let $\phi_2(t)$ and $\phi_3(t)$ denote the characteristic function of the normalized random variables 
$T_x$ and $\overline T_x$, respectively, given by (\ref{eqdefphi2}) and (\ref{eqLe2}).
Then we have uniformly for real $t$ satisfying $\lvert t\rvert\le L^{\tau}$
\[
\bigl\lvert \phi_2(t) - \phi_3(t)\bigr\rvert = O \left( \lvert t\rvert e^{-c_1L^{\kappa}} \right),
\]
where $c_1$ is a certain positive constant depending on $\tau$ and $\kappa$.
\end{proposition}

Note that $e^{-c_1L^{\kappa}} \ll L^{-1}$.
Hence, Proposition~\ref{Pro3} (together with Lemma~\ref{Le1} and Lemma~\ref{Le2})
immediately imply~\eqref{eqPro2refomulated} and, thus, Proposition~\ref{Promain2}.

\section{Comparison of moments}

The key lemma for comparing moments of $T_x$ and $\overline T_x$
is the following property, which will be proved in Section~\ref{sec:keylemma}.

\begin{lemma}\label{Le6}
Let $0< \kappa < \rho < \frac 13 \nu$.
For $1\le d\le L^{\kappa}$, consider
integers $\ijkl_1,\ijkl_2,\ldots,\ijkl_d$ and $\nu_1,\nu_2,\ldots,\nu_d$ 
with
\[
L^{\nu} \le \ijkl_1 ,  \ijkl_2 , \cdots , \ijkl_d \le L - L^{\nu}
\]
and
\[
\nu_1,\nu_2,\ldots,\nu_d \in \{0,1\}.
\]
Then we have uniformly 
\begin{equation*}
\begin{aligned}
\hspace{4em}&\hspace{-4em}
\frac 1{\pi(x)} \# \bigl\{ p \le x : 
\delta_{\ijkl_1}(p) = \nu_1,\ldots, \delta_{\ijkl_d}(p) = \nu_d \bigr\}  \\
&= q_{\ijkl_1,\ldots,\ijkl_d;\nu_1,\ldots,\nu_d} + O\left( e^{-L^\rho} \right),
\end{aligned}
\end{equation*}
where 
\[
q_{\ijkl_1,\ldots,\ijkl_d;\nu_1,\ldots,\nu_d} = \prob\bigl[Z_{\ijkl_1} = \nu_1, \ldots, Z_{\ijkl_d} = \nu_d\bigr]
\]
is related to the probability distribution of the stationary Markov chain $(Z_{\ijkl})$.
\end{lemma}

Now we compare centralized moments of $T_x$ and $\overline T_x$.
\begin{lemma}\label{Le7}
Let $0 < \kappa < \rho < \frac 13 \nu$. Then we have uniformly for $1\le d \le L^\kappa$
\[
\mathbb{E} \left( \frac{T_x - L'\mu}{\sqrt{L'\sigma^2}} \right)^d
= \mathbb{E} \left( \frac{\overline T_x - L'\mu}{\sqrt{L'\sigma^2}} \right)^d
+ O \left( e^{-\frac 12 L^\rho}  \right).
\]
\end{lemma}

\begin{proof}
We expand the difference
\[
\mathbb{E}\left( \sum_{L^{\nu}\le \ijkl\le L-L^{\nu}} \left(D_{\ijkl,x} - \mu \right)\right)^d -
\mathbb{E}\left( \sum_{L^{\nu}\le \ijkl\le L-L^{\nu}} \left(Z_{\ijkl} - \mu \right)\right)^d
\]
and compare it with help of Lemma~\ref{Le6}.
In fact, we have to take into account $\le (2L')^d$ terms and, thus,
the difference is bounded from above by
\[
 \ll (2L)^d e^{-L^\rho} \ll e^{\log(2L)L^\kappa - L^\rho} \ll e^{-\frac 12 L^\rho}.
\]
Of course this proves the lemma, since the factor ${L'}^{-d/2} \sigma^{-d}$ is certainly bounded.
\end{proof}

\def\lemLesix{\ref{Le6}}
\section{Proof of Lemma~\lemLesix}\label{sec:keylemma}

In this section we provide the proof of the (Key) Lemma~\ref{Le6}.
We start with some preliminaries.
Recall that we denote the digits in the Zeckendorf expansion (\ref{eqn_Zeckendorf_rep}) of a non-negative integer $n$ by $\delta_{\ijkl}(n)\in\{0,1\}$ (where $2\le \ijkl \le L$).

\begin{lemma}\label{Leprelim1}
Let $m$ be a positive integer with $m < F_{k-3}$. Then we have
\[
\delta_{\ijkl}(mF_k) = 0
\]
for $\ijkl < k-\log_{\golden} m - 1$ and for $\ijkl > k+\log_{\golden} m +2$.
\end{lemma}

\begin{proof}
We apply Lemma~\ref{Lefirstdigits} for $u=0$ and observe that
\begin{equation}\label{eqappl1}
\lVert n\golden \rVert < \golden^{-L} \quad\mbox{implies}\quad \delta_{\ijkl}(n) = 0 \quad\mbox{for $\ijkl < L$}.
\end{equation}
Since $\lVert F_k \golden\rVert = \golden^{-k}$ (compare to~\eqref{eqn_binet_consequence}), we have
\[
\lVert m F_k \golden\rVert \le m \lVert F_k\golden \rVert \le \frac{m}{\golden^k} < \golden^{-(k-\log_\golden m - 1)}.
\]
Hence we certainly have $\delta_{\ijkl}(m F_k) = 0$ for $\ijkl < k-\log_\golden m - 1$.

On the other hand 
\[
m F_k \le F_{k+ \lfloor \log_\golden m \rfloor + 2}.
\]
Thus, $\delta_{\ijkl}(mF_k) = 0$ for $\ijkl > k+\log_{\golden} m +2$.
\end{proof}

\begin{lemma}\label{Leprelim2}
Suppose that
\[
\delta_{\ijkl}(n) = 0 \quad\mbox{for $\ijkl < N_1$ and $\ijkl > N_2$}
\]
and
\[
\delta_{\ijkl}(m) = 0 \quad\mbox{for $\ijkl < M_1$ and $\ijkl > M_2$.}
\]
Then
\[
\delta_{\ijkl}(\lvert n\pm m\rvert) = 0 \quad\mbox{for $\ijkl < \min\{N_1,M_1\}-3$ and $\ijkl > \max\{N_2,M_2\}+2$.}
\]
\end{lemma}

\begin{proof}
In addition to~\eqref{eqappl1} we apply Lemma~\ref{Lefirstdigits} another time, setting $u=0$, and get
\begin{equation}\label{eqappl2}
\delta_{\ijkl}(n) = 0 \quad\mbox{for all ${\ijkl} < L$}  \quad\mbox{implies}\quad \lVert n\golden \rVert < \golden^{-L+1}.
\end{equation}
Hence, by assumption we have
\[
\lVert n \golden \rVert < \golden^{-N_1+1} \quad\mbox{and}\quad \lVert m \golden \rVert < \golden^{-M_1+1}.
\]
This gives
\[
\lVert (n\pm m) \golden \rVert \le \lVert n \golden \rVert + \lVert m \golden \rVert <  \golden^{-N_1+1} +  \golden^{-M_1+1} < 
\golden^{-\min\{N_1,M_1\}+3}
\]
and implies that $\delta_{\ijkl}(\lvert n\pm m\rvert) = 0$ for $\ijkl < \min\{N_1,M_1\}-3$.

On the other hand we have
\[
\lvert n \pm m \rvert \le \lvert n\rvert + \lvert m\rvert < F_{N_2+1} + F_{M_2+1} < F_{\max\{N_2,M_2\}+3},
\]
which implies that $\delta_{\ijkl}(\lvert n\pm m\rvert) = 0$ for $\ijkl > \max\{N_1,M_1\}+2$.
\end{proof}

\begin{proof}[Proof of Lemma~\ref{Le6}]
Note that we only have to consider the case $\ijkl_1 < \ijkl_2 < \cdots < \ijkl_d$.
All other cases are either trivial or can be directly reduced to this case.

Let $A_0$ and $A_1$ be the rectangles defined in Lemma~\ref{Letiling}
and $\tilde A_j = A_j + \mathbb{Z}\times \mathbb{Z}$, $j=0,1$, their periodic extensions.
Let $\psi_j (x_1,x_2)$, $j=0,1$, be defined by
\[
\psi_j (x_1,x_2) = \begin{cases}1, & \mbox{if $(x_1,x_2) \in \tilde A_j  \setminus \partial \tilde A_j $;} \\
\frac 12, & \mbox{if  $(x_1,x_2) \in \partial \tilde A_j $;}\\
0, & \mbox{otherwise.}
\end{cases}
\]

Clearly,  $\psi_j (x_1,x_2)$ is periodic with period $1$ and has a Fourier expansion\\
$\sum c_{h_1,h_2}(j ) e(h_1x_1 + h_2 x_2)$.
The constant coefficient $c_{0,0}(j )$ is given by
$c_{0,0}(j ) = \lambda_2(A_j )$.
By Lemma~\ref{le_fourier_parallel} (see also \cite[Lemma 1]{Drm96}) these coefficients can be uniformly bounded by
\begin{align*}
\lvert c_{h_1,h_2}(j )\rvert^2 &\ll \frac 1{ \left( 1 + \lvert h_1 + h_2/\golden\rvert \right)^2 
\left( 1 + \lvert h_2 - h_1/\golden\rvert \right)^2  } \\
&= \frac 1{ (1 + |\tilde h_1|)^2(1 + |\tilde h_2|)^2  } ,
\end{align*}
where 
\[
\tilde h_1 =  h_1 + \frac 1 \golden h_2 \quad\mbox{and}\quad
\tilde h_2 =  h_2 - \frac 1 \golden h_1.
\]
For small $\Delta > 0$ we consider the function
\[
f_{\Delta,j  }(x_1,x_2) = \frac 1{\Delta^2}
\int_{-\Delta/2}^{\Delta/2}\int_{-\Delta/2}^{\Delta/2} 
\psi_j \left( x_1 + z_1 - \frac 1\golden z_2, x_2 + \frac 1\golden z_1 + z_2 \right) \,\mathrm dz_1\,\mathrm dz_2.
\]
The Fourier expansion $\sum d_{h_1,h_2}(\Delta,j ) e(h_1x_1+h_2x_2)$ of this function is given by
\[
d_{0,0}(\Delta,j ) = c_{0,0}(j ) = \lambda_2(A_j )
\]
and for $(h_1,h_2)\ne (0,0)$,
\begin{align}
d_{h_1,h_2}(\Delta,j ) &= c_{h_1,h_2}(j )
\frac{ \left( e\left( \frac{\tilde h_1 \Delta}2  \right) - e\left( -\frac{\tilde h_1 \Delta}2  \right) \right) 
\left( e\left( \frac{\tilde h_2 \Delta}2  \right) - e\left( -\frac{\tilde h_2 \Delta}2  \right) \right)  }
{-4\pi^2 \tilde h_1 \tilde h_2 \Delta^2} \nonumber \\
&= c_{h_1,h_2}(j ) \frac{\sin(\pi \tilde h_1 \Delta)}{\pi \tilde h_1 \Delta} 
\frac{\sin(\pi \tilde h_2 \Delta)}{\pi \tilde h_2 \Delta} .  \label{eqdcrep} 
\end{align}
Hence we uniformly have
\begin{equation}\label{eqdcest}
\lvert d_{h_1,h_2}(\Delta,j )\rvert \le \frac {K_1}
{ (1+\lvert \tilde h_1\rvert)(1+\lvert\tilde h_2\rvert) (1+\Delta\lvert\tilde h_1\rvert) (1+\Delta\lvert\tilde h_2\rvert)   }
\end{equation}
with an absolute constant $K_1 > 0$.

By definition it is clear that $0\le f_{\Delta,j }(x_1,x_2)  \le 1$ and that 
\[
f_{\Delta,j }(x_1,x_2)  = \begin{cases}
1, & \mbox{if $(x_1,x_2) \in \tilde A_j  \setminus U_j (\Delta)$;} \\
0, & \mbox{if $(x_1,x_2) \not \in \tilde A_j  \cup U_j (\Delta)$,}
\end{cases}
\]
where we have adapted the notation of Lemma~\ref{Le4}, that is, 
\[
U_j(\Delta) = \left\{ \left( x_1 + y_1 - y_2/\golden, x_1 + y_1/\golden + y_2 \right) :
(x_1,x_2)\in \partial \tilde A_j,\, \lvert y_1\rvert\le \Delta/2, \, \lvert y_1\rvert\le \Delta/2 \right\},
\]

We define
\[
F({\bf x}_1, \ldots, {\bf x}_d) = f_{\Delta,\nu_1 }({\bf x}_1)\cdots f_{\Delta,\nu_d }({\bf x}_d)
\]
and 
\[
t(n) = F\left( \left( n \golden^{-\ijkl_1}, n \golden^{-\ijkl_1-1}\right), \ldots, \left( n \golden^{-\ijkl_d}, n \golden^{-\ijkl_d-1}\right) \right).
\]
Furthermore we assume that the error term in~\eqref{eqZdigcalc} satisfies for $k=\ijkl_1$
\[
O\bigl(\golden^{-\ijkl_1}\bigr) \le \frac{\Delta}2.
\]
Then we have by Lemma~\ref{Le4}
\begin{equation*}
\begin{aligned} 
\hspace{4em}&\hspace{-4em}
\left\lvert\#\bigl\{ p \le x : \delta_{\ijkl_1}(p) = \nu_1, \ldots, \delta_{\ijkl_d}(p) = \nu_d \bigr\} - 
\sum_{p\ge x} t(p) \right\rvert \\
&\le \sum_{\ell = 1}^d \# \left\{p< x :  \left( \{ p \golden^{-\ijkl_\ell} \}, \{ p \golden^{-\ijkl_\ell-1} \} \right) 
 \in U_{\nu_\ell}(\Delta) \right\} \\
 &\ll d \pi(x) \Delta + d \pi(x) e^{-c_3 L^\nu} .
\end{aligned}
\end{equation*}

Next set 
\[
{\bf V} = \left( \golden^{-\ijkl_1}, \golden^{-\ijkl_1-1}, \ldots, \golden^{-\ijkl_d}, \golden^{-\ijkl_d-1} \right)
\]
and denote by ${\bf H}$ an $2d$-dimensional integer vector
\[
{\bf H} = \left( h_{11},h_{12}, \ldots, h_{d1},h_{d2} \right).
\]
Then we have
\[
t(n) = \sum_{{\bf H}} T_{\bf H} \e({\bf H \cdot V}\, n),
\]
where 
\[
T_{\bf H} = d_{h_{11},h_{12}}(\Delta,\nu_1)\cdots d_{h_{d1},h_{d2}}(\Delta,\nu_d),
\]
and consequently
\[
\sum_{p\le x} t(p) = \sum_{\bf H} T_{\bf H} \sum_{p\le x} e({\bf V\cdot H}\, p).
\]

Let $\mathcal{M}_0$ be the set of $2d$-dimensional integer vectors ${\bf H}$ with ${\bf V\cdot H} = 0$.
Note that $(0,0,\ldots, 0,0)$ is always contained in $\mathcal{M}_0$.
We first consider the sum
\[
S_1=\sum_{{\bf H}\in \mathcal{M}_0} T_{\bf H}
\]
and will show that (for some universal constant $\overline K$)
\begin{equation}\label{eqS1est}
S_1 = \sum_{{\bf H}\in \mathcal{M}_0} T_{0,{\bf H}} + O\left(  \Delta^{\frac 1d} \overline K^d d^{2d}  \right),
\end{equation}
where
\[
T_{0,{\bf H}} = c_{h_{11},h_{12}}(\nu_1)\cdots c_{h_{d1},h_{d2}}(\nu_d).
\]
In a second step we will also show that 
\begin{equation}\label{eqPrep}
\sum_{{\bf H}\in \mathcal{M}_0} T_{0,{\bf H}} = \prob\bigl[Z_{\ijkl_1} = \nu_1,\ldots, Z_{\ijkl_d} = \nu_d\bigr].
\end{equation}
We first note that there exists an absolute constant $C> 0$ such that uniformly for all $D\ge 1$ and all real $a_1,\ldots, a_D$ 
\begin{equation}\label{eqprodsin}
\left\lvert \frac{\sin(a_1)}{a_1} \cdots \frac{\sin(a_D)}{a_D} - 1 \right\rvert \le C\bigl(a_1^2 + \cdots a_D^2\bigr).
\end{equation}
Suppose first that $a_1^2 + \cdots a_D^2 \le 1$. Then we can use the estimates
$\left\lvert \frac{\sin x}x\right\rvert\le 1$ and the expansion $\frac{\sin x}x = 1 + O(x^2)$ to obtain
\[
\left\lvert \frac{\sin(a_1)}{a_1} \cdots \frac{\sin(a_D)}{a_D} - 1 \right\rvert \le
\sum_{j=1}^D \left\lvert \frac{\sin(a_j)}{a_j} -1 \right\rvert
\le C_1 \sum_{j=1}^D a_j^2
\]
for some universal constant $C_1 > 0$.
If $a_1^2 + \cdots + a_D^2 > 1$ then the left hand side is still $\le 2$.
Thus,~\eqref{eqprodsin} holds with $C = \max\{C_1,2\}$.

By~\eqref{eqdcrep} this relation implies
\[
T_{\bf H} = T_{0,{\bf H}} \left( 1 + O\left( \Delta^2 \lVert {\bf H} \rVert_\infty^2 \right) \right).
\]
Thus, we have to provide suitable upper bounds for the following three sums:
\begin{align*}
S_2 &=  \Delta^2 \sum_{{\bf H}\in \mathcal{M}_0,\, \|{\bf H}\|_\infty \le H_0}   |T_{0,{\bf H}}| \cdot \| {\bf H} \|_\infty^2, \\
S_3 &= \sum_{{\bf H}\in \mathcal{M}_0,\, \lVert{\bf H}\rVert_\infty >  H_0}  \lvert T_{0,{\bf H}}\rvert, \\
S_4 &= \sum_{{\bf H}\in \mathcal{M}_0,\, \lVert{\bf H}\rVert_\infty >  H_0}  \lvert T_{{\bf H}}\rvert,
\end{align*}
where $H_0$ is chosen suitably. 

Before studying these sums we have to describe the set $\mathcal{M}_0$ more explicitly.
We have (by using the representation $\golden^k = F_k \golden + F_{k-1}$ with the convention $F_{-1} = 1$) 
\begin{align*}
{\bf V\cdot H} &= \sum_{\ell=1}^d \frac{h_{\ell 1}\golden + h_{\ell 2}}{\golden^{\ijkl_\ell+1}} \\
&= \frac 1{\golden^{\ijkl_d+1}}  \sum_{\ell=1}^d \left( h_{\ell 1} \golden^{\ijkl_d-\ijkl_\ell+1} + h_{\ell 2} \golden^{\ijkl_d-\ijkl_\ell} \right)\\
&= \frac 1{\golden^{\ijkl_d+1}}  \sum_{\ell=1}^d \left( h_{\ell 1} (F_{\ijkl_d-\ijkl_\ell+1}\golden + F_{\ijkl_d-\ijkl_\ell} )+ 
h_{\ell 2} (F_{\ijkl_d-\ijkl_\ell}\golden + F_{\ijkl_d-\ijkl_\ell-1} ) \right) \\
&= \frac 1{\golden^{\ijkl_d+1}} \left( \golden \sum_{\ell=1}^d  \left( h_{\ell 1} F_{\ijkl_d-\ijkl_\ell+1} + h_{\ell 2} F_{\ijkl_d-\ijkl_\ell} \right)  + \sum_{\ell=1}^d  \left( h_{\ell 1} F_{\ijkl_d-\ijkl_\ell} + h_{\ell 2} F_{\ijkl_d-\ijkl_\ell-1} \right) \right).
\end{align*}
Hence ${\bf V\cdot H} = 0$ if and only if the last two sums are zero. In particular this means that
\begin{align*}
h_{d1} &= - \sum_{\ell=1}^{d-1}  \left( h_{\ell 1} F_{\ijkl_d-\ijkl_\ell+1} + h_{\ell 2} F_{\ijkl_d-\ijkl_\ell} \right), \\
h_{d2} &= - \sum_{\ell=1}^{d-1}  \left( h_{\ell 1} F_{\ijkl_d-\ijkl_\ell} + h_{\ell 2} F_{\ijkl_d-\ijkl_\ell-1} \right).
\end{align*}
Summing up, this means that we can choose $h_{11}, h_{12}, \ldots, h_{d-1,1}, h_{d-1,2}$ in an arbitrary way, whereas $h_{d1}$ and $h_{d2}$ depend on them.

For notational convenience we write
\[
\tilde h_{\ell 1} =   h_{\ell 1} + \frac 1 \golden h_{\ell 2} \quad\mbox{and}\quad
\tilde h_{\ell 2} =  h_{\ell 2} - \frac 1 \golden h_{\ell 1}.
\]
We note that
\[
\tilde h_{\ell 1}^2 + \tilde h_{\ell 2}^2  = \frac{\golden^2+1}{\golden^2} \left( \tilde h_{\ell 1}^2 + \tilde h_{\ell 2}^2 \right).
\]
Hence, we can replace $\lVert {\bf H} \rVert_\infty$ (up to a universal constant) by $\lVert \tilde {\bf H} \rVert_\infty$ in the sum $S_2$.

In what follows we will use the inequality
\[
\lvert  T_{0,{\bf H}}\rvert \le \frac{K_1^d}{\prod_{\ell=1}^d  (1+\lvert\tilde h_{\ell 1}\rvert) (1+\lvert\tilde h_{\ell 2}\rvert) }.
\]

Let us start with the discussion of the sum $S_2$.
In order to present the idea we consider the special case $d=4$ and the (partial) sum 
\begin{align*}
S_{21} &= \sum_{ {\bf H} \in \mathcal{M}_0,\, \lVert{\bf H}\rVert_\infty \le H_0}   \lvert T_{0,{\bf H}}\rvert \cdot \tilde h_{11}^2 \\
&\le K_1^4  \sum_{\lvert h_{11}\rvert,\lvert h_{12}\rvert,\lvert h_{21}\rvert,\lvert h_{22}\rvert,\lvert h_{31}\rvert,\lvert h_{32}\rvert \le H_0}  
\frac{ \lvert\tilde h_{11}\rvert} { (1 + \lvert\tilde h_{12}\rvert) \prod_{\ell = 2}^4  (1 + \rvert\tilde h_{\ell 1}\rvert)(1 + \lvert\tilde h_{\ell 2}\rvert) },
\end{align*}
where $h_{41}$ and $h_{42}$ are related to $h_{\ell 1}$ and $h_{\ell 2}$ ($1\le \ell \le 3$) via the identities
\begin{equation}\label{eqm4142}
\begin{aligned}
h_{41} &= -  h_{1 1} F_{\ijkl_4-\ijkl_1+1} - h_{1 2} F_{\ijkl_4-\ijkl_1}-  h_{2 1} F_{\ijkl_4-\ijkl_2+1}\\
& - h_{2 2} F_{\ijkl_4-\ijkl_2}- 
 h_{3 1} F_{\ijkl_4-\ijkl_3+1} - h_{3 2} F_{\ijkl_4-\ijkl_3} ,\\
h_{42} &= -  h_{1 1} F_{\ijkl_4-\ijkl_1} - h_{1 2} F_{\ijkl_4-\ijkl_1-1}-  h_{2 1} F_{\ijkl_4-\ijkl_2} \\
& - h_{2 2} F_{\ijkl_4-\ijkl_2-1}- 
 h_{3 1} F_{\ijkl_4-\ijkl_3} - h_{3 2} F_{\ijkl_4-\ijkl_3-1}.
\end{aligned}
\end{equation}
By H\"older's inequality we get 
\begin{align*}
S_{21} &\le K_1^4 \left( {\sum}' 
\left( \frac{ \lvert\tilde h_{11}\rvert} { (1 + \lvert\tilde h_{12}\rvert) (1 + \lvert \tilde h_{2 1}\rvert)(1 + \lvert \tilde h_{2 2}\rvert) 
(1 + \lvert\tilde h_{3 1}\rvert)(1 + \lvert\tilde h_{3 2}\rvert) } \right)^{4/3} \right)^{\frac 14} \\
& \times \left( {\sum}' 
\left( \frac{ \lvert\tilde h_{11}\rvert} { (1 + \lvert\tilde h_{12}\rvert) (1 + \lvert\tilde h_{2 1}\rvert)(1 + \lvert\tilde h_{2 2}\rvert) 
(1 + \lvert\tilde h_{4 1}\rvert)(1 + \lvert\tilde h_{4 2}\rvert) } \right)^{4/3} \right)^{\frac 14} \\
& \times \left( {\sum}' 
\left( \frac{ |\tilde h_{11}|} { (1 + \lvert\tilde h_{12}\rvert) (1 + \lvert\tilde h_{3 1}\rvert)(1 + \lvert\tilde h_{3 2}\rvert) 
(1 + \lvert\tilde h_{4 1}\rvert)(1 + \lvert\tilde h_{4 2}\rvert) } \right)^{4/3} \right)^{\frac 14} \\
& \times \left( {\sum}'
\left( \frac{1} { (1 + \lvert\tilde h_{2 1}\rvert)(1 + \lvert\tilde h_{2 2}\rvert) (1 + \lvert\tilde h_{3 1}\rvert)(1 + \lvert\tilde h_{3 2}\rvert) 
(1 + \lvert\tilde h_{4 1}\rvert)(1 + \lvert\tilde h_{4 2}\rvert) } \right)^{4/3} \right)^{\frac 14},
\end{align*}
where ${\sum}'$ denotes the sum over all integers $h_{11},h_{12},h_{21},h_{22},h_{31},h_{32}$ satisfying 
\[
\lvert h_{11}\rvert,\lvert h_{12}\rvert,\lvert h_{21}\rvert,\lvert h_{22}\rvert,\lvert h_{31}\rvert,\lvert h_{32}\rvert \le H_0,
\]
and $h_{41}$ and $h_{42}$ are given by (\ref{eqm4142}).

For $\vartheta>1$ let $G_{\vartheta}$ denote the sum
\[
G_\vartheta = \sum_{h_1,h_2\in\mathbb{Z}} \frac 1{(1+\lvert h_1+ h_2/\golden\rvert)^\vartheta (1+\lvert h_2 - h_1/\golden\rvert)^\vartheta }
\]
and $R_\vartheta(H_0)$ the sum
\[
R_\vartheta(H_0) = \sum_{\lvert h_1\rvert, \lvert h_2\rvert\le H_0 } \frac {\lvert h_1+ h_2/\golden\rvert^\vartheta}{ (1+\lvert h_2 - h_1/\golden\rvert)^\vartheta }.
\]
It is an easy exercise to show that for $1< \vartheta \le 2$ we uniformly have
\[
G_\vartheta \le \frac {C_1}{(\vartheta - 1)^2} \quad \mbox{and}\quad
R_\vartheta(H_0) \le C_2 \frac{H_0^{\vartheta + 1}}{\vartheta- 1}
\]
for certain positive constants $C_1,C_2$. 

With the help of these sum estimates we can handle the above sums easily. 
For the first sum we directly have
\[
{\sum}' 
\left( \frac{ \lvert\tilde h_{11}\rvert} { (1 + \lvert\tilde h_{12}\rvert) (1 + \lvert\tilde h_{2 1}\rvert)(1 + \lvert\tilde h_{2 2}\rvert) 
(1 + \lvert\tilde h_{3 1}\rvert)(1 + \lvert\tilde h_{3 2}\rvert) } \right)^{4/3}
\le R_{4/3}(H_0)\,  G_{4/3}^2.
\]
In the treatment of the second sum we need to be a bit more careful. 
We first note that for fixed integers $h_{11}, h_{12}, h_{21}, h_{22}$ the map
\[
(h_{31}, h_{32}) \mapsto (h_{41},h_{42})
\]
is a bijective mapping on $\mathbb{Z}^2$. This follows from the fact that the determinant has
absolute value
\[
\bigl\lvert F_{\ijkl_4-\ijkl_3+1} F_{\ijkl_4-\ijkl_3-1} - F_{\ijkl_4-\ijkl_3}^2\bigr\rvert = 1.
\] 
Hence 
\begin{eqnarray*}
\sum_{\lvert h_{31}\rvert, \lvert h_{32}\rvert\le H_0}  \left( \frac 1 {(1 + \lvert\tilde h_{4 1}\rvert)(1 + \lvert\tilde h_{4 2}\rvert) } \right)^{4/3}
&\le & \sum_{h_{31}, h_{32} \in \mathbb{Z}^2 }  \left( \frac 1 {(1 + \lvert\tilde h_{4 1}\rvert)(1 + \lvert\tilde h_{4 2}\rvert) } \right)^{4/3} \\
&=& G_{4/3}.
\end{eqnarray*}
Consequently we have
\[
{\sum}' 
\left( \frac{ \lvert\tilde h_{11}\rvert} { (1 + \lvert\tilde h_{12}\rvert) (1 + \lvert\tilde h_{2 1}\rvert)(1 + \lvert\tilde h_{2 2}\rvert) 
(1 + \lvert\tilde h_{4 1}\rvert)(1 + \lvert\tilde h_{4 2}\rvert) } \right)^{4/3} 
\le R_{4/3}(H_0) \, G_{4/3}^2
\]
as before. The same upper bound holds for the third sum.
Finally the fourth sum is bounded above by
\[
{\sum}' 
\left( \frac{1} { (1 + \lvert\tilde h_{2 1}\rvert)(1 + \lvert\tilde h_{2 2}\rvert) (1 + \lvert\tilde h_{3 1}\rvert)(1 + \lvert\tilde h_{3 2}\rvert) 
(1 + \lvert\tilde h_{4 1}\rvert)(1 + \lvert\tilde h_{4 2}\rvert) } \right)^{4/3}
\le G_{4/3}^3.
\]
This leads to the upper bound
\[
S_{21} \le K_1^4 R_{4/3}(H_0)^{3/4}\, G_{4/3}^{9/4} \le  K_1^4 H_0^{7/4} \,(3 C_2)^{3/4} \, (9C_1)^{9/4}.
\]
and consequently to 
\[
S_2 \le 8 \Delta^2 K_1^4  H_0^{7/4}\,  (3 C_2)^{3/4} \, (9C_1)^{9/4}.
\]
In the general case we apply similar analysis and obtain
\[
S_2 \le 2d \Delta^2 K_1^{d} \left( R_{\frac d{d-1}}(H_0)\right)^{\frac{d-1}d} \, G_{\frac d{d-1}}^{\frac{(d-1)^2}d}
\le \Delta^2 H_0^{2-\frac 1d}\,  C^d \, d^{2d},
\]
where $C = K_1 \max\{ C_1, C_2\}$. 

We note that it is important that $H_0$ can be chosen of order $1/\Delta$ such that $S_2$ is still small.
This will be important in the sequel.

Next we consider the sum $S_3$:
\[
S_3 = \sum_{{\bf H}\in \mathcal{M}_0,\, \lVert{\bf H}\rVert_\infty >  H_0}  \lvert T_{0,{\bf H}}\rvert.
\]
Again we first consider the special case $d=4$ and suppose as a first step that $\lvert h_{11}\rvert > H_0$.
By an application of H\"older's inequality similar to the above we get
\begin{align*}
&S_{31} = {\sum}'' \lvert T_{0,{\bf H}}\rvert \\
&\le K_1^4 {\sum}''  \frac 1 { \prod_{\ell=1}^4  (1 + \lvert\tilde h_{\ell 1}\rvert)(1 + \lvert\tilde h_{\ell 2}\rvert)    } \\
&\le K_1^4 \left( {\sum}''
\left( \frac{1}{ (1+\lvert\tilde h_{11}\rvert) (1 + \lvert\tilde h_{12}\rvert) (1 + \lvert\tilde h_{2 1}\rvert)(1 + \lvert\tilde h_{2 2}\rvert) 
(1 + \lvert\tilde h_{3 1}\rvert)(1 + \lvert\tilde h_{3 2}\rvert) } \right)^{4/3} \right)^{\frac 14} \\
& \times \left( {\sum}'' 
\left( \frac{1} { (1+|\tilde h_{11}|) (1 + \lvert\tilde h_{12}\rvert) (1 + \lvert\tilde h_{2 1}\rvert)(1 + \lvert\tilde h_{2 2}\rvert) 
(1 + \lvert\tilde h_{4 1}\rvert)(1 + \lvert\tilde h_{4 2}\rvert) } \right)^{4/3} \right)^{\frac 14} \\
& \times \left( {\sum}'' 
\left( \frac{1}{ (1+|\tilde h_{11}|)(1 + \lvert\tilde h_{12}\rvert) (1 + \lvert\tilde h_{3 1}\rvert)(1 + \lvert\tilde h_{3 2}\rvert) 
(1 + \lvert\tilde h_{4 1}\rvert)(1 + \lvert\tilde h_{4 2}\rvert) } \right)^{4/3} \right)^{\frac 14} \\
& \times \left( {\sum}'' 
\left( \frac{1} { (1 + \lvert\tilde h_{2 1}\rvert)(1 + \lvert\tilde h_{2 2}\rvert) (1 + \lvert\tilde h_{3 1}\rvert)(1 + \lvert\tilde h_{3 2}\rvert) 
(1 + \lvert\tilde h_{4 1}\rvert)(1 + \lvert\tilde h_{4 2}\rvert) } \right)^{4/3} \right)^{\frac 14},
\end{align*}
the ${\sum}''$ is the sum over all integers  $h_{11},h_{12},h_{21},h_{22},h_{31},h_{32}$ such that
$\lvert h_{11}\rvert > H_0$, and $h_{41}$, $h_{42}$ are given by (\ref{eqm4142}).

In addition to $G_\vartheta$ and $R_\vartheta(H_0)$ we also define
\[
\overline R_\vartheta(H_0) = \sum_{\lvert h_1\rvert > H_0, h_2\in \mathbb{Z}} \frac {1}{ (1+\lvert h_1 + h_2/\golden\rvert)^\vartheta (1+\lvert h_2 - h_1/\golden\rvert)^\vartheta }.
\]
It is an easy exercise to show that
\[
\overline R_\vartheta(H_0) \le  \frac{C_3}{(\vartheta-1)^2 H_0^{\vartheta-1}}
\]
for some absolute constant $C_3>0$. 

With the help of this notation we can estimate the first sum by
\begin{align*}
&{\sum}''
\left( \frac{1}{ (1+\lvert\tilde h_{11}\rvert) (1 + \lvert\tilde h_{12}\rvert) (1 + \lvert\tilde h_{2 1}\rvert)(1 + \lvert\tilde h_{2 2}\rvert) 
(1 + \lvert\tilde h_{3 1}\rvert)(1 + \lvert\tilde h_{3 2}\rvert) } \right)^{4/3}  \\
&\le\overline R_{4/3}(H_0) \, G_{4/3}^2. 
\end{align*}
For the second (and third) sum we get the same bound. We just note 
\[
\sum_{h_{31}, h_{32} \in \mathbb{Z}}  \left( \frac 1 {(1 + \lvert\tilde h_{4 1}\rvert)(1 + \lvert\tilde h_{4 2}\rvert) } \right)^{4/3}
= G_{4/3}
\]
so that we get
\begin{align*}
&{\sum}'' 
\left( \frac{1} { (1+\lvert\tilde h_{11}\rvert) (1 + \lvert\tilde h_{12}\rvert) (1 + \lvert\tilde h_{2 1}\rvert)(1 + \lvert\tilde h_{2 2}\rvert) 
(1 + \lvert\tilde h_{4 1}\rvert)(1 + \lvert\tilde h_{4 2}\rvert) } \right)^{4/3} \\
&\le R_{4/3}(H_0) \, G_{4/3}^2. 
\end{align*}
The treatment of the fourth sum is slightly different.
Here we use the (trivial) bound
\begin{align*}
\sum_{\lvert h_{11}\rvert> H_0, \lvert h_{12}\rvert\in \mathbb{Z}}  \left( \frac 1 {(1 + \lvert\tilde h_{4 1}\rvert)(1 + \lvert\tilde h_{4 2}\rvert) } \right)^{4/3}
&\le \sum_{h_{11}, h_{12} \in \mathbb{Z}}  \left( \frac 1 {(1 + \lvert\tilde h_{4 1}\rvert)(1 + \lvert\tilde h_{4 2}\rvert) } \right)^{4/3} \\
&= G_{4/3}
\end{align*}
that leads to 
\begin{align*}
& {\sum}'' 
\left( \frac{1} { (1 + \lvert\tilde h_{2 1}\rvert)(1 + \lvert\tilde h_{2 2}\rvert) (1 + \lvert\tilde h_{3 1}\rvert)(1 + \lvert\tilde h_{3 2}\rvert) 
(1 + \lvert\tilde h_{4 1}\rvert)(1 + \lvert\tilde h_{4 2}\rvert) } \right)^{4/3} \\
&\le G_{4/3}^3.
\end{align*}
Summing up, this gives
\[
S_{31} \le \frac{K_1^4 \,(9C_3)^{3/4} \,(9C_1)^{9/4}}{H_0^{1/4}}.
\]

Similarly we can deal with the cases $\lvert h_{12}\rvert> H_0$,  $\lvert h_{21}\rvert> H_0$,  $\lvert h_{22}\rvert> H_0$,  $\lvert h_{31}\rvert> H_0$, and $\lvert h_{32}\rvert> H_0$. 
However, if  $\lvert h_{41}\rvert> H_0$ or   $\lvert h_{42}\rvert> H_0$ we have to argue slightly differently. 
Instead of using the relations~\eqref{eqm4142} we use the equivalent relations:
\begin{eqnarray}
(-1)^{\ijkl_4-\ijkl_1} h_{11} &=& -h_{21}( F_{\ijkl_4-\ijkl_2+1}F_{\ijkl_4-\ijkl_1-1} - F_{\ijkl_4-\ijkl_2}F_{\ijkl_4-\ijkl_1}) \nonumber \\
&&-h_{22}( F_{\ijkl_4-\ijkl_2}F_{\ijkl_4-\ijkl_1-1} - F_{\ijkl_4-\ijkl_2-1}F_{\ijkl_4-\ijkl_1}) \nonumber  \\
&&-h_{31}( F_{\ijkl_4-\ijkl_3+1}F_{\ijkl_4-\ijkl_1-1} - F_{\ijkl_4-\ijkl_3}F_{\ijkl_4-\ijkl_1})  \label{eqm1112} \\
&&-h_{32}( F_{\ijkl_4-\ijkl_3}F_{\ijkl_4-\ijkl_1-1} - F_{\ijkl_4-\ijkl_3-1}F_{\ijkl_4-\ijkl_1}) \nonumber   \\
&&-h_{41}F_{\ijkl_4-\ijkl_1-1} +h_{42}F_{\ijkl_4-\ijkl_1},   \nonumber    \\
(-1)^{\ijkl_4-\ijkl_1} h_{21} &=& -h_{21}( F_{\ijkl_4-\ijkl_2}F_{\ijkl_4-\ijkl_1+1} - F_{\ijkl_4-\ijkl_2+1}F_{\ijkl_4-\ijkl_1}) \nonumber  \\
&&-h_{22}( F_{\ijkl_4-\ijkl_2-1}F_{\ijkl_4-\ijkl_1+1} - F_{\ijkl_4-\ijkl_2}F_{\ijkl_4-\ijkl_1}) \nonumber \\
&&-h_{31}( F_{\ijkl_4-\ijkl_3}F_{\ijkl_4-\ijkl_1+1} - F_{\ijkl_4-\ijkl_3+1}F_{\ijkl_4-\ijkl_1})  \label{eqm1112-2} \\
&&-h_{32}( F_{\ijkl_4-\ijkl_3-1}F_{\ijkl_4-\ijkl_1+1} - F_{\ijkl_4-\ijkl_3}F_{\ijkl_4-\ijkl_1}) \nonumber \\
&&-h_{41}F_{\ijkl_4-\ijkl_1+1} +h_{42}F_{\ijkl_4-\ijkl_1}.   \nonumber
\end{eqnarray}
We can therefore replace the above sums by sums over $h_{21},h_{22},h_{31},h_{32},h_{41}$, and $h_{42}$, where $\lvert h_{41}\rvert > H_0$ and
where $h_{11}$ and $h_{12}$ are given by~\eqref{eqm1112}. 
By applying H\"older's inequality (again) we, thus, obtain the same estimate.
This finally leads to 
\[
S_3 \le 8 \frac{K_1^4 \, (9C_3)^{3/4} \, (9C_1)^{9/4}}{H_0^{1/4}}
\]
in the case $d=4$. For the general case we apply the same procedure and obtain
\[
S_3 \le 2d K_1^d \left( \overline R_{\frac d{d-1}}(H_0)\right)^{\frac{d-1}d} \, G_{\frac d{d-1}}^{\frac{(d-1)^2}d}
\le \frac{\overline C^d d^{2d}}{H_0^{1/d}},
\]
where $\overline C = K_1 \max\{C_1,C_3\}$.

In order to handle $S_4$, we just have to observe that $\lvert d_{h_1,h_2}(\Delta,b)\rvert\le\lvert c_{h_1,h_2}(b)\rvert$. 
Hence 
\[
S_4 \le S_3.
\]
This finally leads to 
\begin{align*}
S_1 &= \sum_{{\bf H}\in \mathcal{M}_0} T_{\bf H}  \\
&= \sum_{{\bf H}\in \mathcal{M}_0} T_{0,{\bf H}} + O(S_2+S_3 + S_4) \\
&= \sum_{{\bf H}\in \mathcal{M}_0} T_{0,{\bf H}}  + O\left( \Delta^2 H_0^{2-\frac 1d} C^d d^{2d} 
+ \frac{\overline C^d d^{2d}}{H_0^{1/d}} \right),
\end{align*}
which proves (\ref{eqS1est}) by choosing $H_0 = 1/\Delta$ and $\overline K = \max\{ C, \overline C \}$.

The next main step is to consider those $\bf H$ for which $\theta = {\bf V\cdot H}\ne 0$. 
We distinguish between $\bf H$ with $\lVert{\bf H} \rVert_\infty \le H_1$ and 
those $\bf H$ with $\rVert{\bf H} \rVert_\infty  > H_1$, where $H_1 \ge 1/\Delta$ will be suitably chosen.

In order to handle the second case we consider the sum
\[
S_5 = \sum_{ \lVert {\bf H} \rVert_\infty > H_1 } \lvert T_{\bf H}\rvert.
\]
By using the estimate~\eqref{eqdcest} we observe that
\begin{equation}\label{eqobserve1}
\sum_{h_1,h_2\in \mathbb{Z}} \lvert d_{h_1,h_2}(\Delta,b)\rvert \le K_2 \left( \log(1/\Delta) + 1 \right)^2
\end{equation}
and
\begin{equation}\label{eqobserve2}
\sum_{\lvert h_1\rvert> H_1 ,\, h_2\in \mathbb{Z}} \lvert d_{h_1,h_2}(\Delta,b)\rvert \le \frac{K_2}{\Delta H_1}  \left( \log(1/\Delta) + 1 \right)
\end{equation}
for a universal constant $K_2> 0$. The essential observation is that $\tilde h_1 = h_1 + \frac 1\gamma h_2$ and
$\tilde h_2 = h_1 - \frac 1\gamma h_1$ form a lattice in $\mathbb{R}^2$ and, thus, the value distribution of
\[
(1 + |\tilde h_1|, 1 + \Delta|\tilde h_1|,  1 + |\tilde h_2|, 1 + \Delta|\tilde h_2|)
\]
is comparable with the value distribution of 
\[
(1 + |\ell_1|, 1 + \Delta|\ell_1|,  1 + |\ell_2|, 1 + \Delta|\ell_2|)
\]
if $(\ell_1, \ell_2)$ vary over $\mathbb{Z}^2$. Since 
\begin{align*}
\sum_{\ell_1 \in \mathbb{Z}} \frac 1{(1 + |\ell_1|)(1 + \Delta|\ell_1|) } 
&\le \sum_{|\ell_1| \le 1/\Delta} \frac 1{1 + |\ell_1|} \\
&+  \frac 1{\Delta}  \sum_{|\ell_1| >  1/\Delta}    \frac 1{\ell_1^2} \\
&\ll \log(1/\Delta) + 1 
\end{align*}
we immediately obtain the upper bound (\ref{eqobserve1}). The derivation of (\ref{eqobserve2}) is more involved.
In particular one has to take care of the value distrbution of $(\tilde h_1, \tilde h_2)$ if
$|h_1| > H_1$ and to distinguish between the cases $H_1 \le 1/\Delta$ and $H_1 > 1/\Delta$. For example, if
$H_1 > 1/\Delta$ one has to estimate (among other sums) the sum
\begin{align*}
&\sum_{|\ell_1| > H_1,\, \ell_2 \in \mathbb{Z} } \frac 1{(1 + |\ell_1|)(1 + \Delta|\ell_1|) (1 + |\ell_2|)(1 + \Delta|\ell_2|) }  \\
&\qquad \ll \sum_{|\ell_1| > H_1}  \frac 1{\Delta \ell_1^2}   \sum_{\ell_2 \in \mathbb{Z}} \frac 1{(1 + |\ell_2|)(1 + \Delta|\ell_2|) } \\
&\qquad \ll  \frac 1{\Delta H_1}  \left(  \log(1/\Delta) + 1  \right)
\end{align*}
which corresponds directly to the upper bound (\ref{eqobserve2}).

With the help of these estimates one directly obtains the upper bound
\[
S_5 \le 2d \frac{K_2^d}{\Delta H_1}  \left( \log(1/\Delta) + 1 \right)^{2d-1}.
\]

For the first case we consider the exponential sums $\sum_{p\le x} e(\theta p)$, where
$\theta = {\bf V\cdot H}\ne 0$ and  $\lVert{\bf H} \rVert_\infty \le H_1$.
It is easy to find an upper bound for $\theta$:
\begin{equation}\label{eqthetaupperbound}
\lvert \theta\rvert \ll \frac{H_1}{\golden^{\ijkl_1}} \ll H_1 e^{- \log\golden\, L^\nu}.
\end{equation}
If $d=1$ it is easy to give a lower bound, too:
\[
\lvert \theta\rvert = \frac{\lvert h_{11} \golden + h_{12}\rvert}{\golden^{\ijkl_1+1}} \gg \frac 1{H_1 \golden^{\ijkl_1}} \gg \frac 1{H_1} \frac{e^{ \log\golden\, L^\nu}}x,
\]
since we have the lower bound
\begin{equation}\label{eqDiophbound}
\lvert h_1\golden + h_2\rvert \ge \frac 1{\lvert h_1\rvert+\lvert h_2\rvert} 
\end{equation}
for integer pairs $(h_1,h_2)\ne (0,0)$.

It is, however, more involved to get a useful lower bound for $d>1$.
We consider the (relatively simple) case $d=2$ first:
\begin{align*}
\theta &= \frac{h_{11}\golden + h_{12}}{\golden^{\ijkl_1+1}} + \frac{h_{21}\golden + h_{22}}{\golden^{\ijkl_2+1}} \\
&= \frac{ (h_{11} F_{\ijkl_2-\ijkl_1+1} + h_{12} F_{\ijkl_2-\ijkl_1} + h_{21}) \golden + (h_{11} F_{\ijkl_2-\ijkl_1} + h_{12} F_{\ijkl_2-\ijkl_1-1} + h_{22})}
{\golden^{\ijkl_2+1}},
\end{align*}
where we know that 
\[
\left(  h_{11} F_{\ijkl_2-\ijkl_1+1} + h_{12} F_{\ijkl_2-\ijkl_1} + h_{21}, \, h_{11} F_{\ijkl_2-\ijkl_1} + h_{12} F_{\ijkl_2-\ijkl_1-1} + h_{22}\right) \ne (0,0).
\]
If $h_{11} = h_{12} = 0$ or $h_{21} = h_{22} = 0$ then we are actually in the case $d=1$. So we can skip these cases.
If $(h_{11},h_{12})\ne (0,0)$ and $(h_{21},h_{22})\ne (0,0)$ we distinguish between two cases.
Suppose first that 
\[
\golden^{\ijkl_2-\ijkl_1} \le 2 \golden^6 H_1^2.
\]
Then we get (also with the help of~\eqref{eqDiophbound})
\begin{align*}
\lvert \theta\rvert &\ge \frac 1{ \golden^{\ijkl_2+1}\left( \lvert h_{11} F_{\ijkl_2-\ijkl_1+1} + h_{12} F_{\ijkl_2-\ijkl_1} + h_{21}  \rvert 
   + \lvert h_{11} F_{\ijkl_2-\ijkl_1} + h_{12} F_{\ijkl_2-\ijkl_1-1} + h_{22} \rvert \right)}  \\
 &\ge \frac 1{ \golden^{\ijkl_2+1} H_1 (F_{\ijkl_2-\ijkl_1+1} + 2 F_{\ijkl_2-\ijkl_1} + F_{\ijkl_2-\ijkl_1} + 2)  }  \\
 &\gg \frac 1 { \golden^{\ijkl_2+1} H_1^3 } \\
 & \gg \frac 1{H_1^3} \frac{e^{\log\golden\, L^\nu}}x.
\end{align*}
Secondly, suppose that
\begin{equation}\label{eq2ndcond}
\golden^{\ijkl_2-\ijkl_1} > 2 \golden^6 H_1^2.
\end{equation}
Here we take a closer look at the integers
\[
\overline h_1 =   h_{11} F_{\ijkl_2-\ijkl_1+1} + h_{12} F_{\ijkl_2-\ijkl_1} \quad\mbox{and}\quad
\overline h_2 =   h_{11} F_{\ijkl_2-\ijkl_1} + h_{12} F_{\ijkl_2-\ijkl_1-1}.
\]
By applying Lemma~\ref{Leprelim1} and Lemma~\ref{Leprelim2} it follows that 
$\delta_{\ell}(\lvert\overline h_1\rvert) = 0$ for $\ell < \ijkl_2-\ijkl_1-\log_\golden H_1 -4$ and
that $\delta_{\ell}(\lvert\overline h_2\rvert) = 0$ for $\ell < \ijkl_2-\ijkl_1-\log_\golden H_1 -5$.
By~\eqref{eq2ndcond} we also have the bound $\ijkl_2-\ijkl_1-\log_\golden H_1 -5 > \log_\golden H_1 \ge 2$.
Observe next that $\overline h_1$ and $\overline h_2$ are very similar.
The only difference is the shift in the index of the Fibonacci numbers.
Since the least significant digits of $\overline h_1$ and $\overline h_2$ are zero and not affected
by $h_{11}$ and $h_{12}$ it follows that the Zeckendorf expansions of $\overline h_1$ and $\overline h_2$
can be computed just on the digit level. Hence, the corresponding digits are just shifted:
\[ 
\delta_\ell(\overline h_2) = \delta_{\ell+1}(\overline h_1)  \quad\mbox{or}\quad
\delta_\ell(-\overline h_2) = \delta_{\ell+1}(-\overline h_1) \quad \mbox{ for all $\ell$.}
\]
In particular they are both positive or both negative, and we have the trivial lower bounds
\[
\lvert \overline h_1\rvert\ge F_{\ijkl_2-\ijkl_1 - \lfloor \log_\golden H_1 \rfloor -5 }
\quad\mbox{and}\quad
\lvert\overline h_2\rvert\ge F_{\ijkl_2-\ijkl_1 - \lfloor \log_\golden H_1 \rfloor -6 }.
\]
By~\eqref{eq2ndcond} this also implies that $\lvert\overline h_1\rvert > 2H_1$ and  $\lvert\overline h_2\rvert > 2H_1$.
Consequently,
\begin{align*}
\lvert \theta\rvert &= \left\lvert\frac{ (h_{11} F_{\ijkl_2-\ijkl_1+1} + h_{12} F_{\ijkl_2-\ijkl_1} + h_{21}) \golden + (h_{11} F_{\ijkl_2-\ijkl_1} + h_{12} F_{\ijkl_2-\ijkl_1-1} + h_{22})}
{\golden^{\ijkl_2+1}} \right\rvert \\
&\ge \frac 12 \frac{\lvert\overline h_1 \golden + \overline h_2\rvert}{{\golden^{\ijkl_2+1}}} \\
& \gg \frac{ \golden^{\ijkl_2-\ijkl_1 - \lfloor \log_\golden H_1\rfloor }  }{\golden^{\ijkl_2+1}}   \\
& \gg \frac 1 {H_1 \golden^{\ijkl_1}} \\
& \gg  \frac 1{H_1} \frac{e^{\log\golden\, L^\nu}}x.
\end{align*}

Now suppose that $d\ge 3$. Here we assume that all subsums 
\begin{equation}\label{eqsubsums}
\golden \sum_{\ell=1}^{\ell_0}  \left( h_{\ell 1} F_{\ijkl_d-\ijkl_\ell+1} + h_{\ell 2} F_{\ijkl_d-\ijkl_\ell} \right)  + 
\sum_{\ell=1}^d  \left( h_{\ell 1} F_{\ijkl_d-\ijkl_\ell} + h_{\ell 2} F_{\ijkl_d-\ijkl_\ell-1} \right)
\end{equation}
are non-zero ($1\le \ell_0 \le d$). Otherwise we could reduce $d$ to a smaller number. 

First, if 
\begin{equation}\label{eqjdj1inequ}
\golden^{\ijkl_d-\ijkl_1} \le H_1^{2(d-1)} (2 \golden^{5+2d})^{d-1},
\end{equation}
we have
\begin{align*}
\left\lvert\sum_{\ell=1}^d  \left( h_{\ell 1} F_{\ijkl_d-\ijkl_\ell+1} + h_{\ell 2} F_{\ijkl_d-\ijkl_\ell} \right) \right\rvert &\ll
H_1^{2d-1} (2 \golden^{5+2d})^{d-1}, \\
\left\lvert\sum_{\ell=1}^d  \left( h_{\ell 1} F_{\ijkl_d-\ijkl_\ell} + h_{\ell 2} F_{\ijkl_d-\ijkl_\ell-1} \right) \right\rvert &\ll
H_1^{2d-1} (2 \golden^{5+2d})^{d-1},
\end{align*}
and consequently
\begin{align*}
\lvert \theta\rvert &= \left\lvert\frac 1{\golden^{\ijkl_d+1}} \left( \golden \sum_{\ell=1}^d  \left( h_{\ell 1} F_{\ijkl_d-\ijkl_\ell+1} + h_{\ell 2} F_{\ijkl_d-\ijkl_\ell} \right)  + \sum_{\ell=1}^d  \left( h_{\ell 1} F_{\ijkl_d-\ijkl_\ell} + h_{\ell 2} F_{\ijkl_d-\ijkl_\ell-1} \right) \right)\right\rvert \\
&\gg \frac 1{H_1^{2d-1} (2 \golden^{5+2d})^{d-1} \golden^{\ijkl_d+1} }  \\
&\gg \frac 1{H_1^{2d-1} (2 \golden^{5+2d})^{d-1} } \frac{e^{\log\golden\, L^\nu}}x.
\end{align*}
Conversely, if~\eqref{eqjdj1inequ} does not hold then there is $\ell_0 > 1$ such that
\begin{equation}\label{eqjelljellest}
\golden^{\ijkl_{\ell_0+1}-\ijkl_{\ell_0}} > H_1^{2} 2 \golden^{5+2d}.
\end{equation}
We now set
\begin{align*}
\overline h_1 &= \sum_{\ell=1}^{\ell_0}  \left( h_{\ell 1} F_{\ijkl_d-\ijkl_\ell+1} + h_{\ell 2} F_{\ijkl_d-\ijkl_\ell} \right), \\
\overline h_2 &= \sum_{\ell=1}^{\ell_0}  \left( h_{\ell 1} F_{\ijkl_d-\ijkl_\ell} + h_{\ell 2} F_{\ijkl_d-\ijkl_\ell-1} \right).
\end{align*}
By applying Lemma~\ref{Leprelim1} and Lemma~\ref{Leprelim2} several times it follows that
$\delta_\ell(\lvert\overline h_1\rvert) = 0$ for $\ell < \ijkl_d-\ijkl_{\ell_0}-\log_\golden H_1 -2d-2$ and
that $\delta_\ell(\lvert\overline h_2\rvert) = 0$ for $\ell < \ijkl_d-\ijkl_{\ell_0}-\log_\golden H_1 -2d-3$.
We also have 
\[
\ijkl_d-\ijkl_{\ell_0}-\log_\golden H_1 -2d-3 \ge \ijkl_d-\ijkl_{\ell_{0}+1} + \log_\golden H_1 \ge 2.
\]
Furthermore  the Zeckendorf expansions
of $\overline h_1$ and $\overline h_2$ are (again) just shifted:
\[ 
\delta_\ell(\overline h_2) = \delta_{\ell+1}(\overline h_1)  \quad\mbox{or}\quad
\delta_\ell(-\overline h_2) = \delta_{\ell+1}(-\overline h_1).
\]
So they are both positive or both negative. It is impossible that they are both zero since we have
assumed that (\ref{eqsubsums}) holds. Furthermore we have the trivial lower bounds
\[
\lvert\overline h_1\rvert\ge F_{\ijkl_d-\ijkl_{\ell_0}- \lfloor \log_\golden H_1 \rfloor -3d-1 }
\quad\mbox{and}\quad
\lvert\overline h_2\rvert\ge F_{\ijkl_d-\ijkl_{\ell_0} - \lfloor \log_\golden H_1 \rfloor -3d-2}.
\]
By~\eqref{eqjelljellest} this also implies
\[
\lvert\overline h_1\rvert\ge 2 H_1 \golden^2 F_{\ijkl_d-\ijkl_{\ell_0+1} +2} \ge 2 \left\lvert 
\sum_{\ell=\ell_0+1}^{d}  \left( h_{\ell 1} F_{\ijkl_d-\ijkl_\ell+1} + h_{\ell 2} F_{\ijkl_d-\ijkl_\ell} \right)  \right\rvert
\]
and consequently 
\[
\left\lvert \sum_{\ell=1}^{d}  \left( h_{\ell 1} F_{\ijkl_d-\ijkl_\ell+1} + h_{\ell 2} F_{\ijkl_d-\ijkl_\ell} \right) \right\rvert \ge 
\frac {\lvert\overline h_1\rvert} 2.
\]
Similarly we have
\[
\left\lvert \sum_{\ell=1}^{d}  \left( h_{\ell 1} F_{\ijkl_d-\ijkl_\ell} + h_{\ell 2} F_{\ijkl_d-\ijkl_\ell-1} \right) \right\rvert \ge 
\frac {\lvert\overline h_2\rvert} 2,
\]
which gives
\begin{align*}
\lvert \theta\rvert &\ge \frac 12 \frac {\lvert\overline h_1\rvert\golden + \lvert\overline h_2\rvert}{\golden^{\ijkl_d+1}} \\
&\gg \frac {F_{\ijkl_d-\ijkl_{\ell_0}- \lfloor \log_\golden H_1 \rfloor -3d-1 } } {\golden^{\ijkl_d+1}} \\
&\gg \frac 1 {\golden^{\ijkl_d+1}} \\
&\gg \frac{e^{\log\golden\, L^\nu}}x.
\end{align*}

Summing up we have the upper bound (\ref{eqthetaupperbound}) for $\theta$ and the lower bound
\begin{equation}\label{eqthetalowerbound}
\lvert \theta\rvert \gg \frac 1{H_1^{2d-1} (2 \golden^{5+2d})^{d-1} } \frac{e^{\log\golden\, L^\nu}}x.
\end{equation}
With the help of Lemma~\ref{Leexpsumprimes} we obtain the uniform upper bound
\[
\sum_{p\le x} e({\bf V\cdot H}\, p) \ll  x (\log x)^3  H_1^{d-\frac 12} (2 \golden^{5+2d})^{(d-1)/2}   {e^{ -\frac 12 \log\golden\, L^\nu}}.
\]
Finally we use the upper bound 
\[
\sum_{\lVert {\bf H} \rVert_\infty \le H_1} T_{\bf H} \le K_2^d \left( \log (1/\Delta) + 1 \right)^{2d},
\]
and we obtain 
\begin{align*}
S_6 &= \sum_{\lVert{\bf H}\rVert_\infty \le H_1,\, {\bf V\cdot H} \ne 0} T_{\bf H} \sum_{p\le x} \e({\bf V\cdot H}\,p)  \\
&\ll  x (\log x)^3   K_2^d \left( \log (1/\Delta) + 1 \right)^{2d} 
{ H_1^{d-\frac 12} (2 \golden^{5+2d})^{(d-1)/2}  }  {e^{ -\frac 12 \log\golden\, L^\nu}}.
\end{align*}

Putting everything together leads to
\begin{align*}
& \frac 1{\pi(x)} \# \bigl\{ p \le x : 
\delta_{\ijkl_1}(p) = \nu_1,\ldots, \delta_{\ijkl_d}(p) = \nu_d \bigr\} = \sum_{{\bf V\cdot H} = 0}   T_{0,{\bf H}} \\
& \qquad + O\left( d \Delta + d e^{-c_3 L^\nu}   \right) \\
& \qquad +  O\left(  \Delta^{\frac 1d} \overline K^d d^{2d}   \right) \\
& \qquad + O\left( 2d \frac{K_2^d}{\Delta H_1}  \bigl( \log(1/\Delta) + 1 \bigr)^{2d-1} \right) \\
& \qquad  +  O\left( (\log x)^4   K_2^d \bigl( \log (1/\Delta) + 1 \bigr)^{2d}  
 H_1^{d-\frac 12} (2 \golden^{5+2d})^{(d-1)/2} {e^{ -\frac 12 \log\golden\, L^\nu}}   \right),
\end{align*}
where we have to assume that $\Delta \ge c_0 \golden^{-\ijkl_1} \ge c_0 e^{-\log \golden\, L^\nu}$ (for some constant $c_0>0$) and 
$H_1 \ge 1/\Delta$. 

We recall that $0 < \nu < \frac 12$. 
By assumption we have $0 < \kappa < \rho < \frac 13 \nu$. 
We now choose $\delta$ and $\beta$ with 
\[
\kappa + \rho < \delta < \beta < \frac 23 \nu.
\]
By this choice we certainly have
\[
0 < 2\kappa < \delta < \beta < \nu - \kappa.
\]
We then set
\[
\Delta = e^{-L^\delta}, \quad H_1 = e^{L^\beta},
\]
and assume that
\[
d \le L^\kappa.
\]
It follows that
\begin{align*}
d \Delta + d e^{-c_3 L^\nu} &\ll L^\kappa \left( e^{-L^\delta} + e^{-c_3L^\nu} \right) 
\ll e^{-\frac 12 L^\delta}, \\
\Delta^{\frac 1d} \overline K^d d^{2d}  &\ll e^{-L^{\delta-\kappa} + 2\kappa \log L\, L^\kappa + \log \overline K L^\kappa } 
\ll e^{-\frac 12 L^{\delta-\kappa}}, \\
2d \frac{K_2^d}{\Delta H_1}  \left( \log(1/\Delta) + 1 \right)^{2d-1} &\ll
e^{(\delta+1)\log L\, L^\kappa + \log K_2 L^\kappa + L^\delta - L^\beta } \ll e^{-\frac 12 L^\beta},
\end{align*}
\begin{align*}
\hspace{3em}&\hspace{-3em}
(\log x)^4   K_2^d \left( \log (1/\Delta) + 1 \right)^{2d}  
 H_1^{d-\frac 12} (2 \golden^{5+2d})^{(d-1)/2} {e^{ -\frac 12 \log\golden\, L^\nu}} \\
\quad  &\ll L^4 e^{\log K_2 L^\kappa + (\delta+1) \log L\, L^\kappa + L^{\beta + \kappa} + O(L^{2\kappa}) - \frac 12 \log \golden L^\nu }
 \ll e^{-\frac 13 \log \golden L^\nu }.
\end{align*}
By this choice the dominating term is the second one.
Since $\rho< \delta - \kappa$ we also have
\[
e^{-\frac 12 L^{\delta-\kappa}} \ll e^{-L^\rho},
\]
so that $e^{-L^\rho}$ dominates all error terms. 

What remains is to show the relation~\eqref{eqPrep}.
This will then complete the proof of Lemma~\ref{Le6}.
For this purpose we do all the computations again but we replace the statistics from prime $p\le x$ to all non-negative integers $n \le x$.
This means that we consider the numbers
\[
\# \bigl\{ n \le x : \delta_{\ijkl_1}(n) = \nu_1, \ldots, \delta_{\ijkl_d}(n) = \nu_d\bigr\}
\]
instead of the numbers $\# \{ p \le x : \delta_{\ijkl_1}(p) = \nu_1, \ldots, \delta_{\ijkl_d}(p) = \nu_d\}$. 
Technically this means that we replace the exponential sums $\sum_{p\le x} e({\bf V\cdot H}\, p)$ by
the exponential sums $\sum_{n\le x} e({\bf V\cdot H}\, n)$.
Again we distinguish between the cases ${\bf V\cdot H} = 0$ and ${\bf V\cdot H} \ne 0$.
In the first case the exponential sums are trivial, whereas in the second case we can use the bound
\[
\sum_{n\le x} e(\theta n)  \ll \frac 1{\lVert\theta\rVert},
\]
which gives slightly better upper bounds than Lemma~\ref{Leexpsumprimes}.
Summing up we obtain in completely the same way
\[
 \frac 1{x} \# \bigl\{ n \le x : 
\delta_{\ijkl_1}(n) = \nu_1,\ldots, \delta_{\ijkl_d}(n) = \nu_d \bigr\} = \sum_{{\bf V\cdot H} = 0}   T_{0,{\bf H}} 
 + O\left( e^{-L^\rho} \right).
\]
By comparing this with Lemma~\ref{LeDS02} we immediately deduce the relation~\eqref{eqPrep}.
\end{proof}

\def\proThree{\ref{Pro3}}
\section{Proof of Proposition~\proThree}

Finally, we can complete the proof of Proposition~\ref{Pro3}.
By Taylor's theorem we have for every integer $D>0$ and real $u$
\[
e^{iu} = \sum_{0\le d < D} \frac {(iu)^d}{d!} + O\left( \frac{\lvert u\rvert^D}{D!} \right).
\]
Consequently we have for all random variables $X$ and $Y$
\begin{align*}
\mathbb{E} e^{it X} -  \mathbb{E} e^{it Y} &=
\sum_{d< D} \frac {(it)^d}{d!}\left( \mathbb{E}\, X^d - \mathbb{E}\, Y^d \right)\\
&+ O\left( \frac {\lvert t\rvert^D}{D!}\left| \mathbb{E}\, \lvert X\rvert^D - \mathbb{E}\, \lvert Y\rvert^D \right|
+ 2 \frac {\lvert t\rvert^D}{D!} \mathbb{E}\, \lvert Y\rvert^D \right).
\end{align*}
In particular we will apply this for $X = (T_x - L'\mu)/ (L'\sigma^2)^{1/2}$ and
$Y = (\overline T_x - L'\mu)/ (L'\sigma^2)^{1/2}$. Further we set
$D = \lfloor L^\kappa \rfloor$ for some real $\kappa$ with $0<\kappa < \rho < \frac 13 \nu$.
(Moreover, we assume without loss of
generality that $D$ is even, otherwise we consider $D =  \lfloor L^\kappa \rfloor - 1$.) 
We also suppose that $\lvert t\rvert\le L^\tau$ with $0<\tau < \frac 12 \kappa$.
Hence, by applying Lemma~\ref{Le7} we get
\begin{align*}
\sum_{1\le d\le D} \frac {\lvert t\rvert^d}{d!}\left| \mathbb{E}\, X^d - \mathbb{E}\, Y^d \right|
&\ll \lvert t\rvert \sum_{d\le D} \frac {L^{\tau (d-1)}}{d!}
 e^{-\frac 12 L^{\rho}} \\
&\ll \lvert t\rvert\, e^{ L^\tau - \frac 12 L^\rho }\\
&\ll \lvert t\rvert e^{- \frac 13 L^{\rho}}
\end{align*}
for sufficiently large $x$.

The term 
\[
\frac {\lvert t\rvert^D}{D!}\left\lvert \mathbb{E}\, \lvert X\rvert^D - \mathbb{E}\, \lvert Y\rvert^D \right\rvert
\]
has the same upper bound (recall that we have assumed that $D$ is even).

Finally we have to get some bound for the moments $\mathbb{E}\, \lvert Y\rvert^D$.
By Lemma~\ref{LeCLT} we have for (even) $D= \lfloor L^\kappa \rfloor$ 
and $\lvert t\rvert\le L^\tau$ (where $\tau < \kappa/2$) 
\begin{align*}
\frac {\lvert t\rvert^D}{D!} \mathbb{E}\, \lvert Y\rvert^D & \ll \lvert t\rvert \frac{L^{\tau (D-1)}}{D^{D/2}e^{-D/2}\sqrt{\pi D}} \\
&\ll \lvert t\rvert e^{\tau L^\kappa \log L - \frac 12 \kappa L^\kappa \log L - \frac 12 L^\kappa} \\
&\ll \lvert t\rvert e^{-(\frac 12 \kappa - \tau)L^\kappa \log L}.
\end{align*}
This completes the proof of Proposition~\ref{Pro3}.

\chapter{Extensions and Open Problems}\label{chapter:openproblems}

\section{More questions on the Zeckendorf sum-of-digits function}
In the Introduction, we stated the following straightforward and seemingly intractable problem.
\begin{problem}\label{pb_all_kge1}
Prove that for all $k\geq 1$ there is a prime number $p$ such that $\sz(p)=k$.
\end{problem}
\begin{problem}\label{pb_infinite}
Prove that there is a $k$ such that there exist infinitely many primes $p$ satisfying $\sz(p)=k$, or prove that there is no such $k$.
\end{problem}
Of course, ``straightforward'' is an understatement, as considerable effort would be needed in order to keep track of the constants.
We commented on this after Theorem~\ref{Th1}. However, even if all constant computations can be worked out
the problem might be difficult, too,  because of computational limitations (for the finitely many left cases).
``Intractable'', however, seems to be the correct expression concerning the difficulty of Problem~\ref{pb_infinite}.

\medskip

Mauduit and Rivat~\cite{MR2010} not only handled the sum of digits of prime numbers, but also the sum of digits of squares~\cite{MR2009}.
Therefore the following problem is not hard to come up with.
\begin{problem}\label{pb_sz_squares}
Prove that the Zeckendorf sum of digits of squares is uniformly distributed in residue classes.
\end{problem}
A difference to the base-$q$ expansion is the observation that we do not expect the existence of ``exceptional'' residue classes.
Certainly, the base-$4$ sum-of-digits function function of squares is not uniformly distributed modulo $3$:
\[s_{4}(n^2)\equiv n^2\bmod 3,\]
and squares are never congruent to $2$ modulo $3$.
Meanwhile, the sequence $\bigl(n^2\golden\bigr)_n$ is uniformly distributed modulo $1$, and so the initial digits behave ``randomly'';
it seems reasonable to imagine that this destroys the bias that is present in the base-$q$ case (see~\cite[Th\'eor\`eme~3]{MR2009}).
Concerning primes, an analogous situation arises: in~\cite[Th\'eor\`eme~1]{MR2010}
certain residue classes have to be excluded, but not in our Theorem~\ref{Th4}.
It appears that new ideas are needed in order to handle Problem~\ref{pb_sz_squares}.
The simple fact, used at a crucial position in~\cite{MR2009}, that the lowest $k$ digits of $mq^k$ in base $q$ are zero, is not easily translated to the Zeckendorf case.

\medskip

We could also ask for the joint distribution of the Zeckendorf- and the base-$q$ sum-of-digits functions (extending~\cite[Chapter~5]{Spiegelhofer2014} in the spirit of~\cite{D2001}, and complementing the recent preprint~\cite{VZ2021}).
\begin{problem}
Prove a local limit law for the joint distribution of $\sz(n)$ and $s_q(n)$. 
\end{problem}

\section{Different systems of numeration, and substitutions}
The Zeckendorf expansion is a special Ostrowski expansion, with base $\alpha=\golden-1$.
This immediately demands for the following line of generalization.
\begin{problem}\label{pb_ostrowski}
Prove analoga of Theorems~\ref{Th1}--\ref{Th4} for the $\alpha$-Ostrowski expansion for certain classes of irrational $\alpha\in(0,1)$.
\end{problem}

A different line is represented by~$\beta$-expansions~\cite{R1957,P1960}.
\begin{problem}\label{pb_beta}
Prove analoga of Theorems~\ref{Th1}--\ref{Th4} for the $\beta$-sum-of-digits function, for certain Pisot numbers $\beta$.
\end{problem}

Concerning Problem~\ref{pb_ostrowski}, the Ostrowski expansion allows for one-dimensional detection using $n\alpha$-sequences, as we noted in the introduction (after~\eqref{eqn_central_motivation}).
For the detection of a block of digits with indices in $[a,b)$ something new will have to be found.

Concerning Problem~\ref{pb_beta}, \emph{fractals} will appear on the stage;
this will introduce further considerable technical complications.
An example is given by the \emph{golden ratio base}, where a real number is written as a finite sum of integer powers of $\golden$.
While the restriction on the digits is analogous to the Zeckendorf numeration (no consecutive powers of $\golden$ appear, and only $0,1$ as digits),
the representations of the same integer in the two numeration systems are quite different.
We refer to the paper~\cite{D2020} by Dekking for recent work on the golden ratio numeration system, and in particular the corresponding sum-of-digits function.

However, we are confident that our method is flexible enough to yield nontrivial results for Problems~\ref{pb_ostrowski} and~\ref{pb_beta} as soon as we can control the analytical detection of digits in the respective numeration system.

\medskip

A huge uncharted territory is represented by subsequences of \emph{morphic sequences}, indexed by the sequence of primes.
This problem was addressed in the Preface and the Introduction, after~\eqref{eq_morphic}.
We ask for the following broad generalization of a result by the second author~\cite{Muellner2017} on the Sarnak conjecture.
\begin{problem}\label{pb_sarnak_morphic}
Prove that
\begin{equation}\label{eqn_sarnak_morphic}
\sum_{n\leq N}m(n)\mu(n)=o(N)
\end{equation}
for all complex valued morphic sequences $m(n)$.
\end{problem}
Partial results are known, for example, Sarnak's conjecture holds for substitutions with \emph{discrete spectrum}~\cite{FKL2018}.
An example is given by the fixed point $\mathtt t$ of the \emph{Tribonacci substitution}
\[
\mathtt 1\mapsto \mathtt 1\mathtt 2,\quad \mathtt 2\mapsto \mathtt 1\mathtt 3,\quad \mathtt 3\mapsto\mathtt 1
\]
(see~\cite{BS2005, Siegel2004} for more information).
We are confident that we can also prove a prime number theorem for this sequence (that is, $\Lambda$ takes the place of $\mu$ in~\eqref{eqn_sarnak_morphic}), using the discreteness of the spectrum and the classical \emph{Rauzy fractal}~\cite{Rauzy1982}.
In contrast, proving a prime number theorem for general morphic sequences is wide open.

\medskip
In analogy to the case of the Zeckendorf expansion (where the lowest digit is given by the Fibonacci word),
the Tribonacci word can be recovered from the lowest two digits of the \emph{Tribonacci expansion} of an integer~\cite{Sirvent1999}, which is defined as follows.
A positive integer can be written, in a unique way, as a sum of pairwise different \emph{Tribonacci numbers}
\[T_0 = T_1 = 0,\quad T_2 = 1,\quad T_{n+3} = T_{n+2} + T_{n+1} + T_n\quad\mbox{for }n\geq 0,\]
where no three consecutive Tribonacci numbers are used, and summation starts with the index $3$.
Clearly we can also define the sum-of-digits function for this numeration system and try to prove variants of our theorems.
\begin{problem}\label{pb_tribonacci}
Prove analoga to Theorems~\ref{Th1}--\ref{Th4} for the Tribonacci sum-of-digits function.
\end{problem}
Now that we can handle the Fibonacci case,
we think that this particular problem is a promising, and very attractive, line of research.
We believe that our method is applicable to this situation too;
this will involve, among other things, certain subsets of the Rauzy fractal,
which will play the role of the intervals in the (Ostrowski) one-dimensional detection.
In analogy to the Zeckendorf case, we will have to consider three-dimensional cylinders with the Rauzy fractal as its base, in order to detect Tribonacci digits in an interval $[a,b)$. These cylinders will play the role of our parallelograms in the two-dimensional detection procedure.
Clearly, one could also consider general \emph{linear recurrent numeration systems}, but things are not getting easier in this general setting.
In any case we need suitable procedures for detecting digits in an analytical way, which would be the first step.

\section{Zeckendorf automatic sequences}
An intermediate step on the path towards Problem~\ref{pb_sarnak_morphic} that we would like to pay special attention to is represented by \emph{Zeckendorf block-additive functions}.
We call a function $f$ \emph{Zeckendorf block-additive}, if there exists an integer $m \in \N$ and a function $F: \{0,1\}^m \to \N$ such that
\begin{align*}
	f(n) = \sum_{\ijkl=2}^{L-m+1} F(\delta_\ijkl, \ldots, \delta_{\ijkl+m-1}),
\end{align*}
where the digits $\delta_{\ijkl}$ are given by~\eqref{eqn_Zeckendorf_rep}.
Of course, the Zeckendorf sum-of-digits function is Zeckendorf block-additive, with $m = 1$ and $F(x) = x$.
Again, we can ask the same questions as above.
\begin{problem}\label{pb_zeckendorf_blockadditive}
Generalize Theorems~\ref{Th1}--\ref{Th4} to Zeckendorf block-additive functions $f$.
\end{problem}
We will need certain conditions on the function $f$ (let us, exceptionally, use the adjective ``natural'' for these conditions), and of course $\mu$ and $\sigma$ in Theorem~\ref{Th2} will be different.

\medskip

The next step consists in introducing \emph{Zeckendorf-automatic} sequences. 
We recall that Theorem~\ref{Th4} is concerned with counting primes $p$ such that $\sz(p) \equiv a \bmod m$.
The sequence $a=(\sz(n) \bmod m)_{n\geq 0}$ is an example of a Zeckendorf-automatic sequence --- there exists a deterministic finite automaton with output (DFAO) that accepts exactly the Zeckendorf expansions of nonnegative integers and outputs $a_n$ when fed with the Zeckendorf expansion on $n$.
Being Zeckendorf-automatic is equivalent to being morphic with some special condition on the substitution.
An example is given by the substitution corresponding to $(\sz(n) \bmod 2)_{n\geq 0}$, which is given by~\eqref{eq_morphic}.
We expect that Theorem~\ref{Th4} holds for any primitive\footnote{A substitution $\sigma$ on an alphabet $A$ is called primitive if there exists an integer $n$ such that for all $a,b\in A$ we have that $\sigma^n(a)$ contains the letter $b$.} Zeckendorf-automatic sequence $f$.
\begin{problem}\label{pb_zeckendorf_automatic_PNT}
Let $f$ be a primitive Zeckendorf-automatic sequence.
Prove that for all $\alpha$ there exists some $c_{\alpha}$ such that
\begin{align*}
	\lim_{x \to \infty} \frac{\#\{p \leq x: f(p) = \alpha\}}{\pi(x)} = c_{\alpha}.
\end{align*}
\end{problem}
In order to achieve this result one would need to overcome some major obstacles,
such as proving a Gowers norm estimate for all primitive Zeckendorf-automatic sequences.
A corresponding result for (base-$q$) automatic sequences was proved in~\cite{Byszewski2020}.
Furthermore, the level of distribution is not as straightforward as expected.
Among other things, the carry propagation lemma (Lemma~\ref{lem_z_carry_lemma}) has to be weakened;
also, we would have to generalize our approach to matrix-valued sequences, in other words, import the machinery developed in~\cite{Muellner2017}.

\chapter*{Acknowledgements}
We thank the referee for her/his thorough review and very helpful and significant comments, which helped improving the paper notably.
We also wish to thank Niels Langeveld for providing assistance in reading the Dutch article by Lekkerkerker.

The authors are grateful to the Institut de Math\'ematiques de Luminy in Marseilles, France, where part of the research work for this article was carried out.
We always found optimal working conditions there, including positive atmosphere and agreeable climate.
Christian Mauduit worked at this institute;
it was the work of him and Jo\"el Rivat on digital problems that motivated us to study the problems considered in this paper.
Without their work it would not have been possible for us to prove our main theorems.



\backmatter
\bibliographystyle{amsalpha}

\bibliography{biblio}

\printindex

\end{document}